\documentclass[11pt,oneside]{book}  
\usepackage{amsmath,amsfonts,amssymb,amsthm}
\usepackage{algorithm,algpseudocode}
\usepackage{nicefrac}

\usepackage{mathtools}
\usepackage{mathrsfs}
\usepackage{enumitem}
\usepackage{hyperref}
\usepackage{cleveref}
\usepackage{graphicx}
\usepackage{tikz}
\usepackage{xcolor}

\usepackage{lmodern}
\usepackage[T1]{fontenc}

\theoremstyle{plain}
\newtheorem{theorem}{Theorem}[chapter]
\newtheorem{proposition}[theorem]{Proposition}
\newtheorem{lemma}[theorem]{Lemma}

\newtheorem{claim}[theorem]{Claim}
\newtheorem{corollary}[theorem]{Corollary}

\theoremstyle{definition}
\newtheorem{definition}[theorem]{Definition}
\newtheorem{example}[theorem]{Example}
\newtheorem{exercise}[theorem]{Exercise}

\theoremstyle{remark}
\newtheorem{remark}[theorem]{Remark}


\newcommand{\ie}{\emph{i.e.}}
\newcommand{\eg}{\emph{e.g.}}

\newtheorem{problem}{Problem}

\newcommand{\eqdef}{:=}

\newcommand{\df}[1]{\emph{#1}} 

\renewcommand{\Pr}{\mathbb{P}}
\newcommand{\E}{\mathbb{E}}

\newcommand{\R}{\mathbb{R}}
\newcommand{\N}{\mathbb{N}}

\newcommand{\eps}{\varepsilon}

\newcommand{\uni}{\mathcal{U}} 
\newcommand{\elem}{u} 
\newcommand{\set}{S} 

\newcommand{\mat}{M} 
\newcommand{\act}{A} 

\newcommand{\col}{x} 

\newcommand{\ones}{\vec{1}} 

\renewcommand{\SS}{\mathcal{S}}
\newcommand{\DD}{\mathcal{D}} 
\newcommand{\TT}{\mathcal{T}}
\newcommand{\RR}{\mathcal{R}} 
\newcommand{\rbox}{R} 
\renewcommand{\AA}{\mathcal{A}} 
\newcommand{\II}{\mathcal{I}} 
\newcommand{\CC}{\mathcal{C}} 

\newcommand{\grv}{g} 

\newcommand{\parcol}{h}

\newcommand{\indic}{1} 

\DeclareMathOperator{\ssdeg}{deg}

\DeclareMathOperator{\disc}{disc}
\DeclareMathOperator{\herdisc}{herdisc}
\DeclareMathOperator{\vdisc}{vdisc}
\DeclareMathOperator{\detlb}{detLB}
\DeclareMathOperator{\vollb}{volLB}
\DeclareMathOperator{\vb}{vb}
\DeclareMathOperator{\pdisc}{pdisc}
\DeclareMathOperator{\st}{st}
\DeclareMathOperator{\sser}{ss}
\DeclareMathOperator{\Tr}{tr}
\DeclareMathOperator{\tr}{tr}
\DeclareMathOperator{\rank}{rank}
\DeclareMathOperator{\sign}{sign}
\DeclareMathOperator{\vol}{vol}
\newcommand{\poly}{\ensuremath{\mathsf{poly}}}
\newcommand{\ip}[1]{\langle #1\rangle }
\newcommand{\cut}[1]{}

\newcommand{\diag}{\mathsf{diag}}

\usepackage{color}

\title{Discrepancy Theory: An Algorithmic and Geometric Perspective}

\author{
Nikhil Bansal\footnote{Supported in part by the NWO VICI award 639.023.812 and the NSF awards CCF-2327011 and CCF-2504995.} \and
Aleksandar Nikolov\footnote{Supported in part by an NSERC Discovery Grant (RGPIN-2021-03206), and the Canada Research Chairs program (CRC-2020-00004).}
}

\begin{document}


\frontmatter  
\date{}
\maketitle

\chapter*{Abstract}
\emph{Combinatorial discrepancy theory} is a subject with roots in combinatorics, geometry, and number theory, and with numerous applications to mathematics and computer science. 
At its core, discrepancy theory is  about dividing a collection of objects into two parts that are as balanced as possible.
For example, given a collection of subsets of a finite universe, we may wish to color the elements with two colors so that each set is approximately evenly split.
Other problems in discrepancy are more geometric in flavor, and ask for example, to assign signs to a collection of vectors, so that the sum of the signed vectors is as small as possible.
Classical results, such as the Beck–Fiala theorem and Spencer’s “six deviations” result, show that it is often possible to attain remarkably small discrepancy --- often far smaller than what naive random colorings achieve.

In recent years, discrepancy theory has undergone a transformation, driven by new algorithmic techniques and a rich interplay between probability, optimization, and convex geometry.
These developments have led not only to new constructive proofs of foundational theorems, but also to several new results and research directions.
This monograph aims to provide an accessible and unified introduction to these modern developments, with a focus on the core algorithmic and convex geometric ideas that have driven them. For several results, we provide
new simpler analyses, while highlighting the intuition behind the proofs.

We begin by surveying classical techniques for bounding discrepancy, including iterated rounding, partial colorings, and Banaszczyk’s method, emphasizing the underlying tools from high-dimensional geometry. As most of these methods are non-constructive, they do not yield efficient algorithms for computing the colorings that achieve the discrepancy upper  bounds.
We then describe several algorithmic techniques
that match and even improve upon the classical methods.
We also discuss related topics such as the computational complexity of approximating discrepancy, and discrepancy minimization in the online model.

\tableofcontents

\mainmatter

\allowdisplaybreaks
\chapter{Introduction}
\label{ch1-intro} 

This is a book about \emph{combinatorial discrepancy theory} --- a
subject with roots in combinatorics, geometry, and number theory, and
numerous applications to other areas of mathematics and computer
science. In many of the applications, discrepancy provides a way to
``simplify'' a ``complex'' object. In this introductory chapter, we define the basic concepts in discrepancy theory, and illustrate the ways it can be applied with several examples.

\section{Discrepancy of Sets and Sparsification}

Consider the following balancing problem. We have a collection of subsets $\SS$ of some finite universe $\uni$, which we will call a \df{set system}. We want to color the universe $\uni$ with two colors, say red and blue, so that the red elements $R$, and the blue elements $B$ ``look similar'' to all the sets in $\SS$. I.e., each set in $\SS$ should have, as much as possible, a similar number of elements from $R$ and from $B$. Then the \df{discrepancy} of the coloring is 
\[
\max_{S \in \SS} \Big|\,|S\cap R| - |S \cap B|\,\Big|,
\]
and the discrepancy of the set system $\SS$ is the minimum discrepancy over all colorings, i.e., all partitions of $\uni$ into disjoint sets $R$ and $B$. It is convenient to encode a coloring as a function $\col:\uni \to \{-1,+1\}$,
where $\col(u)=+1$ corresponds to placing $u$ in $R$, and $\col(u)=-1$ corresponds to placing it in $B$. We further define $\col(S) \eqdef \sum_{\elem \in S}\col(\elem)$, which is just $|S\cap R| - |S \cap B|$.
We can now write the discrepancy of a coloring $\col$ and the discrepancy of $\SS$, respectively, as
\begin{align*}
    \disc(\SS,\col) &\eqdef \max_{S \in \SS}|\col(S)|,\\
    \disc(\SS) &\eqdef \min_{\col:\uni\to\{-1,+1\}} \disc(\SS,\col).
\end{align*}
Using Hoeffding's inequality, and the union bound, it is easy to see that, with probability at least $\frac12$, a uniformly random coloring $\col:\uni \to \{-1,+1\}$ achieves $\disc(\SS,\col) \lesssim \sqrt{|\uni|\log(2|\SS|)}$.\footnote{Here and in the
  rest of the book we use the notation $A\lesssim B$ to mean that
  $A \le CB$ for an absolute constant $C > 0$, and $A \gtrsim B$ to
  mean $B \lesssim A$.}  A major theme of this book, and of discrepancy theory in general, is the study of methods to construct colorings that achieve smaller discrepancy than random colorings. 

\paragraph{Application: Epsilon Approximations.} 
Most applications of the discrepancy of set systems rely on the observation that, because each set $S$ in $\SS$ is approximately equally split between the two color classes, restricting $S$ to one of the color classes reduces its size by approximately half. Thus, restricting all sets to the smaller color class gives another set system on a smaller universe, in which the relative sizes of the sets are preserved. 

Let us illustrate this idea with a geometric application. Let us take $P$ to be a finite point set in the plane. Let $\RR_2$ be the family of all axis-aligned rectangles, \ie, all sets of the
form $R = [a,b]\times [c,d]$. A subset $Q$ of $P$ is an $\varepsilon$-approximation of $P$ with respect to $\RR_2$, if for every axis-aligned rectangle $R$ we have that 
\[
\left|\frac{|Q \cap R|}{|Q|} - \frac{|P \cap R|}{|P|}\right| \le \varepsilon.
\]
I.e., $Q$ must be such that the fraction of points in it that land in any axis-aligned rectangle is equal, up to an additive error $\varepsilon$, to the fraction of points of $P$ that land in the same rectangle. Typically, we are interested in $\varepsilon$-approximations of small size. Such an $\varepsilon$-approximation is a concise approximation of the set system $\RR_2|_P \eqdef \{R\cap P: R \in \RR_2\}$ of subsets of $P$ induced by axis-aligned rectangles. Naturally, there is nothing special here about axis-aligned rectangles, and $\varepsilon$-approximations can be defined analogously for other families of geometric sets, e.g., halfspaces, or Euclidean balls, and in dimensions higher than 2. They are a useful concept in computational geometry. For example, a small $\varepsilon$-approximation immediately gives a small-space data structure for approximate range counting: the data structure is the $\varepsilon$-approximation $Q$, and given a query rectangle $R$, the number of points of $P$ in $R$ can be approximated by $|Q \cap R|\cdot |P|/|Q|$, which gives an additive approximation of $\varepsilon |P|$.



Discrepancy gives us a way to construct an $\varepsilon$-approximation $Q$ of half the size of $P$. Let us denote $n\eqdef |P|$. Suppose that, for the set system $\SS \eqdef \RR_2|_P$ induced by axis-aligned rectangles on $P$, the coloring $\col$ achieves $D \eqdef \disc(\SS)$, and let $Q$ be the points in $P$ colored $+1$. The discrepancy constraint implies that for every rectangle $R \in \RR_2$,
\[
\Big||Q\cap R| - |(P\setminus Q) \cap R|\Big| \le D.
\]
On the other hand, $|P\cap R| = |Q\cap R| + |(P\setminus Q) \cap R|$, so
\[
\Big||Q\cap R| - \frac12 |P \cap R|\Big|
= 
\frac12 \Big||Q\cap R| - |(P\setminus Q) \cap R|\Big| \le \frac{D}{2}.
\]
If the size of $Q$ were exactly $n/2$, then dividing both sides by $|Q|$ would give us that $Q$ is a $D/n$-approximation to $P$ with respect to axis-aligned rectangles. But now we can observe that some axis-aligned rectangle contains all the points in $P$, so the discrepancy bound also implies that $\left||Q| -  n/2\right| \le D/2$. Then, we can make sure that $|Q| = n/2$ by adding or removing at most $D/2$ points from it, and this modified $Q$ is a $2D/n$-approximation. 

Let $\disc(\RR_2,n)$ be the maximum of $\disc(\RR_2|_P)$ over all $n$-point sets $P$. The argument above gives us the following lemma.

\begin{lemma}\label{lm:ch1-disc-epsapprox}
    For any $n$-point set $P\subseteq \R^2$, there is an $\varepsilon$-approximation $Q$ of $P$ with respect to axis-aligned rectangles such that $|Q| = n/2$ and $\varepsilon\le 2\disc(\RR_2,n)/n$.
\end{lemma}

We will see later in Chapter~\ref{ch:gamma2} that $\disc(\RR_2, n) \lesssim \log^{1.5}(2n)$.
Thus the $\varepsilon$ in Lemma~\ref{lm:ch1-disc-epsapprox} is very close to $0$ when $n$ is large. Typically, however, we can settle for a constant $\varepsilon$, but would like $|Q|$ to only depend on $\varepsilon$. We can achieve this by applying Lemma~\ref{lm:ch1-disc-epsapprox} repeatedly, together with the following simple composition lemma, whose proof is left as an exercise.

\begin{lemma}\label{lm:ch1-epsapprox-composition}
    If $Q_1$ is an $\varepsilon_1$-approximation of $P$, and also $Q_2$ is an $\varepsilon_2$-approximation of $Q_1$, then $Q_2$ is an $(\varepsilon_1 + \varepsilon_2)$-approximation of $P$.
\end{lemma}

We can now apply Lemma~\ref{lm:ch1-disc-epsapprox} repeatedly, each time halving the size of $Q$, and Lemma~\ref{lm:ch1-epsapprox-composition} tells us that the approximations add up. If we stop at a set $Q$ of $s$ points, we have that $Q$ is an $\varepsilon$-approximation for 
\[
\varepsilon = \frac{2 \disc(\RR_2, n)}{n} + \frac{4 \disc(\RR_2, n/2)}{n} + \ldots + \frac{\disc(\RR_2, 2s)}{s}\lesssim \frac{\disc(\RR_2, n)}{s},
\]
where the last bound follows since $\disc(\RR_2, n)$ grows polylogarithmically in $n$. Plugging in $\disc(\RR_2, n) \lesssim \log^{1.5}(2n)$ and solving for $s$, we get the following theorem.

\begin{theorem}\label{thm:ch1-rectangles}
     For any $\varepsilon \in (0,1)$, and any $n$-point set $P\subseteq \R^2$, there is an $\varepsilon$-approximation $Q$ of $P$ with respect to axis-aligned rectangles such that $|Q| \lesssim (\log^{1.5}(1/\varepsilon))/\varepsilon$.
\end{theorem}

Note that, as long as we can efficiently find a coloring $\col$ of any $n$ points $P$ achieving $\disc(\RR_2|_P,\col) \lesssim \log^{1.5}(2n)$, we can also perform the construction in Theorem~\ref{thm:ch1-rectangles} efficiently. Unfortunately, no such efficient algorithm is known for this bound, but the slightly weaker bound $\disc(\RR_2|_P,\col) \lesssim \log^{2}(2n)$ can be achieved in polynomial time using the methods of Chapter~\ref{ch:algo-komlos-bana}.

The size bound in Theorem~\ref{thm:ch1-rectangles} would be immediately improved by improving the bound on the discrepancy $\disc(\RR_2,n)$. This motivates the following problem, to which we return several times in this book.

\begin{problem}\label{prob:tusnady}
  Determine the maximum discrepancy of a set system induced by the family $\RR_d$ of
  \(d\)-dimensional axis-aligned boxes (cross
products of $d$ intervals) on an \(n\)-point set, \ie,
  give tight bounds on 
  \(
    \disc(\RR_d, n) \eqdef \max\{\disc(\RR_d|_P): P \subseteq \R^d, |P| = n\}.
  \)
\end{problem}
Later in the book we will see upper and lower bounds on $\disc(\RR_d, n)$ that are tight within a $O(\sqrt{\log n})$ factor.

Bounds on $\varepsilon$-approximations with respect to other families of geometric sets similar to Theorem~\ref{thm:ch1-rectangles} can be proved using the same argument and bounds on the discrepancy of the appropriate set system. 

\section{Discrepancy of Matrices and Rounding}

The discrepancy of a set system can be written in a linear algebraic way, which also suggests a natural way to define discrepancy of a matrix. Let $\SS$ be a set system on a universe $\uni$, and let us take
$\mat$ to be the incidence matrix of $\SS$,
\ie, the matrix with rows indexed by $\SS$, columns indexed by $\uni$, and entries equal to 
\[
  \mat(S,\elem) \eqdef
  \begin{cases}
    1 & \text{if } \elem \in S\\
    0 &  \text{if } \elem \not \in S
  \end{cases}
\]
for any $S \in \SS$ and $\elem \in \uni$.\footnote{We identify $n$-dimensional vectors $x$ with functions $x:[n] \to \R$, where $x(i)$ is the $i$-th coordinate of $x$, and $m\times n$ matrices $M$ with functions $M:[m]\times [n] \to \R$, where $M(i,j)$ is the entry in row $i$ column $j$ of $M$.} Then it is easy to check that 
\begin{align*}
\disc(\SS,\col) &= \max_{S \in \SS} |(\mat\col)_S| = \|\mat \col\|_{\infty},\\
\disc(\SS) &= \min_{x \in \{-1,1\}^\uni} \|\mat\col\|_\infty.
\end{align*}
Inspired by these observations, we can then define the discrepancy of any $m\times n$ matrix $\mat$ and a coloring $\col \in \{-1,+1\}^n$ as $\disc(\mat,\col) \eqdef \|\mat\col\|_\infty$, and the discrepancy of the matrix $\mat$ as $\disc(\mat) \eqdef \min_{\col \in \{-1,+1\}^n} \|\mat\col\|_\infty$.

\paragraph{Application: Rounding}
Suppose that $y \in [0,1]^n$ satisfies the constraints $\mat y = b$ for some $m\times n$ matrix $\mat$ and $m$-dimensional vector $b$. For example, $y$ can be the solution to a linear programming relaxation of a combinatorial optimization problem. Matrix discrepancy allows us to round $y$ to a binary vector $z \in \{0,1\}^n$ so that the constraints are approximately satisfied. 

The simplest case is rounding a vector $y$ all of whose coordinates are equal to $\frac12$, i.e., $y = \frac12 \ones$, where $\ones$ is the all-ones vector. Then, if $\col \in \{-1,+1\}^n$ achieves $\|\mat \col\|_\infty = \disc(\mat)$, for the binary vector $z \eqdef  y + \frac12\col$ we have
\[
\left\|\mat z - b\right\|_\infty = \|\mat(z-y)\|_\infty
= 
\frac12 \|\mat \col\|_\infty = \frac{\disc(\mat)}{2}.
\]
This is, in fact, essentially the same observation as the one used in Lemma~\ref{lm:ch1-disc-epsapprox}. If $y$ is instead half-integral, i.e., its coordinates are in $\left\{0,\frac12, 1\right\}$, then we can apply the same construction, but only to the coordinates equal to $\frac12$. The resulting error in our rounding will be half the discrepancy of the submatrix of $\mat$ corresponding to these coordinates. In general, to round arbitrary vectors $y$ using this method, we need a bound on the maximum discrepancy of any submatrix of $\mat$. This motivates the definition of \df{hereditary discrepancy}:
\[
\herdisc(\mat) \eqdef \max_{J \subseteq [n]}\disc(\mat(*,J)),
\]
where $\mat(*,J)$ is the submatrix of $\mat$ consisting of the columns indexed by $J$. By applying the idea above to the bit representation of $y$, we get the following rounding result. 

\begin{theorem}\label{thm:ch1-transference}
  For any $m\times n$ matrix $\mat$,
  and any $y \in [0,1]^n$,
  there exists some $z \in \{0,1\}^n$ such that 
  \[
    \left\|\mat (z-y)\right\|_\infty
    \le
    \herdisc(\mat).
  \]
\end{theorem}
\begin{proof}
  Let us assume that there is some finite $k$ such that, for each $j \in [n]$, $y(j) \in 2^{-k} \mathbb{Z}$, \ie, that each coordinate $y(j)$
  can be written using at most $k$ bits after the radix point. By standard compactness arguments, it is
  enough to show that the theorem holds under this assumption. Suppose that, for every $j$,
  $b_k(j)$ is the $k$-th bit after the radix point of $y(j)$, and $y'(j)$ is the real number given by the first $k-1$ bits after the radix point. Then $y(j) = y'(j) + 2^{-k} b_k(j)$. Let $J = \{j: b_k(j) = 1\}$, and find a
  coloring $\col\in \{-1,+1\}^J$ which achieves
  \[
  \|\mat(*,J) \col\|_\infty = \disc(\mat(*,J)) \le \herdisc(\mat).
  \]
  We set, for all $j \in J$,
  \[
    y''(j) \eqdef  y(j) + \col(j)2^{-k} = y'(j) + 2^{-k} + \col(j)2^{-k},
  \]
  and $y''(j) \eqdef y'(j) = y(j)$ for $j \not \in J$. I.e., for every coordinate $y(j)$ whose $k$-th bit after the radix point is $1$, we round that bit up if $x(j) = +1$, or down otherwise.
  Clearly each 
  \(
  y''(j)
  \)
  can be written with at most $k-1$ bits after the radix point,
  and
  \[
  \|\mat (y''-y)\|_\infty = 2^{-k}\|\mat(*,J)\col\|_\infty\le 2^{-k}\herdisc(\mat).
  \]
  We then replace $y$ with $y''$ and $k$ with $k-1$, and inductively apply the same bit-rounding process until $k = 0$, at which point we have our vector $z \in \{0,1\}^n$. Summing the errors from rounding each bit, we get 
  \[
  \|\mat (z-y)\|_\infty \le (2^{-1} + 2^{-2} + \ldots + 2^{-k})\herdisc(\mat) \le \herdisc(\mat). \qedhere
  \]
\end{proof}

Theorem~\ref{thm:ch1-transference} can be used to derive a result similar to Theorem~\ref{thm:ch1-rectangles}. We let $\mat$ be the incidence matrix of $\RR_2|_P$, and $y = \frac{s}{n}\ones$. Applying Theorem~\ref{thm:ch1-transference} and letting $Q\eqdef \{p \in P:z(p) = 1\}$, we have that, for any rectangle $R \in \RR_2$,
\[
\left||Q\cap R| - \frac{n}{s}|P\cap R|\right| \le \herdisc(\RR_2|_P) \le \disc(\RR_2,n) \lesssim \log^{1.5}(2n).
\]
Moreover, since some rectangle contains all of $P$, we also have $||Q| - s| \lesssim \log^{1.5}(2n)$, so up to a constant factor increase in the bound above, we can assume that $|Q| = s$. Thus, dividing the inequality above by $s$ gives us that $Q$ is an $(\log^{1.5}n)/s$-approximation. Unfortunately, this is not quite enough to get the bound in Theorem~\ref{thm:ch1-rectangles} or any bound that depends entirely on $\varepsilon$. We can address this with a slight strengthening of Theorem~\ref{thm:ch1-transference}. Let us define a parametrized version of hereditary discrepancy by 
\[
\herdisc(\mat,s) \eqdef \max_{J\subseteq [n]: |J| \le s}\disc(\mat(*,J)).
\]
We then have the stronger rounding result stated next, which follows by essentially the same proof.
\begin{theorem}\label{thm:ch1-transference-sizes}
For any $m\times n$ matrix $\mat$, and any $y \in [0,1]^n$, there exists some $z \in \{0,1\}^n$ such that 
\[
    \left\|\mat(z-y)\right\|_\infty
    \le
    \sum_{k = 1}^\infty 2^{-k}\herdisc\left(\mat, 2^{k}\sum_{j = 1}^n y_j\right).
\]
\end{theorem}
\begin{proof}
    We construct $z$ in the same way as in Theorem~\ref{thm:ch1-transference}, and simply observe that the set $J = \{j: b_k(j) = 1\}$ of coordinates of $y$ with $k$-th bit after the radix point equal to $1$ is of size at most $2^k\sum_{j = 1}^n y_j$. 
\end{proof}
Theorem~\ref{thm:ch1-transference-sizes} now allows us to recover Theorem~\ref{thm:ch1-rectangles} since $\herdisc(\RR_2|_P,s) \le \disc(\RR_2,s) \lesssim \log^{1.5}(2s)$. 

Moreover, as with Theorem~\ref{thm:ch1-rectangles}, the roundings in Theorems~\ref{thm:ch1-transference}~and~\ref{thm:ch1-transference-sizes} can found efficiently as long as we can find colorings with discrepancy bounded by the hereditary discrepancy. We return to this question in Chapter~\ref{ch:partial-alg}. 

\section{Vector Balancing and Approximate Carath\'eodory}

While the infinity norm is natural for the discrepancy of set systems, and for measuring how much a rounded solution to a system of linear equations violates the equations, in other applications of matrix discrepancy we may need to consider other norms. The definitions we made above can be generalized easily to such a setting. Suppose $\|\cdot\|_X$ is some norm on $\R^m$, and let $X$ be the corresponding normed space. Then, the discrepancy of a coloring $\col \in \{-1,+1\}^n$ with respect to an $m\times n$ matrix $\mat$ and $X$ is $\disc_X(\mat,\col) \eqdef \|\mat \col\|_X$. The discrepancy of $\mat$ with respect to $X$ is $\disc_X(\mat) \eqdef \min_{\col \in \{-1,+1\}^n} \|\mat\col\|_X$. Hereditary discrepancy is again defined by taking a maximum over submatrices:
\begin{align*}
    \herdisc_X(\mat,s) &\eqdef \max_{J\subseteq [n]: |J| \le s}\disc_X(\mat(*,J)),\\
    \herdisc_X(\mat) &\eqdef \max_{J\subseteq [n]}\disc_X(\mat(*,J)).
\end{align*}
It is easy to check that, since the proofs only use the triangle inequality and homogeneity properties of the $\ell_\infty$ norm, Theorems~\ref{thm:ch1-transference}~and~\ref{thm:ch1-transference-sizes} hold under any norm. Next we consider the $\ell_2$ norm, and give another geometric application of these rounding theorems.

\paragraph{Approximate Carath\'eodory Theorem.} Let $u_0, u_1, \ldots,
u_n$ be vectors in  $\R^m$, and let $v$ be a point in their convex hull. Carath\'eodory's theorem shows that there exists a subset $S
\subseteq \{0\} \cup [n]$ of cardinality at most $m+1$ so that $v$ is
also in the convex hull of $\{u_i: i \in S\}$. This is generally
tight, for example when $u_0 = 0$, $u_i$ is the $i$-th standard basis
vector of $\R^m$, and $v$ is the vector with every coordinate equal to
$\frac{1}{m+1}$. However, if we only ask $v$ to be \emph{close} to the
convex hull of $\{u_i: i \in S\}$, then perhaps we can hope to improve the
bound on the size of $S$, and, moreover, we can hope for bounds
that are independent of the dimension $m$.

Once we allow for approximation, we need to fix a way to measure distances, and a scale. We will, for now, pick the Euclidean, i.e., $\ell_2$ distance. As for fixing a scale, a
natural way to do so is to ask that the diameter
\(
\max_{i = 0}^n\max_{j = 0}^n\|u_i - u_j\|_2
\)
be bounded by $1$. 
Since everything in our set up is invariant under translation, we can
assume that $u_0 = 0$, and, by the diameter condition, $\|u_i\|_2 \le
1$ for all $i \in[n]$.

We will use discrepancy to approximate $v$ by the average of $s$ vectors, for a parameter $s$, so that the approximation error scales like $\min\left\{\frac{\sqrt{m}}{s}, \frac{1}{\sqrt{s}}\right\}$.
To this end, let us explicitly write $v$ as a convex combination of $u_0, \ldots u_n$ as
\begin{align*}
    v = \alpha_0 u_0 + \alpha_1 u_1 + \ldots + \alpha_n u_n
    = \alpha_1 u_1 + \ldots + \alpha_n u_n
\end{align*}
for some $\alpha_1, \ldots, \alpha_n \ge 0$ such that $\sum_{i=1}^n \alpha_i \le 1$. By Carath\'eodory's theorem we can assume that $n \le m$. Inspired by the proof of Theorem~\ref{thm:ch1-rectangles} via Theorem~\ref{thm:ch1-transference}, let us define $y \in \R^n$ by $y(i) \eqdef s \alpha_i$. Take $z \in \{0,1\}^n$ to be the rounding of $y$ guaranteed by Theorem~\ref{thm:ch1-transference} when used with the $\ell_2$ norm and the matrix $\mat$ with columns $u_1, \ldots, u_n$. Let $S \eqdef \{i: z(i) = 1\}$.\footnote{Technically, we use Theorem~\ref{thm:ch1-transference} to round the vector with coordinates $y(i) - \lfloor y(i)\rfloor$, and we first add $\lfloor y(i) \rfloor$ copies of $i$ to $S$ before adding also those $i$ for which $z(i) = 1$.} We then have that
\begin{equation}\label{eq:ch1-apxcar-rounding}
   \left\|sv - \sum_{i\in S} u_i\right\|_2 \le \herdisc_{\ell_2}(\mat). 
\end{equation}
Moreover,  for a uniformly random $\col \in \{-1,+1\}^n$, and a matrix $\mat$ with $n$ columns each of $\ell_2$ norm at most $1$, $\E \|\mat\col\|_2^2 \le n$. Therefore, $\disc_{\ell_2}(\mat) \le \sqrt{n}$, and, since this reasoning applies to any submatrix as well, $\herdisc_{\ell_2}(\mat) \le \sqrt{n}$. We have 
$
\left\|sv - \sum_{i\in S} u_i\right\|_2 \le \sqrt{n}\le \sqrt{m}.$

If we could guarantee that $|S| = s$, then dividing both sides of the inequality by $s$ would guarantee that 
\[
\left\|v - \frac{1}{|S|}\sum_{i\in S} u_i\right\|_2 \le \frac{\sqrt{m}}{s}.
\]
Unfortunately, we cannot automatically deduce that $|S|$ is close to $s$ from~\eqref{eq:ch1-apxcar-rounding}. Nevertheless, there is a simple workaround, given in the following corollary of Theorems~\ref{thm:ch1-transference}~and~\ref{thm:ch1-transference-sizes}.
\begin{corollary}\label{cor:ch1-transference}
    In both Theorem~\ref{thm:ch1-transference} and Theorem~\ref{thm:ch1-transference-sizes}, we can assume that $\sum_{j=1}^n z(j) \le \sum_{j=1}^n y(j)$.
\end{corollary}
\begin{proof}
    Follows from the same proof we gave for the theorems, after noticing that we can always take the coloring $\col \in \{-1,+1\}^J$ to be such that $\sum_{j \in J}\col(j) \le 0$ since if this is not the case, we can just flip the colors without changing the discrepancy.
\end{proof}

We can now show the following result.
\begin{theorem}\label{thm:ch1-apxcarath-L2}
    Suppose that $u_0, \ldots, u_n$ are vectors in $\R^m$, and their diameter is at most $1$. For any vector $v$ in their convex hull, and any positive integer $s$, there exists a (multi-)set $S \subseteq \{0\} \cup [n]$ of size $s$ such that 
    \[
    \left\|v - \frac{1}{s} \sum_{i \in S} u_i\right\|_2 \lesssim \min\left\{\frac{\sqrt{m}}{s}, \frac{1}{\sqrt{s}}\right\}.
    \]
\end{theorem}
\begin{proof}
    As above, we may assume $n \le m$, shift the vectors so that $u_0 = 0$, and write $v = \alpha_1 u_1 + \ldots + \alpha_n u_n$ for non-negative $\alpha_1, \ldots, \alpha_n$ that sum to at most $1$. Define $y \in \R^n$ by $y(i) \eqdef s \alpha_i$. If some $y(i)$ is larger than $1$, then we add $\lfloor y(i)\rfloor$ copies of $i$ to $S$, and replace $y(i)$ by $y(i) - \lfloor y(i)\rfloor$. We can now apply Corollary~\ref{cor:ch1-transference} to $y$ to get a vector $z \in \{0,1\}^n$, and add all $i$ such that $z(i) = 1$ to $S$. By the guarantee of the corollary, we have that $|S| \le s\sum_{i=1}^n\alpha_i = s$. To make sure $S$ has size exactly $s$, we add sufficiently many copies of $0$ to $S$. 
    
    Recall the observation above that $\disc_{\ell_2}(\mat,\col) \le \sqrt{n}$ for an $n$-column matrix $\mat$ with columns whose $\ell_2$ norms are at most $1$. Therefore $\herdisc(\mat,t) \le \sqrt{t}$ for such a matrix, and the theorem follows from the inequality \eqref{eq:ch1-apxcar-rounding} and the more precise bound
    \[
    \left\|sv - \sum_{i\in S} u_i\right\|_2 \le \sum_{k=1}^\infty 2^{-k}\herdisc_{\ell_2}(\mat,2^ks)
    \lesssim \sqrt{s}\sum_{k=1}^\infty 2^{-k/2}\lesssim \sqrt{s}.
    \qedhere \]
\end{proof}

The key quantity in the proof of Theorem~\ref{thm:ch1-apxcarath-L2} is a vector balancing constant. Suppose $K$ is some subset of $\R^m$, and $\|\cdot\|_X$ is, as above, a norm on $\R^m$. The vector balancing constants for $K$ and $X$ are defined by
\begin{align*}
  \vb(K,X,n) &\eqdef \sup\{\disc_X((u_i)_{i=1}^n): u_1, \ldots, u_n  \in K\},\\
  \vb(K,X) &\eqdef \sup_n \vb(K,X,n). 
\end{align*}
In our proof of Theorem~\ref{thm:ch1-apxcarath-L2} we used the simple observation that $\vb(B_{\ell_2},\ell_2,n) \le\sqrt{n}$, where $B_{\ell_2} \eqdef \{u: \|u\|_2 \le 1\}$ is the unit ball. The same proof shows that if the vectors $u_0, \ldots, u_n$ have diameter bounded by $1$ in the norm $\|\cdot\|_X$, then we can guarantee
\[
    \left\|v - \frac{1}{|S|} \sum_{i \in S} u_i\right\|_X \le
    \frac{1}{s} \sum_{k = 1}^\infty 2^{-k}\vb(B_X, X, 2^{k}s),
\]
where $B_X \eqdef \{u \in \R^m: \|u\|_X \le 1\}$.

There is also no reason why
the diameter of $u_0, \ldots u_n$ has to be bounded in the same norm in
which we measure the approximation. For example, in the next chapter,
we will see that $\vb(B_{\ell_1},\ell_\infty) \le 2$ (this is a
result of Beck and Fiala~\cite{BF81}). Then, using analogous reasoning to the above, we have that whenever $u_0, \ldots,
u_n$ have $\ell_1$ diameter at most $1$, for any point $v$ in their
convex hull we can find an average $v'$ of $s$ vectors from
$u_0, \ldots, u_n$ so that $\|v - v'\|_\infty \le
\frac{2}{s}$. Remarkably, it is an open problem whether this 
$\|v - v'\|_\infty \lesssim \frac{1}{s}$
error bound can be achieved under the weaker assumption that $u_0, \ldots,
u_n$ have Euclidean diameter at most $1$. Such an improvement would
follow from a positive solution to the following problem, due to
Koml\'os. 
\begin{problem}\label{prob:komlos}
  Prove or disprove the bound $\vb(B_{\ell_2^m},\ell_\infty^m)
  \lesssim 1$.
\end{problem}

Later in the book, we will see partial progress towards
Problem~\ref{prob:komlos}, and, in particular, a bound of
$O(\sqrt{\log m})$ due to Banaszczyk~\cite{Bana98}. We will see both
Banaszczyk's result in Chapters \ref{ch:bana-intro} and \ref{ch:bana-proof}, and also two different proofs in Chapter \ref{ch:algo-komlos-bana} and \ref{ch:gram-schmidt} that gives an efficient
algorithm to construct the coloring achieving this bound.\footnote{Very recently, Bansal and Jiang \cite{bansaljiang2025} have obtained an $O(\log^{1/4} m)$ discrepancy bound for the Koml\'{o}s problem, building on the techniques we will see in Chapter \ref{ch:algo-komlos-bana}. However, we will not discuss their result in this book.}

\section{Prefix Discrepancy and the Steinitz Problem}
\label{sec:ch1-steinitz}

Suppose an $m\times n$ matrix $\mat$ has columns $u_1, \ldots, u_n \in \R^m$. A useful strengthening of matrix discrepancy is prefix discrepancy, which requires that the same coloring simultaneously balances each prefix of columns $u_1, \ldots, u_i$. Formally, we define the \df{prefix discrepancy} of a matrix $\mat$ with columns $u_1, \ldots, u_n$, and a coloring $\col \in \{-1,+1\}^n$   as 
\[
\pdisc(\mat, \col) \eqdef
\max_{j=1}^n \left\|\sum_{i = 1}^j\col(i)u_i \right\|_\infty.
\]
The prefix discrepancy of $\mat$ is the minimum of $\pdisc(\mat, \col)$ over all colorings $\col \in \{-1,+1\}^n$. Similarly to matrix discrepancy, we can also define a variant $\pdisc_X(\mat)$ in which the $\ell_\infty$ norm is replaced by some other norm $\|\cdot\|_X$ on $\R^m$.

This is a natural generalization of discrepancy when we decide on the colors $\col(1), \ldots, \col(n)$ one at a time, e.g., because the column vectors $u_1, \ldots, u_n$ are specified online. We will also see that prefix discrepancy can be used to prove bounds on standard discrepancy in problems with some implicit ordering: this will be the approach we take for Problem~\ref{prob:tusnady} in Chapter~\ref{ch:gamma2}. Here we will see another application to a classical geometric problem. 

\paragraph{The Steinitz Problem.} Suppose we have vectors $u_1, \ldots, u_n$ in $\R^m$, so that each vector is ``short'',
\ie, $u_i \in [-1,+1]^m$ for all $i$. If
$u_1 + \ldots + u_n = 0$, then the partial sums $u_1 + \ldots + u_j$
trace a path in space from the origin back to itself. At intermediate
steps, the path can be arbitrarily far from $0$, but it always takes
short steps. Can we re-arrange the vectors, so that the path always
stays close to $0$, and, in particular, the norm
\( \left\|\sum_{i = 1}^j u_i\right\|_\infty\) of each partial sum is
always bounded by a function depending only on the dimension $m$? See
Figure~\ref{fig:ch1-steinitz} for an illustration.

\begin{figure}[htp]
  \centering
  \includegraphics{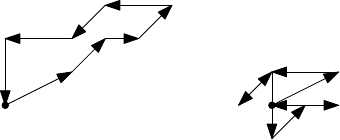}
  \caption[Rearrangement of a path]{A path of short steps returning to
    the origin, and a rearrangement of the steps so that the path
    stays much closer to the origin.}
  \label{fig:ch1-steinitz}
\end{figure}

This question was first asked (and answered affirmatively) by
Steinitz~\cite{Steinitz1913}. Note that the requirement that the vectors
sum to $0$ is not essential. In general, if $v = u_1 + \cdots + u_n$
is the sum of the vectors; then we can make sure that the path
traced by the partial sums stays close to the line from $0$ to $v$,
\ie, that
\(
\left\|\frac{j}{n}v - \sum_{i = 1}^j u_i\right\|_\infty  
\)
is as small as possible for all $j \in [n]$. This question is, however, essentially equivalent to the version where the vectors sum to $0$, which we can see by replacing each $u_i$ with $u_i - \frac{1}{n}v$. 

The Steinitz problem arises
naturally in many contexts, e.g.~in designing algorithms for integer
programming~\cite{EisenbrandW20}, and approximation algorithms for
scheduling problems~\cite{Sev78}. Let us use the notation $\st((u_i)_{i=1}^n)$ for how close we can keep the path to the line from $0$ to $v$, i.e., 
\[
\st((u_i)_{i=1}^n) \eqdef \min_{\pi \in S_n} \max_{j=1}^n \Bigg\|\frac{j}{n} v - \sum_{i=1}^j u_{\pi(i)}\Bigg\|_\infty,
\]
where $S_n$ is the set of permutations of $[n]$, and $v \eqdef \sum_{i = 1}^n u_i$. Once again, this question also makes sense for any norm $\|\cdot\|_X$ in place of the $\ell_\infty$ norm, and we denote the corresponding quantity by $\st_X((u_i)_{i=1}^n)$.

Prefix discrepancy is a very useful tool in bounding $\st((u_i)_{i=1}^n)$. From now on, let us return to the assumption $u_1 + \ldots + u_n = 0$. If some prefix sums are too large, then prefix discrepancy allows us to rearrange the vectors so that the prefix sum norms improve. The rearrangement is simple: we take a coloring $\col$, and place the vectors colored $+1$ first, followed by the vectors colored $-1$ in reverse order. For some intuition, let us take some prefix sum $u_1 + \ldots + u_i$. After the rearrangement, this prefix gets split into a prefix in the first half of the rearranged sequence (the vectors among $u_1, \ldots u_i$ colored $+1$), and a suffix in the second  half (the vectors colored $-1$). If the prefix discrepancy is small, then the new prefix and suffix will have sums approximately equal to $\frac12 \sum_{j = 1}^i u_i$, which has smaller norm than $\sum_{j = 1}^i u_i$. Since the vectors add up to $0$, the norms of the suffix sums equal the norms of their complementary prefix sums, so we have improved the norms of all prefix sums by this rearrangement.

The main theorem and its proof follow.
\begin{theorem}\label{thm:ch1-chobanyan}
  For any integer $n \ge 1$, and any
  vectors $u_1 \ldots, u_n \in \R^m$ such that $u_1 + \ldots + u_n = 0$, 
  \begin{equation}\label{eq:ch1-chobanyan}
    \st((u_i)_{i=1}^n)\le
    \max_{\sigma \in S_n} \pdisc((u_{\sigma(i)})_{i = 1}^n).
  \end{equation}
  where $S_n$ is the set of permutations of $[n]$.
\end{theorem}
\begin{proof}
  Let $\pi$ be the permutation of $[n]$ that minimizes $\max_{j = 1}^n\left\|\sum_{i = 1}^j u_{\pi(i)} \right\|_\infty$, i.e., achieves $\st((u_i)_{i=1}^n)$. Let
  $\col:[n]\to\{-1,+1\}$ achieve $\pdisc((u_{\pi(i)})_{i = 1}^n)$. We
  will show that if \eqref{eq:ch1-chobanyan} does \emph{not} hold, then we can choose
  another permutation $\pi'$ for which 
  \begin{equation}\label{eq:chobanyan-step}
  \max_{j = 1}^n\left\|\sum_{i = 1}^j u_{\pi'(i)} \right\|_\infty
  < 
  \max_{j = 1}^n\left\|\sum_{i = 1}^j u_{\pi(i)} \right\|_\infty = \st((u_i)_{i=1}^n),
  \end{equation}
  which contradicts the definition of $\st((u_i)_{i=1}^n)$. We choose $\pi'$ using the
  swapping operation alluded to above. We set $S = \{i: \col(i) = +1\}$, we place
  the vectors $u_i$ indexed by $i \in S$ first, ordered according to
  $\pi$, followed by the vectors indexed by $i \not \in S$, ordered in
  reverse order of $\pi$. Using $1_S(\cdot)$ for the indicator function of $S$, we have
  \begin{align*}
    \max_{j =1}^{|S|} \left\|\sum_{i = 1}^j u_{\pi'(i)} \right\|_\infty
    & =
    \max_{j=1}^n \left\|\sum_{i = 1}^j \indic_S(i)u_{\pi(i)}  \right\|_\infty =
      \max_{j=1}^n \left\|\sum_{i = 1}^j \frac{1 + \col(i)}{2} u_{\pi(i)}\right\|_\infty\\
    &\le
      \frac12\max_{j=1}^n \left\|\sum_{i = 1}^j u_{\pi(i)}\right\|_\infty
      +\frac12\max_{j=1}^n \left\|\sum_{i = 1}^j x(i)u_{\pi(i)}\right\|_\infty\\
    &\le 
      \frac12\st((u_i)_{i=1}^n)+ \frac12\pdisc((u_{\pi(i)})_{i = 1}^n).
  \end{align*}
  A similar calculation shows that
  \[
    \max_{j =|S| + 1}^n \left\|\sum_{i = 1}^j u_{\pi'(i)} \right\|_\infty
    \le
    \frac12\max_{j=1}^n \st((u_i)_{i=1}^n)   
    + \frac12\pdisc((u_{\pi(i)})_{i = 1}^n).  
  \]
  Thus, if \(\pdisc((u_{\pi(i)})_{i = 1}^n) < \st((u_i)_{i=1}^n)\), then \eqref{eq:chobanyan-step} holds. 
\end{proof}


Theorem~\ref{thm:ch1-chobanyan} is due to Chobanyan~\cite{Chobanyan94},
and is known as Chobanyan's transference theorem. The proof above only uses the triangle inequality for $\ell_\infty^m$, so the transference theorem in fact holds for $\st_X(\cdot)$ and $\pdisc_X(\cdot)$ for any normed space $X$. As with our other applications of discrepancy, it is not hard to see that if we can efficiently find the colorings achieving the prefix discrepancy bound, we can also efficiently find the re-arrangement in the theorem. 

In order to use Theorem~\ref{thm:ch1-chobanyan} to bound $\st((u_i)_{i=1}^n)$ for $u_1, \ldots, u_n \in [-1,1]^m$, we need to bound the partial coloring discrepancy $\pdisc(\mat)$ of any $m\times n$ matrix $\mat$ with entries in $[-1,1]$. A simple way to do so is to take a uniformly random coloring $\col:[n]\to \{-1,+1\}$.
Hoeffding's inequality and the union bound then give us 
\( \pdisc(M, \col) \lesssim \sqrt{n
  \log(m+n)}, \) with positive probability. This argument and Theorem~\ref{thm:ch1-chobanyan} then imply that this bound also holds for $\st((u_i)_{i=1}^n)$ whenever $u_1, \ldots, u_n \in [-1,1]^m$.
  
This is much less than the trivial bound $n$, unless $m$ is
exponential in $n$. But what about bounds that depend only on $m$? 
In
the next chapter, we will show that $\st((u_i)_{i=1}^n) \le 2m$ holds under the same assumptions, and, moreover, it holds for \emph{any} norm, and vectors $u_1, \ldots, u_n$ in its unit ball. Thus, the path of partial sums in the Steinitz
problem can be made to stay within a bounded distance from the origin, no matter how many steps we take! 

The next tantalizing problem is open.
\begin{problem}\label{prob:steinitz}
  Give tight bounds (in terms of $m$ and $n$) on $\st((u_i)_{i=1}^n)$ when $u_1, \ldots, u_n \in [-1,1]^m$, and on $\pdisc(M)$ for $m\times n$ matrices $M$ with entries in $[-1,1]$. 

  Similarly, give tight bounds (in terms of $m$ and $n$) on $\st_{\ell_2^m}((u_i)_{i=1}^n)$ when $u_1, \ldots, u_n \in B_{\ell_2^m} := \{x \in \R^m: \|x\|_2 \le 1\}$. Also give tight bounds on $\pdisc_{\ell_2^m}(\mat)$ for an $m\times n$ matrix $\mat$ with columns bounded by $1$ in $\ell_2$ norm.
\end{problem}

It is conjectured that the right bound in terms of $m$ for both problems is $O(\sqrt{m})$. However, the best known bound for both is $O(m)$,
which also holds for every $m$-dimensional
norm~\cite{Sev78,GS80,baranygrinberg}. Partial progress towards the
conjecture for $\ell_2^m$ was made by Banaszczyk~\cite{B12}, who
showed that, when $u_1, \ldots, u_m \in B_{\ell_2^m}$, 
\begin{align*}
\pdisc_{\ell_2^m}((u_i)_{i=1}^n) &\lesssim \sqrt{m} + \sqrt{\log n},\\
st_{\ell_2^m}((u_i)_{i=1}^n) &\lesssim \sqrt{m} +
\sqrt{\log n},
\end{align*}
where the second bound follows from the first via Theorem~\ref{thm:ch1-chobanyan}.

\section{Organization of the Book}

In this book we survey methods of proving upper bounds on the notions
of discrepancy introduced above: discrepancy of set systems and matrices, vector balancing in different normed spaces, prefix discrepancy. We emphasize
methods from high dimensional probability and convex geometry. We also
spend considerable attention on methods of bounding discrepancy that
also give efficient algorithms for finding a coloring. This
is an important consideration for applications of discrepancy theory
to computational problems.

There are several other book-length surveys of discrepancy theory. The monographs of Back and Chen~\cite{BeckChen-book}, and of Drmota and Tichy~\cite{DrmotaTichy97} are primarily about continuous discrepancy questions, and irregularities / uniformity of distribution. The books of Chazelle~\cite{chazelle2001discrepancy} and Matou\v{s}ek~\cite{matousek2010geometric} cover a mix of combinatorial discrepancy and continuous discrepancy, and the connection between the two. We differ from these books in our focus on combinatorial discrepancy, and our emphasis on geometric techniques and efficient algorithms. Most of the results and algorithms we describe in the later chapters of this book were discovered after Chazelle and Matou\v{s}ek's books were published. 

We begin with linear algebraic arguments for bounding discrepancy in
Chapter~\ref{ch:linalg}. These techniques are flexible, and we will
apply them to all problems in this introductory chapter, as well as
some new ones. They also give efficient algorithms. The bounds proved
with linear algebra, however, are often not tight. To address this, in
Chapter~\ref{ch:partial} we introduce Beck's partial coloring method,
and a more refined geometric variant of it due to Giannopoulos. 

While the classical
partial coloring method often (but not always!)  gives tight bounds on
discrepancy, its classic variant does not give efficient constructions
of low discrepancy colorings. In Chapter~\ref{ch:partial-alg}, we introduce two recent
algorithmic variants of the partial coloring method: one is based on
randomizing the linear algebraic method due to Lovett and Meka, and another is an algorithmic
version of Giannopoulos's partial coloring theorem, due to Rothvoss. We also explain how semidefinite programming can be used to efficiently construct colorings whose discrepancy is competitive against the hereditary discrepancy. 

Then, in
Chapter~\ref{ch:bana-intro}, we introduce a method of Banaszczyk to prove discrepancy upper bounds that goes beyond partial colorings. We
describe Banaszczyk's original proof and the powerful geometric ideas
it utilizes in Chapter~\ref{ch:bana-proof}. Banaszczyk's proof does not
give an efficient coloring algorithm, but we give another,
algorithmic, proof of his main result in
Chapter~\ref{ch:gram-schmidt} based on the so-called Gram-
Schmidt Walk. Before this, we also give a different algorithmic proof of the bound it implies for the Koml\'os conjecture (Problem~\ref{prob:komlos}) in Chapter~\ref{ch:algo-komlos-bana}, based on using semidefinite programming. 

In Chapter~\ref{ch:gamma2}, we turn to
the computational complexity of approximating discrepancy and
hereditary discrepancy, based on $\gamma_2$-norms from convex geometry. We show that an optimal application of
Banaszczyk's theorem allows approximating hereditary discrepancy of
arbitrary set systems and matrices up to polylogarithmic
factors. 

Finally, in Chapter~\ref{ch:online} we consider a model in
which the elements to be colored arrive online and must be colored as
they come. We describe the Self-Balancing Walk algorithm, which shows that almost optimal discrepancy upper bounds hold even in
this online setting.

\section{Notation}

For two quantities $A$ and $B$, that may depend on other parameters,
we use the notation $A\lesssim B$ to mean that $A \le CB$ for an
absolute constant $C > 0$, independent of all other quantities. We use
$A \gtrsim B$ to denote $B \lesssim A$. We also use standard $O(\cdot)$ and $\Omega(\cdot)$ notation.

We use the notation $[n] \eqdef \{1, \ldots, n\}$. We identify vectors
$u \in \R^n$ with functions $u:[n] \to \R$, and use $u(i)$ for the
$i$-th coordinate of $u$. Similarly, for a (finite) set $S$, we treat
$\R^S$ as both a vector space of real-valued vectors indexed by $S$,
and as the space of functions from $S$ to $\R$. For a set $T\subseteq
[n]$, and a vector $x \in \R^n$, we let $x(T)$ equal the sum $\sum_{i
  \in T}{x(i)}$, and similarly for $x \in \R^S$ and $T \subseteq S$.

We also extend this notation to matrices: the entry of an $m$ by $n$
real valued matrix $\mat$ in row $i$ and column $j$ is denoted
$\mat(i,j)$, \ie, we identify $\mat$ with a function from
$[m]\times [n]$ to $\R$. We use analogous notation when the rows and
columns of $\mat$ are indexed by sets. We let $\mat(S,T)$ be the
submatrix of $\mat$ indexed by the rows in $S$ and columns in
$T$. (Note that this is different from the meaning of $x(S)$ for a
vector $x$ above.) We use $\mat(*,S)$ instead of $\mat([m], S)$ and $\mat(S,*)$ instead of $\mat(S,[n])$. We also identify a sequence of vectors
$u_1, \ldots, u_n \in \R^m$ with the $m\times n$ matrix whose $i$-th
column is $u_i$.

The conventions above are also followed when $\R$ is replaced with
another set, e.g., $\{0,1\}$ or $\{-1,1\}$.

The $\ell_p^n$ norm is defined on $\R^n$ by $\|u\|_p \eqdef \left(\sum_{i
  =1}^n |u(i)|^p\right)^{1/p}$. We denote the standard inner product
on $\R^n$ by $\ip{u,v} \eqdef \sum_{i=1}^n u(i)v(i)$.

Let us, finally, collect the basic discrepancy definitions introduced
above.

For a finite dimensional normed space $X$ with norm $\|\cdot \|_X$, and a matrix $\mat$ with $n$ columns in $X$, we define
\begin{align*}
  \disc_X(\mat,\col)&\eqdef \left\|\mat\col\right\|_X;\\
  \disc_X(\mat) &\eqdef \min_{\col \in \{-1,+1\}^n}\disc_X(\mat,\col);\\
  \herdisc_X(\mat, s) &\eqdef
                                 \max_{J \subseteq [n]: |J| \le s}
                                 \disc_X(\mat(*,J));\\
  \herdisc_X(\mat) &\eqdef \max_{s} \herdisc_X(\mat, s).
\end{align*}
For a subset $K$ of $X$, we define
\begin{align*}
  \vb(K,X,n) &\eqdef \sup\{\disc_X((u_i)_{i=1}^n): u_1, \ldots, u_n  \in K\},\\
  \vb(K,X) &\eqdef \sup_s \vb(K,X,n). 
\end{align*}

When we just write $\disc$ and $\herdisc$ without a subscript, we mean
that $X = \ell_\infty^m$. For a set system $\SS$ over a universe
$\uni$, $\disc(\SS,\col)$, $\disc(\SS)$, $\herdisc(\SS,s)$ and
$\herdisc(\SS)$ are defined to equal, respectively,
$\disc(\mat,\col)$, $\disc(\mat)$, $\herdisc(\mat,s)$ and
$\herdisc(\mat)$ where $\mat$ is the incidence matrix of $\SS$. 

For a finite dimensional normed space $X$ with norm $\|\cdot \|_X$, and
vectors $u_1, \ldots, u_n$ in $X$, we define
\begin{align*}
  \pdisc_X((u_i)_{i = 1}^n, \col) &\eqdef
  \max_{j=1}^n \left\|\sum_{i = 1}^j\col(i)u_i \right\|_X;\\
  \pdisc_X((u_i)_{i=1}^n) &\eqdef \min_{\col \in \{-1,+1\}^n}\pdisc_X((u_i)_{i=1}^n,x).
\end{align*}
For a subset $K$ of $X$, we define
\[
  \sser(K, X) \eqdef \sup\{\pdisc_X((u_i)_{i=1}^n): u_1, \ldots, u_n \in K\}.
\]

The Steinitz
constant of \(u_1,\ldots, u_n\) is defined as 
\[
  \st_X((u_i)_{i=1}^n) = \min_{\pi \in S_n} \max_{j = 1}^n \Big\|\sum_{i = 1}^j
     u_{\pi(i)} - \frac{j}{n} \sum_{k = 1}^n u_k\Big\|_X.
\]
For a set \(K
\subseteq \R^m\), and a normed space \(X = (\R^m,\|\cdot\|_X)\), we  have
\begin{align*}
  &\st(K, X,n) \eqdef \sup\{\st_X((u_i)_{i=1}^n): u_1, \ldots, u_n \in K\},\\
  &\st(K, X) \eqdef \sup_n \st(K, X,n).
\end{align*}

\section{Bibliographic Notes}

Epsilon approximations were defined by Vapnik and Chervonenkis~\cite{VC71}. Discrepancy theory was used by Matou\v{s}ek, Welzl, and Wernisch to construct small $\varepsilon$-approximations of families of sets of bounded VC dimension~\cite{MatWW93}. Their argument is analogous to our proof of Theorem~\ref{thm:ch1-rectangles}. A general version of this argument can be found in Matou\v{s}ek's book~\cite{matousek2010geometric}.

Problem~\ref{prob:tusnady} was first asked by Tusn\'ady in the early
1980s. In particular, he asked if $\disc(\RR_2,n) \lesssim 1$. Beck
answered Tusn\'ady's question negatively~\cite{beck-rect}, by proving that
$\disc(\RR_2, n)$ is at least a universal constant times the minimum
Lebesgue measure discrepancy of an $n$-point set in the plane. Here,
the Lebesgue measure discrepancy of a set $P \subseteq [0,1]^d$ of
size $n$ equals $\sup_{R \in \RR_d} ||R \cap P| - n\lambda^d(R)|$,
where $\lambda^d$ is the Lebesgue measure in $\R^d$. Beck's proof that Lebesgue measure discrepancy is bounded by  $O(\disc(\RR_2,n))$ uses an argument similar to the proof of Theorem~\ref{thm:ch1-rectangles}.
Schmidt~\cite{schmidt-irregVII} proved the Lebesgue measure
discrepancy of any $n$-point set $P$ in $[0,1]^2$ is at least
$\Omega(\log n)$, which is tight. Together with Beck's result, Schmidt's lower bound shows that
$\disc(\RR_2,n) = \Omega(\log n)$. This remains the best known lower
bound on $\disc(\RR_2,n)$. In higher dimensions, Matou\v{s}ek,
Nikolov, and Talwar proved that $\disc(\RR_d,n) \gtrsim \log(n)^{d-1}$
in any constant dimension $d$~\cite{MNT}. We will cover this lower
bound in Chapter~\ref{ch:gamma2}. The strongest known upper bound for Tusn\'ady's
problem is $\disc(\RR_d,n) \lesssim \log(2n)^{d-0.5}$, and is due to
Nikolov~\cite{tusnady-ub}. We cover it in Chapter~\ref{ch:bana-intro}.

It is worth noting also that determining tight bounds on the best
possible Lebesgue measure discrepancy (with respect to axis-aligned
boxes) of an $n$-point set in $[0,1]^d$ is a notorious open problem for $d > 2$,
and the gaps in high dimensions are much bigger than they are for $\disc(\RR_d,n)$. The
best upper bounds are in $O(\log(n)^{d-1})$ and the best known lower
bounds are in $\Omega(\log(n)^{(d-1)/2 + \eta_d})$ for some
$\eta_d > 0$ that goes to $0$ with $d$~\cite{BilykLV08}. We refer to
the books of Beck and Chen~\cite{BeckChen90}, and
Matou\v{s}ek~\cite{matousek2010geometric} for further discussion of
this fascinating problem.

Theorem~\ref{thm:ch1-transference} is due to
Lov\'asz, Spencer, and Vesztergombi, who also defined the notion of
hereditary discrepancy~\cite{LSV}. 

The vector
balancing question can be traced back to work by Anatole Beck, who
showed an equivalence between the strong law of large numbers holding
in a Banach space $X$, and the condition that
$\min_{x \in \{-1,+1\}^n}\|x(1) v_1 + \ldots + x(n) v_n\|_X \le cn$ for
some fixed constant $c < 1$, all $n$, and all $v_1, \ldots, v_n$ in
the unit ball of $X$~\cite{ABeck62}. Around the same time, Dvoretzky asked what can be said in
general about the maximum of the quantity $\min_{x \in \{-1,1\}^n}\|x(1) v_1 + \ldots
+ x(n) v_n\|_X$ over all $v_1, \ldots, v_n$ in the unit ball of
$X$ (see the Unsolved Problems chapter of~\cite{Dvoretzky-problem}). 
Vector balancing problems were
investigated by Gluskin~\cite{Glu89}, Giannopoulos~\cite{Gia93},
Banaszczyk~\cite{Bana93,Bana98}, Dadush, Nikolov, Talwar,
Tomczak-Jaegermann~\cite{anynorm}, Reis and Rothvoss~\cite{ReisR23}, among others. 

The approximate Carath\'eodory theorem (proved with a different
argument) seems to be due to Maurey, see, e.g.,
\cite{Pisier81}. Maurey's proof, covered, for example, by
Vershynin~\cite{vershynin}, is based on randomly sampling from the
distribution given by the coefficients of a convex
combination. Similar results can be shown also using optimization
techniques, either using the Frank-Wolfe algorithm~\cite{CP23}, or using mirror
descent~\cite{MLVW17}. Theorem~\ref{thm:ch1-transference-sizes} is due to Dadush, Nikolov, Talwar,
Tomczak-Jaegermann~\cite{anynorm}, who also first observed the
connection to approximate Carath\'eodory theorems. This connection was
further pursued by Reis and Rothvoss~\cite{RR22-apxcar}.

A bound independent of $n$ on the Steinitz constant $\st_X((u_i)_{i=1}^n)$ for any finite
dimensional norm $X$ and vectors $u_1, \ldots, u_n$ in unit ball,
was first proved by Steinitz~\cite{Steinitz1913}. He needed such a
bound as a lemma in his higher-dimensional generalization of Riemann's theorem that any
conditionally convergent series in $\R$ can be rearranged to converge
to any number in $\R$. Steinitz used his lemma to show that the set of
values to which the rearrangements of a series of vectors in $\R^m$
converges is an affine subspace of $\R^m$. This result, and the lemma,
were rediscovered later by Kadec~\cite{Kadec53}.  The best bound for
the Steinitz constant in terms of the dimension $m$ of the normed space $X$ is $m$, and is due to Sevastjanov
and Grinberg~\cite{Sev78,GS80}. Theorem~\ref{thm:ch1-chobanyan} is due to Chobanyan~\cite{Chobanyan94}, and an algorithmic version of it was described by Harvey and Samadi~\cite{HarveySamadi14}.

\section{Recent Work not Covered Here}

Several results closely related to the topics covered in this  book have been made public since the writing of the book was completed. Aden-Ali~\cite{adenali2026} presented a computationally efficient (in fact, linear time) online discrepancy minimization algorithm that satisfies the same discrepancy guarantees as the inefficient algorithm of Kulkarni, Reis, and Rothvoss~\cite{rothvoss-online}, and significantly improves on the algorithm presented in Chapter~\ref{ch:online}. The algorithm of Aden-Ali (discovered by ChatGPT 5.5-Pro) gives a new constructive proof of Theorem~\ref{thm:bana2}, and resolves several open problems in this book. 

Shortly after the paper by Aden-Ali became available online, Altschuler and Tikhomirov~\cite{altschulerT2026} used similar techniques to give improved bounds for the Beck-Fiala problem, showing that the Beck-Fiala Conjecture (see Chapter~\ref{ch:linalg}) holds for set systems on $n$ elements of degree at least $\Omega(\log(n)^{1+\eta})$ for any fixed $\eta>0$.



\chapter{Linear Algebraic Methods}
\label{ch:linalg}

In this chapter we see the first example of a general and flexible
method for bounding the discrepancy of set systems and matrices. The
method is based on linear algebraic arguments, and happens to also
give efficient deterministic algorithms to compute low discrepancy
colorings. It is very versatile, and allows proving non-trivial, and
sometimes even tight bounds in many different settings: we will use it
to bound the discrepancy of set systems with bounded degree, and for systems of
axis-aligned boxes, and of set systems induced by permutations, and to
give bounds for vector balancing and for the Steinitz problem in
arbitrary norms. Some of the ideas developed here will form the basis of
more powerful randomized algorithms presented in later chapters.

\section{The Beck-Fiala Theorem}

Recall that a set system is a pair \((\SS,\uni)\), where $\SS$ is a
collection of subsets of the (finite) universe $\uni$. Let us use the notation
\(\ssdeg(\elem)\) for the degree of \(\elem \in \uni\), i.e., the number
of sets in \(\SS\) that \(\elem\) appears in. The maximum degree
\(\ssdeg(\SS)\) of \(\SS\) is the maximum degree \(\ssdeg(\elem)\) over all
\(\elem \in \uni\). The maximum degree of $\SS$ is a natural measure
of how sparse it is, and we will see that small degree set systems
have a low discrepancy. The next exercise is the easiest interesting
case of this phenomenon.

\begin{exercise}\label{ex:ch2-bf-graphs}
  Given an undirected graph $G = (V,E)$, define $\SS$ with universe $\uni=E$ and  sets $S_u$,
  for each vertex $u \in V$, consisting of all edges with endpoint
  $u$. Then $\ssdeg(\SS) = 2$. Show that if $G$ is
  bipartite, then $\disc(\SS) \le 1$. Show furthermore that $\disc(\SS)
  \le 2$ for all graphs $G$.
\end{exercise}

More generally we have the following fundamental question in
discrepancy theory.
\begin{problem}
  Determine the best possible  bound on the discrepancy \(\disc(\SS)\)
  of a set system $\SS$ as a function of \(\ssdeg(\SS)\).
\end{problem}

In this section, we prove the Beck-Fiala theorem, stated next.
\begin{theorem}\label{thm:ch2-bf}
  For any set system \((\SS, \uni)\) there exists a coloring \(\col\),
  computable in polynomial time, such that \(\disc(\SS,\col) \le 2\ssdeg(\SS)-1\).
\end{theorem}

It should be surprising that
\emph{any} discrepancy bound can be shown purely in terms of $\ssdeg(\SS)$. A key
insight is that a set system with small maximum degree only has few ``large'' sets, i.e., sets of size \(\gg \ssdeg(\SS)\), as we show next.
\begin{lemma}\label{lm:ch2-bf-counting}
  For any set system \(\SS\) over a universe $\uni$ of size \(n\), and
  any \(s>0\),   \[
    |\{\set \in \SS: |\set| > s\}| < n\ssdeg(\SS)/s.
  \]
\end{lemma}
\begin{proof}
  Let \(I \eqdef \{(\set, \elem): \elem \in \set\}\) be the set of
  set-element incidences in \(\SS\). We will count \(|I|\) in two
  different ways. On the one hand, $|I| \le n\ssdeg(\SS)$ as each
  element of $\uni$ contributes at most $\ssdeg(\SS)$ incidences. 
  On the other hand, denoting by \(k\) the number of sets in \(\SS\)
  of size greater than \(s\), we have \(|I| > ks\).
  Combining the two inequalities gives \(ks < n\ssdeg(\SS)\), as desired.
\end{proof}

The proof of the Beck-Fiala theorem relies on two observations. The
first is that a set system consisting only of small sets has small
discrepancy. Indeed, \emph{any} coloring of a set of size \(s\) has
discrepancy at most \(s\). The second observation is that requiring that a few sets have small discrepancy still leaves many degrees of freedom for the coloring. In this chapter we formalize this second observation using linear algebra. In Chapter~\ref{ch:partial}, we will see a different formalization through geometry and Gaussian measure, and in Chapter~\ref{ch:partial-alg} we combine the two approaches.

\subsection{The Beck-Fiala Algorithm}
We now prove Theorem \ref{thm:ch2-bf}, using a linear algebraic approach.

 The key idea that enables using linear algebra is to relax the concept of a coloring to \df{fractional coloring} $\col \in [-1,+1]^\uni$. The algorithm will compute a sequence of fractional colorings $\col_0, \ldots, \col_T$,
 where initially, $\col_0$ assigns color $0$ to every element, and at each step $t$, $\col_t$ colors at least one more element with $\pm 1$ than $\col_{t-1}$. So by step $T \le n$, $\col_T \in \{-1,+1\}^\uni$.
 
 Whenever some color $\col_t(\elem)$ reaches $\{-1,+1\}$, we fix $\col_t(\elem)$ forever, and we also maintain the invariant that $\col_t(S) = \col_{t-1}(S)$ for all sets $S \in \SS$ that have more than $\ssdeg(\SS)$ elements that are not yet colored $\pm 1$. These are the ``dangerous sets'' -- with potential to incur large discrepancy. 
 The existence of such $\col_t$ will follow from Lemma~\ref{lm:ch2-bf-counting} and a simple linear algebraic argument.

\paragraph{The Formal Algorithm.} The coloring \(\col\) is computed by an iterative algorithm. At each step \(t\), the algorithm maintains a fractional coloring \(\col_t \in [-1,1]^\uni\), a set of fixed elements \(F_t \eqdef \{\elem \in \uni: |\col_{t-1}(\elem)| = 1\}\), and a set of ``dangerous'' sets \(\DD_t \eqdef \{\set \in \SS:|\set \setminus F_t| > \ssdeg(\SS)\}\). 
Initially, we set \(\col_0 \eqdef 0\), and the algorithm terminates when \(F_{t+1} = \uni\), at which point we output \(\col \eqdef \col_t\). 

To compute $\col_t$ from $\col_{t-1}$, in step \(t \ge 1\), we find a nonzero update vector $\Delta\col_t$ such that 
\begin{enumerate}
\item \(\Delta\col_t(\set) := \sum_{\elem \in \set}\Delta\col_t(\elem) = 0\) for all dangerous sets \(\set \in \DD_t\);\label{enum:bf-dangerous}
\item $\Delta\col_t(\elem) = 0$ for all fixed elements \(\elem \in F_t\).\label{enum:bf-fixed}
\end{enumerate}
We argue why such a $\Delta\col_t$ exists below. To compute the coloring $\col_t$, we set $\col_t = \col_{t-1} + \gamma \Delta \col_t$, where we choose $\gamma > 0$ to be the largest possible value such that $\col_t \in [-1,1]^\uni$. See Figure~\ref{fig:ch2-bflemma}. 

\begin{figure}[htp]
\centering
\includegraphics{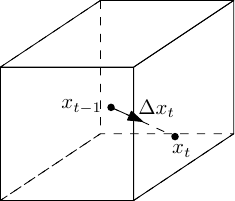}
\caption{We find \(\col_t\) by starting from \(\col_{t-1}\) and going in the direction of \(\Delta \col_t\) until we hit a boundary of the cube \([-1, +1]^\uni\).}
\label{fig:ch2-bflemma}
\end{figure}

\paragraph{Analysis.} Let us use the notation $\act_t \eqdef \uni \setminus F_t$ for the ``active'' elements, i.e., the ones not yet fixed. First we verify that, as long as $F_t \neq \uni$ (equivalently $\act_t \neq \emptyset$), there exists a $\Delta\col_t \neq 0$ satisfying \ref{enum:bf-dangerous}.~and~\ref{enum:bf-fixed}. The two conditions form a system of linear equations, and to show that the system has a nonzero solution $\Delta\col_t$, we will argue that it is under-constrained. By Lemma~\ref{lm:ch2-bf-counting}, applied to the set system \(\SS|_{\act_{t}}\) with \(s = \ssdeg(\SS) \ge \ssdeg(\SS|_{\act_{t}})\), we have that \(|\DD_t| < |\act_{t}|\). Thus, $\Delta \col_t$ must satisfy at most $|\DD_t| + |F_t| = |\DD_t| + (n-|\act_{t}|) < n$ linear equations, and the subspace of solutions to these equations has dimension at least $1$. This shows that we can find some $\Delta \col_t \neq 0$ satisfying \ref{enum:bf-dangerous}.~and~\ref{enum:bf-fixed}.

Next, the algorithm terminates in at most $n$ steps as $F_{t+1} \supsetneq F_{t}$, by the choice of $\gamma$ and thus $|F_t|$ increases by at least $1$ in each step. 

Finally, we analyze the discrepancy of the final coloring $\col = \col_T$. Notice that, as $F_t$ is increasing in $t$, \(\DD_{t} \subseteq \DD_{t-1}\), i.e., a set can only stop being dangerous once, and it never becomes dangerous again. Since initially $\col_0(S) = 0$ for all $S \in \SS$, we have that \(\col_t(\set) = 0\) as long as \(\set \in \DD_t\). 
To bound the final discrepancy, consider an arbitrary set \(\set \in \SS\), and suppose that it was last dangerous at time step \(t=t(S)\), i.e., \(\set \in \DD_{t} \setminus \DD_{t+1}\). Then,
\[
    |\col(\set)| \le |\col_{t}(\set)| + |\col(\set) - \col_{t}(\set)|
    = |\col(\set) - \col_{t}(\set)|.
\]
As the color stays unchanged for any element \(\elem \in F_{t+1}\) after time $t$, we have \(\col(\elem) = \col_t(\elem)\) and thus 
\begin{align*}
    & |\col(\set) - \col_{t}(\set)|
    =
    \Big|\sum_{\elem \in \set \setminus F_{t+1}} (\col(\elem) -  \col_t(\elem))\Big|\\
    &\le \sum_{\elem \in \set \setminus F_{t+1}} |\col(\elem) -  \col_t(\elem)| < 2|\set \setminus F_{t+1}| \le 2 \ssdeg(\SS).
\end{align*}
The strict inequality follows since \(\col(\elem) \in \{-1,+1\}\), and \(\col_t(\elem) \in (-1, +1)\) for any \(\elem \not\in F_{t+1}\), thus \(\col(\elem) - \col_t(\elem) \in (-2, 2)\). The final inequality follows as \(\set \not \in \DD_{t+1}\). 

We have now shown that \(|\col(\set)| < 2 \ssdeg(\SS)\), and, as \(|\col(\set)|\) is an integer,  we must have \(|\col(\set)| \le 2 \ssdeg(\SS)-1\). As \(\set\) was arbitrary, the result follows.

\paragraph{Improvements.} Theorem~\ref{thm:ch2-bf} is not tight, and Beck and Fiala
conjectured that the tight bound is \(\disc(\SS) \lesssim
\sqrt{\ssdeg(\SS)}\). The best known bounds (for general \(\ssdeg(\SS)\))  that only depend on
\(\ssdeg(\SS)\) are linear in \(\ssdeg(\SS)\), and, in fact, even the leading
constant \(2\) has not been improved.\footnote{Very recently, Bansal and Jiang \cite{bansaljiang2025} proved the conjectured $O(\sqrt{\ssdeg(\SS)})$ bound when $\ssdeg(\SS) = \Omega(\log^2 n)$.} Later in this book we will see
bounds that have only mild dependence on the number of elements
\(n\). 

For now, the next exercise gives a small improvement to
Theorem~\ref{thm:ch2-bf}.

\begin{exercise}\label{ex:ch2-bf-improvement}
  Modify the proof of Theorem~\ref{thm:ch2-bf} to show the improved
  bound $\disc(\SS) \le \max\{2\ssdeg(\SS) - 3, \ssdeg(\SS)\}$.

  \textsc{Hint:} By an argument analogous to
  Lemma~\ref{lm:ch2-bf-counting}, a set system on $n$ elements with maximum degree $s$
  has at most $n$ sets with size at least $s$. What happens when the
  set system has exactly $n$ such sets?
\end{exercise}

Note that, unlike Theorem~\ref{thm:ch2-bf},
Exercise~\ref{ex:ch2-bf-improvement} recovers the (tight) bound
$\disc(\SS) \le 2$ when $\ssdeg(\SS) \le 2$. 

\section{Template of a Linear Algebraic Algorithm}

In this section we generalize the approach taken in the proof of the Beck-Fiala theorem to formulate a template for an algorithm computing low discrepancy colorings using linear algebra. We will use it later for various applications. We keep the notation from the proof of Theorem~\ref{thm:ch2-bf} whenever possible.

\paragraph{The Generic Algorithm.} We start with
some fractional coloring $\col_0 \in [-1,1]^n$, usually
$\col_0 \eqdef 0.$ Then we construct a coloring in \emph{phases}, where in the $t$-th phase we compute $\col_t$. 

In phase $t$, we define a matrix $\mat_t$ of \emph{dangerous directions}, and a set of \emph{active} elements $\act_t \subseteq [n]$. The set of \emph{fixed} elements is always $F_t \eqdef \{i \in [n]: |\col_{t-1}(i)| = 1\}$, and we require $\act_t \cap F_t = \emptyset$. We choose the next coloring $\col_t$ so that
\begin{enumerate}
    \item $\mat_t\col_t = \mat_t \col_{t-1}$;
    \item $\col_t(i) = \col_{t-1}(i)$ for each
$i \not \in \act_{t}$.
\end{enumerate}
We also aim to fix as many elements as possible, i.e., to maximize $|F_{t+1}| - |F_t|$. At the first $T$ for which $F_{T+1} = [n]$, we output $\col_T$.

Intuitively, each row of $\mat_t$ defines a ``dangerous direction'', and moving the coloring in such a direction may incur too much discrepancy. For this reason we ask the coloring update $\Delta x_t = x_t - x_{t-1}$ to lie in the kernel of $\mat_t$, which is our ``safe subspace''.  

The proof of the Beck-Fiala theorem follows the template above. In it, each phase fixes the color of at least one element, leading up to $n$ phases. In other applications, we may want to fix more elements in a single phase so that the number of phases becomes much smaller. To compute $\col_t$ from $\col_{t-1}$ we shall use the following lemma. 

Recall that \(\mat(*, J)\) denotes the submatrix of \(\mat\) consisting of the columns indexed by the set \(J\).

\begin{lemma}\label{lm:ch2-linalg-gen}
  Let \(\mat\) be an \(m \times n\) matrix, let \(\col \in
  [-1,+1]^n\), and let \(\act \subseteq [n]\) be such that for any \(i
  \in \act\), \(|\col(i)| < 1\).  If \(\rank \mat(*,\act) <
  |\act|\), then there exists a \(\col' \in [-1,
  +1]^n\) such that
  \begin{enumerate}
  \item \(\mat \col' = \mat \col\);
  \item for any \(i \in [n] \setminus \act\), \(\col'(i) = \col(i)\)
  \item for at least \(|\act| - \rank \mat(*,\act)\) elements 
    \(i \in \act\), we have \(\col'(i) \in \{-1, +1\}\).
  \end{enumerate}
  Moreover, given \(\mat\) and \(\col\), \(\col'\) can be computed in
  polynomial time.
\end{lemma}

The proof of Lemma~\ref{lm:ch2-linalg-gen} is simple and left as an exercise.

Similar to the proof of the Beck-Fiala theorem, we start initially with $\col_0= 0^n$ and iteratively update the coloring $\col_t$ to $\col_{t+1}= \col_t + \Delta \col_t$ at each step using Lemma~\ref{lm:ch2-linalg-gen} with $x=\col_t$ and $x'=\col_{t+1}$.


One may wonder if we can gain something by picking the update $\Delta \col_t$ more carefully. In Chapter~\ref{ch:partial-alg} we will see an example of this, where the directions are picked randomly from some appropriate distribution.

Applying Lemma~\ref{lm:ch2-linalg-gen}  to the incidence matrix of
a set system gives the following immediate corollary that we record for ease of reference.

\begin{lemma}\label{lm:ch2-fewsets}
  Let \((\SS, \uni)\) be a set system, let \(\col\in [-1,+1]^\uni\) be a
  fractional coloring, and let
  \(\act \subseteq \uni\) be the set of elements with \(|\col(\elem)| < 1\). If \(|\SS|_{\act}| < |\act|\), then there is a \(\col' \in [-1,+1]^\uni\) such that
  \begin{enumerate}
  \item for all \(\set \in \SS\), \(\col'(\set) = \col(\set)\);
  \item for all \(\elem \in \uni\setminus\act\), \(\col'(\elem) = \col(\elem)\);
  \item for at least \(|\act|-|\SS_{|\act}|\) elements \(\elem \in
    \act\) we have
    \(\col'(\elem) \in \{-1,+1\}\).
  \end{enumerate}
  Moreover, given \(\SS\) and \(\col\), \(\col'\) can be computed in
  polynomial time.
\end{lemma}

\paragraph{Simple Application.} As an application of Lemma~\ref{lm:ch2-linalg-gen}, we first show
a useful structural result about discrepancy, that the
hereditary discrepancy of any \(m\times n\) matrix is always achieved, up to a
factor of \(2\), on a submatrix consisting of at most \(m\) columns
from \(\mat\). For this reason, discrepancy bounds in terms of \(n\)
can often be automatically strengthened to hold in terms of
\(\min\{m,n\}\) instead.

\begin{theorem}\label{thm:ch2-col-reduction}
  Let \(\|\cdot\|_X\) be a norm on \(\R^m\), and let \(\mat\) be an
  \(m\times n\) matrix. Then there exists a set \(J \subseteq [n]\) of
  size \(|J| \le m\) such that
  \[
    \disc_X(\mat) \le 
     2 \herdisc_X(\mat(*,J)) \le 2\herdisc_X(\mat,m)
  \]
\end{theorem}
\begin{proof}
  We apply Lemma~\ref{lm:ch2-linalg-gen} with the matrix \(\mat\),
  \(\act = [n]\), and \(\col = 0\). This gives us a fractional
  coloring \(\col'\) for which \(\mat \col' = 0\), and the set
  \(J \eqdef \{i: |\col'(i)| < 1\}\) has size at most
  \(\rank \mat \le m\). 
  Let \(y\in [-1,+1]^J\) be \(\col'\)
  restricted to \(J\). 
  By Theorem~\ref{thm:ch1-transference}, there
  exists some coloring \(z \in \{-1,+1\}^J\) such that
  \(\|\mat(*,J)(z-y)\|_X \le 2\herdisc_X(\mat(*,J))\). 
  
  We can then
  extend \(\col'\) to a coloring \(\col''\in \{-1,+1\}^n\) by setting
  \[
    \col''(i) \eqdef
    \begin{cases}
      z(i) & i \in J\\
      \col'(i) & i \not \in J. 
    \end{cases} 
  \]
  We now have
  \[
  \disc_X(\mat,\col'') 
  \le
  \|\mat \col'\|_X + \|\mat(*,J)(z-y)\|_X = 2\herdisc_X(\mat(*,J)),
  \]
  as we needed to show.
\end{proof}

\section{Discrepancy of Permutations}

For our next application of the linear algebraic method, we study the
discrepancy of set systems consisting of prefixes of \(k\)
permutations. We are given \(k\) permutations \(\pi_1, \ldots, \pi_k\)
of \([n]\), and the set system \((\SS, [n])\) consists of prefixes of
the permutations, i.e., all sets of
the form \(\{\pi_i(1), \ldots, \pi_i(j)\}\) where \(i\) and \(j\)
range over \([k]\) and \([n]\), respectively. We say that \(\SS\)
is the set system induced by \(\pi_1, \ldots, \pi_k\). One of the
problems we will return to later in the book is determining the
discrepancy of such set systems.
\begin{problem}
  For a set system \((\SS, [n])\) induced by \(k\) permutations of
  \([n]\), determine the best possible bound on \(\disc(\SS)\) as a
  function of \(n\) and \(k\).
\end{problem}
In addition to being
natural set systems in their own right, \(k\) permutation systems are
useful when studying the discrepancy of rectangles, and they have 
applications to approximation algorithms. We will see the first connection later in this chapter. 

The discrepancy of set systems induced by \(k\) permutations exhibits
an interesting transition: it is equal to \(1\) when \(k\) is \(1\) or
\(2\), and can be as large as \(\Omega(\log n)\) for any \(k\ge
3\). When \(k = 1\), we can just color \(\pi(1), \ldots,
\pi(n)\) with alternating colors, in this order, and the resulting
coloring clearly has discrepancy \(1\). 

\medskip 
The case of \(2\) permutations
is more interesting, and we briefly sketch the construction of a
coloring \(\col\) with discrepancy \(1\). The idea is to ensure, for
every \(i \in [2]\), and \(j \in [\lfloor n/2\rfloor]\) that
\(\col(\pi_i(2j - 1)) \neq \col(\pi_i(2j))\). Then, any prefix
\(\set\) of either permutation can be partitioned into the pairs \(\{\pi_i(2j-1),
\pi_i(2j)\}\), and at most one additional element, and each pair
contributes discrepancy \(0\) to \(\col(\set)\).

To find such a coloring \(\col\) we construct a graph \(G = ([n], E)\), where the
edge set \(E\) consists of all such pairs \(\{\pi_i(2j-1),
\pi_i(2j)\}\).
Then, the coloring \(\col\) that we
need corresponds to a proper 2-coloring of the vertices of \(G\),
which exists as \(G\) is a union of two
matchings, and hence a disjoint union of paths and even length
cycles.


\paragraph{$k$ Permutations.}
The combinatorial argument above does not extend to \(3\) or more
permutations, but the linear algebraic method allows us to prove an \(O(k \log n)\) bound on the discrepancy of \(k\)
permutations. 

The idea is to divide the permutations into intervals, and
make sure that each interval receives zero total discrepancy, but with respect to a fractional coloring, rather than a full \(\pm 1\)
coloring. In order to find such a fractional coloring using
Lemma~\ref{lm:ch2-fewsets}, we need that the total number of intervals whose discrepancy we keep at \(0\) is small. In particular, we make sure that the number of intervals is at most half the number of active elements, and thus we can fix half of them in each phase. This guarantees we can find a proper coloring in $O(\log n)$ phases, giving the logarithmic discrepancy bound. 

Formally, we show the following.

\begin{theorem}\label{thm:ch2-kperm}
  For any integers \(k \ge 1\) and \(n\ge 1\), and any permutations \(\pi_1,
  \ldots, \pi_k\) of \([n]\), there exists a coloring \(\col\),
  computable in polynomial time, such that for the set system \((\SS,
  [n])\) induced by \(\pi_1, \ldots, \pi_k\) we have
  \(
  \disc(\SS) \lesssim k \log 2n.
  \)
\end{theorem}

\paragraph{The Algorithm.} We apply our template. The initial coloring is set to \(\col_0 \eqdef 0\), and recall that the fixed elements are $F_t \eqdef \{i: |\col_{t-1}(i)| = 1\}$. We will always set $A_t \eqdef [n]\setminus F_t$. We compute $\col_t$ from $\col_{t-1}$ using Lemma~\ref{lm:ch2-fewsets}, applied to the coloring $\col_{t-1}$, the active set $\act_t$, and a set system $\II_t$ of at most $|\act_t|/2$ intervals that we describe next.

To define $\II_t$,  we construct a set system \(\II_{i,t}\) of size \(|\II_{i,t}| \le |\act_t|/2k\) for each permutation \(\pi_i\), and set \(\II_t = \II_{1,t} \cup \ldots \cup \II_{k,t}\). To construct \(\II_{i,t}\), take the interval \(\{\pi_i(1), \ldots, \pi_i(b_1)\}\) that contains exactly \(2k\) elements of \(\act_t\); then the next interval \(\{\pi_i(b_1+1), \ldots, \pi_i(b_2)\}\) that contains \(2k\) elements of \(\act_t\), and so on, until no such interval remains.

We run our generic algorithm with these set systems until the first time $T$ when $F_{T+1} = [n]$, and at that point we output $\col \eqdef \col_T$.

\paragraph{Analysis.} Lemma~\ref{lm:ch2-fewsets} guarantees that, after phase $t$, at least $|\act_t|/2$ new elements become fixed, i.e., $|\act_{t+1}| \le |\act_t|/2$. Since $T$ is the smallest integer for which $\act_{T+1} < 1$, we have $T \le \log_2(2n)$. 

Next we bound how the discrepancy changes from one fractional coloring to the next. Let us fix some prefix \(\set \in \SS\), where $\set = \{\pi_i(1), \ldots, \pi_i(j)\}$. We can partition $S \cap A_t$ into disjoint sets $I_1, \ldots, I_\ell \in \II_{t}$ and a ``remainder'' set $R$ of size at most $|R| \le 2k-1$, where $I_1, \ldots, I_\ell$ are all the sets $\II_{i,t} \subseteq \II_t$ fully contained in $\set\cap \act_t$. So we must have $I_1 \cup \ldots \cup I_\ell = \{\pi_i(1), \ldots, \pi_i(b)\} \cap \act_t$ for some $b \le j$, and also $|\{\pi_i(b+1), \ldots, \pi_i(j)\} \cap A_t| < 2k$.
Therefore, $R = \{\pi_i(b+1), \ldots, \pi_i(j)\}\cap\act_t$ has size at most $2k-1$.

Now, using that \(\col_t(I_j) = \col_{t-1}(I_j)\) for all \(I_1, \ldots, I_\ell\), we have
\begin{align}
    \left|\col_t(\set) - \col_{t-1}(\set)\right|
    &\le \sum_{j = 1}^\ell|\col_t(I_j) -  \col_{t-1}(I_j)|+ |\col_t(R) - \col_{t-1}(R)|\notag\\
    &= |\col_t(R) - \col_{t-1}(R)| \le 4k-2.\label{eq:ch2-perm-dchange}
\end{align}
As there are $T=\log_2(2n)$ phases,
and \(\col_0(\set) = 0\), 
the
final discrepancy of \(\set\) is at most 
\[ |\col_T(\set)| \le \sum_{t = 1}^T \left|\col_t(\set) - \col_{t-1}(\set)\right|  = O(k\log n).\]
%

\section{Discrepancy of Boxes}
\label{sec:ch2-tusnady}

In this section we apply the Beck-Fiala theorem (Theorem~\ref{thm:ch2-bf}), as well as Theorem~\ref{thm:ch2-kperm}, to give two simple bounds on the
discrepancy of axis-aligned boxes in \(\R^d\). Let us recall some
notation from the Introduction. We define \(\RR_d\) to be the set of
all axis-aligned boxes in \(\R^d\), i.e., all sets of the type
\( \rbox_{a,b} \eqdef \{u \in \R^d: a \le u \le b\}, \) where
\(a,b\in \R^d\), and \(a \le u \le b\) is understood to hold
coordinate by coordinate, i.e., \(a(i) \le u(i) \le b(i)\) for all
\(i \in [d]\). A set \(P\) of \(n\) points in \(\R^d\) induces the set
system \(\RR_d|_P = \{\rbox_{a,b} \cap P: a,b \in \R^d\}\) of subsets
of \(P\) induced by boxes. Recall Problem~\ref{prob:tusnady}, asking
us to determine tight bounds on 
  \[
    \disc(\RR_d, n) \eqdef \sup\{\disc(\RR_d|_P): P \subseteq \R^d, |P| = n\}.
  \]

To bound \(\disc(\RR_d, n)\), it is slightly easier to restrict the sets to 
\(
\AA_d|_P \eqdef \{\rbox_{0,b} \cap P: b \in \R^d\}
\)
of anchored boxes, i.e., boxes with one corner fixed at the origin, and bound
\(
\disc(\AA_d, n) \eqdef \sup\{\disc(\AA_d|_P): P \subseteq \R^d, |P| = n\}.
\)
This suffices, as each axis-aligned rectangle can be expressed as a $\pm$ combination of $2^d$ anchored boxes (we leave the simple proof as an exercise), and thus
  we have the inequalities
  \[
  \disc(\AA_d,n)  \le 
  \disc(\RR_d, n) \le 2^d \disc(\AA_d,n).
  \]
Moreover, we can always modify the points in $P$ so that $P$ is contained in the grid $[n]^d$, without changing the set system $\AA_d|_P$. This also gives us
   \[
     \disc(\AA_d, n) =
     \sup\{\disc(\AA_d|_P): P \subseteq [n]^d, |P| = n\}.
   \]

We will make use of the method of \df{canonical intervals}. For a non-negative
integer \(k\), we define $\CC_k$ to consist of consecutive disjoint intervals (restricted to the integers) of size $2^k$. For a given \(n\), we can take the set
system \(\CC\) to consist of all the intervals in
\(
\bigcup_{k = 0}^{\lfloor \log_2 n \rfloor}\CC_k
\)
that are entirely contained in \([n]\). \(\CC\) has small degree:
because each \(i\in [n]\) is contained in a unique canonical interval
in \(\CC_k\), the degree of \(\CC\) is bounded as \(\ssdeg(\CC) \le 1+ \log_2 n\). It is an easy exercise to show that every anchored interval $I = \{1, \ldots, i\} \subseteq [n]$ can be decomposed into the union of at most $1 + \log_2 n$ disjoint intervals from $\CC$, each from a different $\CC_k$.

A typical application of canonical intervals is given in the next theorem, which applies the Beck-Fiala theorem to bound the discrepancy of boxes. We only give a proof sketch and leave the details as an exercise. 

\begin{theorem}\label{thm:ch2-bftusnady}
  For any integers \(n \ge 1\) and \(d \ge 2\),
  \(\disc(\RR_d, n) \lesssim \log(2n)^{2d}\), and, moreover, for any
  fixed \(d\), a coloring of any set \(P\) of \(n\) points in \(\R^d\)
  achieving this discrepancy can be computed in time polynomial in
  \(n\).
\end{theorem}
\begin{proof}[Proof Sketch]
    For simplicity, we sketch the proof in the case $d=2$, which already captures the key ideas.
    As we saw above it is enough to bound the discrepancy of the set system $\AA_2|_P$ induced by anchored boxes on arbitrary $n$-point subset $P$ of $[n]^2$. In order to use the Beck-Fiala theorem, we instead consider a set system with low degree. To this end, let $\mathcal{CR}_2 := \CC\times \CC$ be the set system of canonical rectangles, i.e., rectangles whose sides are given by canonical intervals, and let $\SS := \mathcal{CR}_2|_P$ be its restriction to $P$. Then, since $\CC$ has degree at most $1+\log_2 n$, $\mathcal{CR}_2$, and therefore also $\SS$, has degree at most $(1+\log_2 n)^2$. The Beck-Fiala theorem (Theorem~\ref{thm:ch2-bf}) then implies that $\disc(\SS) \lesssim \log(2n)^2$.
    
    Moreover, since any anchored interval in $[n]$ can be written as the disjoint union of at most $1+\log_2 n$ canonical intervals, we have that any anchored rectangle $A \in \AA_2$ in $[n]^2$ can be written as the disjoint union of at most $(1+\log_2 n)^2$ canonical rectangles from $\mathcal{CR}_2$. Restricting to $P$, this means that any set in $\AA_2|_P$ is the disjoint union of at most $(1+\log_2 n)^2$ sets from $\SS$. Therefore, the coloring that achieves discrepancy $\disc(\SS) \lesssim \log(2n)^2$ also gives $\disc(\AA_2|_P) \lesssim \log(2n)^4$. 
\end{proof}

We can get a better bound by encoding one dimension of the boxes as permutations, only using canonical intervals on the other dimensions, and using Theorem~\ref{thm:ch2-kperm} in place of the Beck-Fiala theorem. 

\begin{theorem}\label{thm:ch2-tusnady}
  For any integers \(n \ge 1\) and \(d \ge 2\),
  \(\disc(\RR_d, n) \lesssim \log(2n)^{2d-1}\), and, moreover, for any
  fixed \(d\), a coloring of any set \(P\) of \(n\) points in \(\R^d\)
  achieving this discrepancy can be computed in time polynomial in
  \(n\).
\end{theorem}

\begin{proof}[Proof Sketch]
  Once again, we focus on the case $d=2$, leaving the extension to higher dimensions as an exercise. As before, we only need to bound the discrepancy of \(\AA_d|_P\) for an arbitrary $n$-point set \(P \subseteq [n]^2\).  To reduce this
  problem to the discrepancy of permutations, we construct
  \(m \le 1 + \log_2 n\) permutations \(\pi_1, \ldots, \pi_m\)
  of \(P\), one for each integer \(k\) between \(0\) and
  \(\lfloor \log_2 n \rfloor\). We will think of each permutation as an
  ordering of \(P\). The permutation $\pi_k$ is defined so that all points of $P$ whose $y$-coordinates are in the same canonical interval in $\CC_k$ are listed contiguously, and in increasing order of the $x$-coordinate  (see Figure~\ref{fig:ch2-boxes-perm} for an example).

  \begin{figure}[htp]
    \centering
    \includegraphics{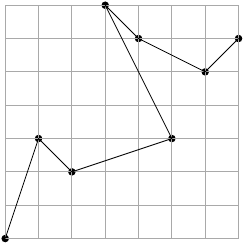}
    \caption{The permutation \(\pi_2\) for a pointset in \([8]^2\), listing the
      points with \(y\)-coordinate in \(\{1,2,3,4\}\) first, in
      increasing order of their \(x\)-coordinate, and then listing the
      points with \(y\)-coordinate in \(\{5,6,7,8\}\). The bottom-left
      corner is \((1,1)\).}
    \label{fig:ch2-boxes-perm}
  \end{figure}

  If \(\SS\) is the set system induced by \(\pi_1, \ldots,
  \pi_m\) on \(P\), then, by Theorem~\ref{thm:ch2-kperm},
  \[
  \disc(\SS) \lesssim m \log(2n) \lesssim \log(2n)^2.
  \]
  The theorem is now implied by the bound
  \begin{equation}\label{eq:ch2-tusnady}
    \disc(\AA_d|_P, \col) \le 2(1 + \log_2 n) \disc(\SS, \col),
  \end{equation}
  valid for any coloring $\col$ of $P$.
  Notice that, since we can express any interval of a permutation (i.e., a set
  of the type \(\{\pi(i), \ldots, \pi(j)\}\) for \(i \le j\)) as the
  difference of two prefixes, the discrepancy of any interval of any
  of the permutations \(\pi_k\) is bounded by \(2\disc(\SS,
  \col)\).
  To prove \eqref{eq:ch2-tusnady}, it is then enough show that every anchored box in $\AA_2|_P$ can be written as the disjoint union of $1+\log_2 n$  intervals of the permutations $\pi_1, \ldots, \pi_m$. Take an arbitrary anchored rectangle $\rbox_{0,b} \cap P$. Decomposing its vertical side $[0, b(2)]$ into a disjoint union of canonical intervals allows us to write $\rbox_{0,b}$ as a disjoint union $(\rbox'_1 \cap P) \cup \ldots\cup (\rbox'_\ell\cap P)$ of  $\ell \le 1+\log_2 n$ rectangles, where the vertical side of each $\rbox'_i$ is a canonical interval, and the horizontal side is $[0,b(1)]$. The vertical side of $\rbox'_i$ being a canonical interval implies that all points in $\rbox'_i\cap P$ are listed contiguously in one of the permutations $\pi_k$, i.e., they form an interval of $\pi_k$. This proves our claim, and therefore inequality \eqref{eq:ch2-tusnady} and the theorem.
\end{proof}

\section{Vector Balancing}

Suppose that \(\|\cdot\|_X\) is a norm on \(\R^m\), and recall our
notation for the discrepancy of an \(m\times n\) matrix \(\mat\)
measured in terms of \(\|\cdot\|_X\):
\[
\disc_X(\mat) \eqdef \min_{\col \in \{-1,+1\}^n} \|\mat \col\|_X.
\]
Recall further that the vector balancing constant of a set \(K
\subseteq \R^m\) with respect to the normed space \(X\) is
\[
  \vb(K, X) \eqdef \sup\{\disc_X((u_i)_{i=1}^n): u_1, \ldots, u_n \in
  K, n \in \N\}.
\]

It is natural to set \(K\) to be the unit ball \(B_X = \{u \in \R^m:
\|u\|_X \le 1\}\) of \(X\), and ask how \(\vb(B_X, X)\) can be bounded
in terms of \(m\) and of geometric properties of \(B_X\). The triangle
inequality shows that \(\disc_X(\mat) \le \sum_{i = 1}^n \|u_i\|_X \le
n\) for any matrix \(\mat = (u_i)_{i = 1}^n\) with columns in \(B_X\), 
and this bound holds for \emph{any} coloring \(\col \in \{-1,
+1\}^n\). This bound is weak when \(n\) is much bigger than \(m\), but
the linear algebraic method
allows us to reduce to the
\(n \le m\) case, and show a bound for an arbitrary normed space \(X\)
that only depends on the dimension. 
\begin{theorem}\label{thm:ch2-gen-vb}
  Let \(\|\cdot\|_X\) be a norm on \(\R^m\) with unit ball
  \(B_X\).
  Then \(\vb(B_X, X) \le 2m\), and moreover, for any matrix
  $M$ with columns \(u_1, \ldots, u_n \in B_X\), a
  coloring \(\col\) achieving \(\disc_X(M) \le 2m\) can
  be computed in polynomial time. 
\end{theorem}
\begin{proof}
  To show that \(\disc_X(\mat) \le
  2m\), by Theorem~\ref{thm:ch2-col-reduction} it suffices to
  show that \(\herdisc_X(\mat,m) \le  m\). The latter follows trivially from the triangle inequality.
\end{proof}

\remark The bound above can be improved to \(\vb(B_X, X) \le m\) by modifying the proof of Theorem~\ref{thm:ch2-col-reduction} in this special case. In particular, rounding the coordinates of the vector $y \in [-1,+1]^J$ in the proof of Theorem~\ref{thm:ch2-col-reduction} to the nearest element of $\{-1,+1\}$, instead of using discrepancy based rounding. We leave the details as an exercise. 

This improvement to Theorem~\ref{thm:ch2-gen-vb} is tight for the \(\ell_1^m\) norm, since
\(\disc_{\ell^m_1}(I) = m\), where \(I\) is the \(m\times m\) identity
matrix. On the other hand, a better bound is possible for every
\(\ell^m_p\) norm when \(p > 1\), as shown in the next theorem. 

\begin{theorem}\label{thm:ch2-ellp-vb}
  Let
  \(B_p^m =\{u\in\R^m:\|u\|_p\le 1\}\) be the unit ball of 
  \(\ell_p^m\). 
  Then,  \(\vb(B_p^m, \ell_p^m) \lesssim m^{1/p},\) for \(1 \le p \le 2\),
  and  \(\vb(B_p^m, \ell_p^m) \lesssim
  \min\{\sqrt{p},\sqrt{\log 2m}\}\sqrt{m}\) for \(p \ge 2\).
\end{theorem}
\begin{proof}
  As in the proof of Theorem~\ref{thm:ch2-gen-vb}, let \(u_1, \ldots,
  u_n \in B_p\), and define the matrix \(\mat \eqdef (u_i)_{i =    1}^n\). 
  By Theorem~\ref{thm:ch2-col-reduction}, it is enough to bound
  \(\herdisc_X(\mat(*,J))\) for any \(J\subseteq [n]\) of size at most
  \(m\), i.e., it suffices to show the discrepancy bounds
  with \(n\) in place of \(m\):
  \begin{align}
    &1 \le p \le 2 \implies \disc_{\ell_p^m}(\mat) \lesssim n^{1/p},\\
    &p \ge 2 \implies \disc_{\ell_p^m}(\mat) \lesssim \min\{\sqrt{p},\sqrt{\log 2m}\}\sqrt{n}.
  \end{align}

  Let us start with the Euclidean case \(p=2\). Then, for \(\col \in \{-1,+1\}^n\)
  chosen uniformly at random, we have the parallelogram law
  \begin{equation}\label{eq:ch2-paralellogram}
    \E\|\mat \col\|_2^2 = 
    \E \Big\|\sum_{i = 1}^n \col(i)u_i\Big\|_2^2 = \sum_{i = 1}^n
    \|u_i\|_2^2 \le n.
  \end{equation}
  So there must then exist at least one coloring \(\col\) such that
  \(\|\mat \col\|_2^2 \le n\), and thus \(\disc_{\ell_2^m}(\mat) \le
  \sqrt{n}\).

  Interestingly, an approximate version of the upper bound in
  \eqref{eq:ch2-paralellogram} holds for any \(p \ge 2\) by the
  type-\(2\) inequality for \(\ell_p\)~\cite[Chapter
  9]{LedouxTalagrand91}. The inequality states that, for a uniformly
  random \(\col \in \{-1,+1\}^n\),
  \begin{equation}\label{eq:ch2-type2}
    \E \Big\|\sum_{i = 1}^n \col(i)u_i\Big\|_p^2 \lesssim {p} \sum_{i = 1}^n
    \|u_i\|_p^2 \le p n.
  \end{equation}
  This implies that, for \(p \ge 2\),
  \(
  \disc_{\ell_p^m}(\mat) \lesssim \sqrt{pn}.
  \)
  
  By the inequalities \(\|x\|_q \le \|x\|_p \le m^{\frac1p -
    \frac1q} \|x\|_q\), which hold for any \(1 \le p \le q \le \infty\), all \(\ell_p^m\) norms are equivalent up to a constant
  for \(p \ge \log m\). Therefore, for \(p \ge \max\{2,\log
  m\}\), we also have the stronger bound
  \(
  \disc_{\ell_p^m}(\mat) \lesssim \sqrt{\log(2m) n}.
  \)
  This takes care of the case \(p \ge 2\).

  The appropriate analog of the type-2 inequality~\eqref{eq:ch2-type2} for \(1 \le p < 2\) is the type-\(p\) inequality, which states
  that for a uniformly random \(\col \in \{-1,+1\}^n\),
  \[
    \E \Big\|\sum_{i = 1}^n \col(i)u_i\Big\|_p^p \lesssim \sum_{i = 1}^n
    \|u_i\|_p^p \le n.
  \]
  This implies that, for \(1 \le p \le 2\), we have 
  \(
  \disc_{\ell_p^m}(\mat) \lesssim n^{1/p}.
  \)
\end{proof}


The example of the \(m\times m\) identity matrix again shows that the
bounds in Theorem~\ref{thm:ch2-ellp-vb} are tight up to constants for
\(1 \le p \le 2\). The example of an appropriately scaled Hadamard matrix (i.e., an $m\times m$ matrix $M$ with entries in $\{-1,+1\}$ and pairwise orthogonal rows and columns) shows that the bounds are also tight up to the \(\max\{\sqrt{p},\sqrt{\log 2m}\}\)
term for \(p \ge 2\): we leave the details of the argument as an exercise.

As we will see later 
however, a better bound holds for \(p = \infty\) (equivalently, for $p \ge \log m$): this is a landmark result of Spencer (Theorem~\ref{thm:ch3-spencer}), which we prove in Chapter~\ref{ch:partial}.

Let us finally remark that the proof of the Beck-Fiala theorem
(Theorem~\ref{thm:ch2-bf}) can be generalized to show that \(\vb(B_1^m,
\ell_\infty^m) \le 2\). We leave this generalization as an
exercise. 

We recall here the  Koml\'os problem
(Problem~\ref{prob:komlos} from the Introduction), asking to determine
if \(\vb(B_2^m, \ell_\infty^m) \lesssim 1\). A positive resolution of this problem also implies a
positive resolution of the Beck-Fiala conjecture mentioned earlier.
In general, it is an interesting problem to determine how the vector
balancing constant \(\vb(B_p^m, \ell_q^m)\) depends on \(m\) when
\(p\) is much smaller than \(q\).

\section{Prefix Discrepancy and Signed Series}

Next, we recall the Steinitz problem, where given \(n\) ``short''
vectors \(u_1, \ldots, u_n \in \R^m\) whose sum is \(0\), we ask
if the vectors can be rearranged so that all partial sums are
bounded independently of \(n\). As we saw in the Introduction, the Steinitz problem is closely related to the
prefix discrepancy of a matrix \(M = (u_i)_{i = 1}^n\), defined as
\[
  \pdisc_X(M) \eqdef \min_{\col \in \{-1, +1\}^n}
  \max_{j=1}^n \Big\|\sum_{i = 1}^j \col(i) u_i\Big\|_X,
\]
where \(\|\cdot\|_X\) is a norm. Namely, by
Theorem~\ref{thm:ch1-chobanyan} (Chobanyan's transference theorem), 
the Steinitz
constant of \(M = (u_i)_{i = 1}^n\), 
\[
  \st_X(M) = \min_{\pi \in S_n} \max_{j = 1}^n \Big\|\sum_{i = 1}^j
     u_i - \frac{j}{n} \sum_{k = 1}^n u_k\Big\|_X,
\]
satisfies the bound
\[
  \st_X(M) \le \max_{\pi \in S_n}\pdisc_X(\pi \circ M),
\]
where \(\pi \circ M \eqdef (u_{\pi(i)})_{i = 1}^n\) is the rearrangement of
the columns of \(M\) according to \(\pi\). In other words, the
worst-case prefix discrepancy over re-arrangements of the columns of
\(M\) bounds the Steinitz constant.

Let us recall the notation for the worst-case Steinitz constant
and the worst-case prefix discrepancy: for a set \(K
\subseteq \R^m\), and a normed space \(X = (\R^m,\|\cdot\|_X)\), we  have
\begin{align*}
  &\st(K, X) \eqdef \sup\{\st_X((u_i)_{i=1}^n): u_1, \ldots, u_n \in K, n \in \N\},\\
  &\sser(K, X) \eqdef \sup\{\pdisc_X((u_i)_{i=1}^n): u_1, \ldots, u_n \in K, n \in \N\}.
\end{align*}
Chobanyan's lemma then implies that \(\st(K,X) \le \sser(K,X)\) for any
\(K\) and \(X\). 
The next theorem bounds these by $2m$.
This significantly strengthens Theorem~\ref{thm:ch2-gen-vb}, by extending the bound there for discrepancy to that for prefix discrepancy.

\begin{theorem}\label{thm:ch2-gen-st}
  Let \(\|\cdot\|_X\) be a norm on \(\R^m\) with unit ball
  \(B_X\).
  Then
  \[
    \st(B_X, X) \le \sser(B_X, X) \le 2m.
  \]
  Moreover, for any matrix \(M\) with columns \(u_1, \ldots, u_n \in B_X\), a
  coloring \(\col\) achieving \(\pdisc_X(M) \le 2m\) can
  be computed in polynomial time. 
\end{theorem}

\paragraph{The Algorithm.} We use our template via Lemma~\ref{lm:ch2-linalg-gen}. However, unlike previous examples, we will choose the active coordinates \(\act\) more carefully (among the coordinates that have not yet reached \(+1\) or \(-1\)). 
Also, we can assume that \(m < n\),
otherwise the theorem holds trivially for any arbitrary coloring by the triangle inequality for \(\|\cdot\|_X\).

As usual, let $\col_{t-1}$ be the coloring at the  beginning of step $t$, and 
initially, \(\col_0 \eqdef
  0\).  At each step \(t
  \ge 1,\) we choose \(\act_{t}\) to consist of the first \(m+1\) indices
  \(i_1, \ldots, i_{m+1}\) satisfying \(|\col_{t-1}(i_j)|<1\) for all
  \(j \in [m+1]\). In particular, at $t=1$, \(\act_1 \eqdef [m+1]\). 
In other words, $A_t$ can be viewed as a sliding window of the first $m+1$ strictly fractional coordinates.  
At step $t$, applying Lemma~\ref{lm:ch2-linalg-gen} to \(\mat\),
  \(\col_{t-1}, A_t \)
  gives a fractional coloring
  \(\col_{t}\) such that
  \begin{enumerate}
  \item \(\mat \col_t = \mat x_{t-1}\);
  \item for all \(i \in [n]\setminus\act_{t}\), \(\col_t(i) = \col_{t-1}(i)\);
  \item at least one element \(i \in \act_{t}\) satisfies
    \(|\col_t(i)| = 1\).
  \end{enumerate}

We terminate whenever fewer
than \(m+1\) such indices are left -- once that happens,
  say at step \(T\), we set \(\col(i) \eqdef \col_T(i)\) whenever
  \(|\col_T(i)| = 1\), and otherwise we set \(\col(i)\) arbitrarily, and output $\col$.

\paragraph{Analysis.} As the rank of \(\mat(*,\act_{t})\) is at most \(m < |\act_{t}| = m+1\), at each step $t$, at least one variable in $A_t$ reaches $\pm 1$, and the algorithm terminates in at most $n$ steps.

To analyze the (prefix) discrepancy, let us fix some arbitrary \(j \in [n]\), and track how the discrepancy 
$d_t = \|\sum_{i=1}^j x_t(i) u_i \|_X$ evolves over time. Clearly, $d_0=0$ as each $x_0(i)=0$ initially.
The key observation is that as long as $A_t \subseteq [j]$, i.e., the sliding window $A_t$ is contained in the first $j$ coordinates, the discrepancy $d_t$ of this prefix stays zero.

To see this, let $\max(\act_t)$ denote the largest coordinate in $\act_t$, and let $t_j$ be the last time step when
  \( \max(\act_t)\leq j\).   
  As \(\max (\act_s)\) is non-decreasing in $s$ by our choice of
  \(\act_s\), we have that, for all \(1 \le s \le t_j\), \(\act_{s} \subseteq
  [j]\). Since the colorings \(\col_s\) and \(\col_{s-1}\) are equal
  outside \(\act_{s}\),
  \[
    \sum_{i=1}^j (\col_s(i) - \col_{s-1}(i))u_i =  \mat (\col_s -
    \col_{s-1}) = 0,
  \]
  and thus $d_{s}-d_{s-1} = 0$.
  Adding up these equalities over \(s \in [t_j]\), and using
  \(d_0 = 0\) yields $d_{t_j}=0$. 

  To finish the analysis, notice that the final coloring \(\col\) can differ 
  from the coloring
  \(\col_{t_j}\)
  in at most \(m\) coordinates in
  \([j]\). This is because for any coordinate \(i \le j\),  \(\col(i)\) can differ from  \(\col_{t_j}(i)\)
  only if \(i \in \act_{t_j+1} \cap [j]\), and there are only
  \(m\) such \(i\) by our choice of $t_j$.
  As \(|\col(i) -\col_{t_j}(i)|\leq 2\) for any such $i$ and $\|u_i\|_X\leq 1$,  by the triangle inequality we have that $d_t \leq 2m$. This completes the analysis of the algorithm.
  

\remark
This bound can be slightly improved to \(\sser(B_X, X)\le 2m-1\), by optimizing the proof above further.
For $\st(B_X, X)$, a bound of \(\st(B_X, X) \le m\) can also be shown for all
normed spaces \(X\) using a different linear algebraic argument. Both bounds are tight up to a constant factor for \(X =
\ell_1^m\). 

Recall the conjecture from the introduction that both bounds can be improved to
\(O(\sqrt{m})\) for \(X = \ell_2^m\) and \(X = \ell_\infty^m\). We remark that
there is \emph{no \(m\)-dimensional normed space}, let alone these
two, for which a bound that is asymptotically better than \(O(m)\) has been shown. 

\section{Bibliographic Notes}

Theorem~\ref{thm:ch2-bf} is due to Beck and Fiala~\cite{BF81}, who
also pioneered the linear algebraic method presented in this
chapter. Exercise~\ref{ex:ch2-bf-improvement} is due to Bednarchak and
Helm~\cite{BH97-beckfiala}. The best known upper bound on discrepancy purely in terms of
the degree is $2\ssdeg(\SS) - \log^*(\ssdeg(\SS))$ and is due to
Bukh~\cite{Bukh16}. Very recently, Bansal and Jiang showed an $O(\sqrt{\ssdeg(\SS)})$ bound whenever $\ssdeg(\SS) = \Omega(\log^2n)$. For general $\ssdeg(\SS)$, they obtain an $\widetilde{O}(\sqrt{\ssdeg(\SS)} + \sqrt{\log n})$ bound, where $\widetilde{O}(\cdot)$ hides factors polynomial in $\log \log n$.
Their techniques also give an $O(\log^{1/4} m)$ discrepancy bound for the Koml\'{o}s problem.

The application of the Beck-Fiala theorem to the discrepancy of
axis-aligned boxes in Theorem~\ref{thm:ch2-bftusnady} is due
to Beck~\cite{beck-rect}. Theorems~\ref{thm:ch2-kperm}~and~\ref{thm:ch2-tusnady} are due to
Bohus~\cite{Bohus}. Beck asked whether the $\log n$ term can be
removed in Theorem~\ref{thm:ch2-kperm} for $3$ (or more)
permutations. He also proved that the discrepancy of $2$ permutations
is at most $1$: see~\cite{spencer-lectures}. Newman, Neiman, and Nikolov
answered Beck's question negatively and showed that the $O(\log n)$ bound is
tight for $3$ permutations~\cite{beck3perm}. Franks gave a simplified proof of this
result~\cite{Franks21}.  While the dependence on $n$ in
Theorem~\ref{thm:ch2-kperm} is tight, it is possible to improve the
dependence on $k$~\cite{SST-perms}, as we will see in the next
chapter. An upper bound on the discrepancy of $k$ permutations that is
simultaneously tight in terms of both $k$ and $n$ is not known.

Theorems~\ref{thm:ch2-gen-vb}~and~\ref{thm:ch2-gen-st}
are due to B\'ar\'any and Grinberg~\cite{baranygrinberg}. 
Sevastjanov, and Sevastjanov and Grinberg used linear algebraic
arguments to give better bounds for the Steinitz constant, showing
that it is bounded by $d$ in any $d$-dimensional normed
space~\cite{Sev78,GS80}. Their proof develops a linear algebraic
argument directly for the Steinitz problem, rather than going through
discrepancy, and gives a better bound than
Theorem~\ref{thm:ch2-gen-st}. This argument, and more applications of
linear algebraic methods to vector balancing and related problems can
be found in B\'ar\'any's survey~\cite{B08}. 

Linear algebraic arguments similar to the ones in this chapter have
also found applications in the design of approximation algorithms for
computationally hard problems. An early example is the work of
Karmarkar and Karp on the bin packing
problem~\cite{KarmarkarK82}. Many more algorithms using these
techniques are surveyed by Lau, Ravi, and
Singh~\cite{iterative-book}. 














\chapter{Partial Coloring Methods}\label{ch:partial}

The linear algebraic method from Chapter~\ref{ch:linalg} is a powerful
tool for taking advantage of the structure of a set system when
proving discrepancy upper bounds. We saw that it can be used to give significantly improved
discrepancy upper bounds 
over those 
achieved by random colorings for many natural set systems, e.g.,
set systems of bounded degree, and set systems induced by rectangles
or by permutations. 

However, it is not at all clear how to use this method to match, let alone improve the upper bound $\disc(\SS)
\lesssim \sqrt{n \log(2n)}$ for a set system $\SS$ of $O(n)$ subsets of a
universe of $n$ elements, achieved by a random coloring.
In this
chapter, we develop the method of partial colorings, which combines the
power of randomization with the ability of the linear algebraic method
to take advantage of weak dependencies between sets in a set
system. 
We give a number of applications of the method, including a
proof of Spencer's Six Standard Deviations Suffice Theorem showing the discrepancy upper bound
above can be improved to $\disc(\SS) \lesssim \sqrt{n}$.

\section{Beck's Partial Coloring Lemma}

To build some intuition for partial coloring methods, we start with
the simplest of them: the partial coloring lemma of Beck. 

A general
theme of discrepancy upper bounds is decomposing a set system into
a potentially large number of small sets, and a small number of
potentially large sets. We saw this in the proof of the Beck-Fiala
theorem in Chapter~\ref{ch:linalg}, where we showed that a set system
\(\SS\) of maximum degree \(\ssdeg(\SS)\) on a universe of size \(n\)
has at most \(\frac{n}{c}\) sets of size greater than
\(c\ssdeg(\SS)\). This is useful, since it is easier to control the
discrepancy of small sets, and when the large sets are not too many,
keeping their discrepancy low imposes a small number of constraints on
our colorings. In Chapter~\ref{ch:linalg} we used linear algebra to
formalize this intuition. Beck's lemma below gives another
formalization, based on a beautiful combination of combinatorial and
probabilistic reasoning.

In preparation, we prove a simple combinatorial lemma, showing that in
any large enough set of colorings, at least two colorings must
disagree on many elements. Stronger statements are known, but we state
one with an elementary proof, as it suffices for our needs.

\begin{lemma}\label{lm:ch3-hamming}
  Suppose \(\mathcal{C} \subseteq \{-1,+1\}^n\) is a set of colorings
  of size \(|\mathcal{C}| \ge 2^{n/2}\). Then there exist two colorings
  \(\col,\col' \in \mathcal{C}\) such that
  \[
    |\{i: \col(i) \neq \col'(i)\}| \ge n/10.
  \]
\end{lemma}
\begin{proof}
  Let us fix any $\col \in \mathcal{C}$. The
  number of colorings \(\col'\in \{-1,+1\}^n\) for which \(|\{i
  : \col(i) \neq \col'(i)\}| \le n/10\) is
  \[
    \sum_{i = 0}^{\lfloor n/10\rfloor} {n \choose i}
    \le
    2^{n H(1/10)} < 2^{n/2} \le |\mathcal{C}|.
  \]
  where \(H(p) \eqdef -p \log_2 p - (1-p)\log_2(1-p)\) is the binary
  entropy function, and we used standard estimates on the volumes of
  Hamming balls (e.g.~Proposition 3.3.3 in~\cite{guruswami2012essential}).
  Then there must be at least one $\col' \in \mathcal{C}$
  for which $|\{i: \col(i) \neq \col'(i)\}| \ge
  n/10$.
\end{proof}

We are now ready to prove Beck's partial coloring lemma. The lemma applies to set systems $\SS$ partitioned into a set system $\SS_1$, typically consisting of the large sets, which should be relatively few, and a set system $\SS_2$, typically consisting of the small sets. The lemma guarantees that one can find a partial coloring (\ie, a coloring that's allowed to assign color $0$ to a constant fraction of elements), which has discrepancy $0$ on sets in $\SS_1$, and has discrepancy matching the one achieved by a random coloring on $\SS_2$.

\begin{lemma}\label{lm:ch3-beck}
  Let $\SS_1$ and $\SS_2$ be set systems on a set $\uni$ of size $n
  \ge 3$
  such that  $|S| \le s$ for all $S \in \SS_2$, and
  \[
  \prod_{S \in \SS_1}{(2|S| + 1)} \le 2^{(n-1)/5}.
  \]
  Then there exists a {partial coloring} $\col:\uni \to \{-1, 0, +1\}$
  such that
  \begin{enumerate}
  \item $\col(S) = 0$ for all $S\in \SS_1$;
  \item $|\col(S)| \lesssim \sqrt{s \log(2|\SS_2|)}$ for
    all $S \in \SS_2$;
  \item \(|\{\elem \in \uni: \col(\elem) \neq 0\}| \ge n/10.\)
  \end{enumerate}
  \end{lemma}
\begin{proof}
  Let $\col'$ be chosen uniformly at random from $\{-1, +1\}^\uni$. By the union bound and Hoeffding's inequality,
  \[
    \Pr\left[\max_{S \in \SS_2} |\col'(S)| > t \right]
    \le 2|\SS_2| e^{-t^2/2s}.
  \]
  When \(t \eqdef \sqrt{2s \ln(4|\SS_2|)},\) this is at most $1/2$, and therefore the set
  \[
    \mathcal{C} \eqdef \left\{\col'\in \{-1,+1\}^\uni : \disc(\SS_2,\col')\le \sqrt{2s \ln(4|\SS_2|)}\right\}
  \]
  has size $|\mathcal{C}| \ge 2^{n-1}$. 

  On the other hand, the number of different values that the vector
  $(\col'(S))_{S \in \SS_1}$ can take as \(\col'\) ranges over
  \(\{-1,+1\}^\uni\) is at most $\prod_{S \in
    \SS_1}{(2|S| + 1)}$, which is bounded by $2^{(n-1)/5}$ by
  assumption. 
  By the pigeonhole
  principle, there must be a set $\mathcal{C}' \subseteq \mathcal{C}$
  of colorings of $\uni$ of size  $|\mathcal{C}'| \ge
|\mathcal{C}|/2^{(n-1)/5} \ge 2^{4(n-1)/5}$ such that any two
  $\col', \col'' \in \mathcal{C}'$ satisfy $\col'(S) = \col''(S)$ for
  all $S\in \SS_1$. As \(4(n-1)/5 \ge n/2\) for \(n
  \ge 3\), applying Lemma~\ref{lm:ch3-hamming} to \(\mathcal{C}'\)
  shows that there exist two colorings \(\col',\col''\in
  \mathcal{C}'\) that differ on at least
  \(n/10\) elements of \(\uni\). 
  We can now take $\col \eqdef \frac12(\col' - \col'')$ as our
  partial coloring. As \(\col(\elem) \in \{-1,1\} \) whenever \(\col'(\elem)
  \neq \col''(\elem)\), \(\col\) colors properly at least
  \(n/10\) elements of \(\uni\), proving the third property of \(\col\). The
  first property holds because \(\col'(S) = \col''(S)\) for all \(S
  \in \SS_1\), so \(\col(S) = \frac12(\col'(S) - \col''(S)) = 0\). The
  second property holds since \(\col',\col''\in \mathcal{C}\), so for
  any \(S \in \SS_2\) we have
  \[
    |\col(S)| \le \frac12(|\col'(S)| + |\col''(S)|) \le \sqrt{2s \ln(4|\SS_2|)}. \qedhere
  \]
\end{proof}

Note that the proof of
Lemma~\ref{lm:ch3-beck}, unlike the linear algebraic arguments in
Chapter~\ref{ch:linalg}, does not suggest an efficient algorithm to
find the partial coloring. The main issue is the use of the pigeonhole
argument on the exponentially large set of colorings \(\mathcal{C}\)
to construct $\mathcal{C}'$.

Lemma~\ref{lm:ch3-beck} is a refinement of the discrepancy bounds we can
get from a uniformly random coloring. In particular, if $\SS_1$ is
empty, then we can, with high probability, achieve the same
discrepancy on $\SS_2$ as in the lemma with a full uniformly random
coloring. Being able to set the discrepancy of some sets to $0$,
however, allows us to prove discrepancy bounds that are not achieved
by a random coloring (except with exponentially small probability). As an illustration, we
prove a bound for the Beck-Fiala problem. Later in the chapter
we will see more advanced techniques that allow proving sharper
bounds. 

\begin{theorem}\label{thm:ch3-bf-weak}
  For any set system \((\SS, \uni)\) with maximum degree
  \(\ssdeg(\SS)\) and \(|\uni| = n > 1\), 
  we have \(\disc(\SS) \lesssim \sqrt{\ssdeg(\SS)} \log(n)^{2}\). 
\end{theorem}
\begin{proof}
  Let us assume that \(n \ge 3\), as otherwise the theorem is
  trivially true. Similarly, we can assume
  that \(\ssdeg(\SS) \le n^2\), since the trivial upper
  bound of \(\disc(\SS) \le n\) holds for any coloring.
  As in the proof of Lemma~\ref{lm:ch2-bf-counting}, 
  the number of sets in \(\SS\) satisfies \(|\SS| \leq n\ssdeg(\SS) \le
  n^3\). 
  
  We define \(\SS_1 \eqdef \{\set \in \SS: |\set| > s\}\) and
  \(\SS_2 \eqdef \{\set \in \SS: |\set| \le s\}\) for a parameter
  \(s\) that we choose later. By Lemma~\ref{lm:ch2-bf-counting},
  \(
  |\SS_1| < n\ssdeg(\SS)/s,
  \)
  and, therefore,
  \[
    \prod_{\set \in \SS_1}{(2|\set| + 1)} \le (2n+1)^{n\ssdeg(\SS)/s}
    = 2^{n\log_2(2n+1)\ssdeg(\SS)/s}.
  \]
  In order to apply Lemma~\ref{lm:ch3-beck}, we need that the exponent
  on the right is at most \((n-1)/5\), so we set
  \(
  s \eqdef C\log_2(2n+1) \ssdeg(\SS),
  \)
  for a sufficiently large constant $C>0$.
  Then the lemma gives us a partial coloring \(\col:\uni \to \{-1, 0,
  +1\}\) with discrepancy
  \[
    \disc(\SS, \col) \lesssim  \sqrt{\ssdeg(\SS)\log_2(n) \log(2|\SS_2|)}
    \lesssim  \sqrt{\ssdeg(\SS)} \log(2n),
  \]
   where we used the trivial
  bound \(|\SS_2| \le |\SS| \le  n^3\). Moreover,
  Lemma~\ref{lm:ch3-beck} guarantees that the set \(\act \eqdef
  \{\elem \in \uni: \col(\elem) = 0\}\) satisfies \(|\act|
  \le 9n/10\).

  To obtain a full coloring, we can iterate this process. In
  particular, as $\ssdeg(\SS|_\act) \le \ssdeg(\SS)$, we can
  inductively
  find a coloring \(\col'\) of the restriction \(\SS|_\act\), and assign the color in $\col'$ to the elements that are given color $0$ by $\col$. As there are $O(\log n)$ iterations, the overall discrepancy can increase by a factor at most $O(\log n)$, giving the bound $O(\sqrt{\ssdeg(\SS)}\log(n)^2)$, as desired. 
\end{proof}

\section{A Stronger Partial Coloring Lemma, and Applications}

While Beck's partial coloring lemma is a powerful tool, it often gives
suboptimal bounds with large poly-logarithmic factors, as we saw in
the case of the Beck-Fiala problem. Intuitively, one reason for this
is that the lemma only allows setting two types of discrepancy bounds
for the partial coloring: a set either receives \(0\) discrepancy, or
the discrepancy it would get under a random coloring. More precise
bounds can be derived by allowing more freedom in setting discrepancy
bounds. In this section, we state a stronger partial coloring lemma
that follows this approach, and present some applications of it. We
prove the lemma later in the chapter, as a corollary of a more general
geometric result.

\subsection{Statement}

Next we state this stronger partial coloring lemma. Rather than
restrict ourselves to set systems, we adopt a more general linear
algebraic formalism. The proof of Lemma~\ref{lm:ch3-partial} is deferred to Section~\ref{sec:ch3-giann}.

\begin{lemma}\label{lm:ch3-partial}
  Let \(v_1, \ldots, v_m \in \R^n\) be such that \(\|v_i\|_2 \le 1\)
  for each \(i\), and let \(\lambda_1, \ldots, \lambda_m\) be
  positive reals. For any \(s > 0\), define the function
  \begin{equation}\label{eq:ch3-partial}
    \parcol(s) :=  \begin{cases}
      1.5\exp(-s^2/8) & s\ge 4\\
      \ln\left(20/s^2\right) & s < {4}
    \end{cases}.
  \end{equation}
  If \(\sum_{i=1}^m \parcol(\lambda_i) \le n/8\), then there exists a partial coloring
  \(\col \in \{-1,0,+1\}^n\) such that, for all \(i \in [m]\),
  \(|\ip{v_i,\col}| \le \lambda_i\|v_i\|_2\), and \(|\{j: \col(j) \neq 0\}|
  \ge n/10\). 
\end{lemma}

To apply Lemma~\ref{lm:ch3-partial} to a set system
$\SS \eqdef \{S_1, \ldots, S_m\}$, we can choose each $v_i$ to be the indicator vector of the set $S_i$, and, if we want the bound $|\col(S_i)| \le d_i$,
we can set $\lambda_i \eqdef d_i/\sqrt{|S_i|}$.

We can now compare
Lemmas~\ref{lm:ch3-beck}~and~\ref{lm:ch3-partial}. Suppose that
$\SS_1 = \{S_1, \ldots S_k\}$, and $\SS_2 = \{S_{k+1}, \ldots, S_m\}$,
where any set in $\SS_2$ has size at most $s$, and define
$v_1, \ldots, v_m$ as above. If we set
$\lambda_i = 1/\sqrt{2|S_i|}$ for $i \in [k]$, and
$\lambda_i = \sqrt{8\ln(12|\SS_2|)}$ for each $i \in \{k+1, \ldots, m\}$, then a partial
coloring $\col$ such that $|\ip{v_i,\col}| \le \lambda_i$ for all
$i \in [m]$ would satisfy $|\col(S)| < 1$, and thus
$\col(S) = 0$ for $S \in \SS_1$, and also
$|\col(S)| \lesssim \sqrt{s \log(2|\SS_2|)}$ for $S \in \SS_2$, as
in Lemma~\ref{lm:ch3-beck}. Lemma~\ref{lm:ch3-partial} guarantees such
a coloring if
\[
  \sum_{S \in \SS_1} \ln(40|S_i|) \le (n-1)/8.
\]
For comparison, the requirement in Lemma~\ref{lm:ch3-beck}
can be re-written as
\[
  \sum_{S \in \SS_1} \log_2(2|S| + 1) \le (n-1)/5,
\]
which is of the same form. Thus,  Lemma~\ref{lm:ch3-partial} allows us
to recover an analogous statement to
Lemma~\ref{lm:ch3-beck}. However, Lemma~\ref{lm:ch3-partial}  gives
us much more freedom in setting the $\lambda_i$ parameters. We
illustrate how this freedom translates to better discrepancy bounds
with the next two applications.

\subsection{Applications}


\paragraph{Spencer's Theorem.}
For our first application we consider an arbitrary set system
\((\SS, \uni)\) of \(m \eqdef |\SS|\) sets on a universe of size \(n
\eqdef |\uni|\). Recall that a random coloring achieves a discrepancy
of \(O(\sqrt{n \log m})\), and, in the absence of additional structure,
it is reasonable to conjecture that this is tight. Surprisingly, this
is not the case, and a celebrated result of Spencer shows that, for
any \(m \ge n\),
\(\disc(\SS) \lesssim \sqrt{n \log(2m/n)}\). In particular, when $m=O(n)$, we get a discrepancy bound of \(O(\sqrt{n})\),
showing that the logarithmic term incurred by a random coloring is
unnecessary. In fact, Spencer showed that when \(m = n\) the
discrepancy is at most \(6\sqrt{n}\), and his result is often
referred to as the Six Standard Deviations Suffice theorem.

Using Lemma~\ref{lm:ch3-partial}, we prove the following slightly stronger statement for matrices with bounded entries.
\begin{theorem}\label{thm:ch3-spencer}
  For any \(m \times n\) matrix \(\mat \in [-1, +1]^{m\times n}\),
  where $m \ge n$, we have
  \(\disc(\mat) \lesssim \sqrt{n\log(2m/n)}\). 
\end{theorem}
Spencer's result follows with $M$ as the incidence matrix of \(\SS\).

The core of the proof of Theorem~\ref{thm:ch3-spencer} is the
following partial coloring result.
\begin{lemma}\label{lm:ch3-spencer-partial}
For any \(m \times n\) matrix \(\mat \in [-1, +1]^{m\times n}\),
  where $m \ge n$, there exists a partial coloring
  \(\col \in \{-1,0,+1\}^n\) such that
  \(\|\mat \col\|_\infty \lesssim \sqrt{n\log(2m/n)}\) and
  \(|\{i: \col(i) \neq 0\}| \ge n/10\).
\end{lemma}
\begin{proof}
  We use Lemma~\ref{lm:ch3-partial}, where vector $v_i$ is the $i$-th
  row of $\mat$, and
  \(
  \lambda_i \eqdef \sqrt{8\ln(12m/n)}.
  \)
  Then the statement follows from Lemma~\ref{lm:ch3-partial}.
\end{proof}

We can show
Theorem~\ref{thm:ch3-spencer} by iteratively applying
Lemma~\ref{lm:ch3-spencer-partial}  to obtain
a full coloring. The key observation in the case of Spencer's theorem
is that the discrepancy of partial colorings decreases rapidly
with \(n\) and thus the discrepancy of the full coloring only increases by a constant factor.

\begin{proof}[Proof of Theorem~\ref{thm:ch3-spencer}]
  Let \(\col\) be as in Lemma~\ref{lm:ch3-spencer-partial}
  Then the set \(\act \eqdef
  \{\elem \in \uni: \col(\elem) = 0\}\) has size bounded by \(|\act|
  \le 9n/10\). We inductively find a coloring \(\col'\) of
  \(\mat(*,\act)\), and assign color \(\col(i)\) to any \(i \not \in
  \act\), and color \(\col'(i)\) to any \(i \in \act\). 
  The overall discrepancy is then
  bounded by
  \[
    \disc(\mat) \lesssim
    \sum_{k = 0}^{\lfloor\log_{10}n\rfloor} \sqrt{n (0.9)^k\log\left(\frac{2m}{(0.9)^kn}\right)}
    \lesssim \sqrt{n\log\left(\frac{2m}{n}\right)},
  \]
  which proves the theorem. 
\end{proof}

\paragraph{The Beck-Fiala Problem.}
Next, we see how to use Lemma~\ref{lm:ch3-partial} to give an
improvement to Theorem~\ref{thm:ch3-bf-weak}. 

\begin{theorem}\label{thm:ch3-bf}
  For any set system \((\SS, \uni)\) with maximum degree
  \(\ssdeg(\SS)\) and \(|\uni| = n\), 
  we have \(\disc(\SS) \lesssim \sqrt{\ssdeg(\SS)} \log(n)\). 
\end{theorem}
Again, we will iteratively apply the following partial coloring bound.

\begin{lemma}\label{lm:ch3-bf-partial}
  For any set system \((\SS, \uni)\) with maximum degree
  \(\ssdeg(\SS)\), there
  exists a partial coloring \(\col \in \{-1, 0,+1\}^n\) such that 
  \(\disc(\SS,\col) \lesssim \sqrt{\ssdeg(\SS)}, \) and \(|\{i: \col(i) \neq 0\}|
  \ge n/10\).
\end{lemma}
\begin{proof}
  Recall that, by Lemma~\ref{lm:ch2-bf-counting}, for any $s > 0$ we
  have that the number of sets in $\SS$ of size larger than $s$ is at
  most $n\ssdeg(\SS)/s$. We use this observation together with
  Lemma~\ref{lm:ch3-partial} to find a partial coloring.

  Let $\SS \eqdef \{S_1, \ldots, S_m\}$, and, as usual, let $v_i$ be the indicator vector of $S_i$.  Let
  \(\lambda_i := \sqrt{C\ssdeg(\SS)/|S_i|}\) for a large enough
  constant \(C > 0\) to be decided later.
  For any integer \(t\) between \(0\) and
  \(\lceil \log_2(n)\rceil\), let \(I(t)\) consist of those
  \(i \in [m]\) for which \(|S_i| \in (2^{t-1},2^{t}]\). By
  Lemma~\ref{lm:ch2-bf-counting}, $|I(t)| < 2n\ssdeg(\SS)/2^t$. We
  can then write the sum in Lemma~\ref{lm:ch3-partial} as
  \begin{align*}
    \sum_{i = 1}^m \parcol(\lambda_i)
    =
      \sum_{t = 0}^{\lceil \log_2(n)\rceil} \sum_{i \in I(t)} \parcol(\lambda_i)
    &<       \sum_{t = 0}^{\lceil \log_2(n)\rceil}
      \frac{2n\ssdeg(\SS)}{2^t} \parcol\left(\sqrt{C\ssdeg(\SS)/2^t}\right),
  \end{align*}
  where, for the  inequality, we used that the function
  \(\parcol(s)\) defined in \eqref{eq:ch3-partial} is non-increasing with $s$.
  It is therefore enough to verify this
  \begin{equation}\label{eq:ch3-komlos-sum}
    \sum_{t = 0}^{\lceil \log_2(n)\rceil}
    \ssdeg(\SS)2^{-t} \parcol\left(\sqrt{C\ssdeg(\SS)/2^t}\right) \le \frac{1}{16}
  \end{equation}
  for a large enough constant $C$. Consider the value $t^* :=
  \log_2\left(C\ssdeg(\SS)/16\right)$, which is the value of $t$ for which
  \(\sqrt{C\ssdeg(\SS)/2^t} = 4.\)
  For any $t > t^*$,  we have
  \[
    \ssdeg(\SS)2^{-t} \parcol\left(\sqrt{C\ssdeg(\SS)/2^{t}}\right)
    = \frac{16}{C}\cdot 2^{-(t-t^*)}\ln\left(5\cdot
          2^{t-t^*}/4\right).
  \]
  It is clear that the sum over terms $t > t^*$ is $O\left(1/C\right)$.

  For
  any $t \le t^*$, we have
  \[
    \ssdeg(\SS)2^{-t} \parcol\left(\sqrt{C\ssdeg(\SS)/2^{t}}\right)
    =
    \frac{24}{C}\cdot 2^{t^* - t}e^{-2^{t^*-t+1}}.
  \]
  Again, it is easily seen that the sum over terms  $t <
  t^*$ is $O(1/C)$. Thus \eqref{eq:ch3-komlos-sum} holds for a large enough $C$.
\end{proof}

The next exercise extends   Theorem~\ref{thm:ch3-bf} to a bound for
the Koml\'os problem.
\begin{exercise}\label{ex:ch3-komlos}
  Modify the proof of Lemma~\ref{lm:ch3-bf-partial} and
  Theorem~\ref{thm:ch3-bf} to show that for any $m\times n$ matrix
  \(\mat\) with columns $u_1, \ldots, u_n$ such that
  $\|u_j\|_2 \le 1$ for each $j$, we have
  \(\disc(\mat)\lesssim \log(2n).\)
\end{exercise}

Matou\v{s}ek's book~\cite{matousek2010geometric} gives a number of other applications of
Lemma~\ref{lm:ch3-partial}. 

\section{Proof of the Partial Coloring Lemma}
\label{sec:ch3-giann}

Our next goal is to prove Lemma~\ref{lm:ch3-partial}. The original
proofs of the lemma, due to a sequence of works by Spencer, Boppana, and Matou\v{s}ek, are
similar to the proof of Beck's partial coloring lemma
(Lemma~\ref{lm:ch3-beck}), and relied on a combination of probabilistic, combinatorial,
and information theoretic arguments. Here we describe a different
geometric approach, due to Giannopoulos~\cite{Gia93}, that
allows us to prove essentially equivalent results, while being more
flexible. In Section~\ref{sect:ch3-corr}, we will see that using the geometric
arguments directly can also sometimes give easier proofs of existence of
good partial colorings than applying  Lemma~\ref{lm:ch3-partial}. 
Moreover, we will see later in the book that the geometric viewpoint is 
helpful in developing efficient algorithms.

\subsection{A Geometric Partial Coloring Lemma}

The first step we take is to encode discrepancy constraints as a
centrally symmetric convex body \(K \subseteq \R^n\), \ie, a convex
set \(K\) with non-empty interior such that \(K = -K\). For example,
if \(\mat\) is an \(m \times n\) matrix, the constraint
\(\disc(\mat,\col) \le t\) can be encoded as \(\col \in tK\) where
\(K \eqdef \{\col \in \R^n: \|\mat\col\|_\infty \le 1\}\). The
convexity and symmetry of \(K\) follow from the triangle inequality,
homogeneity, and symmetry of the \(\ell_\infty^m\) norm. This encoding
of the problem is the key link between discrepancy and geometry
exploited in the arguments below.

Finding a good
partial coloring with discrepancy at most \(t\) is now equivalent to
finding some point of \(\{-1, 0, +1\}^n\) with many nonzero
coordinates inside of \(K\). Intuitively, this task should be easier
when \(K\) is large, and the measure of largeness that will be
fruitful for us is Gaussian measure: the probability that a standard
Gaussian random vector \(\grv\) lies in \(K\). We use the notation
\[
  \gamma^n(K) \eqdef \Pr[g \in K] = 
  {(2\pi)^{-n/2}}\int_K e^{-\|x\|_2^2/2}dx
\]
for the $n$-dimensional Gaussian measure.

We first need a standard lemma about how
Gaussian measure changes when we shift \(K\).

\begin{lemma}\label{lm:ch3-gauss-shift}
  Let \(K\) be a centrally symmetric set in \(\R^n\). Then,
  for any \(t \in \R^n\), we have
  \(
    \gamma^n(K+t) \ge e^{-\|t\|_2^2/2} \gamma^n(K).
  \)
\end{lemma}
\begin{proof}
  By the symmetry of \(K\), we have $K+t = -K+t$, so
  \begin{align}
    \gamma^n(K+t) &= \frac12 (\gamma^n(-K + t) + \gamma^n(K+t))\notag\\
                  &=  
                  {(2\pi)^{-n/2}}\int_K \frac{1}{2} (e^{-\|t+x\|_2^2/2} +e^{-\|t-x\|_2^2/2})\, dx \notag\\
                  & \geq 
                  {(2\pi)^{-n/2}} \int_K (e^{-(\|t+x\|_2^2 +\|t-x\|_2^2)/4})\, dx
                   \tag{A.M. $\geq$ G.M.} \\
                  & = 
                  {(2\pi)^{-n/2}} \int_K (e^{-\|t\|^2_2/2 -\|x\|_2^2/2}) \, dx =
                    e^{-\|t\|_2^2/2} \gamma^n(K). \notag \ \ \ \ \qedhere
  \end{align}
\end{proof}

We are now ready to prove a  partial coloring lemma due to
Giannopoulos's, that will in turn imply Lemma~\ref{lm:ch3-partial}.

\begin{lemma}\label{lm:ch3-gian}
  Suppose that \(K \subseteq \R^n\) is a centrally symmetric convex
  body such that \(\gamma^n(K) \ge  2^{-n/4}\).
  Then there exists a partial coloring \(\col \in \{-1, 0, +1\}^n\)
  such that \(\col \in 2K\) and \(|\{i: \col(i) \neq 0\}| \ge n/10\). 
  %
\end{lemma}
\begin{proof}
 The idea will be to find a set \(\mathcal{C} \subseteq \{-1,+1\}^n\) of size at
  least \(2^{n/2}\) such that any two colorings
  \(\col',\col'' \in \mathcal{C}\) satisfy
  \(\col'-\col'' \in O(1)\cdot K\). Then $\col = (\col'-\col'')/2$ for any \(\col',\col''
  \in \mathcal{C}\) that
  disagree in at least \(n/10\) coordinates gives the desired partial coloring.

  Let \(\grv\) be a standard Gaussian in \(\R^n\), and define
  \[
  C(g) \eqdef |\{\col \in \{-1,+1\}^n: g \in   \col/2 +K\}|.
  \]
  We have
  \begin{align*}
    \E[C(g)] &= \sum_{\col \in \{-1,+1\}^n} \gamma^n\left(\col/2 + K\right).
  \end{align*}
  By Lemma~\ref{lm:ch3-gauss-shift}, each term on the right
  is at least \(e^{-n/8}\gamma^n(K) \ge e^{-n/8}2^{-n/4}\), and thus
  \[
    \E[C(g)] \ge 2^{3n/4}e^{-n/8} \ge 2^{n/2}.\]
Pick some $g$ for which 
  \(C(g) \ge 2^{n/2}\) and let $\mathcal{C}$ be the set $\{\col \in \{-1,+1\}^n: g \in   \col/2 +K\}$, i.e., the set whose size is $C(g)$. By Lemma~\ref{lm:ch3-hamming},
  there must exist two \(\col',\col'' \in \mathcal{C}\) for which
  \(
  |\{i: \col'(i) \neq \col''(i)\}| \ge n/10.
  \)
  Then we take our partial coloring as \(\col \eqdef \frac12(\col' -
  \col'')\). Note that $g - \col'/2 \in K$, or equivalently, $ \col'/2 - g \in K$, and $g -  \col''/2
  \in K$ by the definition of
  $\mathcal{C}$ and the symmetry of $K$. This
  implies
  \[
    \col = \left(\frac12 \col'  - g\right) + \left(g - \frac12 \col''\right) \in 2K.
  \]
  Moreover, \(\col(i) \neq 0\) whenever \(\col'(i) \neq \col''(i)\),
  which happens for at least \(n/10\) choices of \(i\). This
  proves the lemma.
\end{proof}

\begin{remark}
  The constants in the statement of Lemma~\ref{lm:ch3-gian}, and also Lemma~\ref{lm:ch3-hamming}, 
were not optimized and are somewhat arbitrary.  In general, for any $c<1$ and $\gamma^n(K) \ge 2^{-cn}$ we can conclude
that there exists a partial coloring $x \in O_c(1) \cdot K$ with $\Omega_c(n)$ nonzero coordinates. 
The proof method, however, does
not extend to $K$ such that $\gamma^n(K) \le 2^{-n}$. 
\end{remark}

The following exercise suggests a strengthening of
Lemma~\ref{lm:ch3-gian}.
\begin{exercise}\label{ex:ch3-gian-unit}
  Let $K \subseteq \R^m$ be a symmetric convex body such that
  \(\gamma^m(K) \ge 2^{-n/4}\), and let \(\mat \eqdef (u_i)_{i =
    1}^n\) be an \(m\times n\) matrix whose columns satisfy
  \(\|u_i\|_2 \le 1\) for all \(i\).  Prove that there  exists a
  partial coloring \(\col \in \{-1, 0, +1\}^n\) such that \(\mat \col
  \in 2 K\) and \(|\{i: \col(i) \neq 0\}| \ge n/10\).
\smallskip 

  \textsc{Hint:} Define $C(g) \eqdef |\{\col \in \{-1,+1\}^n: g \in K +
  \frac12 \mat\col\}|$, and follow the proof of
  Lemma~\ref{lm:ch3-gian} to show that
  \[
    \E[C(g)] \ge 2^{-n/4}\sum_{\col \in \{-1,+1\}^n} e^{-\|\mat\col\|_2^2/8}.
  \]
  Now use the A.M.-G.M. inequality to lower bound the sum on
  the right. 
\end{exercise}

  Exercise~\ref{ex:ch3-gian-unit} allows us to define $K$ as a set of
good discrepancy vectors as opposed to a set of good fractional
colorings. For example, suppose that, in the context of the Koml\'os
problem, we have an $m \times n$ matrix $\mat$ all of whose columns
have $\ell_2$ norm at most $1$. Then setting $K = [-t,+t]^m$ for $t \ge \sqrt{2\ln(2m)}$ gives $\gamma^m(K) \ge \frac12 \ge
2^{-n/4}$ (for $n \ge 4$) and shows that there exists a partial
coloring $\col$ such that $\|\mat\col\|_\infty \lesssim
\sqrt{\log(2m)}$. In Chapter~\ref{ch:bana-intro} we will see a
different geometric argument by Banaszczyk that gives the same
conclusion for a \emph{full} coloring $\col$.

\paragraph{Comparison with Random Coloring.}
It is worth comparing Lemma~\ref{lm:ch3-gian} to what we can get from
a random coloring.
Suppose that \(K\) is a symmetric
convex set in \(\R^n\), and $\col \in \{-1,+1\}^n$ is
chosen uniformly at random, \ie, every coordinate is an independent
and uniform sample from $\{-1,+1\}$. Let us denote by $\|\cdot\|_K
\eqdef \inf\{t \ge 0: x \in tK\}$ the norm with unit ball $K$.
Because a uniformly random $\col \in \{-1,+1\}^n$ is more concentrated
than a standard Gaussian $\grv \in \R^n$, $\E\|\col\|_K$ is dominated
by $\E\|\grv\|_K$. More precisely, we have the inequality (inequality
(4.8) in Ledoux and Talagrand's book~\cite{LedouxTalagrand91})
\begin{align*}
  \E\|\col\|_K \le \sqrt{\frac{\pi}{2}}\E\|\grv\|_K.
\end{align*}
Therefore, there always exists a \emph{full} coloring
\(\col \in \{-1,+1\}^n\) that lies in \(\sqrt{\frac{\pi}{2}}\E[\|\grv\|_K] \cdot
K\).  Since it is a Lipschitz function of a Gaussian, the random
variable \(\|\grv\|_K\) is tightly concentrated around its median
(see, e.g., Theorem 5.2.2 in~\cite{vershynin}). From this, it is not
hard to show that when \(\gamma^n(K) \ge \frac12\),
\(\E[\|\grv\|_K] \lesssim 1\), and, therefore, we can find a full
coloring in a constant scaling of \(K\). (A similar argument is
carried out in detail in Chapter~\ref{ch:bana-intro}.)

Thus, a random
coloring works well in the large measure setting of
\(\gamma^n(K) \ge \frac12\). The power of Lemma~\ref{lm:ch3-gian}, and
the reason why it allows us to prove tighter bounds than possible by a
random coloring, is that it allows us to work in the ``small ball''
setting, in which \(\gamma^n(K)\) is allowed to be exponentially
small. 

As a simple illustration, consider the incidence matrix $\mat$
of the set system of all prefix intervals $\{1, \ldots, i\}$, $i \in
[n]$, on the universe $[n]$, and
\(K\eqdef\{\col \in \R^n: \|\mat \col\|_\infty \le 1\}.\)
Then
\[
  \|\col\|_K = \max_{i = 1}^n |\col(1) + \ldots + \col(i)|,
\]
and $\|\col\|_K \gtrsim \sqrt{n}$ with high probability for a
uniformly random coloring $\col$.  This is much worse than the
discrepancy of $\mat$, which is just $1$ for $\col$ that alternates
between $-1$ and $+1$. On the other hand, classical small ball
inequalities imply that $\gamma^n(tK) \gtrsim e^{-cn/t^2}$ for a
constant $c > 0$, which is enough to get $\gamma^n(tK) \ge 2^{-n/4}$
for a constant $t$. Thus, Lemma~\ref{lm:ch3-gian} is strong enough to
give a partial coloring with constant discrepancy. 
Later, we expand on this example, and use the small ball inequality to show that any set system induced by $k$ permutations has
a partial coloring of discrepancy $O(\sqrt{k})$. The $k$ permutations
example will also show that Lemma~\ref{lm:ch3-gian} cannot be
strengthened to give a full coloring, since there are set systems
induced by $3$ permutations on $[n]$ that have discrepancy
$\Omega(\log n)$.

\subsection{Gaussian Correlation}

To derive Lemma~\ref{lm:ch3-partial} from Lemma~\ref{lm:ch3-gian} we
need a way to prove lower bounds on the Gaussian measure $\gamma^n(K)$
of the convex body
\[
K\eqdef \{\col \in \R^n: |\ip{v_i,\col}| \le 0.5 \lambda_i\|v_i\|_2 \ \ \forall i
\in [m]\}
\]
under the assumptions of Lemma~\ref{lm:ch3-partial}. Then any $\col \in 2K$ would satisfy $|\ip{v_i,\col}| \le \lambda_i\|v_i\|_2$. The only
property of \(K\) we will use is that it is the intersection of slabs
of bounded width: we have
\(
K = \bigcap_{i = 1}^m S_i,
\)
where
\(
S_i \eqdef \{\col \in \R^n: |\ip{v_i,\col}| \le 0.5 \lambda_i\|v_i\|_2 \}
\)
is a symmetric slab, \ie, the region of width \(\lambda_i\|v_i\|_2\) between two hyperplanes orthogonal to
\(v_i\). The following key lemma, due to Sidak~\cite{Sidak67}, shows
why this property is useful. 

\begin{lemma}[Sidak's Lemma]\label{lm:ch3-sidak}
  Suppose that \(K = -K\) is a centrally symmetric closed convex set in
  \(\R^n\), and that \(S \eqdef \{x \in \R^n: |\ip{u,x}| \le t\}\) is a
  symmetric slab defined by the vector \(u \in \R^n\). Then,
  \begin{equation}\label{eq:ch3-sidak}
  \gamma^n(K \cap S) \ge \gamma^n(K) \gamma^n(S).
  \end{equation}
  In particular, if \(K = \bigcap_{i = 1}^m S_i\) and \(S_1, \ldots,
  S_m\) are symmetric slabs defined as \(S_i \eqdef \{x \in \R^n:
  |\ip{u_i,x}| \le t_i\}\) for some vectors \(u_1, \ldots, u_m \in
  \R^n\) s.t.~$\|u_i\|_2 = 1$ for all $i$, and $t_1, \ldots, t_m > 0$, then
  \[
    \gamma^n(K) \ge \prod_{i = 1}^n \gamma^n(S_i) =
    \prod_{i =  1}^n\gamma^1\left(\left[-t_i,t_i\right]\right).
  \]
\end{lemma}

Lemma~\ref{lm:ch3-sidak} allows us to reduce estimating the Gaussian
measure of \(K\) to estimating Gaussian measures
of slabs, and it is easy to see that the Gaussian measure of a slab can be computed as the Gaussian measure of a one-dimensional interval. We can get a pretty good handle on the
latter using standard concentration inequalities, for example.

We first derive Lemma~\ref{lm:ch3-partial} from
Lemmas~\ref{lm:ch3-gian}~and~\ref{lm:ch3-sidak}. Then, later in this
section, we give a self contained short proof of
Lemma~\ref{lm:ch3-sidak}.

\begin{proof}[Proof of Lemma~\ref{lm:ch3-partial}]
  As mentioned above, we work with the convex body 
  \(
  K\eqdef \bigcap_{i = 1}^m S_i,
  \)
  defined by the slabs \(S_i := \{x \in \R^n: |\ip{v_i,x}| \le
  0.5 \lambda_i\|v_i\|_2 \}\). Equivalently, we can define 
  \[
  S_i = \{x \in \R^n: |\ip{u_i,x}| \le 0.5\lambda_i\}
  \]
  for $u_i \eqdef \frac{v_i}{\|v_i\|}$, and
  Lemma~\ref{lm:ch3-sidak} gives us
  \[
    \gamma^n(K) \ge \prod_{i=1}^m\gamma^n(S_i)
    \ge
    \prod_{i=1}^m\gamma^1\left([-0.5\lambda_i, 0.5\lambda_i]\right)
  \]
  Using the inequality \(\Pr[|g| \ge s] \le e^{-s^2/2}\) for a
  standard one-dimensional Gaussian \(g\), we have  $\gamma^1([-0.5\lambda_i, 0.5\lambda_i]) \geq 1-e^{-\lambda_i^2/8}$, and thus
  \begin{equation}
    \label{eq:ch3-partial-K}
    \gamma^n(K) \ge \prod_{i=1}^m(1-e^{-\lambda_i^2/8}).
  \end{equation}

  We now have the inequalities
  \begin{equation}\label{eq:ch3-exp-ineq}
    1-e^{-s^2/8} \ge
    \begin{cases}
      e^{-1.5 e^{-s^2/8}} & s\ge {4}\\
      s^2/20 & s \le {4}
    \end{cases},
  \end{equation}
  where the right hand side is $e^{-\parcol(s)}$, for the function $\parcol(s)$ defined in \eqref{eq:ch3-partial}.
  Both inequalities can be
  verified by observing that the function $f_{L,c}(t) := e^{-Lt} -
  (1-ct)$ is convex in $t$ and satisfies $f_{L,c}(0) = 0$ for all $L$ and
  $c$. Then, $f_{L,c}(t) \le 0$ for all $t \in [0,a]$ if and only if
  $f_{L,c}(a) \le 0$. The first inequality then follows from
  substituting $e^{-s^2/8}$ for $t$ and setting $L := 1.5$, $c := 1$, $a
  := e^{-2}$. The second inequality follows from substituting
  $s^2/8$ for $t$, and setting $L := 1$, $c := 2/5$, $a :=
  2$. 

  Plugging in \eqref{eq:ch3-exp-ineq} into \eqref{eq:ch3-partial-K}, we get
  \[
    \gamma^n(K) \ge \exp\left(-\sum_{i=1}^m\parcol(\lambda_i)\right) \ge
    e^{-n/8} \ge 2^{-n/4},
  \]
  where the penultimate inequality is by assumption. The lemma now
  follows from Lemma~\ref{lm:ch3-gian}.
\end{proof}

\begin{remark}
  Lemma~\ref{lm:ch3-partial} clearly holds with $h(s) :=
  -\ln\left(1-e^{-s^2/2}\right)$ in place of $\parcol(s)$, and the function
  $\parcol(s)$ is simply an upper bound on $h(s)$. In applications, however,
  $\parcol(s)$ is usually easier to use. 
\end{remark}

\paragraph{Proof of Sidak's Lemma.} We now take a slight detour and give a short proof of
Lemma~\ref{lm:ch3-sidak}, following an argument of Giannopoulos (see
the proof and discussion in~\cite{SzW99}). Let us first introduce one of the central
theorems in geometric analysis: the Prekopa-Leindler inequality.

\begin{lemma}\label{lm:ch3-pl}
  Let \(\lambda \in (0,1)\), and let \(f,g,h:\R^n \to \R_{\ge 0}\) be
  three non-negative measurable functions satisfying
  \begin{equation}\label{eq:ch3-pl-ass}
    h(\lambda x + (1-\lambda)y) \ge f(x)^\lambda g(y)^{1-\lambda}
  \end{equation}
  for all \(x,y \in \R^n\). Then we have
  \[
    \int_{\R^n} h(x) dx \ge \left(\int_{\R^n} f(x) dx\right)^\lambda\left(\int_{\R^n} g(x) dx\right)^{1-\lambda}.
  \]
\end{lemma}
We do not prove this inequality here, but refer the reader to Lecture~5 in Ball's
survey~\cite{Ball-modern}, or Barthe's survey~\cite{Barthe06} for short proofs. A classical
application of Lemma~\ref{lm:ch3-pl} is to prove the Brunn-Minkowski
inequality which states that, under
appropriate measurability assumptions, for any sets \(A, B \subseteq
\R^n\), and any \(\lambda \in (0,1)\),
\begin{equation}\label{eq:ch3-bw}
  \vol^n(\lambda A + (1-\lambda)B) \ge \vol^n(A)^\lambda\vol^n(B)^{1-\lambda}.
\end{equation}
Here, \(\vol^n(\cdot)\) denotes \(n\)-dimensional volume, \ie,
Lebesgue measure, and \(\lambda A + (1-\lambda)B\) is the Minkowski
sum \(\{\lambda x + (1-\lambda)y: x \in A, y \in B\}\).
To prove Sidak's lemma, we, instead, need a version of the
Brunn-Minkowski inequality for Gaussian measure.
\begin{lemma}\label{lm:ch3-bm-gauss}
  For any Borel sets \(A,B \subseteq \R^n\), and any \(\lambda \in
  (0,1)\), we have
  \[
    \gamma^n(\lambda A + (1-\lambda)B) \ge \gamma^n(A)^\lambda\gamma^n(B)^{1-\lambda}.
  \]
\end{lemma}
\begin{proof}
  Let \(p(x) \eqdef (2\pi)^{-n/2} e^{-\|x\|_2^2/2} \) be the
  Gaussian density function. 
  Aiming to use Lemma~\ref{lm:ch3-pl}, we define \(f,g,h:\R^n\to \R_{\ge 0}\)  by
  \begin{align*}
    f(x) &\eqdef p(x) 1\{x \in A\};\\
    g(y) &\eqdef p(y) 1\{y \in B\};\\
    h(z) &\eqdef p(z) 1\{z \in \lambda A + (1-\lambda)B\}.
  \end{align*}
  Measurability of \(f,g,h\) follows from the assumption that \(A\) and
  \(B\) are Borel (and, therefore, so is \(A + B\)), and from the
  continuity of \(p\). 
  We need to then verify that \eqref{eq:ch3-pl-ass} holds for these \(f\), \(g\), and
  \(h\). Let us fix some \(x,y \in \R^n\). It is easy to check that \(p(x)\) 
  is log-concave, \ie,
  \begin{align}
    p(\lambda x + (1-\lambda)y) \ge
    p(x)^\lambda p(y)^{1-\lambda}.\label{eq:ch3-gauss-lc}
  \end{align}
  Moreover, if \(x \in A\) and \(y \in B\), then, by definition, \(\lambda
  x + (1-\lambda)y \in \lambda A + (1-\lambda)B\), which implies that 
  \begin{equation}\label{eq:ch3-ind-lc}
    1\{\lambda x + (1-\lambda)y \in \lambda A + (1-\lambda)B\}
    \ge
    1\{x \in A\}^\lambda 1\{y \in B\}^{1-\lambda}.
  \end{equation}
  Multiplying \eqref{eq:ch3-gauss-lc} and \eqref{eq:ch3-ind-lc}
  together, we get \eqref{eq:ch3-pl-ass}. The lemma is now implied by
  Lemma~\ref{lm:ch3-pl}. 
\end{proof}


\begin{exercise}
  Modify the proof of Lemma~\ref{lm:ch3-bm-gauss} to prove the
  inequality \eqref{eq:ch3-bw}. Then derive from it the stronger
  inequality
  \[
    \vol^n(\lambda A + (1-\lambda)B)^{1/n} \ge \lambda \vol^n(A)^{1/n} +
    (1-\lambda)\vol^n(B)^{1/n},
  \]
  using the elementary fact that $\vol^n(\lambda A) = \lambda^n
  \vol^n(A)$ for any $\lambda \ge 0$.
\end{exercise}

\begin{figure}[htp]
  \centering
  \includegraphics[scale=0.9, trim=0 0 0 0]{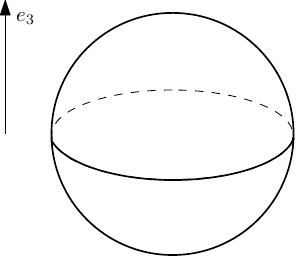}
  \caption[Slice of a ball]{The set $K_0$ where $K$ is a
    three-dimensional Euclidean ball centered at the origin. $K_0$ is
    a slice of the ball given by intersecting it with a horizontal
    plane through the origin.}
  \label{fig:ball-slices}
\end{figure}

\begin{figure}[htp]
  \centering
  \includegraphics[scale=0.5, trim = 0 0 0 0]{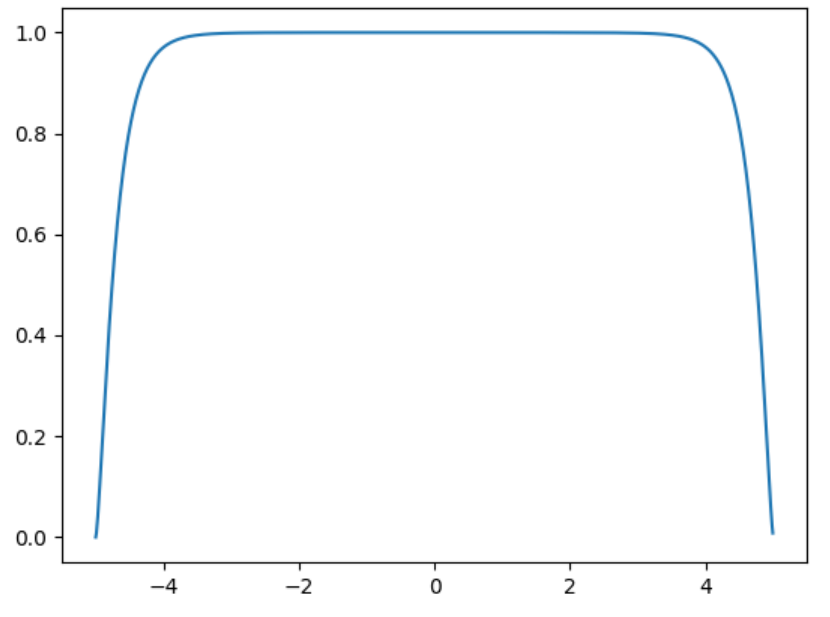}
  \caption[Measures of slices of a ball]{The Gaussian measure of $K_t$
    as a function of $t$, where $K$ is a three-dimensional Euclidean
    ball of radius $5$. Here $K_t$ is a two-dimensional circle of
    radius $\sqrt{25-t^2}$. Notice the measure is quasiconcave, i.e.,
    the set $\{t: \gamma^3(K_t) \ge \alpha\}$ is an interval for any
    \(\alpha\).}
  \label{fig:ball-slices-vol3d}
\end{figure}

Let \(K\subseteq \R^n\) be a closed convex
set. Lemma~\ref{lm:ch3-bm-gauss} allows us to characterize how the
Gaussian measure of sections of \(K\) changes as we slice it in a given
direction. For ease of notation, we choose the direction to be
\(e_n\), the \(n\)-th standard basis vector. This does not really cost
us any generality, since the standard Gaussian measure is invariant
under rotations. Let us then define the function
\(
  f_K(t) \eqdef \gamma^{n-1}(K_t),
\)
where \(K_t \eqdef \{y \in \R^{n-1}: (y,t) \in K\}\) is the slice of
 \(K\) consisting of points whose  \(n\)-th coordinate is equal to \(t\). See
Figures~\ref{fig:ball-slices}~and~\ref{fig:ball-slices-vol3d} for an example. Lemma~\ref{lm:ch3-bm-gauss} now tells us that \(f_K\) is
log-concave, \ie, for any \(s,t \in \R\) and any \(\lambda \in
(0,1)\),
\begin{equation}
  \label{eq:ch3-logc}
  f_K(\lambda s + (1-\lambda)t) \ge f_K(s)^\lambda f_K(t)^{1-\lambda}.
\end{equation}
Indeed, by the convexity of \(K\), \(\lambda K_s + (1-\lambda)K_t
\subseteq K_{\lambda s + (1-\lambda)t}\), and
Lemma~\ref{lm:ch3-bm-gauss} gives
\[
  f_K(\lambda s + (1-\lambda)t)
  \ge
  \gamma^{n-1}(\lambda K_s + (1-\lambda)K_t)
  \ge f_K(s)^\lambda f_K(t)^{1-\lambda}.
\]

Notice that when \(K\) is centrally symmetric, i.e.,
\(K = -K\), the function \(f_K\) is even: \(f_K(-t) = f_{-K}(t) =
f_K(t)\). Together with the log-concavity inequality
\eqref{eq:ch3-logc}, used with $\lambda \eqdef \frac12$ and $s = -t$, this means that \(f_K\) is maximized at
\(0\). Moreover, observe that $\gamma^n(K)$ is the expectation of $f_K(g)$ for a one-dimensional Gaussian $g$:
\begin{align}
\gamma^n(K) &= \frac{1}{(2\pi)^{n/2}} \int_K e^{-\|x\|_2^2/2}dx\notag\\
&= \frac{1}{(2\pi)^{n/2}} \int_{-\infty}^\infty \int_{K_t}e^{-(\|y\|_2^2 + t^2)/2}dy dt\notag\\
&= \frac{1}{\sqrt{2\pi}} \int_{-\infty}^\infty e^{-t^2/2} f_K(t)dt.\label{eq:ch3-int-slices}
\end{align}
Altogether, we have that, for centrally symmetric $K$, $f_K(0)\ge\gamma^n(K)$.



Equipped with the log-concavity of slice measures (inequality
\eqref{eq:ch3-logc}), we are ready to prove Sidak's lemma.

\begin{proof}[Proof of Lemma~\ref{lm:ch3-sidak}]
  We first note that the lemma is trivial for \(n=1\). In that case,
  \(K\) and \(S\) are both closed intervals symmetric around
  \(0\), and \(K \cap S\) is the smaller of the two intervals. We have
  \[
    \gamma^1(K \cap S) = \min\{\gamma^1(K),\gamma^1(S)\}
    \ge \gamma^1(K)\gamma^1(S).
  \]

  Inequality \eqref{eq:ch3-logc} and a simple computation let us
  reduce the higher dimensional case to the one dimensional one. We can  assume, by rescaling $u$, that $S = \{x \in \R^n: |\ip{u,x}| \le t\}$ and $\|u\|_2 = 1$. By
  the rotation invariance of the standard Gaussian measure, we may also
  assume that \(u = e_n\), the $n$-th standard basis vector.  Therefore,
  \(
  S = \left\{x \in \R^n: x_n \in \left[-t,t\right]\right\}.
  \)
  Note that, using \(\grv\) for a standard Gaussian random variable in \(\R^n\),
  \begin{align*}
    \gamma^n(S) = \Pr\left[|\grv(n)| \le t\right]
    = \gamma^1\left(\left[-t,t\right]\right).
  \end{align*}
  Applying \eqref{eq:ch3-int-slices} to $K \cap S$, and noticing that $f_{K \cap S}(s)$ equals $f_K(s)$ if $s \in [-t,t]$, and $0$ otherwise, we have
  \begin{equation}\label{eq:ch3-int-sidak}
      \gamma^n(K \cap S) = \frac{1}{\sqrt{2\pi}} \int_{-t}^{t} e^{-s^2/2} f_K(s)ds.
  \end{equation}
  The right hand side is just the expected value of $f_K(g)1\{g \in [-t,t]\}$ for a standard $1$-dimensional Gaussian random variable $g$. Using integration by parts, this expectation equals
  \begin{equation} \label{eq:ch3-int-parts}
      \int_{0}^\infty \Pr[f_K(g)1\{g \in [-t,t]\} \ge r]dr
      = \int_{0}^\infty \gamma^1\left(I_K(r)\cap\left[-t,t\right]\right)dr,
  \end{equation}
  where \(I_K(r) \eqdef \{s\in \R: f_K(s) \ge r\}\). 
  By the log-concavity of \(f_K\) (inequality \eqref{eq:ch3-logc}), the set
  \(I_K(r)\) is convex, \ie, an
  interval. Moreover, \(f_K\) is clearly continuous, so \(I_K(r)\) is
  closed, and, for \(K\) symmetric around the origin, \(f_K\) is even, and \(I_K(r)\) is also
  symmetric around the origin. Thus, by the one-dimensional case of
  the lemma,
  \[
    \gamma^1\left(I_K(r)\cap \left[-t, t\right]\right)
    \ge
    \gamma^1(I_K(r)) \gamma^1\left(\left[-t,t\right]\right)
    = \gamma^1(I_K(r)) \gamma^n(S).
  \]
  Plugging this inequality back into \eqref{eq:ch3-int-sidak}~and~\eqref{eq:ch3-int-parts}
  gives us
  \begin{align*}
    \gamma^n(K \cap S) 
    &\ge \gamma^n(S) \int_{0}^\infty \gamma^1(I_K(r))dr\\
    &= \gamma^n(S) \int_{0}^\infty \Pr[f_K(g) \ge r]dr\\
    &= \gamma^n(S) \E[f_K(g)]\\
    &= \gamma^n(S)\int_{-\infty}^\infty e^{-s^2/2}f_K(s)ds\\
    &= \gamma^n(S)\gamma^n(K).
  \end{align*}
  Above, the first equality uses the definition of $I_K(r)$, the second uses integration by parts again, with $g$ a standard 1-dimensional Gaussian, and the last equality uses \eqref{eq:ch3-int-slices}.
\end{proof}

Let us now state a significant generalization
of Lemma~\ref{lm:ch3-sidak}: Royen's Gaussian correlation inequality.
\begin{theorem}\label{thm:ch3-gauss-cor}
  For any two closed convex sets \(K, L\) symmetric around the origin
  (\ie, \(K = -K\) and \(L = -L\)), we have
  \begin{equation}\label{eq:ch3-gauss-cor}
  \gamma^n(K \cap L) \ge \gamma^n(K) \gamma^n(L).
  \end{equation}
\end{theorem}
Theorem~\ref{thm:ch3-gauss-cor}, whose proof we omit, was known as the
Gaussian correlation conjecture and was open for many years until it
was proved by Royen~\cite{Royen14} (see also the
exposition~\cite{LM-royen}). We will use it next to bound the discrepancy of set systems induced by
permutations.
\section{Discrepancy of Permutations and Gaussian Correlation}\label{sect:ch3-corr}
Recall that the set system \((\SS,[n])\) induced by \(k\) permutations
\(\pi_1, \ldots, \pi_k\) of \([n]\) consists of prefix sets of the type
\(\{\pi_i(1), \ldots, \pi_i(j)\}\). In Chapter~\ref{ch:linalg} we used
the linear algebraic method to show that the discrepancy of such a set
system is bounded by \(O(k\log(n))\). While the dependence on \(n\) in
this bound is optimal for a constant \(k \ge n\), the dependence on
\(k\) can be improved using the partial coloring method. For example,
a careful application of Lemma~\ref{lm:ch3-partial} shows that there
exists a partial coloring of discrepancy \(O(\sqrt{k})\), as indicated
in the following exercise. In this exercise we use the notion of
canonical interval defined in Chapter~\ref{ch:linalg}.

\begin{exercise}\label{ex:ch3-kperms}
  Let \(\pi_1, \ldots, \pi_k\) be permutations of \([n]\). Consider
  the set system \((\TT,[n])\) consisting of the sets \(\pi_i(C)
  \eqdef \{\pi_i(j): j \in C\}\) for
  all \(i \in [k]\) and all canonical intervals \(C \in \CC\) as
  defined in Chapter~\ref{ch:linalg}. Using Lemma~\ref{lm:ch3-partial},
  show that there exists a partial coloring \(\col\in \{-1, 0,+1\}^n\)
  such that, for any \(C \in \CC_p \subseteq \CC\), \(|\col(C)| \lesssim
  \sqrt{k}\cdot\phi(p)\) and \(|\{i: \col(i) \neq 0\}| \ge
  n/10\). Here, \(\phi(p)\) should be a function on the
  positive integers, chosen so that \(\sum_{p = 0}^{\infty}\phi(p)\)
  converges to an absolute constant.

  Then, using canonical intervals (see Section~\ref{sec:ch2-tusnady}) conclude that \(\col\)
  satisfies \(|\col(\set)|\lesssim \sqrt{k}\) for all sets \(\set\) in
  the set system \(\SS\) induced by \(\pi_1, \ldots, \pi_k\).
\end{exercise}
We now take a different approach to show the same
result. Recall that we proved Lemma~\ref{lm:ch3-partial} using Sidak's
lemma (Lemma~\ref{lm:ch3-sidak}) and Giannopoulos's partial coloring
lemma (Lemma~\ref{lm:ch3-gian}). We will replace Sidak's lemma with
the deeper correlation inequality of Royen
(Theorem~\ref{thm:ch3-gauss-cor}), which will allow us to make the
calculations in the proof significantly simpler.

To apply Lemma~\ref{lm:ch3-gian} to the 
bound the discrepancy of a set system \((\SS,[n])\) induced by
\(k\) permutations \(\pi_1, \ldots, \pi_k\) on \([n]\), we need to
estimate the Gaussian measure of $tK$, where $K$ is the symmetric convex set
\begin{equation}\label{eq:ch3-K-perms}
K :=\left\{\col \in \R^n: \max_{\set \in \SS} |\col(\set)| \le 1\right\}.
\end{equation}
Instead of writing \(K\) as the
intersection of \(|\SS|\) many slabs, one for each set \(\set \in
\SS\), we define \(k\) convex sets \(K_1, \ldots, K_k\) where
\begin{equation}\label{eq:ch3-Ki}
K_i :=\left\{\col \in \R^n: \max_{\set \in \SS_i} |\col(\set)| \le 1\right\}
\end{equation}
and \(\SS_i\) consists of all prefix sets of the form \(\{\pi_i(1),
\ldots, \pi_i(j)\}\) for one fixed permutation \(\pi_i\). Clearly, \(K
= K_1 \cap \ldots \cap K_k\). 

We will show that, for some \(t \lesssim
\sqrt{k}\), \(\gamma^n(tK_i) \ge 2^{-n/(4k)}\), and then \(\gamma^n(tK) \ge
2^{-n/4}\) follows from Theorem~\ref{thm:ch3-gauss-cor}.
Notice that $\gamma^n(tK_i)$ is simply
\[
  \Pr[\max_{j = 1}^n\left|g(\pi_i(1)) + \ldots +
    g(\pi_i(j))\right| \le t],
\]
for a standard Gaussian
 vector \(\grv \in \R^n\). 
 
 Since the distribution of \(g\) is
invariant under permuting coordinates, this is equivalent to
estimating
\[
  \Pr\left[\max_{j = 1}^n\left|\grv(1) + \ldots + \grv(j))\right| \le t\right].
\]
That is, the probability that an \(n\) step
random walk on the real line with standard Gaussian steps starting at
\(0\) stays between \(-t\) and \(+t\). This is a question we
can answer using standard results from probability theory. It is
convenient to work instead with a continuous random process, in
particular, Brownian motion. Standard Brownian motion is a random
process $\{B(s): s \ge 0\}$ such that
\begin{itemize}
\item $B(0) = 0$ with probability 1;
\item for all integers $k > 1$, and all times $0 \le s_1 \le \ldots
  \le s_m$, the increments $B(s_1)-B(0)$, $B(s_2) - B(s_{1})$, $\ldots$, $B(s_m) -
  B(s_{m-1})$ are jointly independent;
\item for all $s \ge 0$ and $\Delta\ge 0$, $B(s+\Delta) - B(s)$ is
  distributed as a Gaussian with mean $0$ and variance $\Delta$;
\item $B(s)$ is almost surely continuous as a function of $s$.
\end{itemize}
For a proof of the non-trivial fact that Brownian motion exists,
see~\cite{PeresMorters}. It is well-known that Brownian motion
satisfies the following small ball inequality: for all $t > 0$, 
\begin{equation}
  \label{eq:ch3-bm-small}
  \Pr\Big[\sup_{s \in [0,1]} |B(s)| \le t\Big] \gtrsim \exp(-\pi^2/8t^2).
\end{equation}
For a derivation, see,
\eg, equation (7.1) of Ledoux's lecture notes~\cite{Ledoux96}. We
claim that \eqref{eq:ch3-bm-small} implies a similar bound for
$\max_{j = 1}^n\left|\grv(1) + \ldots + \grv(j)\right|$.
To see this, note that the vectors $(\grv(1), \ldots, \grv(n))$ and
$\sqrt{n}(B(1/n) - B(0),\ldots, B(1) - B((n-1)/n))$ are identically
distributed. Therefore,
\begin{align}
  \Pr\left[\max_{j = 1}^n\left|\grv(1) + \cdots + \grv(j)\right| \le t\right]
  &=
    \Pr\left[\max_{j = 1}^n\left|B(j/n)\right| \le t/\sqrt{n}\right]\notag\\
  &\ge
    \Pr\Big[\sup_{s \in [0,1]} |B(s)| \le t/\sqrt{n}\Big]\notag\\
  &\gtrsim    \exp(-\pi^2n/8t^2).\label{eq:ch3-smallball}
\end{align}

We can now prove the existence of a partial
coloring of discrepancy \(O(\sqrt{k})\) for any set system induced by
\(k\) permutations.

\begin{lemma}\label{lm:ch3-perms-partial}
  For any set system \((\SS,[n])\)
  induced by \(k\geq 1\) permutations of \([n]\), there exists a partial
  coloring \(\col \in \{-1, 0,+1\}^n\) such that for every
  \(\set \in \SS\), \(|\col(S)| \lesssim \sqrt{k}\), and
  \(|\{i: \col(i) \neq 0\}| \ge n/10\).
\end{lemma}
\begin{proof}
  As discussed above, it is enough to show that, for each
  \(i \in [k]\), \(\gamma^n(tK_i) \ge 2^{-n/4k}\) for some
  \(t \lesssim \sqrt{k}\). Let us first handle the case $n \ge Ck$ for
  a large enough constant $C>0$. By \eqref{eq:ch3-smallball}, we
  have
  \[
    \gamma^n(tK_i) =
    \Pr\left[\max_{j = 1}^n\left|\grv(1) + \ldots +  \grv(j)\right|  \le t\right]
    \ge c e^{-\pi^2n/8t^2},
  \]
  for a constant \(c > 0\). If $C \ge \frac{8\ln(1/c)}{\ln(2)}$, then
  it is enough to set $t \eqdef \sqrt{\frac{\pi^2 k}{\ln(2)}} \lesssim
  \sqrt{k}$ for the right hand side to be at least
  $2^{-n/4k}$.

  The proof in the case $n \le Ck$ is analogous, but, instead of
  \eqref{eq:ch3-smallball}, we use the bound
  \begin{equation}
    \label{eq:ch3-rw-tail}
    \Pr\left[\max_{j = 1}^n\left|\grv(1) + \ldots + \grv(j)\right|  \ge t\right]
    \le
    2e^{-t^2/2n}.    
  \end{equation}
  This inequality can be derived from an analogous one for Brownian
  motion (see, e.g., Theorem 2.18 in M\"{o}rters and
  Peres's book~\cite{PeresMorters}) as we did for
  \eqref{eq:ch3-smallball}. Therefore, $\gamma^n(tK_i) \ge 1- 2
  e^{-\frac{t^2}{2n}}$ for every $i \in [k]$, and  it suffices to make sure that
  \(
  1- 2e^{-\frac{t^2}{2n}} \ge 2^{-\frac{n}{4k}}.
  \) 
    Analogously to inequality \eqref{eq:ch3-exp-ineq}, we can argue that this inequality holds
  for some $t \lesssim \sqrt{n\log\left({Ck}/{n}\right)}$. We can
  now check that, because $n \le Ck$, the right hand side is bounded
  by $O(\sqrt{k})$.
\end{proof}

The following theorem follows from Lemma~\ref{lm:ch3-perms-partial} by
completing the partial coloring to a full coloring in the usual way. 
\begin{theorem}\label{thm:ch3-perms}
  For any integer \(k \ge 1\), and any set system \((\SS,[n])\)
  induced by \(k\) permutations of \([n]\), we have
  \(\disc(\SS) \lesssim \sqrt{k}\log(2n).\)
\end{theorem}

As mentioned in Chapter~\ref{ch:linalg}, 
the dependence on $n$ in Theorem~\ref{thm:ch3-perms} is tight for any
$k \ge 3$ \cite{beck3perm}. It is open, however, whether the bound $\sqrt{k}\log(2n)$
is asymptotically tight both in $k$ and $n$. It is consistent with current knowledge that this bound can be improved to $O(\sqrt{k} + \log n)$, for example.
Note that
Lemma~\ref{lm:ch3-perms-partial} implies that the assumptions in
Lemma~\ref{lm:ch3-gian} are not sufficient to guarantee the existence
of a full, rather than partial, coloring. If they were sufficient,
then we would be able prove a constant upper bound on the discrepancy
of any set system induced by $3$ permutations, contradicting the lower
bound of Neiman, Newman, and Nikolov~\cite{beck3perm}.

An interesting question we close with is whether other classical small
ball inequalities for Gaussian processes, as well as Royen's
correlation inequality, can be useful in proving discrepancy upper
bounds.  For example, the small ball inequality for the Brownian sheet
(see again Chapter 7 of Ledoux's notes~\cite{Ledoux96}) appears to be
related to Tusn\'ady's problem (Problem~\ref{prob:tusnady}) in $2$
dimensions.

\section{Bibliographic Notes}

The partial coloring method was first developed by Beck, who proved
Lemma~\ref{lm:ch3-beck}~\cite{beck-roth,Beck88}. He used his method to
show that the discrepancy of arithmetic progressions restricted to
$[n]$ is bounded by $n^{1/4}$ up to a factor polynomial in
$\log(n)$. This nearly matches a lower bound of Roth~\cite{roth-ap}. 

The partial coloring method of Beck was
further refined by Spencer~\cite{spencer1985} in order to prove
Theorem~\ref{thm:ch3-spencer}. Spencer's techniques are combinatorial
and do not use the geometric ideas presented in
Section~\ref{sec:ch3-giann}. The method was further refined by
Matou\v{s}ek~\cite{Matousek-halfspaces} and Matou\v{s}ek and
Spencer~\cite{MatousekSpencer-ap}. The latter paper formulated a
general partial coloring lemma similar to
Lemma~\ref{lm:ch3-partial}. Matou\v{s}ek and Spencer proved their
partial coloring lemma using an information theoretic argument due to
Boppana, initially proposed as a simplification of Spencer's proof of
Theorem~\ref{thm:ch3-spencer}~\cite{alon2000probabilistic}. Their
paper also removed the logarithmic terms in Beck's upper bound for the
discrepancy of arithmetic progressions, matching Roth's lower bound
up to a universal constant. The Matou\v{s}ek and Spencer
partial coloring lemma is sometimes also known as the entropy method.

In our presentation of the partial coloring method we depart from the
combinatorial approach and instead adopt geometric techniques
developed in the study of the vector balancing problem. The geometric
approach presented here is originally due to~Gluskin~\cite{Glu89}, who independently
proved the same discrepancy upper bound as
Spencer~\cite{spencer1985}. Lemma~\ref{lm:ch3-gian} is due to
Giannopoulos~\cite{Gia93}, who strengthened and simplified Gluskin's
results. 

Theorem~\ref{thm:ch3-bf} and Exercise~\ref{ex:ch3-komlos} are due to Giannopoulos~\cite{Gia93}, who
used Lemma~\ref{lm:ch3-gian}, and a geometric estimate of the Gaussian
measure of slices of a cube, and, independently, to
Srinivasan~\cite{Srinivasan97} who used Lemma~\ref{lm:ch3-partial}. As already mentioned, a stronger upper bound of
$O(\sqrt{\log n})$ for the Koml\'os problem was given by
Banaszczyk~\cite{Bana98} using a completely different argument not
based on partial colorings. We explain Banaszczyk's argument and its
various applications in
Chapters~\ref{ch:bana-intro}~and~\ref{ch:bana-proof}.

Lemma~\ref{lm:ch3-perms-partial} and Theorem~\ref{thm:ch3-perms} are
due to unpublished work of Spencer, Srinivasan, and
Tetali~\cite{SST-perms}. Their proof used Lemma~\ref{lm:ch3-partial},
along the lines of Exercise~\ref{ex:ch3-kperms}. The proof we present
here using Royen's correlation inequality and the small ball
inequality for Brownian motion has not been published before. We note
that some proofs of the small ball inequality \eqref{eq:ch3-bm-small}
use arguments similar to Spencer, Srinivasan, and Tetali's
argument. Other proofs rely on special properties of Brownian motion,
\eg, its connection to the Dirichlet problem.

Many other applications of the partial coloring method -- to
Tusn\'ady's problem, to giving tight bounds on the discrepancy of
halfspaces and other geometric sets, and others -- are surveyed in
Matou\v{s}ek's book~\cite{matousek2010geometric}.



\chapter{Algorithms for Partial Coloring}
\label{ch:partial-alg}
In Chapter \ref{ch:partial}, we saw the partial coloring method, and how it can give substantially improved bounds for various problems.
For a long time however, no efficient algorithms were known for these methods, which prevented the use of discrepancy in many algorithmic applications. Interestingly, in recent years, several algorithmic approaches have been developed for these, based on ideas from probability, linear algebra, geometry and optimization.

In this chapter, we describe two such elegant results. First, the Edge-Walk algorithm due to Lovett and Meka, that gives an algorithmic version of Theorem \ref{lm:ch3-partial}. Second, an algorithm due to Rothvoss for the more general setting of arbitrary symmetric convex bodies in Theorem \ref{lm:ch3-gian}.
In Section \ref{sec:vec-herdisc}, we describe the notion of vector discrepancy and an algorithm for finding low discrepancy colorings based on semidefinite programming.

\section{Lovett-Meka Algorithm}
\label{sec:lm}
Recall the setting of Theorem \ref{lm:ch3-partial}, where we are given $m$ vectors $v_1, \ldots, v_m \in \R^n$, and real parameters $\lambda_1,\ldots, \lambda_m\geq 0$, and the goal is to find a partial coloring $x$ such that $ |\ip{v_i,x}| \leq \lambda_i \|v_i\|_2 $. 

We will show the following result.


\begin{theorem}
\label{th:lm}
Let $v_1, \ldots, v_m \in \R^n$,  and let $x_0 \in [-1,1]^n$ be some arbitrary starting point, and let   $\lambda_1,\ldots,\lambda_m \geq 0$ be such that 
\begin{equation}
\label{eqcond21}
 \sum_i \exp(-\lambda_i^2/16) \leq n/16.
 \end{equation}
 Let $\delta>0$ be an arbitrarily small parameter.
Then there is a randomized algorithm that runs in time polynomial in $n,m$ and $1/\delta$, and,  with constant probability, finds $x\in [-1,1]^n$ such that
\smallskip 

(i) $|\{i: |x(i)| \geq 1-\delta\}| \ge \frac{n}{2}$, and 

\smallskip

(ii)  $|\ip{v_i,x-x_0}| \leq \lambda_i \|v_i\|_2$ for each $i\in [m]$.
\end{theorem}

\begin{remark} We set $\delta = 1/\text{poly}(n,m)$ to keep the running time polynomial, 
 but it is convenient to think of $\delta$ as $0$. So (i) says that the colors of at least $n/2$ elements reach $\pm 1$, and (ii) controls the discrepancy incurred for each vector.
 
The ``colors'' $x(i)$ produced by the algorithm in Theorem \ref{th:lm} lie in $[-1,1]$, in contrast to $\{-1,0,1\}$ in Theorem \ref{lm:ch3-partial}. 
But this makes no difference in applications as the starting point $x_0$ is allowed to be arbitrary, and we can iterate the algorithm for $O(\log n)$ rounds to obtain a full coloring.
In fact, this flexibility in colors allows the condition in \eqref{eqcond21} to be more relaxed than in \eqref{eq:ch3-partial}. In particular, the contribution of any $\lambda_i$ to the left side in \eqref{eqcond21} never exceeds $1$, while $\parcol(\lambda_i)$ in \eqref{eq:ch3-partial} gets arbitrarily large as $\lambda_i$ approaches $0$.
\end{remark}





\subsection{The Algorithm}
At a high level, the algorithm is based on combining the linear algebraic approach from  Chapter \ref{ch:linalg} with probabilistic techniques.

\medskip
\noindent {\bf The idea.}
The algorithm starts at the point $x_0$, and performs a standard Brownian motion (approximated by a random walk with tiny Gaussian increments). See Figure \ref{fig:lm} below. Let $x_t$ denote the coloring at time $t$. If some coordinate $x_t(j)$ reaches $\pm 1$, its value is fixed and no longer updated (i.e.,~all subsequent steps of the walk are orthogonal to $e_j$, the $j$-th standard basis vector).
Moreover, if any discrepancy constraint becomes tight, i.e.,~$|\ip{v_i, x_t-x_0}| = \lambda_i \|v_i\|_2$, all the subsequent steps of the walk are orthogonal to $v_i$. 

Clearly, by design, the coloring $x_t$ stays inside the cube $[-1,1]^n$, and no discrepancy constraint  is violated.  
Notice that every time some coordinate reaches $\pm 1$, or a discrepancy constraint becomes tight, the dimension of the subspace where the walk is performed may reduce by $1$.
The key idea of the analysis is to show that the condition \eqref{eqcond21} ensures that many coordinates are likely to reach $\pm 1$, before the walk gets ``stuck''. 
To see why this might be the case, consider the case when all $\lambda_i$ are equal, and, therefore, equal to a large enough constant $C$. Suppose also for simplicity, that $x_0 = 0$. Then the boundary of the strip $\{x: |\ip{v_i, x}| \le C \|v_i\|_2\}$, which encodes a discrepancy constraint, is at Euclidean distance $C$ from the origin, while the boundary of the cube $[-1,1]^n$ is at distance $1$ from the origin. So, intuitively, the Brownian motion is more likely to hit the boundary of the cube before a discrepancy constraint becomes tight.

\begin{figure}[hbtp!]
    \centering
\includegraphics[scale=0.7]{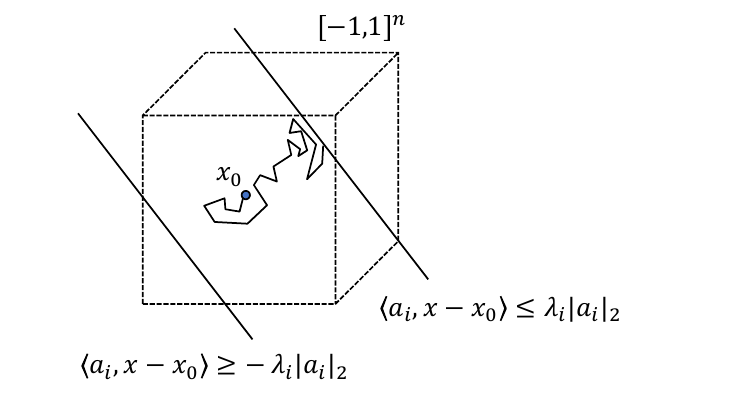}
    \caption{Lovett-Meka Algorithm. The random walk starting from $x_0$, and hitting a discrepancy constraint.}
    \label{fig:lm}
\end{figure}
\paragraph{The Formal Algorithm.}
Fix a small step size $\gamma > 0$, such that $\delta \ge c \gamma \log^{1/2}(2nm/\gamma)$ for a sufficiently large constant $c$, to be determined later. 
The algorithm proceeds in discrete time steps $t=1,2,\ldots,\ell=4/\gamma^2$.
Let $x_{t-1}$ denote the coloring at the beginning of time step $t$ (equivalently, at the end of time step  $t-1$). 
At each $t=1,\ldots,\ell$, the coloring $x_{t-1}$ is updated to $x_t$ as follows:

\smallskip

(i) Let 
$F_{t}:=\{j \in [n]: |x_{t-1}(j)|> 1-\delta \}$
denote the set of tight (\ie, \emph{fixed}) elements thus far, and let \[D_t := \{i \in [m]: |\ip{v_i,x_{t-1}-x_0}| \geq  (\lambda_i -\delta)\|v_i\|_2 \}\] denote the set of tight discrepancy constraints thus far (\ie, the \emph{dangerous} directions).

\smallskip

(ii) Let $W_t := \{ u \in \R^n: u(j)=0 \text{ for all } j \in F_t,\, \ip{v_i,u}=0 \text{ for all } i \in D_t\}$
be the subspace orthogonal to the tight constraints.

\smallskip 
(iii) 
Sample a random standard Gaussian vector  $g_t \sim N(W_t)$  in the subspace $W_t$, and set  $x_t = x_{t-1} + \gamma g_t$.

Above, for a linear subspace $W$ of $\R^n$, we use the notation $N(W)$ to denote the standard Gaussian distribution restricted to $W$, and $g \sim N(V)$ to mean that $g$ is a random variable distributed according to $N(V)$. 


\subsection{Analysis}
We show that the algorithm satisfies the guarantee in Theorem \ref{th:lm}. 

Notice that at any time $t$, we have that $x_t-x_0 = \sum_{s=1}^t \gamma g_{s}$, where each $g_s \sim N(W_s)$. As the subspace $W_s$ depends on $x_{s-1}$, which in turn depends on $g_1,\ldots,g_{s-1}$, the coloring $x_t$ evolves as a martingale with Gaussian increments.

So let us begin by recalling some basic properties of Gaussians and martingales. The analysis is elementary, and follows directly from these properties.

\medskip 
\noindent {\bf Some properties of Gaussians.} Let us use the notation $N(\mu,\sigma^2)$ for the one-dimensional Gaussian distribution with mean $\mu$ and variance $\sigma^2$, and $g\sim N(\mu,\sigma^2)$ to denote that $g$ is a random variable distributed according to $N(\mu,\sigma^2)$. If $g_1 \sim N(\mu_1,\sigma_1^2)$  and $g_2 \sim N(\mu_2,\sigma_2^2)$ are independent Gaussians, their sum is also Gaussian with $g_1 + g_2 \sim N(\mu_1 + \mu_2, \sigma_1^2 + \sigma_2^2)$ . 
This implies that if $g = (g(1),\ldots,g(n))$ is the standard gaussian random vector in $\R^n$, with $g(1),\ldots,g(n)$ independent samples from $ N(0,1)$, then for any $u \in \R^n$, we have  \[\ip{g,u}  = \sum_i g(i) u(i) \sim N(0,\|u\|_2^2).\]
For a linear subspace $W$, and $N(W)$, as above, the standard Gaussian distribution restricted to $W$, rotational invariance implies that a random vector $g \sim N(W)$ can be written as 
$g=  g(1) w_1+ \cdots + g(d) w_d$ for any orthonormal basis
$\{w_1,\ldots,w_d\}$ for $W$ and $g(1),\ldots,g(d)$ independent samples from $N(0,1)$. 
This implies the following useful facts.
\begin{lemma}
\label{lem:gaussian-subspace}
Let $W$ be a subspace of $\mathbb{R}^n$ of dimension $d$ and $g\sim N(W)$. Then for all $u \in \mathbb{R}^n$,  $\langle g,u \rangle \sim N(0,\sigma^2)$ with $\sigma^2 \leq \|u\|_2^2$. Moreover, $\sum_{i=1}^n \E[g(i)^2] = d$.
\end{lemma}
\begin{proof}
Let $P$ be the orthogonal projection onto $W$.
 As $g\in W$, for any $u \in \R^n$, $\langle g,u \rangle  =  \langle g,Pu\rangle$ and  hence $\ip{g,u} \sim N(0,\|Pu\|^2)$. Clearly, $\|Pu\|_2 \leq \|u\|_2$. 
 
 For the second part, if $w_1,\ldots,w_d$ is an orthonormal basis for $W$, and $e_1, \ldots, e_n$ is the standard basis of $\R^n$, then by the above observation, 
 \[
 \E[g(i)^2] = \E[\ip{g, e_i}^2] = \|Pe_i\|_2^2 =\sum_{j=1}^d \ip{e_i,w_j}^2.
 \]
 We thus have
 \begin{align*}
     \sum_{i=1}^n \E[g(i)^2]
     = \sum_{i=1}^n \sum_{j=1}^d \ip{e_i,w_j}^2
     = \sum_{j=1}^d \sum_{i=1}^n \ip{e_i,w_j}^2
     = \sum_{j=1}^d \|w_j\|_2^2 = d,
 \end{align*}
 where the third equality follows as $\{e_1,\ldots,e_n\}$ is an orthonormal basis for $\mathbb{R}^n$.
 \end{proof}
Finally, we need some concentration properties for martingales with Gaussian increments.
Recall that for  $g \sim N(0,1)$, for any $\lambda \geq  0$, \[\Pr[ |g| \geq \lambda] \leq 2\exp(-\lambda^2/2).\]
(In fact, the factor $2$ is not necessary.) More generally, we have the following.
\begin{lemma}
\label{ch4:lm-martingale-bound}
Let $g_1,\ldots,g_\ell$ be random variables over $\R$ such that, for all $1 \leq t \leq \ell$,  $g_t$ only depends on $g_1,\ldots,g_{t-1}$, and the law of $g_t$ conditioned on $g_1, \ldots g_{t-1}$ is, with probability $1$, equal to that of a Gaussian with mean zero and variance at most $\sigma^2$.
  Then for any $\lambda > 0$, \[\Pr[|g_1 + \cdots + g_\ell| \geq  \lambda  \sqrt{\ell}\sigma] \leq  2\exp(-\lambda^2/2).\]
\end{lemma}
\begin{proof}
Recall that for $g \sim N(0,\sigma^2)$, its moment generating function is $\E[ \exp(\beta g)] = \exp(\beta^2\sigma^2/2)$ for all $\beta \geq 0$. 
Let $S_t =g_1 + \cdots + g_t$. Then
\begin{align*}
\E[\exp(\beta S_t)]   = & \E[ \E[ \exp(\beta (S_{t-1} + g_t)) | g_1,\ldots,g_{t-1}]] \\
=  & \E  [\exp(\beta S_{t-1}) \E[ \exp(\beta g_t) | g_1,\ldots,g_{t-1}]] \\ 
\leq &  \E [\exp(\beta S_{t-1})] \exp(\beta^2\sigma^2/2),
\end{align*}
and hence by induction $\E[\exp(\beta S_\ell)] \leq \exp(\beta^2 \ell\sigma^2/2)$.
For any $\beta\geq 0$,
 \[\Pr[S_\ell \geq \lambda \sqrt{\ell}\sigma] = \Pr[\exp(\beta S_\ell) \geq \exp{ \beta \lambda \sqrt{\ell}\sigma}] \leq  \exp(\beta^2 \ell\sigma^2/2 -\beta  \lambda \sqrt{\ell}\sigma)\] by Markov's inequality, and making the choice $\beta = \lambda /\sqrt{\ell}$ that minimizes the left hand side gives the result.
\end{proof}

\medskip
\noindent{\bf Basic properties of the algorithm.} 
Let us note some simple observations about the algorithm.
\begin{enumerate}
    \item Once an element is fixed, its value does not change and it remains fixed. Similarly, once. Once a discrepancy constraint becomes tight, it stays tight since subsequent coloring updates are orthogonal to $v_i$. So $F_t$, $D_t$ can only increase with time, and, in particular, $\text{dim}(W_t)$ is non-increasing.
\item \label{property:lm-2} The final coloring $x_\ell$ returned by the algorithm satisfies $x_\ell \in [-1,1]^n$ and $\ip{v_i,x_\ell - x_0} \leq \lambda_i \|v_i\|_2 $ for all $i \in [m]$ with high probability. 
Indeed, consider the first time $t$ when $x_t(j)>1$. Then it must be that $x_{t-1}(j)< 1-\delta$ (otherwise $x_{t-1}(j)$ would not have been updated at time $t$), and thus it must be that $g_t(j) \geq \delta/\gamma$. 
As $g_t(j) = \ip{g_t,e_j} \sim N(0,\sigma^2)$ with $\sigma^2 \leq 1$ (recall Lemma \ref{lem:gaussian-subspace}),
this probability is at most $2 \exp(-(\delta/\gamma)^2/2) \leq (mn/\gamma)^{-c}$, for $c$ arbitrarily large,  by the choice of $\gamma$. 
An identical argument works for discrepancy constraints, and the claim follows by a union bound over the $O(mn/\gamma^2)$ choices for $i,j$ and $t$.
\end{enumerate}

 \noindent{\bf The main argument.} It remains to show that  $|F_\ell| \geq n/2$ (that is, $|x_{\ell}(i)| \geq 1-\delta$ for at least $n/2$ elements) with some constant probability.
As $|F_\ell| \leq n$, 
by Markov's inequality applied to $n-|F_\ell|$,
it suffices to show that $\E[|F_\ell|] \geq 0.6n$. 

This will follow from the following two claims.
First, many constraints must become tight on average. In particular,
\begin{lemma}
\label{lem:lm-vl}
$\E[|F_{\ell}| + |D_{\ell}|] \geq 3n/4$.
\end{lemma}
Second, not many discrepancy constraints can become tight. In particular, 
\begin{lemma}
\label{lem:lm-dl}
 Assuming $\delta \le 0.1$, $\E[|D_{\ell}|] \leq (1.1) n/8$.
\end{lemma}
We now prove Lemmas \ref{lem:lm-vl} and \ref{lem:lm-dl}.

\begin{proof}[Proof of Lemma \ref{lem:lm-vl}]
    As $\dim(W_\ell) \geq  n -  |F_\ell| - |D_{\ell}|$, it suffices to show that $\E[\dim(W_\ell)] \leq n/4$.
To do this, we track how $\|x_t\|_2^2$ evolves.

To this end, notice that since
 $x_t  = x_{t-1}+ \gamma g_t$, we have
\[\E[ \|x_t\|_2^2 - \|x_{t-1}\|_2^2 | x_{t-1}]  = \E[ 2 \gamma \ip{x_{t-1}, g_t} + \|g_t\|^2 | x_{t-1} ] = \gamma^2 \dim(W_t),\]
where we use that $\E[\ip{x_{t-1}, g_t}|x_{t-1}]=0$ and $\E[\|g_t\|^2 ] = \dim(W_t)$ by Lemma \ref{lem:gaussian-subspace}.
Summing up over all time steps and averaging, and observing that $\dim(W_t)$ is non-increasing, this gives  us
\[\E[\|x_\ell\|_2^2] = \|x_0\|_2^2 + \sum_{1 \leq t\leq \ell}  \gamma^2\,  \E[\dim(W_t)]  \geq  \gamma^2  \ell \, \E[ \dim(W_\ell)] \]
As $\ell = 4/\gamma^2$, this gives $\E[ \dim(W_\ell)] \leq \E[\|x_{\ell}\|_2^2]/4$. 
The result then follows, as $\|x_\ell\|_2^2 \leq n$ trivially for any $x_{\ell} \in [-1,1]^n$.\footnote{Actually, we need some care as $x_\ell$ could end up outside $[-1,1]^n$. However, the argument in Observation \ref{property:lm-2} implies that for any $j \in [n]$, $\E[x_\ell(j)^2] \leq (1-\delta)^2 + \gamma^2 < 1$. So we have $\E[\|x_{\ell}\|_2^2] < n$, which suffices for the proof of the Lemma.}
\end{proof}

\begin{proof}[Proof of Lemma \ref{lem:lm-dl}]
Consider some vector $v_i$. The corresponding discrepancy constraint becomes tight if $|\ip{v_i, x_t-x_0}| \geq (\lambda_i -\delta)\|v_i\|_2$ at some time $t$. As the subsequent coloring updates are orthogonal to $v_i$, this is same as $|\ip{v_i, x_\ell-x_0}| \geq (\lambda_i -\delta)\|v_i\|_2$ (i.e.,~the constraint is tight at time $\ell$). Let us call this event $E_i$.

As $\langle v_i,x_\ell -x_0\rangle $ is a martingale with Gaussian increments, with variance at most $\gamma^2 \|v_i\|_2^2$ per step and $\ell=4/\gamma^2$, by Lemma \ref{ch4:lm-martingale-bound} we have 
\[  \Pr[E_i] \leq 2 \exp(-(\lambda_i -\delta)_+^2/8),\]
     where we denote $(\lambda_i -\delta)_+ = \max(0,\lambda_i -\delta)$. So,
     \[ \E [ |D_\ell|] = \sum_i \Pr[E_i]  \leq \sum_i 2 \exp(-(\lambda_i -\delta)_+^2/8).\]
To relate this to condition \eqref{eqcond21}, we can use that, by the Cauchy-Schwarz inequality, 
\(
\lambda_i^2 = ((\lambda_i-\delta) + \delta)^2 \le 2(\lambda_i-\delta)^2 + 2\delta^2.
\)
Therefore, since $\delta \le 0.1$, $(\lambda_i -\delta)^2_+  \geq \lambda_i^2/2 -  \delta^2 \geq \lambda_i^2/2 - 0.01$.
Then $\Pr[E_i] \leq 2 \exp( -\lambda_i^2/16 + 0.01/8) \leq (2.2) \exp(-\lambda_i^2/16)$, and \eqref{eqcond21} gives
\[ \E[|D_\ell|] = \sum_i \Pr[E_i]  \leq \sum_i (2.2) \exp(-\lambda_i^2/16) \leq  (1.1) n/8. \qedhere\]
\end{proof}

This completes the proof of Theorem \ref{th:lm}.
\section{Rothvoss' Algorithm}
\label{sec:ch4-rothvoss}
While the Lovett-Meka algorithm suffices for most applications of the partial coloring method, it does not easily generalize to general convex bodies, so it does not give a constructive proof of Giannopoulos's partial coloring Lemma (Lemma~\ref{lm:ch3-gian}). Recall that Lemma~\ref{lm:ch3-gian} guarantees that any symmetric convex body $K$ in $\R^n$ of Gaussian measure at least $2^{-n/4}$ contains a partial coloring $\col$ with at least $\frac{n}{10}$ colored coordinates. Attempting to apply the Lovett-Meka algorithm to compute such a coloring, we are faced with the problem that the body $K$ can be defined by infinitely many inequalities, making it impossible to apply the condition~\eqref{eqcond21}.

Using an entirely different approach, Rothvoss was able to give an algorithmic version of Lemma~\ref{lm:ch3-gian}, stated in the next theorem.

\begin{theorem}
\label{thm:rothvoss-convex-discrepancy}
    Let $\epsilon$ be a sufficiently small constant, let $\delta>0$ be sufficiently small compared to $\epsilon$,\footnote{The constants $\epsilon,\delta$ can be computed easily from the proof below, but we avoid plugging specific values to avoid clutter.} and let $K \subseteq \R^n$ be a symmetric convex body with Gaussian measure $\gamma^n(K) \geq \exp(-\epsilon n)$. Then there is a randomized polynomial time algorithm that, with probability $1-\exp(-\Omega(n))$ (with the implied constant depending on $\epsilon$ and $\delta$), finds a point $x \in K \cap[-1,1]^n$ with $|\{i: x(i)\in \{-1,1\}\}| \ge \delta n$.  
\end{theorem}

 \subsection{Algorithm}
The algorithm can be described in a single line!

 Sample a random Gaussian $g \sim N(0,I_n)$, and output the point $x^*$ in $K \cap [-1,1]^n$ closest to $g$.
  That is, output
\begin{equation}
    \label{eq:rothvoss-algorithm}
x^* = \arg\min \,\{ \|x-g\|_2 : x \in K \cap [-1,1]^n\}.
\end{equation}

That's it! See Figure~\ref{fig:ch4-rothvoss} for an illustration.

\begin{figure}[hbtp!]
    \centering
\includegraphics[scale=0.88, trim= {0mm 0mm 0mm 5mm}]{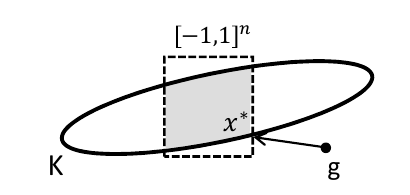}
    \caption{Rothvoss' Algorithm. The shaded area indicates $K\cap [-1,1]^n$, and $x^*$ is the closest point in $K\cap [-1,1]^n$ to the random Gaussian $g$.}
    \label{fig:ch4-rothvoss}
\end{figure}

Notice that $x^*$ can be determined efficiently, to any desired accuracy, by solving a convex program, as $f(x):=\|x-g\|_2$ is a convex function for any fixed $g$, and $K \cap [-1,1]^n$ is a convex set.
All one needs is a membership oracle for $K$.
In fact, the objective function is strictly convex\footnote{i.e.,~$f(\lambda x + (1-\lambda) y) < \lambda f(x) + (1-\lambda) f(y)$ whenever $x\neq y$, $\lambda \in (0,1)$. Indeed, we can equivalently consider the objective $\|x-g\|^2_2$ and here, $\|\lambda x + (1-\lambda) y-g\|_2^2 - \lambda \|x-g\|_2^2 - (1-\lambda) \|y-g\|_2^2 = -\lambda(1-\lambda) \|x-y\|_2^2$ for all $g$.} and thus the optimum $x^*$ is always unique. 

\subsection{Analysis}
We now show that the algorithm satisfies Theorem \ref{thm:rothvoss-convex-discrepancy}. The analysis is very elegant, and uses two key observations about how 
the objective value $v^*:=\|x^*-g\|_2$ of the program \eqref{eq:rothvoss-algorithm} behaves. Note that both $x^*$
and $v^*$ are functions of $g$.
\medskip 

(i) With high probability, $v^* \geq c\sqrt{n}$, where $c>0$ is a universal constant independent of $\epsilon$. 

\medskip
(ii) If $x^*$ has fewer than $\delta n$ coordinates $x^*(i) \in \{-1,1\}$, then with high probability, $v^*$ must be small, roughly $O(\sqrt{(\epsilon + \delta \ln 1/\delta) n})$. 


\medskip 
\noindent The result then follows directly from these observations.

\medskip

In light of this, let us take a brief detour to understand how the distance
$d(g,A):= \min\{\|x-g\|_2: x\in A\}$
of a point $g$ to a set $A \subset \R^n$ behaves, when $g \sim N(0,I_n)$ is chosen randomly. 

\medskip 
\noindent {\bf Fattenings and Gaussian Isoperimetric Inequality.}
 For $\delta>0$, let $A_\delta = \{x:d(x,A)\leq \delta\}$, the $\delta$-fattening of $A$, denote the set consisting of points within distance $\delta$ from $A$.
 Clearly, for any $g$,
 as $d(g,A) \leq \delta$ if and only if $g \in A_\delta$, we have that, for a standard Gaussian $g$ in $\R^n$,  \[\Pr[d(g,A) \leq \delta] = \gamma^n(A_\delta).\] 
 In other words, we can understand the distribution of the distance $d(g,A)$ by looking at the Gaussian measure of fattenings of $A$. 

 \smallskip
 The key tool to get a handle on $\gamma^n(A_\delta)$ is the classical Gaussian isoperimetric (GI) inequality~\cite{Borell75}, which states that among all sets with equal Gaussian measure, fattenings increase the measure of halfspaces the least. This is a Gaussian analog of the more classical isoperimetric inequality for volume (Lebesgue measure), which states that among all sets of equal volume, the Euclidean sphere has the smallest surface area. While the classical isoperimetric inequality is easily derived from the Brunn-Minkowski inequality \eqref{eq:ch3-bw}, the Gaussian isoperimetric inequality is a deeper fact. 
  
Formally, a halfspace is a set of the form $H:= \{x \in \R^n: \ip{u,x} \leq t\}$ for some unit vector $u \in \R^n$, and scalar $t\in \R$. The $\delta$-fattening of $H$ is $H_\delta:= \{x \in \R^n: \ip{u,x} \leq t+\delta\}$. By the rotational symmetry of Gaussians, we may assume that $u=e_1$ and thus,  
\[\gamma^n(H) = \Phi(t) = \gamma^1((-\infty,t]) = \int_{-\infty}^t (2\pi)^{-1/2} \exp(-t^2/2) \,dt.\] 

\begin{theorem}[GI Inequality]
\label{thm:gii}
    Let $A \subset \R^n$ be any measurable set and $H$ be any halfspace so that $\gamma^n(A) = \gamma^n(H)$. Then for every $\delta>0$, we have that  $\gamma^n(A_\delta) \geq  \gamma^n(H_\delta)$.
Equivalently,
for a set $A$, define the Gaussian isoperimetric function $t(A) := \Phi^{-1}(\gamma^n(A))$. 
Then, \[t(A_\delta) \geq t(H_\delta) = t(H)+\delta = t(A) + \delta.\]
\end{theorem}
This gives the following key fact --- if $\gamma^n(A)$ is not too small, then the measure concentrates close to $A$.
More quantitatively,

\begin{lemma} 
 \label{lm:ch4:measure-conc}
 Let $A\subset \R^n$ be a set with $\gamma^n(A) = \exp(-\beta n) \leq 1/2$. Then,
   $\Pr_{g\sim N(0,I_n)}[d(g,A) \geq  (8\beta n)^{1/2} ] < \exp(-\beta n)$.
\end{lemma}  
Let us first recall that $\Phi(0)=1/2$, and that  $\Phi(t) \leq \exp(-t^2/2)$ for $t \leq 0$, and by symmetry,  $\Phi(t) \geq 1- \exp(-t^2/2)$ for $t \geq 0$.

   \begin{proof}
Let $\delta := (8\beta n)^{1/2}$. 
We will show that $\gamma^n(A_\delta) \geq 1- \exp(-\beta n)$.
We use the notation in Theorem \ref{thm:gii}.

As $\Phi(t(A)) =\gamma^n(A) =\exp(-\beta n)\leq 1/2$, we have $t(A) \leq 0$ (as $\Phi(0)=1/2$). Using $\Phi(t) \leq \exp(-t^2/2)$ for $t \leq 0$, this gives $\exp(-\beta n) \leq \exp(-t(A)^2/2)$, and hence $t(A) \geq -(2\beta n)^{1/2}$.

By Theorem \ref{thm:gii}, $t(A_\delta) \geq -(2\beta n)^{1/2} + \delta = (2\beta n)^{1/2}$, and as $\Phi(t) \geq 1-\exp(-t^2/2)$ for $t>0$, this gives 
$\gamma^n(A_\delta) \geq 1- \exp(-\beta n)$.
\end{proof}

We now get back to proving Theorem \ref{thm:rothvoss-convex-discrepancy}.
We first show that the objective value $v^*=\|x^*-g\|_2$ in $\eqref{eq:rothvoss-algorithm}$ is typically quite large. 
 \begin{lemma}
 \label{lm:ch4:large-value}
        $\Pr_{g \sim N(0,I_n)}[v^* \leq  \sqrt{n}/6 ] = \exp(-\Omega(n))$.
        
    \end{lemma}
\begin{proof}
As the domain $K \cap [-1,1]^n$ of $x$ in \eqref{eq:rothvoss-algorithm} is contained in  $[-1,1]^n$, it suffices to show that $\Pr[d(g,[-1,1]^n) \leq \sqrt{n}/6] = \exp(-\Omega(n))$.

Consider the squared  distance $d^2(g,[-1,1]^n) $.
Then each coordinate $i$ with $|g(i)|\geq 2$ contributes at least $1$, and this happens independently with probability $2 \gamma^1([2,\infty)) \geq 1/25$ for each $i$. Thus by standard Chernoff bounds $\Pr[d^2(g,[-1,1]^n) < n/36] = \exp(-\Omega(n))$.
\end{proof}

 \paragraph{Lower bound on the number of $\pm 1$ coordinates in $x^*$.}
To get a handle on when $x^*$ has few $\pm 1$ coordinates, we need a basic observation about convex programs.


For a subset of coordinates $I \subseteq [n]$,
define the set $S(I):= 
[-1,1]^I \times \R^{[n]\setminus I}$. In particular, $S([n])=[-1,1]^n$.  We have the following.
\begin{lemma}
\label{lm:ch4-rothvoss-convex}
    Fix some $g \in \R^n$. Let $x_g^*$ be the (unique) solution to \eqref{eq:rothvoss-algorithm},\footnote{We put the subscript $g$ to emphasize the dependence on $g$.} and let $I_g = \{i \in [n]: x_g^*(i) \in \{-1,1\}\}$ be the set of coordinates for which the constraints $x(i) \in [-1,1]$ are tight at $x_g^*$. Then \[\arg\min\{ \|x-g\|_2: x \in K \cap S(I_g)\} = x_g^*.\]

In other words, the optimizer to the convex program \eqref{eq:rothvoss-algorithm}  does not change upon dropping the constraints $x(i) \in [-1,1]$ for $i \notin I_g$.
\end{lemma}

\begin{proof}
Let us denote \[y_g^*=\arg\min\{ \|x-g\|_2: x \in K \cap S(I_g)\},\] and
suppose for the sake of contradiction that $y_g^*\neq x_g^*$.
Then, it must be that  $\|y_g^*-g\|_2 < \|x_g^*-g\|$, by strong convexity and as there are fewer constraints on $y_g^*$ (as $[-1,1]^n \subset  S(I_g^*)$).
Also, clearly, 
$y_g^*(i),x_g^*(i)  \in [-1,1]$ for each $i\in I_g^*$ (as these coordinate constraints are enforced). 

Next, crucially, by the definition of $I_g^*$ we have that $x_g^*(i) \in (0,1)$ for $i\notin I_g^*$.
So even though $y_g^*(i)$ may not lie in $[-1,1]$ for $i \notin I_g^*$, there exists some $\lambda \in (0,1)$ such that the point \[\tilde{x} = \lambda x_g^* + (1-\lambda) y_g^*\] lies in $[-1,1]^n$. As both $x_g^*,y_g^* \in K$ and $K$ is convex, we also have that $\tilde{x} \in K \cap [-1,1]^n$. This implies that $\tilde{x}$ is also a feasible solution to \eqref{eq:rothvoss-algorithm}. However, by strong convexity, \[\|\tilde{x}-g\|^2_2 < \lambda \|x_g^*-g\|_2^2 + (1-\lambda) \|y_g^*-g\|_2^2 <  \|x_g^*-g\|^2_2,\] which contradicts the optimality of $x_g^*$.
\end{proof}

For a subset of coordinates $I \subset [n]$, consider the objective value  
\[  v_g^*(I) := \text{min}\{ \|x-g\|_2: x \in K \cap S(I)\},\]
 in \eqref{eq:rothvoss-algorithm} when the constraints $x(i) \in [-1,1]$ are enforced only for $i\in I$. 
 
Clearly, the objective to \eqref{eq:rothvoss-algorithm} satisfies  $v^*_g =  v^*_g([n])\geq v^*_g(I)$ for all $I \subset [n]$ (as the domain $K \cap [-1,1]^n \subset K\cap S(I) $). 
On the other hand, 
Lemma \ref{lm:ch4-rothvoss-convex} gives that  \begin{equation}
    \label{eq:ch-rothvoss-val-relation}
    v^*_g = v^*_g(I_g) \text{ for every $g \in \R^n$}.
\end{equation}

\paragraph{Completing the proof.}
We now show that, with high probability over the choice of $g$, the value $v^*_g(I)$ is small for {\em every} $I$ such that $|I| \leq \delta n$.
By Lemma \ref{lm:ch4:large-value} and \eqref{eq:ch-rothvoss-val-relation}, this would imply that $|I_g|> \delta n $ with high probability, giving Theorem \ref{thm:rothvoss-convex-discrepancy}.
\begin{lemma}
For any fixed subset $I \subset [n]$ with $|I|\leq \delta n$, 
    \[\Pr_g[ v_g^*(I) \geq \sqrt{8(\epsilon+\delta) n}] \leq \exp(-(\epsilon+ \delta) n)   .\] 
\end{lemma}
\begin{proof}
As $v_g^*(I) = d(g,K\cap S(I))$, we apply Lemma \ref{lm:ch4:measure-conc} with $A= K \cap S(I)$.
To bound  $\gamma^n(A)$ from below, note that  $S(I) = \cap_{i\in I} S_i$ is an intersection of strips $S_i = \{x: |x(i)|\leq 1\}$,  and thus by the Sidak's Lemma \ref{lm:ch3-sidak}, 
   \[\gamma^n(K \cap S(I))  \geq \exp(-\epsilon n) \cdot (\gamma^1([-1,1]))^{|I|} \geq \exp(-(\epsilon+ \delta) n),\]
   where we use that $\gamma^1([-1,1]) \geq 0.682 \geq 1/e$ and $|I|\leq \delta n$. 
\end{proof}
Recall that, by standard estimates (e.g., Proposition 3.3.3 in~\cite{guruswami2012essential}), there are at most $2^{H(\delta)n}$ subsets $I \subset [n]$ with $|I|\leq \delta n$, where $H(\delta) = -\delta \log_2 \delta - (1-\delta)\log_2(1-\delta)$ is the binary entropy function. Applying a union bound over all such subsets, we have that,  with probability at least  $1- 2^{H(\delta) n} \exp(-(\epsilon+\delta)n)$, 
 \[  v^*_g(I) \leq \sqrt{8 (\epsilon + \delta)n} \,\,\,\, \text{ for all $I$ with }  |I|\leq \delta n.\]
%
By \eqref{eq:ch-rothvoss-val-relation} this means that 
\[
\Pr[v_g^* \le \sqrt{8 (\epsilon + \delta)n}]
\ge 
1-2^{H(\delta) n} \exp(-(\epsilon + \delta)n)
- \Pr[|I_g| \ge \delta n].
\]
Choosing $\epsilon$ to be a sufficiently small constant, and $\delta \leq c \eps/ \ln (1/\eps)$ for sufficiently small $c$, and using Lemma \ref{lm:ch4:large-value}, we have that $|I_g| \ge \delta n$ with probability exponentially close to $1$. This completes the proof of Theorem \ref{thm:rothvoss-convex-discrepancy}. 

\section{Vector Discrepancy and SDPs}
\label{sec:vec-herdisc}
The partial coloring method can often give close to tight bounds for various general classes of matrices. E.g. any matrix $\mat$ with $m=n$ rows and entries in $[-1,1]$ has $O(\sqrt{n})$ discrepancy.
However, one can ask a stronger question ---
given {\em any} specific matrix $\mat$, can we efficiently find a coloring with discrepancy close to $\disc(\mat)$?

\smallskip 

Unfortunately, this is too much to hope for and discrepancy is very hard to approximate in general. For example, we will prove the following result in Chapter \ref{ch:bana-intro}.
\begin{theorem}
  For a set system $(\SS, [n])$ with $|\SS| = O(n)$ sets, it is $\mathsf{NP}$-hard to distinguish whether  $\disc(\SS) = 0$ or  $\disc(\SS) = \Omega(\sqrt{n})$.
\end{theorem}

This hardness result also suggests that LP or SDP based convex programming relaxations for discrepancy are unlikely to give useful information about the actual discrepancy $\disc(\mat)$.

While this is true in a sense, it turns out that SDPs can be very useful. In fact, the first algorithmic results for partial colorings were based on SDPs and random walks. In general, SDPs gives a flexible and powerful way to ``design'' random walks. Let us begin by exploring relaxations for discrepancy more closely. 

\medskip

\noindent {\bf LP relaxation.}
Recall that a linear program (LP) consists of variables $x(1),\ldots,x(n) \in \R$, viewed as a vector $x = (x(1),\ldots,x(n)) \in \R^n$, and the goal is to optimize some linear objective subject to some linear constraints. LPs can be solved optimally in polynomial time.

Given an $m\times n$ matrix $\mat$ with rows $v_1, \ldots, v_m \in \R^n$, the following is a natural LP relaxation for discrepancy,
\[ \min t \,\,\,\,  \text{ s.t. } \,\,    -t\leq  \ip{v_i ,x} \leq t,  \,\,\, \forall i \in [m]   \,\,\, \text{ and } \, -1 \leq x(j) \leq 1, \,\,\,\forall  j \in [n].\]
However this LP always has the trivial solution $x=\mathbf{0}$ with objective $t=0$, and hence is not at all useful.


\paragraph{Semidefinite Programs (SDPs).} 
These 
are a more general class of optimization problems, and an SDP can be viewed as an LP with
variables of the form $X(i,j)$ for $1 \leq i,j \leq n$, arranged as entries of an $n\times n$ matrix $X$, where we require that $X$ be symmetric and positive semidefinite (PSD), denoted  by $X \succeq 0$.

Formally, for matrices $A,B$, let $\ip{A, B} = \text{Tr}(A^TB) = \sum_{ij} A(i,j) B(i,j)$ denote the usual trace inner product for matrices. A SDP has the form
\[
\max \ip{C,X}   \quad \text{s.t.} \quad  \ip{A_k, X}  \leq  b_k, \quad 1 \leq k \leq m,
\quad X  \succeq  0  \]
where $C,A_1,\ldots,A_m \in \R^{n\times n}$ are arbitrary, and it can be solved in polynomial time to any desired accuracy under mild assumptions. 

Recall that a matrix $X \succeq 0$ is PSD if and only if it is the Gram matrix of some vectors $w_1,\ldots,w_n \in \mathbb{R}^n$, i.e.,~$X(i,j)= \langle w_i,w_j \rangle$. So we can equivalently view SDPs as vector programs --- where the variables are the vectors $w_i$ and we can impose arbitrary linear constraints on their inner products (but not on the $w_i$ themselves).

\medskip

\noindent {\bf SDP relaxation for discrepancy.} 
Given a matrix $\mat \in \R^{m\times n}$, and some target discrepancy bound $\lambda$, consider the following SDP, where the colors $x(j)\in \pm 1$ are allowed to be arbitrary unit vectors, and the analogous quantity for the discrepancy $|\sum_j \mat(i,j) x(j)|$ of row $i$ becomes $\|\sum_j \mat(i,j) w_j\|_2$. See Figure \ref{fig:ch4-sdp} for an example.
\begin{equation}
\label{eq:disc-sdp}
  \left\|\sum_{j} \mat(i,j) w_j\right\|_2^2 \leq\lambda^2 \quad \text{for } i \in [m],  \qquad \|w_j\|_2^2  =  1  \quad {j\in [n].}
\end{equation}
\begin{figure}[hbtp!]
    \centering
\includegraphics[scale=0.88, trim= {0mm 0mm 0mm 5mm}]{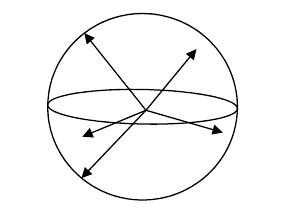}
    \caption{The unit vectors $w_j \in \R^n$ in the SDP solution.}
    \label{fig:ch4-sdp}
\end{figure}

Let us define the vector discrepancy of matrix as follows.
\begin{definition} We define
the vector discrepancy 
    $\vdisc(\mat)$ of a matrix $\mat$ as the smallest $\lambda\geq 0$ for which \eqref{eq:disc-sdp} is feasible.
\end{definition}
Clearly \ref{eq:disc-sdp} implies $\vdisc(\mat) \leq \disc(\mat)$.
However, at first glance, vector colorings do not seem to give useful information. E.g.,~for any matrix $\mat$ with entries in $[-1,1]$ (as in Spencer's theorem~\ref{thm:ch3-spencer}), the solution $w_i=e_i$, the $i$-th standard basis vector, is always feasible with $\lambda = n^{1/2}$.

Interestingly however, this SDP becomes quite useful when $\lambda \ll n^{1/2}$, as it gives non-trivial correlations between the vectors $w_i$ that we can exploit. In particular, even though $\disc(\mat)$ is hard to approximate (as we will prove in Chapter~\ref{ch:gamma2}), we can show the following instance-wise result.

\begin{theorem}
\label{thm:pseduo-alg}
Given any $m\times n$ matrix $\mat$, one can find a coloring with discrepancy $O((\log m \log n)^{1/2} \herdisc(\mat))$ in polynomial time.
\end{theorem}
We first describe the algorithm and then prove Theorem \ref{thm:pseduo-alg}.
\subsection{Algorithm}
At a high level, the algorithm performs a random walk as in the Lovett-Meka algorithm. However, instead of taking a random step in a subspace, the random increments for coordinates are carefully correlated so that the discrepancy increment for each row of $\mat$ is small. In particular, at each step it solves the SDP in \eqref{eq:disc-sdp}  (restricted to coordinates with colors not yet $\pm 1$) with $\lambda = \herdisc(\mat)$, and generates a step of the walk by using this solution as the covariance matrix.  
We now describe this formally.



\paragraph{ Algorithm.} Fix a small step size $\gamma = \max_{i,j}|\mat(i,j)|/n^2$  and let $\delta = 4 \gamma \log(2mn/\gamma)$. 
Let $\lambda = \herdisc(\mat)$.\footnote{While we do not know $\herdisc(\mat)$, we can do a binary search on $\lambda$.}  

Let $x_{t-1}$ denote the coloring and let $F_t = \{j: |x_{t-1}(j)| > 1-\delta\}$ be the set of fixed elements at the end of time step $t-1$. 
Initialize $x_0(j) =0$ for $j\in [n]$ and $F_1= \emptyset$. 

\medskip 

 \noindent At each time $t=1,2,\ldots,\ell = (4 \log n)/\gamma^2$, do the following. 
 
\smallskip 

(i)  Solve the SDP
 \begin{align*}
     & \|\sum_{j} \mat(i,j) w^t_j\|_2^2  \leq   \lambda^2 \quad \forall i\in [m],\\
     &  \|w^t_j\|_2^2  =   1   \text{ if }  j \not\in F_t,\text{ and }
   \| w^t_j\|_2^2 =  0 \text{ otherwise}. 
 \end{align*}
 

(ii) Sample $g_t \sim N(0,I_n)$, and set 
\[x_t(j) = x_{t-1}(j) + \gamma \langle g_t,w^t_j \rangle \text{ for each $j\in [n]$}.\]
\,\,\,\,\,\,\,\,\, Set $F_{t+1} = \{j: |x_t(j)| > 1-\delta\}$.

\medskip 
\noindent Output the final coloring $x(j)=\text{sign}(x_{\ell}(j))$.

\subsection{Analysis}
The ideas are similar to those in Section \ref{sec:lm}.

First, note that the SDP is always feasible at each time $t$. This is because we set $\lambda = \herdisc(\mat)$, and hence  $\disc(\mat(*,[n]\setminus F_t)) \le \lambda.$ 

Let us now consider how the colorings $x_t$ and the row discrepancies evolve.
Fix some element $j$. Its color $x_t(j)$ starts at $0$ and evolves as a martingale with updates 
\(\Delta x_t(j)= \gamma \langle w^t_j,g_t\rangle,\) 
which are distributed as a Gaussian with mean $0$ and variance $\gamma^2 \|w^t_j\|^2$.
Similarly, for the $i$-th row $v_i$ of $\mat$, its discrepancy $x_t(v_i) := \sum_{j} \mat(i,j) x_t(j)$ is $0$ at $t=0$, and evolves as a martingale with increments \[\sum_j \mat(i,j) \Delta x_t(j)  = \sum_{j} \gamma \langle g_t,\sum_j \mat(i,j) w^t_j \rangle,\]
where the right hand side is a Gaussian with mean $0$ and variance equal to \(\gamma^2 \left\|\sum_j \mat(i,j) w^t_j\right\|^2.\)

\medskip
\noindent {\bf The discrepancy bound.}
 Fix some row $i$. As $\| \sum_j \mat(i,j) w^t_j\|^2 \leq \lambda^2$ by the SDP constraint, the discrepancy increment  $\Delta x_t(v_i)$ is a Gaussian with mean $0$ and variance bounded above by $\gamma^2 \lambda^2$. 
So 
by Lemma \ref{ch4:lm-martingale-bound}, 
\[
\Pr [ |x_\ell(v_i)| \geq \ell^{1/2}  \gamma \lambda \cdot c (\log m)^{1/2}] \leq  2\exp(-(c^2 \log m)/2).\]
Plugging $\ell = (4 \log n)/\gamma^2$ and choosing $c$ large enough, a union bound over the $m$ rows gives that $\disc(\mat,\col) = O(\lambda\, (\log m \log n)^{1/2})$ whp.

\medskip
\noindent {\bf All elements become fixed.}
As long as $j \not \in F_t$ is alive, the SDP constraint ensures that 
$\|w^t_j\|_2=1$. So, $\Delta x_t(j)$ is a mean $0$ Gaussian with variance $\gamma^2$, and at each step,
$x_{t}^2(j)$ increases  on average by \[\E[x_{t}^2(j)|x_{t-1}(j)] - x^2_{t-1}(j)  = \E[(\Delta x_t(j))^2|x_{t-1}] =\gamma^2.\] 
A similar argument as in Lemma \ref{lem:lm-vl} gives that $x_t(j)$ will reach $\pm 1-\delta$ in
$(2/\gamma^2)$ steps with probability at least $1/2$ (irrespective of where it started).
Dividing $\ell = 4\log n/\gamma^2$ into $2 \log n$ intervals of size $2/\gamma^2$, the probability that $j$ stays active (i.e., is not fixed) till the end of the algorithm is at most $1/n^2$. The result then follows by a union bound over the $n$ elements.

\medskip
\noindent
{\bf Rounding Error.}
We condition on the event that at the end of the algorithm all elements are fixed, i.e., $F_{\ell+1} = [n]$.
As a color is not updated once $|x_t(j)|\geq 1-\delta$, our choice of $\delta$ ensures that  $|x_\ell(j)| \leq 1$ for all $j$ with high probability.
So, rounding the $x_\ell(j)$ to $\pm 1$ at the end incurs error at most $n \delta =  o(\max_{ij}|\mat(i,j)|)$ per row, which is negligible as $\herdisc(\mat) \geq \max_{ij} |\mat(i,j)|$.

\section{Bibliographic Notes}
\label{sec:ch4-notes}
The first algorithmic approach for partial coloring was given by Bansal \cite{B10}, who gave efficient algorithms for various applications such as the $O(n^{1/2})$ bound for Spencer's problem with $m=O(n)$ sets, and the $O(t^{1/2} \log n)$ bound for the Beck-Fiala problem. His approach was based on SDPs and random walks. Theorem \ref{thm:pseduo-alg} is also from \cite{B10}. 

Even though the approach of \cite{B10} gave efficient algorithms, the algorithms relied on the existence of partial coloring as given by Theorem \ref{lm:ch3-partial} to argue the feasibility of the underlying SDPs, and hence was not fully constructive.
The first truly constructive algorithm for the partial coloring method was given by Lovett and Meka \cite{lovettmeka}, which we described in Section \ref{sec:lm},  

The algorithm in Section \ref{sec:ch4-rothvoss}, for general convex bodies, is due to Rothvoss \cite{Rothvoss17}.
Another elegant algorithm, that also works for general symmetric convex bodies $K$, is due to Eldan and Singh \cite{EldanS18}. This algorithm is also very simple to state ----- Pick a uniformly random direction $\theta \in \R^n$ and solve the linear program $\max \ip{\theta,x}$ subject to $x \in K \cap [-1,1]$. 

The Gaussian Isoperimetric inequality was proved by Sudakov and Tsirelson \cite{sudakov-tsirelson78} and independently by Borell \cite{Borell75}. Several other proofs are now known and we refer to the note by Ledoux \cite{Ledoux-survey}. We will see a generalization of this inequality called  Ehrhard's inequality in Chapter \ref{ch:bana-proof}.

Interestingly Rothvoss's algorithm enjoys a guarantee similar to Theorem~\ref{thm:pseduo-alg}. Dadush, Nikolov, Talwar, and Tomczak-Jaegermann showed that, using Rothvoss's algorithm with $K = \{\col \in \R^n: \|\mat \col\|_{\infty} \le C\herdisc(\mat)\}$ for a large enough constant $C$, we get a partial coloring with discrepancy $O(\herdisc(\mat))$. We can then extend the partial coloring to a full coloring $\col$ with discrepancy $O(\log(n)\herdisc(\mat))$~\cite{anynorm}. They also showed that this holds more generally for discrepancy measured in any norm.

The SDP-based approach in Section \ref{sec:vec-herdisc} was also used in \cite{BDG16}, with some additional constraints, to obtain the first algorithmic bound for the Koml\'{o}s problem, matching Banaszczyk's non-constructive  bound \cite{Bana98}. We will see Banaszczyk's method in Chapter \ref{ch:bana-intro} and the algorithm of \cite{BDG16} in Chapter \ref{ch:algo-komlos-bana}.

Several other algorithmic approaches have also been developed. 
These are typically stated for specific problems, but often the underlying idea is more generally applicable.
Harvey, Schwartz and Singh \cite{HSS14}  developed an algorithm based on inscribing ellipsoids and random walks.  
Rothvoss and Reis \cite{ReisR23} consider another variant that combines ideas from the algorithms in Section \ref{sec:lm} and \ref{sec:ch4-rothvoss}.
More recently, deterministic algorithms based on 
barrier potential functions \cite{BansalLV22} and regularized optimization \cite{PesentiV23} have also been developed. Remarkably, the approach of \cite{PesentiV23}, in addition to being algorithmic, also improves the constant for Spencer's problem for $m=n$ sets from  $5.32$ in \cite{spencer1985} to $3.67$.

Beyond discrepancy, these approaches have also led to several new results in areas such as approximation algorithms, computational geometry, machine learning, differential privacy and graph sparsification. We do not discuss these here due to lack of space, but would like to highlight the breakthrough result of Rothvoss \cite{Rothvoss16} and Hoberg and Rothvoss  \cite{HobergR19}, who used these ideas to obtain an additive $O(\log n)$ approximation for the classical bin packing problem, improving on the long-standing $O(\log^2 n)$ bound by Karmarkar and Karp \cite{KarmarkarK82}. For a general connection between discrepancy and rounding techniques, we refer to \cite{Bansal19}.

There has also been a lot of recent interest in designing discrepancy algorithms with fast running times \cite{JainSS23, JambulapatiRT24}. In Chapter \ref{ch:online}, we will see a beautiful algorithm due to Alweiss, Liu and Sawhney \cite{alweiss2020discrepancy}, that runs in time linear in $\text{nnz}(\mat)$, the number of non-zero entries of $\mat$, and also works in the online setting.

Matou\v{s}ek \cite{Matousek-detlb} used the connection between vector discrepancy and hereditary discrepancy in Section \ref{sec:vec-herdisc}, together with SDP duality, to show a remarkable result that the hereditary discrepancy of any matrix $\mat$ is tightly characterized by the so-called determinant lower bound (see Chapter \ref{ch:gamma2} for the definition). The lower bound  $\detlb(\mat) \leq \herdisc(\mat)$ was shown by Lov\'asz, Spencer and Vesztergombi~\cite{LSV}, and we give an exposition of their proof in Chapter~\ref{ch:gamma2}. Matou\v{s}ek showed the upper bound $\herdisc(\mat) \leq O(\log n)^{3/2} \detlb(\mat)$.
This was further improved by Jiang and Reis \cite{JiangReis22} to a logarithmic bound, which is also the best possible~\cite{LiNikolov24}.
We will explore some related ideas further in Chapter \ref{ch:gamma2}.

\chapter{Banaszczyk's Method }\label{ch:bana-intro}

A problem with the partial coloring method is that it requires $O(\log
n)$ rounds to obtain a full coloring, often resulting in an extra
$O(\log n)$ factor loss in the overall discrepancy bound. For example,
for the Koml\'os problem  (Problem~\ref{prob:komlos}), we incurred $O(\log n)$ discrepancy in
Exercise~\ref{ex:ch3-komlos}, even though there exists a partial
coloring with  $O(1)$ discrepancy.

In a remarkable result~\cite{Bana98}, Banaszczyk gave a different way
to find a full coloring directly, which often gives better
bounds\footnote{Though it does not seem to be directly comparable with
  the partial coloring method. For example, we do not know how to
  obtain the $O(\sqrt{n})$ for Spencer's problem using Banaszczyk's
  result}, and, in particular, improves the bound 
to $O(\sqrt{\log n})$. 
This was the best known bound for the Koml\'os problem until very recently \cite{bansaljiang2025}.

 We now state Banaszczyk's theorem. 
 In  addition to giving an improved bound for the Koml\'os problem,
 it gives a powerful and general geometric result which has many other
 surprising applications. We  will discuss some of them in this chapter, and others later in the book. 

\begin{theorem}[Banaszczyk~\cite{Bana98}]\label{thm:bana}
Let $K$ be an arbitrary closed convex set in $\mathbb{R}^m$ with Gaussian measure $\gamma^m(K)\ge 1/2$, and let $v_1,\dots,v_n\in\mathbb{R}^m$ be arbitrary vectors with length $\|v_i\|_2$  at most $1/5$. Then there exist signs $x(1),\ldots,x(n) \in \{-1,1\}$ such that $\sum_{i=1}^n x(i) v_i \in K$. 
\end{theorem}

\begin{remark}
    At first glance, the constants $1/2$ and $1/5$ above may seem
    somewhat arbitrary. However, $\gamma^m(K)\geq 1/2$ ensures that
    the origin ${0}$ is in $K$, which is necessary if say, each
    $v_i={0}$. Similarly, some upper bound on $\|v_i\|_2$ is also
    necessary; e.g.,~if $K$ is the strip $\{x \in \R^m: |x(1)|\leq
    a\}$ with $a\approx 0.6745$ so that $\gamma^m(K)=1/2$, then no
    signs may exist if $\|v_i\|_2>a$, e.g.,~if $n=1$ and $v_1=be_1$
    (i.e., the vector with $b$ in the first coordinate and $0$'s
    everywhere else) for some $b>a$.
\end{remark}
For most applications, these constants do not really matter and there is a lot of slack. For example, if $K$ is the unit $\ell_2$-ball of radius $\sqrt{m}$, then $\gamma^{m}(K) = \Theta(1)$, but $\gamma^m(2K)$ is exponentially close to $1$.

\section{Applications}
The ability to choose the body $K$ in Theorem \ref{thm:bana} makes the result quite powerful. Let us consider some  examples. 

\subsection{Koml\'os Problem}
Theorem \ref{thm:bana} directly gives the following for the Koml\'{o}s problem.
    
\begin{theorem}
\label{thm:komlos-bana}
Given vectors $v_1,\ldots,v_n \in \R^m$ with $\|v_i\|_2 \leq 1$, there
exists a coloring $\col \in \{-1,+1\}^n$ such that   $\|\sum_i \col(i)
v_i\|_\infty \lesssim \sqrt{\log(2m)}$.
\end{theorem}
\begin{proof}
    We apply Theorem~\ref{thm:bana} to $K\eqdef[-t,t]^m$ with $t$ large enough so that $\gamma^m(K)\geq 1/2$. As
    \[\gamma^m(K) = \Pr[ g \in K] = \Pr[\|g\|_\infty \leq t] \geq 1/2 ,\]   
    where $g=(g(1),\ldots,g(m))$ is a standard Gaussian random vector
    in $\R^m$. We have
     $\Pr[|g_i| \geq t] \leq e^{-t^2/2}$ and setting $t= \sqrt{2\ln(2m)}$ suffices by a union bound over the $m$ coordinates.
\end{proof}

\begin{remark}
\label{rem:bana-komlos-trick}
   This also gives an $O(\sqrt{\log n})$ bound using a simple
   trick. \footnote{Even though the span of $v_1,\ldots,v_n$ has
     dimension at most $n$, we cannot assume that $m=n$ here as the
     coordinate directions of $\R^m$ may not be aligned with this
     subspace.} Let $\mat$ denote the matrix with columns
   $v_1,\ldots,v_n$, and let the rows of $\mat$ be $u_1, \ldots, u_m
   \in \R^n$. Then $\sum_{i=1}^m\sum_{j=1}^n \mat(i,j)^2 \leq n$  as
   each $\|v_j\|_2\leq 1$, and thus there are at most $n^2$ rows $u_i$
   with $\|u_i\|_2^2 \geq 1/n$. For the remaining rows, by
   Cauchy-Schwarz, any coloring $\col \in \{-1,+1\}^n$  achieves
   discrepancy at most
   \(
     |\ip{u_i,\col}| \le \|u_i\|_2 \|\col\|_2 \le 1,
   \)
   as $\|\col\|_2 = \sqrt{n}$. Thus we can assume that $m\leq n^2$.
\end{remark}

\subsection{Discrepancy and the $\gamma_2$ Norm}

The above example can be substantially generalized by changing the
choice of the body $K$. In Theorem~\ref{thm:komlos-bana}, we chose $K$
to be the $\ell_\infty^m$ ball, sufficiently scaled up, and bounded
the discrepancy of a matrix $\mat$ whose columns have $\ell_2$ norm at
most $1$. On the other hand, if $\mat$ has rows with $\ell_2$ norm at
most $1$, we can choose $K$ to be $\{x: \|\mat x\|_\infty \le t\}$ for
a large enough $t$, and apply Theorem~\ref{thm:bana} to the standard
basis vectors. (For this case, Theorem~\ref{thm:bana} is an overkill,
and picking a uniformly random coloring gives the same result.) The
following result interpolates between these two cases.


\begin{lemma}\label{lm:fact-ub}
  For any $m\times n$ matrix $\mat$ that can be written as the matrix
  product $\mat = UV$ of two matrices $U \in \R^{m\times k}$ and $V\in
  \R^{k\times n}$, where each row
  $u_i$ of $U$ satisfies $\|u_i\|_2 \le r$, and each column $v_j$ of
  $V$ satisfies $\|v_j\|_2 \le c$, 
  we have
  \(
  \disc(\mat) \lesssim rc \sqrt{\ln(2m)}.
  \)
\end{lemma}
\begin{proof}
  Let us define the convex body
  \[
    K = \{ y \in \R^k: \|Uy \|_{\infty} \leq r\sqrt{2\ln{(2m)}}  \}
  \]
  To prove the lemma, we will apply Banaszczyk's theorem to the body
  $K$ and the column vectors of the matrix
  $\frac{1}{5c} V$. We first show that
  the Gaussian measure of $K$ is at least $\frac12$. For a standard
  Gaussian random vector $\grv$ in $\R^k$, by the union bound we have
  that, for $t \eqdef r\sqrt{2\ln(2m)}$,
  \begin{align*}
    \gamma^k(K) & = \Pr[\grv \in K]
    = \Pr[\|U\grv\|_\infty \le t]
    \ge 1 - \sum_{i = 1}^m \Pr[|\ip{u_i,\grv}| \ge t]\\
    &\ge 1 - \sum_{i = 1}^m \exp(-t^2/2\|u_i\|_2^2)
      \ge 1 - m \exp(-t^2/2r^2) \ge \frac12.
  \end{align*}
  Above we used that $\ip{u_i,\grv}$ is distributed like a
  one-dimensional Gaussian with mean $0$ and variance
  $\|u_i\|_2^2$.
  
  Now Theorem~\ref{thm:bana} applied to the columns $\frac{1}{5c}v_1,
  \ldots, \frac{1}{5c}v_n$ of $\frac{1}{5c}V$ gives that there exists a $\pm 1$ coloring $\col$ such that
    $\sum_{j=1}^n \col(j) \frac{v_j}{5c}\in K$ or equivalently that $V\col \in 5c K$.
  Recalling the definition of $K$, this gives
  \[
    \|\mat\col\|_\infty = \|UV\col\|_\infty \le 5rc\sqrt{2\ln{(2m)}}. \qedhere
  \]
\end{proof}

\medskip
\noindent {\bf The $\gamma_2$-norm of a matrix $M$.}
The infimum of the product $rc$ over all factorizations $\mat = UV$ is
called the $\gamma_2$-norm of $\mat$, and denoted
$\gamma_2(\mat)$. In this notation, Lemma~\ref{lm:fact-ub} can be
restated as
\begin{equation}
  \label{eq:fact-ub}
  \disc(\mat) \le \gamma_2(\mat)  \sqrt{\ln(2m)}
  \quad \text{for all matrices $\mat \in \R^{m\times n}$}
\end{equation}

It has several interesting
properties and is closely related to the hereditary discrepancy of
$\mat$. In Chapter~\ref{ch:gamma2} we explore this further, and
develop a general method of proving upper and lower bounds on
hereditary discrepancy using the $\gamma_2$ norm.

\section{Vector Discrepancy for Koml\'os and the Cloning Trick}
\label{sec:komlos-sdp-bound}


Recall that in vector discrepancy,
instead of asking for a $\pm$
coloring,
we relax each $\col(j)$ to be some vector on the unit sphere in
$\R^d$ for an arbitrary dimension $d$. As we saw in Section \ref{sec:vec-herdisc}, vector coloring is
equivalent to a semidefinite programming relaxation of discrepancy.

We now give a surprising application of Banaszczyk's theorem to show a constant upper bound for the vector discrepancy of instances of the Koml\'os problem, i.e., matrices with columns whose $\ell_2$ norm is at most $1$.
The proof uses what we call ``the cloning trick'', which has several
other applications.


\begin{theorem} 
\label{thm:komlos-vec-disc}
For any vectors $v_1, \ldots, v_n \in \R^m$ with $\|v_j\|_2 \le 1$ for
all $j \in [n]$, there exists a positive integer $d$ and a coloring
$x:[n] \to \R^d$ with $\|x(j)\|_2 = 1$ for all $j \in [n]$ such that the maximum vector discrepancy over all rows $i\in [m]$ satisfies
\begin{equation}\label{eq:komlos-vec-disc}
  \max_{i = 1}^m \Bigg\|\sum_{j = 1}^n x(j) v_j(i) \Bigg\|_2 \lesssim 1.
\end{equation}
\end{theorem}
\begin{proof}
The idea is to create several disjoint copies of the instance and
apply Banaszczyk's theorem (with a suitable body $K$) to obtain
several $\pm 1$ colorings, such that for each row $i$, the colorings
have low discrepancy on {\em average}. We can these combine these
colorings into a single vector coloring by associating each coloring to  one coordinate of the vector
coloring. The key point is that requiring each row to have low discrepancy on average is a much more relaxed condition, and we do not need to lose the $\sqrt{\log n}$ factor that we incur, if we are only allowed a single coloring.  The details follow.

Let $d$ be a parameter that we will set later. (It will turn out that $d =
O(\log m)$ suffices.)  Let, further, $e_1, \ldots, e_d$ be the
standard basis of $\R^d$, and define $v_{\ell,j} \eqdef e_\ell\otimes v_j \in \R^{dm}$, where $\otimes$ is the Kronecker product. I.e., vector
$v_{\ell,j}$ consists of $d$ blocks of $m$ coordinates each, and each
block is all zeroes, except the $\ell$-th block, which is equal to
$v_j$. See Figure~\ref{fig:ch6a-cloning} for a picture of the block
structure of $(v_{\ell,j})_{\ell,j}$.
We have $\|v_{\ell,j}\|_2 = \|e_\ell\|_2 \|v_j\|_2  = \|v_j\|_2
\le 1$,
by basic properties of Kronecker products.

\begin{figure}[htp]
  \centering
  \includegraphics{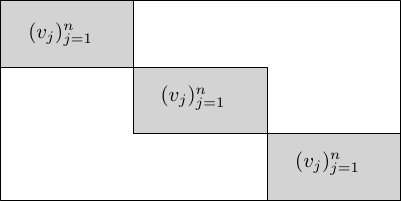}
  \caption{The block structure of $(v_{\ell,j})_{\ell,j}$ for
    $d = 3$. All unshaded entries are $0$'s. }
  \label{fig:ch6a-cloning}
\end{figure}


Consider now the convex body  $K = \cap_{i=1}^m K_i$ in $\R^{dm}$ defined by the intersection of the bodies
\[K_i=\{ y\in \R^{dm} :  \sum_{\ell=1}^d y(\ell,i)^2 \leq 2d\},\]
where we use $y(\ell,i)$ as shorthand for $y((\ell-1)m + i)$. In other
words, we think of vectors in $\R^{dm}$ as having $d$ blocks of
$m$ coordinates each, and $y(\ell, i)$ is the $i$-th coordinate in the
$\ell$-th block.
We claim that $K$ has large Gaussian measure.
\begin{claim} 
  For $d \lesssim \log(2m)$ large enough,  $\gamma^{dm} (K) \geq 1/2$.   
\end{claim}
\begin{proof}
  As $K = \cap_{i=1}^m K_i$, by the union bound it suffices to show
  that $\gamma^{dm} (\R^{dm} \setminus K_i ) = \Pr[ \sum_{\ell=1}^d
  \grv(\ell,i)^2 > 2 d] \leq 1/2m$,  where the $g(\ell,i)$ are
  independent standard Gaussian random variables.

  For each $i$ and $\ell$, we have $\E[g(\ell,i)^2]=1$ and $g(\ell,i)^2$ has subexponential tails, and thus by 
 standard tail bounds, \eg, \cite[Corollary 2.8.3]{vershynin}, we have
 \[\Pr\Big[\sum_{\ell=1}^d g(i,\ell)^2   \geq (1+\eps) d \Big]  \leq  2\exp(- cd \min (\eps,\eps^2)),\]
for some fixed $c>0$. Plugging $\eps=1$, $d \ge \frac{1}{c}
\ln(4m)$ gives the result.
\end{proof}
Consider the coloring $x' \in \{-1,1\}^{dn}$
obtained by applying Theorem \ref{thm:bana}
to the $dn$ vectors $v_{\ell,j}$ and the body $K$. Let us use
$x'(\ell,j)$ for the color (\ie, value in $\{-1, +1\}$) given to
$v_{\ell,j}$.  Let $y$ denote the
resulting discrepancy vector $\sum_{\ell=1}^d\sum_{j = 1}^n x'(\ell,j)
v_{\ell,j}$.
 Then $\frac{1}{5} y \in K$, and, by the definition of $K$, we have
 that  $\frac{1}{5} y \in K_i$ for each $i\in [m]$. Thus, $\sum_{\ell=1}^d y(\ell,i)^2 \leq 10 d$.
 Also notice that, as $v_{\ell,j}(\ell,i) = v_{j}(i)$ for each $\ell$,
 and $v_{\ell,j}(\ell',i) = 0$ whenever $\ell \neq \ell'$, the entries of $y$ are
\[  y(\ell,i)  = \sum_{j=1}^n x'(\ell,j) v_{\ell,j}(\ell,i) =  \sum_{j=1}^n x'(\ell,j) v_{j}(i).\]
This gives us that
\begin{equation}
\label{eq:cloning-trick-bound}
\sum_{\ell=1}^d \Big(\sum_{j=1}^n x'(\ell,j) v_{j} (i)\Big)^2 =
\sum_{\ell=1}^d y(\ell,i)^2  \leq 50 d.
\end{equation}
Consider the vector coloring $x:[n] \to \R^d$ with vectors
\[x(j) \eqdef \frac{1}{\sqrt{d}} (x'(1,j),\ldots,x'(d,j)) \qquad \text{ for $j \in [n]$}.\]
Clearly, $\|x(j)\|_2=1$ as each $x'(\ell,j)$ is in $\{-1,+1\}$. 
Moreover, for each $i\in [m]$, the vector discrepancy is at most  
    \begin{align*}
      \bigg\| \sum_{j=1}^n x(j) v_j(i) \bigg\|_2^2
      & = \frac{1}{d} \sum_{\ell=1}^d \Big( \sum_{j=1}^n x'(\ell,j) v_j(i)\Big)^2   = \frac{1}{d} \sum_{\ell=1}^d  y(\ell,i)^2 \leq 10,  
         \end{align*}
   where the last two steps follow from \eqref{eq:cloning-trick-bound}.
\end{proof}

Let us mention another equivalent view of
Theorem~\ref{thm:komlos-vec-disc}. Instead of defining the vectors
$x(j) = \frac{1}{\sqrt{d}} (x'(1,j), \ldots, x'(d,j))$, we think of
$x'$ as defining a random coloring. In particular, we can set
$x(j) \eqdef x(\ell,j)$, where $\ell \in [d]$ is chosen uniformly at
random. Then the proof of Theorem~\ref{thm:komlos-vec-disc} shows that
\begin{equation}\label{eq:ch6a-veckomlos-rand}
  \max_{i = 1}^m \E\Big(\sum_{j} x(j) v_j(i)\Big)^2 \lesssim 1.
\end{equation}

Theorem~\ref{thm:komlos-vec-disc} in fact holds with the implied constant on the right hand side of \ref{eq:komlos-vec-disc} equal to $1$, which is clearly best possible. We give a different proof of a more general version of this better upper bound in Chapter~\ref{ch:online}. 

\section{The Main Geometric Core}
\label{sec:bana-core}
The original proof of Theorem \ref{thm:bana} by Banaszczyk was non-algorithmic and based on clever convex geometric argument and properties of Gaussian measures. 
The main technical core of his proof is the following geometric fact. See also Figure \ref{fig:bana:kstar}.
\begin{theorem}
\label{thm:bana-star}
For any convex body $K \in \R^m$ with $\gamma^m(K) \geq 1/2$ and any vector $u \in \R^m$ with $\|u\|_2 \leq 1/5$, there is a convex body 
$K*u$ contained in $(K-u) \cup (K+u)$ such that $\gamma^m(K*u) \geq \gamma^m(K)$. 
\end{theorem}

\begin{figure}[hbtp!]
    \centering
\includegraphics[scale=0.5, trim={3mm 2mm 2mm 3mm}]{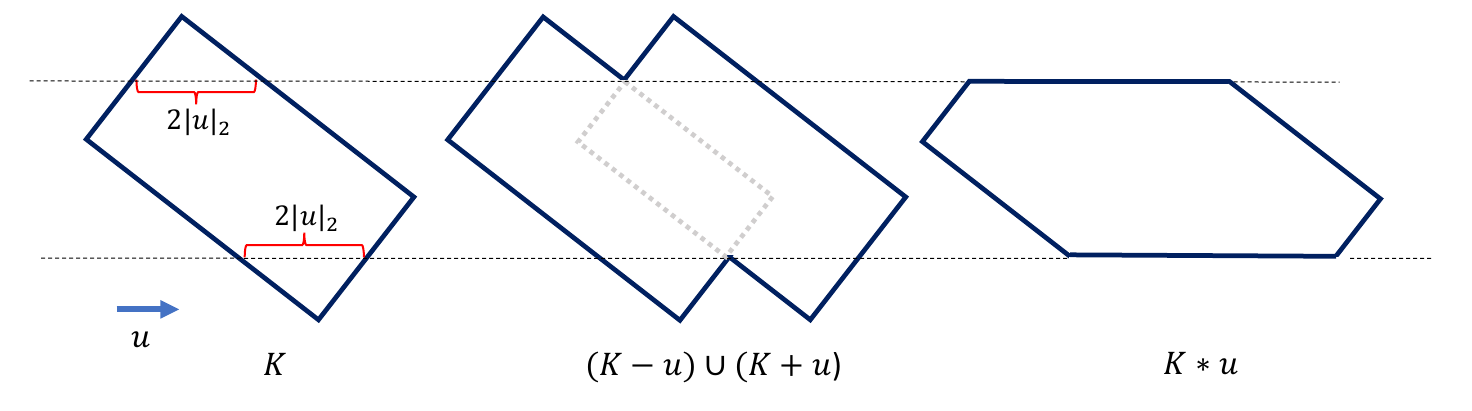}
    \caption{The body $K*u$ obtained from $K$.}
    \label{fig:bana:kstar}
\end{figure}

Chapter \ref{ch:bana-proof} will be devoted entirely to the proof of Theorem \ref{thm:bana-star}.
For now, it does not matter how the body $K*u$ is obtained, and let us see how Theorem \ref{thm:bana} follows directly from Theorem \ref{thm:bana-star}. 
\begin{proof}(Theorem \ref{thm:bana}).
We apply induction on the number of vectors $n$.

The base case of $n=0$ holds as the origin $\mathbf{0}$ is
in $K$. Indeed, if $\mathbf{0} \notin K$, then, because $K$ is convex
and closed,
there must be a hyperplane $h$ separating $\mathbf{0}$ and $K$, \ie,
there is some $y \in \R^m$ and $\alpha < 0$ such that $\ip{x,y} \le  \alpha$ for all $x \in
K$ (this is the hyperplane separator theorem, see~\cite[Corollary~11.4.2]{Rockafellar}). However,
\[
\gamma^m(K) \le \gamma^m(\{x\in \R^m: \ip{x,y} \le \alpha\}) =
\Pr[\ip{\grv,y} \le \alpha] < \frac12,
\]
for a standard Gaussian random vector $\grv \in \R^m$, contradicting
the assumption $\gamma^m(K) \geq 1/2$.

Suppose the result holds for $n-1$.
Consider the convex body $K'=K*v_n$. Then $\gamma^m(K') \geq 1/2$ by Theorem \ref{thm:bana-star}, and  
by the inductive hypothesis 
there exist signs $x_1,\ldots,x_{n-1} \in \{-1,1\}$ such that 
$u := x_1 v_1 + \cdots + x_{n-1} v_{n-1} \in K'$. Now as $K' \subseteq (K-v_n) \cup (K+ v_n)$, 
either $u \in K-v_n$ in which case we set $x_n=1$ so that 
$x_1v_1+ \ldots + x_{n-1}v_{n-1} + x_n v_n =  u+ v_n \in K$, or, otherwise
$u \in K+v_n$ and we can set $x_{n}=-1$.
\end{proof}
Despite the simplicity of this inductive argument, it is unclear
how to use it to obtain an efficient algorithm to find the
coloring $x$. In general, a convex body $K$ can be very complicated
and require exponential size to describe it in any reasonably accurate
way. Even if the initial body $K$ is simple (like the unit
cube or the Euclidean ball), a few applications of the $*$ operator
above might, potentially, make it very complicated. In particular,
there is no known way to construct an efficient membership oracle for
the intermediate bodies in the inductive proof of
Theorem~\ref{thm:bana} above using a membership oracle for the initial $K$.

 



\section{Equivalence to Sub-Gaussianity}
\label{sec:bana-subgaussian}
Recall that a Gaussian random variable $\grv$ with mean $0$ and
variance $\sigma^2$ has a  moment generating function
$ \E[\exp(\lambda \grv)] = \exp(\lambda^2 \sigma^2/2)$ for all $\lambda
\in \R$. This property motivates the following (standard) definition.

\begin{definition}[sub-Gaussian Random Variable]
We say that a mean $0$ real-valued random variable $X$ is $\sigma$-\df{sub-Gaussian} if  
\[ \E[ \exp(\lambda X)] \leq \exp(\sigma^2 \lambda^2/2) \quad \text{for all $\lambda \in \R$}.\]  
\end{definition}
As the name suggests, if $X$ is $\sigma$-sub-Gaussian then it has sub-Gaussian tails.\footnote{Sometimes in the literature, this property is referred to $\sigma^2$-sub-Gaussian instead, but we use $\sigma$-subguassian throughout the text.} Indeed, setting $\lambda = t/\sigma^2$ above and applying Markov's inequality gives that for any $t>0$,
\begin{align*}\Pr[ X \geq t] &  \le \inf_{\lambda \ge 0} \Pr[\exp(\lambda X)  \geq
  \exp(\lambda t)] \\ 
  & \leq \inf_{\lambda \ge 0} \exp(\lambda^2 \sigma^2/2 - \lambda t) = \exp(-t^2/2\sigma^2).\end{align*}
Notice that if $X_1,\ldots,X_n$ are independent and $X_i$ is $\sigma_i$-sub-Gaussian with mean-zero, the sum $X = \sum_i X_i$ is $\left(\sum_i \sigma_i^2\right)^{1/2}$-sub-Gaussian as
\[ \E\exp(\lambda X) = \prod_i \E[\exp(\lambda X_i)] \leq \prod_i \exp(\sigma_i^2 \lambda^2/2).\]
Also, any $X$ distributed on $[-c,c]$ with mean-zero is
$c$-sub-Gaussian: this result is also known as Hoeffding's lemma.

The definition above is extended to random vectors by asking that each
one-dimensional marginal is sub-Gaussian.

\begin{definition}[sub-Gaussian Random Vector] A vector valued random
  variable $X \in \R^m$ with mean $0$ is $\sigma$-sub-Gaussian if for all vectors $x \in \R^m$, we have that  
\[ \E \exp( \ip{X,x})] \leq \exp(\sigma^2 \|x\|^2/2).\]
Equivalently, the projection $\ip{X,\theta}$ of $X$ in any direction
$\theta$, $\|\theta\|_2 = 1$ is a $\sigma$-sub-Gaussian real-valued
random variable.  
\end{definition} 

\paragraph{An equivalent version of Banaszczyk's Theorem.}
Surprisingly,  Theorem \ref{thm:bana} 
is equivalent (up to various $O(1)$ factors) to a statement that does not involve mentioning any convex body $K$.


Consider the following statement.
\begin{theorem}
\label{thm:bana-subgaussian}
    Given arbitrary vectors $v_1,\ldots,v_n \in \R^m$ with $\|v_j\|_2
    \leq 1$ for every $j \in [n]$, there is a probability distribution
    $D$ over colorings $x \in \{-1,1\}^n$ such that the discrepancy
    vector $Y \eqdef \sum_{j=1}^n x(j) v_j$ is $O(1)$-sub-Gaussian
    when $x$ is sampled according to $D$.
\end{theorem}

It turns out that Theorems~\ref{thm:bana}~and~\ref{thm:bana-subgaussian} are equivalent up to constant factors. We sketch the argument below.

To show that Theorem \ref{thm:bana} implies Theorem~\ref{thm:bana-subgaussian} one can
 use a cloning trick similar to that in Section~\ref{sec:komlos-sdp-bound}.
The idea is to make multiple copies of the instance and apply Theorem
\ref{thm:bana} with a suitably defined convex body $K$, and view the
uniform distribution on the colorings for these copies as the
distribution $D$. This is similar to the alternative view of
Theorem~\ref{thm:komlos-vec-disc} above, but, instead of asking for
the second moment bound \eqref{eq:ch6a-veckomlos-rand}, we ask that
the random vector $\sum_{j = 1}^n x(j) v_j$ is $O(1)$-sub-Gaussian. 
The details of showing this are analogous to the proof of
Theorem~\ref{thm:komlos-vec-disc}, but are more involved, as one needs
to make exponentially many copies. We refer to \cite{rothvoss-online} for details.

The reverse implication will be more relevant for us, since we will
use it to give an algorithm for Theorem~\ref{thm:bana} in Chapter
\ref{ch:gram-schmidt}. In particular, we will show that a
distribution $D$ satisfying Theorem~\ref{thm:bana-subgaussian} can be
sampled efficiently.

 Intuitively, Theorem \ref{thm:bana-subgaussian} should imply Theorem \ref{thm:bana} because 
 the random vector $Y$ in Theorem \ref{thm:bana-subgaussian} is $O(1)$ sub-Gaussian, and as $\gamma^m(K)\geq 1/2$, a standard Gaussian $\grv \in \R^m$ lies in $K$ with probability at least $1/2$.
So $Y$, being ``more concentrated'' than a Gaussian with variance
$O(1)$ in every direction, should also lie in some $O(1)$-factor scaling of $K$ with
constant probability. 

While this intuition can be made rigorous, a
significant hurdle is that the definition of a sub-Gaussian random
vector only requires that the one-dimensional marginals are
sub-Gaussian, and does not make any requirements about
higher-dimensional marginals. Nevertheless, the following comparison
theorem, which is a deep result of Talagrand, allows us to get around
this obstacle. We refer to Talagrand's book~\cite{Talagrand-book} for
a proof; see also~\cite[Section 8.6]{vershynin}.

\begin{theorem}[Talagrand's Comparison Theorem]\label{thm:ch6a-talagrand}
  For any bounded set $T \subseteq \R^m$, and any $\sigma$-sub-Gaussian
  random vector in $\R^m$, we have
  \[
    \E \sup_{t \in T} \ip{X,t} \lesssim \sigma \,\,  \E\sup_{t \in T}
    \ip{\grv, t},
  \]
  where $\grv$ is a standard Gaussian random vector. 
\end{theorem}

To use Theorem~\ref{thm:ch6a-talagrand}, we need the notion of a polar set, defined for a set $K
\subseteq \R^m$ by
\[
  K^\circ \eqdef \{y \in \R^m: \ip{x,y} \le 1 \,\,\,\, \forall x \in K\}.
\]
As a consequence of the hyperplane separator theorem, when $K$ is convex, closed, and
contains $0$, $K^{\circ \circ} = K$~\cite[Chapter 5]{Matousek-discgeom}. 

We give the main implication next. To avoid some technicalities, we
consider only the case where $\gamma^m(K)\geq 3/4$. This does not make
a difference for most applications of Theorem~\ref{thm:bana}.

\begin{theorem}\label{thm:bana-subg-suff}
Suppose that the  random vector $Y$ in Theorem \ref{thm:bana-subgaussian}  is $\sigma$-sub-Gaussian.
There is a universal constant $c > 0$ such that, for any closed convex
set $K \subseteq \R^m$ containing $0$ with Gaussian measure $\gamma^m(K) \geq 3/4$, we have 
\[\Pr[ Y \in c\sigma K] \geq 1/2.\]
\end{theorem}
\begin{remark}
    Notice that if we could sample efficiently from the distribution $D$, then this also give an efficient algorithm to find a coloring $x$ with $\sum_j x_j v_j \in 2cc'K$, by sampling a couple of times from $D$.
\end{remark}
\begin{proof}
Let $B_{\ell_2^m}$ denote the unit $\ell_2^m$-ball. We claim that $\frac12
B_{\ell_2^m}\subseteq K$. Otherwise, there is a point $u$, $\|u\|_2 = \frac12$,
that is not in $K$, and, since $K$ is closed and convex, by the
hyperplane separator theorem there is a vector $y \in \R^m$ such that
$\ip{x,y} \le 1$ but $\ip{u,y} > 1$. By the Cauchy-Schwarz inequality,
$\ip{u,y} > 1$ implies $\|y\|_2 > 2$. Then, for the halfspace $H
\eqdef \{x \in \R^m: \ip{x,y} \le 1\}$ we have
\[
  \gamma^m(K) \le \gamma^m(H) = \Pr[\ip{\grv,y} \le 1] <
  \Phi(1/\|y\|_2) \leq \Phi(1/2) < 3/4.
\]
Above, $\grv$ is a standard Gaussian vector in $\R^m$, $\Phi$ is the
CDF of the one-dimensional standard Gaussian,  and we used the
fact that $\ip{\grv,y}$ is a one-dimensional Gaussian with mean $0$
and standard deviation $\|y\|_2 > 2$. This inequality contradicts
$\gamma^m(K) \ge 3/4$.

Using polarity, we have that $tK = \{x \in \R^m: \ip{x,y} \le t \
\forall y \in K^\circ\}$, for any $t \ge
0$. Theorem~\ref{thm:ch6a-talagrand} gives us 
\begin{equation}\label{eq:ch6a-comp}
  \E \sup_{z \in K^\circ} \ip{Y,z} \le c_1\sigma \E\sup_{z \in
    K^\circ} \ip{\grv,z}
\end{equation}
for a standard Gaussian random vector $\grv \in \R^m$ and some
absolute constant $c_1 >
0$. 

We want to show that $\E\sup_{z \in  K^\circ} \ip{\grv,z} \lesssim
1$. Since $\frac12 B_{\ell_2^m} \subseteq K$, the Cauchy-Schwarz theorem
implies that $K^\circ \subseteq 2B_{\ell_2^m}$, and the function
$f_K(x) \eqdef \sup_{z \in K^\circ} \ip{x,z}$
is $2$-Lipschitz with respect to the $\ell_2^m$ norm. (Note that $f_K$
is just the norm with unit ball $K$ when $K = -K$ and $K$ is bounded
and has nonempty interior.) Suppose, for contradiction, that
$\E[f_K(\grv)] > C$ for some sufficiently large constant $C > 1$. Then, the
Gaussian concentration inequality~\cite[Corollary~8.6.2]{vershynin} gives us that, for
some absolute constant $c_2 > 0$,
\[
  \gamma^m(K) = \Pr[f_K(\grv) \le 1]
  \leq \Pr[f_K(\grv) \le \E[f_K(\grv)] - (C-1)] \le e^{-c_2 (C-1)^2}.
\]
Choosing $C$ sufficiently large contradicts $\gamma^m(K) \ge
\frac34$, so $\E\sup_{z \in  K^\circ} \ip{\grv,z} = \E[f_K(\grv)]
\le C$. Together with \eqref{eq:ch6a-comp}, this gives us
\(
\E \sup_{z \in K^\circ} \ip{Y,z}\le c_1C\sigma ,
\)
as well. 
Now, by Markov's inequality, $\sup_{z \in K^\circ} \ip{Y,z} \le 2c_1C\sigma$ with probability at
least $\frac12$,  and
therefore, $Y \in (2c_1 C)K$ with the same probability.
\end{proof}

It is not hard to extend Theorem~\ref{thm:bana-subg-suff} to \emph{symmetric} closed
convex sets $K$ with $\gamma^m(K) \ge \frac12$. The main observation
is that such a set $K$ still must contain a Euclidean ball of some
fixed constant radius $r$, since, if it did not, it would be
sandwiched between two hyperplanes of distance less than $2r$
apart. Extending the theorem to non-symmetric $K$ is, however,
non-trivial. For example, if $K$ is the halfspace $\{x \in \R^m: x(1)
\le 0\}$, then $\gamma^m(K) = \frac12$, but $K$ does not contain a
Euclidean ball of any positive radius around the origin, and its polar
set $K^\circ$ is unbounded, so $\E \sup_{z \in K^\circ} \ip{Y,z} =
\infty$. Nevertheless, it is possible to algorithmically reduce the
case of non-symmetric $K$ with $\gamma^m(K) \ge \frac12$ to the
symmetric case: we refer to~\cite{DGLN-journal} for the details.

\section{Prefix discrepancy}
In many applications of discrepancy, one might actually want to bound the discrepancy of every prefix sum, instead of just the overall sum $\sum_i x(i) v_i$. One example is the application to the Steinitz Problem in Section~\ref{sec:ch1-steinitz}, which we revisit later in this section.
The following result, also due to Banaszczyk, controls the discrepancy of every prefix sum, at the expense of a slight increase in the Gaussian measure of $K$. 

\begin{theorem}\label{thm:bana2}
 Given arbitrary vectors $v_1,\dots,v_n\in\mathbb{R}^m$ with $\|v_i\|_2 \leq 1/5$ and any
convex body $K\subseteq \mathbb{R}^m$ with $\gamma^m(K)\ge 1-1/(2n)$,
there exists a $\pm 1$ coloring $x$ such that $\sum_{j=1}^k x(j) v_j \in K$ for each $k\in [n]$.
\end{theorem}

Similarly to Theorem \ref{thm:bana}, this result also follows from Theorem \ref{thm:bana-star} based on a clever induction argument.  
\begin{proof}
Consider the convex bodies $K_n,K_{n-1},\ldots,K_1$ defined
iteratively as follows:  $K_n\eqdef K$ and $K_{j}\eqdef(K_{j+1}*v_{j+1}) \cap K$, for $j=n-1,\ldots,1$.  

We first show by backwards induction that $\gamma^m(K_j) \geq 1/2 + (j-1)/2n$. This implies that $\gamma^m(K_j) \geq 1/2$ for all $j\geq 1$.

The base case of $j=n$ holds as $\gamma^m(K_n) = \gamma^m(K) \ge 1-1/2n$.
Suppose the claim holds for some $j \leq n$. Then 
\begin{align*}\gamma^{m}(K_{j-1}) &  = \gamma^m((K_{j}*v_{j}) \cap
                                      K) \geq  \gamma^m(K_{j}*v_{j}) - \gamma^m(\R^m\setminus {K}) \\
& \geq    \gamma^m(K_{j})  - \gamma^m(\R^m\setminus{K}) \geq  \frac{1}{2} +\frac{j-1}{2n} -\frac{1}{2n} = \frac{1}{2} + \frac{ j-2}{2n},
\end{align*}
where we use that $\gamma^m(K_{j}*v_{j}) \geq \gamma^m(K_{j})$ as
$\gamma^m(K_{j}) \geq 1/2$ by Theorem \ref{thm:bana-star}, and that
$\gamma^m(\R^m \setminus {K})\leq 1/(2n)$.

We now show that any sequence $x(1), \ldots, x(k) \in \{-1,+1\}$ satisfying
$\sum_{j=1}^\ell x(j) v_j \in K_\ell$ for all $\ell \in [k]$, can be extended to
$x(1), \ldots, x(k+1) \in \{-1,+1\}$ such that $\sum_{j=1}^{k+1} x(j)
v_j \in K_{k+1}$. This will imply the result as the final sequence
$x(1),\ldots,x(n)$ will satisfy $\sum_{j=1}^k x(j) v_j \in K_k
\subseteq K$ for all $k\in [n]$.

To initialize the sequence, we pick some $x(1)$ such that $x(1) v_1 \in K_1$. Such an $x(1)$ exists by Theorem \ref{thm:bana} as $\gamma^m(K_1) \geq 1/2$. 
Then, for $k\ge 1$, let $x(1),\ldots,x(k)$ be the sequence chosen so
far, and define $y_k\eqdef\sum_{j=1}^k x(j) v_j$.  Then
\[y_k\in K_k \subseteq K_{k+1}*v_{k+1} \subseteq (K_{k+1}+v_{k+1}) \cup
  (K_{k+1}- v_{k+1}),\] and at least one of $y_k+v_{k+1}$ or $y_k-v_{k+1}$ must lie in $K_{k+1}$. We set $x(k+1)$ accordingly and extend the sequence.
\end{proof}

It is possible to use the proof of Theorem~\ref{thm:bana2} and the
cloning trick in order to formulate a statement similar to
Theorem~\ref{thm:bana-subgaussian} for prefix sums. We state this
theorem next, and refer to the work of Kulkarni, Reis, and Rothvoss~\cite{rothvoss-online} for the
proof.
\begin{theorem}\label{thm:bana-prefix-subgaussian}
  Given arbitrary vectors $v_1,\ldots,v_n \in \R^m$ with $\|v_i\|_2
  \leq 1$, there is a distribution $D$ over colorings $x \in
  \{-1,1\}^n$ such that, for each $k \in [n]$, the prefix discrepancy
  vector $Y_k \eqdef \sum_{j=1}^k x(j) v_j$ is $O(1)$-sub-Gaussian when
  $x$ is sampled from $D$.
\end{theorem}

Unlike Theorem~\ref{thm:bana-subgaussian}, there is no known efficient
algorithm to sample from the distribution $D$. More generally, it is
an intriguing open problem how to make Theorem \ref{thm:bana2}
algorithmic. 

\subsection{Prefix Koml\'os Problem and the Steinitz Lemma}
Theorem \ref{thm:bana2} directly implies the following bound for the prefix version of the Koml\'os problem.

\begin{theorem}[Prefix Koml\'os]\label{thm:prefix-komlos}
  Given vectors $v_1,\dots,v_n\in\mathbb{R}^m$ with $\|v_j\|_2 \leq 1$, 
  there exists a coloring $x:[n]\to\{-1,+1\}$ such that  $\|\sum_{j=1}^k x(j) v_j\|_2 \lesssim \sqrt{\log(2n)}$ for all $k\in [n]$. I.e., the signed series constant satisfies 
  \[
  \sser(B_{\ell_2^m},\ell_\infty^m,n) \lesssim \sqrt{\log(2n)}.
  \]
\end{theorem}
\begin{proof}
We apply Theorem~\ref{thm:bana2} with $K \eqdef [-t,t]^m$ and $t>0$ such
that  \[\gamma^m(K) \geq 1-1/2n.\]  As in the proof of
Theorem~\ref{thm:komlos-bana}, we see that setting $t \lesssim \sqrt{\log(2mn)}$ suffices.
By Remark \ref{rem:bana-komlos-trick}, we can also assume that $m\leq n^2$.
\end{proof}

A similar proof shows the following bound, which gets tantalizingly
close to resolving Problem~\ref{prob:steinitz} for the case of the
Euclidean norm.

\begin{theorem}\label{thm:sser-ellp}
  Let $p \ge 2$ and suppose that the vectors $v_1, \ldots, v_n \in
  \R^m$ each satisfy $\|v_j\|_2 \le 1$. There exists a coloring $x:[n]
  \to \{-1,+1\}$ such that $\|\sum_{j=1}^kx(j)v_k\|_p \lesssim
  \sqrt{p}m^{1/p} + \sqrt{\log n}$. I.e., the signed series constant satisfies 
  \[
  \sser(B_{\ell_2^m},\ell_p^m,n) \lesssim \sqrt{p}m^{1/p} + \sqrt{\log n}.
  \]
\end{theorem}
\begin{proof}
  We apply Theorem~\ref{thm:bana2} with $K = tB_{\ell_p^m}$, where
  $B_{\ell_p^m}$ is the unit ball of the $\ell_p^m$ norm, and $t$ is
  chosen so that $\gamma^m(K) \ge 1 - 1/2n$. 
  By Jensen's inequality
  and standard bounds on the $p$-th moment of a
  Gaussian~\cite[Proposition 2.5.2]{vershynin}, we have
  \[
    \E[\|\grv\|_p] \le \E[\|\grv\|_p^p]^{1/p}
    = \E\left[\sum_{i=1}^m |\grv(i)|^p\right]^{1/p} \lesssim \sqrt{p}m^{1/p}.
  \]
  By the Gaussian concentration inequality, since $\|\cdot\|_p$ is
  1-Lipschitz with respect to the Euclidean norm when $p \ge 2$,
  \[
    \Pr[\|\grv\|_p > \E[\|\grv\|_p] + s] \le e^{-c s^2}
  \]
  for an absolute constant $c > 0$. Choosing $s \eqdef c'
  \sqrt{\log(2n)}$ for a large enough constant $c' > 0$ makes  the
  right hand side at most $\frac{1}{2n}$. Since the left hand side
  equals $\gamma^m(\R^m\setminus K)$ for $t = \E[\|\grv\|_p] + s
  \lesssim \sqrt{p}m^{1/p} + \sqrt{\log(2n)}$, this proves the lemma.
\end{proof}

If Theorem~\ref{thm:sser-ellp} could be improved by removing the
additive $\sqrt{\log 2n}$ term, then it would give tight bounds for
Problem~\ref{prob:steinitz} for any constant $p \ge 2$. In particular, via
Chobanyan's transference theorem (Theorem~\ref{thm:ch1-chobanyan}), it would imply
\[
  \st(B_{\ell_p^m}, \ell_p^m) \le m^{\frac12 - \frac1p}\st(B_{\ell_2^m}, \ell_p^m) 
  \le m^{\frac12 - \frac1p}\sser(B_{\ell_2^m}, \ell_p^m) 
  \lesssim \sqrt{pm},
\]
which is tight for constant $p \ge 2$. 

Another open
problem is to make Theorem~\ref{thm:sser-ellp} or Theorem~\ref{thm:prefix-komlos},
algorithmic. For both theorems, there is no known efficient algorithm to find
the coloring $x$ given the vectors $v_1, \ldots, v_n$.



\subsection{Prefix Discrepancy and the $\gamma_2$ Norm}

Theorem~\ref{thm:bana2} also immediately implies that
Lemma~\ref{lm:fact-ub} can be extended to bound prefix discrepancy, at
the cost of a small increase in the logarithmic factor. Here we recall
that the prefix discrepancy of an $m\times n$ real matrix $\mat = (w_i)_{i=1}^n$ is
the quantity
\[
  \pdisc((w_i)_{i = 1}^n) \eqdef
  \min_{\col:[n]\to\{-1,+1\}}\max_{j=1}^n \left\|\sum_{i = 1}^j\col(i)w_i \right\|_\infty.
\]
This definition coincides with our definition of
$\pdisc_{\ell_\infty^m}$ from Chapter~\ref{ch1-intro}.

\begin{lemma}\label{lm:fact-ub-prefix}
  For any $m\times n$ matrix $\mat$ 
  we have
  \[
  \pdisc(\mat) \lesssim \gamma_2(\mat) \sqrt{\ln(2mn)}.
  \]   
\end{lemma}
\begin{proof}
  By the definition of $\gamma_2(\mat)$, we need to show  that if $\mat$ can be written as the matrix
  product $\mat = UV$ of two matrices $U \in \R^{m\times k}$ and $V\in
  \R^{k\times n}$, where each row
  $u_i$ of $U$ satisfies $\|u_i\|_2 \le r$, and each column $v_j$ of
  $V$ satisfies $\|v_j\|_2 \le c$, then $\pdisc(\mat) \lesssim rc\sqrt{\ln(2mn)}$.
  Similarly to the proof of Lemma~\ref{lm:fact-ub}, we define the convex body
  \[
    K = \{ y \in \R^k: \|Uy \|_{\infty} \leq r\sqrt{2\ln{(2mn)}}  \}.
  \]
  Arguing as in the proof of Lemma~\ref{lm:fact-ub},
  we have $\gamma^m(K) \ge 1 - \frac{1}{2n}$.
  Now Theorem~\ref{thm:bana2} applied to the columns $\frac{1}{5c}v_1,
  \ldots, \frac{1}{5c}v_n$ of $\frac{1}{5c}V$ implies the lemma.
\end{proof}

In Chapter~\ref{ch:gamma2} we will use Lemma~\ref{lm:fact-ub-prefix}
to prove the best known upper bound for Tusn\'ady's problem (Problem~\ref{prob:tusnady}).

\section{Bibliographic Notes}

Theorems~\ref{thm:bana}~and~\ref{thm:komlos-bana} are due to Banaszczyk \cite{Bana98}. The
original proof was non-algorithmic and based on deep results from convex geometry. We will describe Banaszczyk's proof in detail in
Chapter \ref{ch:bana-proof}. 

Lemma~\ref{lm:fact-ub} is implicit in Larsen's work on data structure
lower bounds~\cite{disc-larsen}, and was made explicit by
Matou\v{s}ek, Nikolov, and Talwar~\cite{MNT}. 

The cloning trick in Section \ref{sec:komlos-sdp-bound} is due to
unpublished work by Shachar Lovett, Raghu Meka, and Oded
Regev. Nikolov~\cite{nikolov2013komlos} gave a different and direct proof of
Theorem~\ref{thm:komlos-vec-disc} where the implied constant on the
right hand side of \eqref{eq:komlos-vec-disc} is equal to $1$. Yet another proof of Theorem~\ref{thm:komlos-vec-disc}, also with the constant $1$ on the right hand side, was given by Dadush, Garg, Lovett, and Nikolov~\cite{DGLN16}.
Interestingly, it is unclear how to show this tight
bound using the cloning trick, as the body $K$ in
Theorem~\ref{thm:bana} needs to be scaled by $5$ if the vectors $v_j$
have unit length. We give another (previously unpublished) proof of the tight upper bound in Chapter~\ref{ch:online}.

The equivalence between
Theorems~\ref{thm:bana}~and~\ref{thm:bana-subgaussian} was shown by
Dadush, Garg, Lovett and Nikolov \cite{DGLN16}, motivated by the
question of finding an efficient algorithm for Theorem
\ref{thm:bana}. A direct proof of Theorem \ref{thm:bana-subgaussian},
and an efficient sampling algorithm for the distribution $D$, were given
later by Bansal, Dadush, Garg and Lovett~\cite{BansalDGL19} using an algorithm called the Gram-Schmidt Walk. We will
describe this algorithm and its analysis in detail in Section~\ref{ch:gram-schmidt}.

The argument we present here that Theorem \ref{thm:bana} implies
Theorem \ref{thm:bana-subgaussian} via the cloning trick is
different from the Dadush, Garg, Lovett, and Nikolov argument, which
was based on the minimax theorem.  The cloning trick based argument
for
Theorems~\ref{thm:bana-subgaussian}~and~\ref{thm:bana-prefix-subgaussian}
appears in unpublished work of Nikolov. Kulkarni, Rothvoss and Reis
apply a similar argument to prove that the sequence of signs can even
be computed online, albeit via an inefficient
algorithm~\cite{rothvoss-online}. We consider online algorithms for
discrepancy minimization in more detail in Chapter~\ref{ch:online}.

Theorem \ref{thm:bana2} for prefix discrepancy, and
Theorem~\ref{thm:prefix-komlos} for the prefix Koml\'os problem are
also due to Banaszczyk, in a more recent work~\cite{B12}. An intriguing open problem is to
find an efficient algorithmic version of either theorem. The
best known algorithmic bound in the prefix setting is $O(\log mn)$, which
remarkably also works in the online
setting~\cite{alweiss2020discrepancy}, and is based on a beautiful algorithm called the Self-Balancing Walk, that we will describe in detail
in Chapter \ref{ch:online}.  See also \cite{rothvoss-online,
  ChewiGRT22, LiuSS22} for related progress. 
  Finally, we remark that an algorithmic
$O((\log mn)^{1/2})$ discrepancy bound for the prefix Koml\'{o}s
problem would already be interesting to obtain in the special case
where each vector $v_i$ has only two non-zero entries, as it would
have implications for approximation algorithms for scheduling
problems~\cite{BansalRS22}.







\chapter{Proof of Banaszczyk's Theorem}
\label{ch:bana-proof}
We now describe the original non-constructive proof of  Theorem \ref{thm:bana} due to Banaszczyk \cite{Bana98}. 
The proof uses several non-trivial and interesting properties of convex bodies and Gaussian measures, that are not widely used in theoretical computer science. So we give a detailed exposition of these ideas.


\paragraph{Proof Overview.}

Recall from Section \ref{sec:bana-core} that the key technical core underlying Banaszczyk's theorem was the following geometric result.

\begin{theorem}
\label{thm:bana-star-2}
Given any convex body $K \subset \R^m$ with Gaussian measure $\gamma^m(K) \geq 1/2$ and vector $u \in \R^m$ with $\|u\|_2 \leq 1/5$,\footnote{   Let us denote $r = \|u\|_2$. The only property of $r$ we need is that $\gamma^1([-r,r]) \leq \gamma^1((-\infty,-1])$.
%
 In particular, $\gamma^1([-1/5,1/5])\approx0.1585$ and $\gamma^1((-\infty,-1]) \approx 0.1586$.} there is a convex body 
$K*u$ contained in $(K-u) \cup (K+u)$ such that $\gamma^m(K*u) \geq \gamma^m(K)$. 
\end{theorem}

The body $K*u$ in Theorem \ref{thm:bana-star-2} is obtained in a natural way from $K$ and $u$, see Figure \ref{fig:bana-proof-1} and Section \ref{sec:bana-proof-body} for a formal description.
The transformation of $K$ to $K*u$ loses some points from $K$ and gains some new points, and the proof is based on showing that the Gaussian measure gained is at least the measure lost.

A key difficulty however is that unlike the standard Lebesgue measure, the Gaussian measure changes in complicated ways upon moving points. E.g.,~while shifting $K$ to $K+u$ does not change the Lebesgue measure, it is much harder to relate $\gamma^m(K+u)$ to $\gamma^m(K)$, 
as the Gaussian density of each point $x \in K$ changes from $\exp(-\|x\|^2/2)$ to $\exp(-\|x+u\|^2/2)$  depending on $x$. 

The key idea is to apply several interesting transformations to $K$ and $K*u$, while preserving various properties, and eventually reduce them to very simple two dimensional bodies. In fact, our final task reduces to comparing the Gaussian measures of  $1$-dimensional intervals! 
These final computations however are still quite delicate and use the log-concavity of Gaussian measure in non-trivial ways.

\medskip
\noindent{\bf Organization.}
The rest of the chapter is organized as follows. We describe the body $K*u$ in Section \ref{sec:bana-proof-body} and describe some facts about Gaussians including the notion of Ehrhard symmetrization. In Section \ref{sec:bana-proof-symmetrize-u}, we apply the symmetrization along $u$ to transform $K$ and $K*u$ into simpler bodies, but still lying in $\R^m$. In Section \ref{sec:bana-proof-reducing-two-dimensions} we transform them further to $2$-dimensional bodies and finally in Section \ref{sec:bana-proof-comparing-1-dim-intervals} we compare their volumes and discuss some facts about log-concave measures.

In Section \ref{sec:bana-proof-ehrhard} we discuss Ehrhard's inequality, which is used in the transformations in Sections \ref{sec:bana-proof-symmetrize-u} and \ref{sec:bana-proof-reducing-two-dimensions}.

\section{The body $K*u$} 
\label{sec:bana-proof-body}
Given $K$ and $u$, we describe how $K*u$ is obtained.

First, by rotational symmetry of the Gaussian, we can assume that $u=r e_1$, where $r\leq 1/5$. 
As $u=re_1$, let us view a point $x =(x_1,\ldots,x_m) \in \R^m$ as $(x_1,y) \in \R \times \R^{m-1}$ where $y=(x_2,\ldots,x_{m}) \in \R^{m-1}$. 

For each $y \in \R^{m-1}$, let   $K_y := \{t: (t,y)\in K\}$ denote the segment obtained by intersecting $K$ with the horizontal line $\{(t,y): t \in \R\}$.
Consider the body 
\[K_1 \;\eqdef\; \{(t,y)\in\R\times\R^{m-1} : (t,y)\in K,\ |K_y|\ge 2r\}.\]
 consisting of segments of $K$ along $u$ that are at least $2r$ wide.

The body $K*u$ (see Figure \ref{fig:bana-proof-1}) is defined as \begin{equation}
\label{eq:bana-proof-kstaru}
    K*u = (K_1-u) \cup (K_1+u).
\end{equation}
\begin{figure}[hbtp!]
    \centering
\includegraphics[scale=0.55, trim={10mm 2mm 10mm 3mm}]{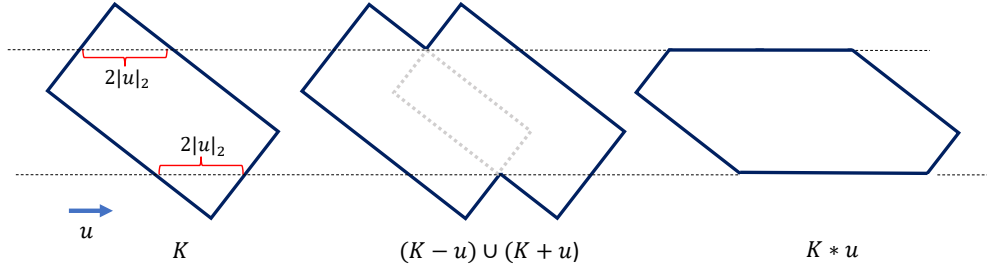}
    \caption{The convex body $K*u$ obtained from $K$ and $u$.}
    \label{fig:bana-proof-1}
\end{figure}
In other words, we remove points where $K$ is thinner than $2r$ along $u$, and extend $K$ by $u$ and $-u$ where it is wider than $2r$ along $u$.
Convexity of $K*u$ follows since, for each fixed $y$, $(K*u)_y$ is an interval
(in fact $K_y+[-r,r]$ when non-empty), and the set of $y$ with non-empty fibers,
together with the endpoints of these fibers, vary affinely in $y$ because $K$ is convex.





To prove Theorem \ref{thm:bana-star-2}, we need to show that the Gaussian measure we lose consisting of the segments $K_y$ with $|K_y|<2r$, is no more that the amount we gain by extending the intervals $K_y$ to $K_y + [-r,r]$ whenever $|K_y| \geq 2r$.
Figure \ref{fig:bana-proof-2} shows the region lost by red and the region gained by blue.
\begin{figure}[hbtp!]
    \centering
    \includegraphics[scale=0.55, trim={15mm 2mm 15mm 3mm}]{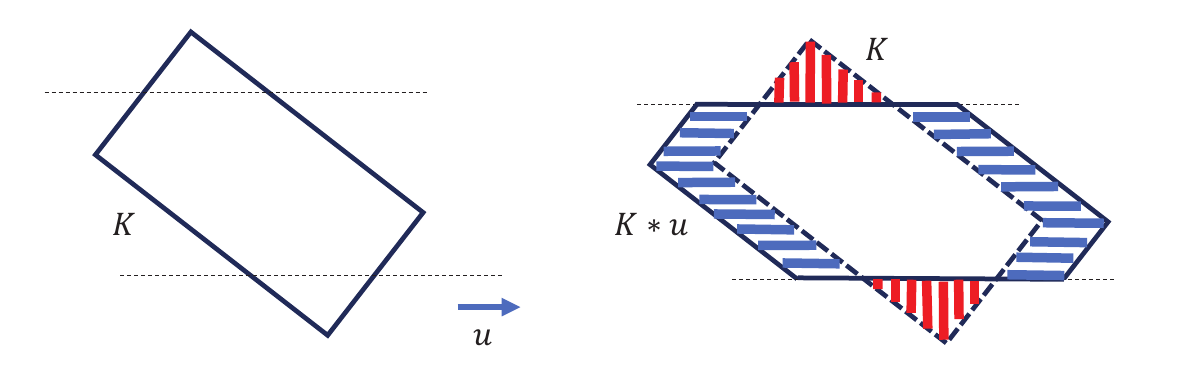}
    \caption{The region gained (blue) and lost (red) in $K*u$ compared to $K$.}
    \label{fig:bana-proof-2}
\end{figure}
\paragraph{Some Gaussian Facts.}
Let $\phi(x) = (2 \pi)^{-1/2} \exp(-x^2/2)$ denote the standard Gaussian density.
In $d$-dimensions the Gaussian density is a product distribution with \[\phi(x) = \prod_{i=1}^d \phi(x_i)  = (2 \pi)^{-d/2} \exp(-\|x\|^2/2) \text{ for $x = (x_1,\ldots,x_d)$}.\]

For a set $A \subset \R$, let $\gamma^1(A) = \int_{A} \phi(x) dx$ denote its measure, and let 
\[  \Phi(x) = \gamma^1(-\infty,x) = (2\pi)^{-1/2} \int_{-\infty}^x \exp(-t^2/2) dt,\]
denote the cumulative distribution function of $\phi$.

The function $\Phi^{-1}:[0,1] \rightarrow  \R \cup \{-\infty,\infty\}$ will play an important role.
For a set $A\subset \R$, we define the left-infinite interval  \[I(A):=(-\infty,\Phi^{-1}(\gamma^1(A))].\]   Note that $I(A)$ same Gaussian measure as $A$, as $\gamma^1((-\infty,x]) = \Phi(x)$ for all $x$ and thus $\gamma^1(I(A))= \gamma^1(A)$.

When working with Gaussian measures, it is often convenient to replace a set $A$ by $I(A)$, as it makes it more extremal while  preserving the Gaussian measure. 
In particular for $A\subset \R$, let  \[A+[-r,r] := \{x+y:x\in A, y\in [-r,r]\},\]
denote the expansion of $A$ by $r$.
Then, by the Gaussian isoperimetric inequality (in $\R$) we have that 
 \begin{equation}
     \label{eq:bana-proof-gsi}
      \gamma^1(I(A)+[-r,r]) \leq \gamma^1(A+[-r,r])
 \end{equation}
  for all $r\geq 0$.
Also notice that $\Phi^{-1}(\gamma^1(I(A)+[-r,r]))$ is simply $ \Phi^{-1}(\gamma^1(A)) + r$. So we can write \eqref{eq:bana-proof-gsi} as
\begin{equation}
    \label{eq:bana-proof-gsi-cor} \Phi^{-1}(\gamma^1(A+[-r,r]))\geq  \Phi^{-1}(  \gamma^1(I(A))) +r.
\end{equation}

\paragraph{ Gaussian symmetrization along a direction.}
For a set $A \subset \R^d$, and say $u=e_1$, the Gaussian symmetrization $I_u(A)$ of $A$ along $u$ is obtained by replacing each horizontal segment $A_y = \{t:(t,y)\in A\}$ by the segment $I(A_y)$.
Formally,
\[ I_u(A) := \{ (t,y): y\in\R^{m-1}, t \leq \Phi^{-1}(\gamma^1(A_y))\}.\]
For example, Figure \ref{fig:bana-proof-3} shows the body $I_u(K)$ obtained from $K$.

Notice that $\gamma^d(I_u(A)) =\gamma^d(A)$ as $\gamma^1(I(A_y)) = \gamma^1(A_y)$ for each segment $A_y$ (and the coordinates other than $x_1$ do not change).

A very interesting fact proved by Ehrhard \cite{Ehrhard1983} is the following.
\begin{theorem}
\label{thm:baby-ehrhard}
    If $A \subset \R^d$ is convex, then its Gaussian symmetrization $I_u(A)$ along any direction $u$ is also convex.
\end{theorem}
We will discuss this in more detail in Section \ref{sec:bana-proof-ehrhard}. 
In Section \ref{sec:bana-proof-reducing-two-dimensions}, we will also consider a more general notion called Ehrhard symmetrization, where symmetrizing $A \in \R^d$ along a subspace $U$ gives a $d-\text{dim}(U)+1$ dimensional body.
This will be used to reduce our problem to two dimensions. 

\section{Symmetrizations of $K$  and $K*u$ along $u$}
\label{sec:bana-proof-symmetrize-u}
Let us consider the symmetrizations of $K$ and $K*u$ along $u$. So,
\[I_u(K) := \{(t,y): y\in \R^{m-1}, t \leq f_K(y)\},\]
where we denote $f_K(y) = \Phi^{-1}(\gamma^1(K_y))$ for ease of notation. 

\begin{figure}[hbtp!]
    \centering
\includegraphics[scale=0.58, trim={15mm 4mm 0 3mm }]{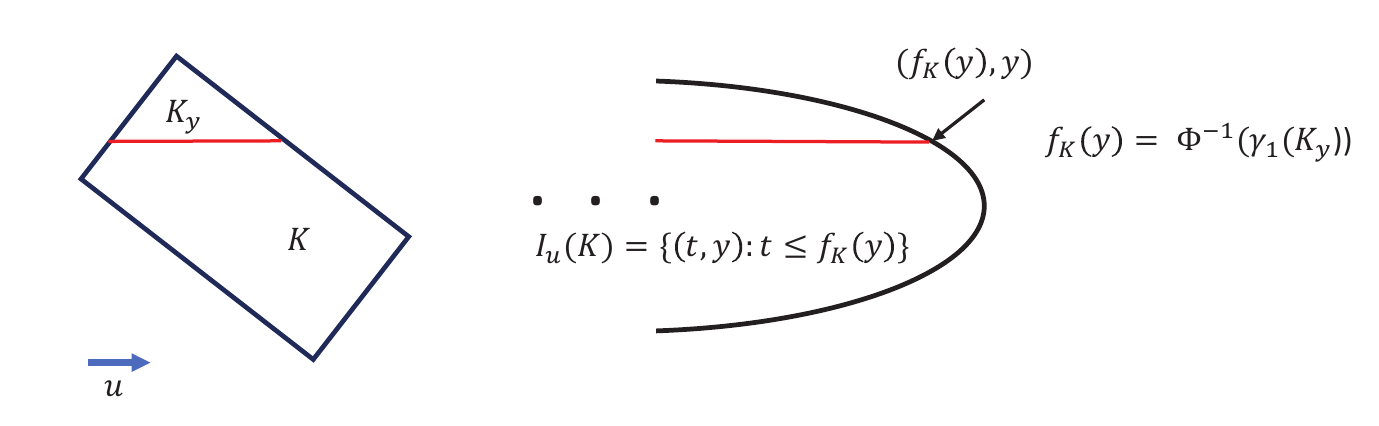}
    \caption{
    The symmetrization of $K$ along $u$. The interval $K_y$ of $K$ is transformed to the interval $(-\infty,f_K(y)$. The body $I_u(K)$ extends infinitely to the left.}
    \label{fig:bana-proof-3}
\end{figure}

Similarly, 
\[I_u(K*u) =\{(t,y): y \in \R^{m-1},t \leq f_{K*u}(y)\},\]
where $f_{K*u}(y) =\Phi^{-1}(\gamma^1((K*u)_y))$.  See Figure \ref{fig:bana-proof-4}.
Notice that by the way $K*u$ is obtained from $K$, we have
\begin{align} f_{K*u}(y) = \begin{cases}  -\infty   \qquad & |K_y|< 2r\\ 
\Phi^{-1} ( \gamma^1(K_y+[-r,r]))  \qquad & |K_y| \geq 2r. 
\end{cases} \label{eq:fkstartu} 
\end{align}
By Theorem \ref{thm:baby-ehrhard}, both $I_u(K)$ and $I_u(K*u)$ are convex. Hence the functions of $f_K(y)$ and $f_{K*u}(y)$ are concave.
Also as symmetrization does not change the Gaussian measure, our goal is the same as showing that $\gamma^m(I_u(K*u))\geq \gamma^m(I_u(K))$.

\begin{figure}[hbtp!]
    \centering
\includegraphics[scale=0.6, trim={10mm 5mm 0 5mm }]{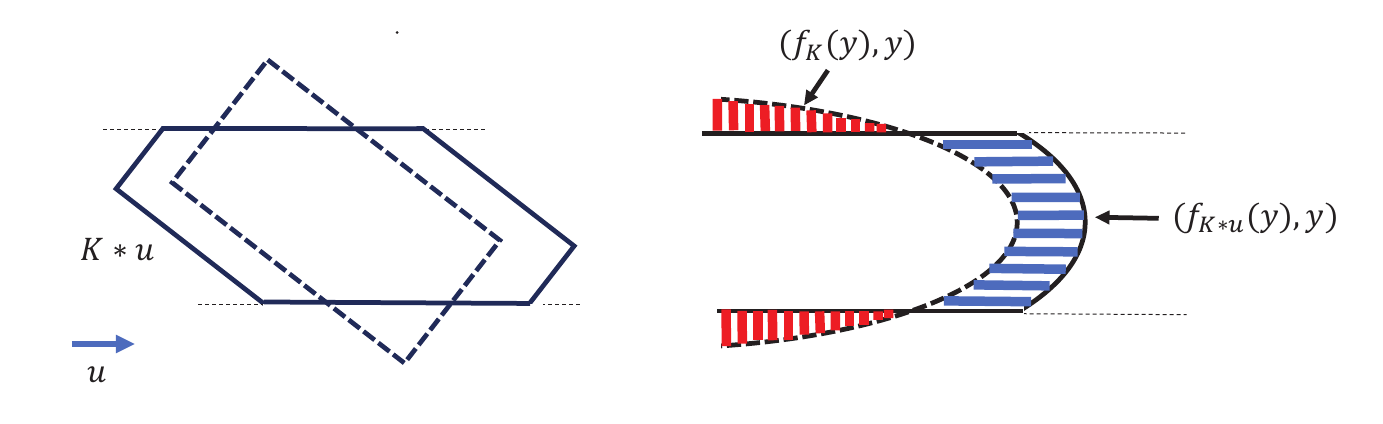}
    \caption{The body $I_u(K*u)$ obtained by symmetrizating $K^*$ along $u$, and described by the function $f_{K*u}(y)$. The blue and red regions depict the volume gained and lost by $I_u(K*u)$ over $I_u(K)$.}
    \label{fig:bana-proof-4}
\end{figure}

\paragraph{The body $K'$.}
We now shrink the body $I_u(K*u)$ and replace it a simpler body $K'$ that
 is more closely related to $I_u(K)$. 
This is depicted in Figure \ref{fig:bana-proof-5}. 
Define the threshold 
\begin{equation}
\label{eq:bana-pf-defn-p}
    -p:=\Phi^{-1}(\gamma^1([-r,r])).
\end{equation}
Recall that $r$ was chosen in Theorem~\ref{thm:bana-star-2} to satisfy $\gamma^1([-r,r]) \leq \gamma^1(-\infty,-1)$. This implies that $-p \leq -1$ (equivalently $p\geq 1)$.

Consider the body 
\begin{equation}
    \label{eq:kprime}
    K'=\{(t,y): y \in \R^{m-1}, t \leq g_K(y)\},
\end{equation}
where $g_K:\R^{m-1}\rightarrow \R$ is defined as
 \begin{align} g_K(y) = \begin{cases}   -\infty  \qquad & \text{ if } f_K(y) < -p \\
 f_K(y)+r   \qquad & \text{ otherwise},
 \end{cases}
 \label{eq:gky}
\end{align}

\begin{figure}[hbtp!]
    \centering
\includegraphics[scale=0.6, trim={10mm 0 10mm 0 }]{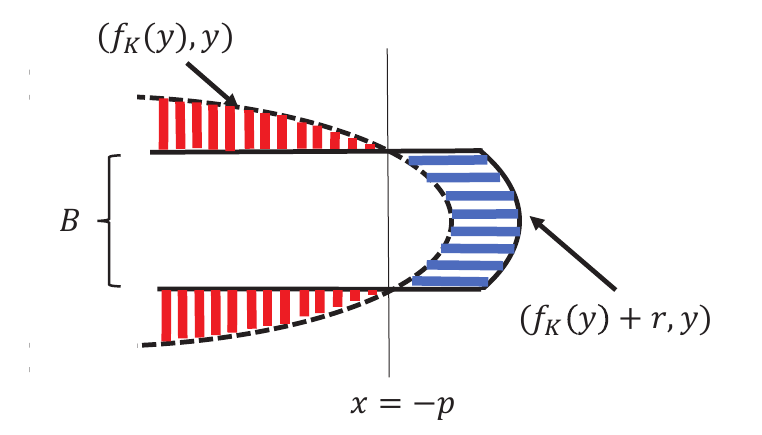}
    \caption{The body $K'$ obtained by shrinking $I_u(K*u)$. The blue region only decreases and the red region only increases.}
    \label{fig:bana-proof-5}
\end{figure}

Notice that $K'$ only depends on $f_K(y)$ and is more closely related to $I_u(K)$. $K'$ is  also convex, as $K' + (r,0)$ is the intersection of the translated convex set $I_u(K)+(r,0)$
with the convex superlevel set $\{(t,y): f_K(y)\ge -p\}$.

We now show that $K' \subset I_u(K*u)$.
\begin{lemma} For all $y$,  $g_K(y)  \leq f_{K*u}(y)$  and thus $K'\subset I_u(K*u)$.
\end{lemma}
\begin{proof}
Fix some $y$. We first show that if $f_{K*u}(y)=-\infty$, then $g_K(y)$ is also $-\infty$. Indeed, if $f_{K*u}(y)=-\infty$ this means that $|K_y|<2r$ by \eqref{eq:fkstartu}. But then $\gamma^1(K_y) < \gamma^1([-r,r])$  
 as $[-r,r]$ has the largest Gaussian measure among all $2r$ length intervals. So
\[f_K(y) = \Phi^{-1}(\gamma^1(K_y)) < \Phi^{-1}(\gamma^1([-r,r]))= -p,\]
and  thus $g_K(y)= -\infty$ by definition \eqref{eq:gky}.

Now, suppose $|K_y|\geq 2r$. Then $f_{K*u}(y) = \Phi^{-1}(\gamma^1(K_y+[-r,r]))$ by \eqref{eq:fkstartu}, and 
by the  Gaussian isoperimetric inequality \eqref{eq:bana-proof-gsi-cor},
\[
    f_{K*u}(y) \geq  \Phi^{-1}(\gamma^1(K_y)) +r = f_K(y) +r \geq  g_K(y).  \qedhere  \]
\end{proof} 

As $K' \subset I_u(K*u)$, to prove Theorem \ref{thm:bana-star-2} it suffices to show that $\gamma^m(K') \geq \gamma^m(I_u(K))$.

\section{Reducing to two dimensions}
\label{sec:bana-proof-reducing-two-dimensions}
Abusing notation a bit, let us use $K$ to denote $I_u(K)$ (as  we will not need the original body $K$ anymore).

Our goal is to compare the Gaussian measures of  $K$ and $K'$ and show that the lost measure $\gamma^m(K\setminus K')$ (depicted by red in Figure \ref{fig:bana-proof-5}) is at most the gained measure $\gamma^m(K'\setminus K)$ (depicted by blue).

\paragraph{Slices.} We can compare $K$
 and $K'$ nicely by looking at their $(m-1)$ dimensional slices orthogonal to $x_1$.
 For $x \in \R$, let 
 \[K_x:= \{y \in \R^{m-1}: (x,y) \in K\}\] 
denote the slice of $K$ passing through $x$. 
Similarly, let $K'_x:= \{y \in \R^{m-1}: (x,y) \in K'\}$. 
Clearly,
the measure of $K,K'$ can be expressed in terms of their slices as,
\[\gamma^m(K) = \int_x\gamma^{m-1}(K_x) \phi_1(x) dx \, \text{ and } \, \gamma^m(K') = \int_x \gamma^{m-1}(K'_x) \phi_1(x) dx.\]

The slices of $K'$ are closely related to those of $K$.
Define the set \begin{equation}
    B :=\{y \in \R^{m-1}: f_K(y) \geq -p\}.
    \label{eq:banaB}
\end{equation}
As $f_K$ is concave (as $K$ is convex), the slice of $K$ at $x=-p$ is exactly $K_{-p}=B$. So by the definition of $K'$ in \eqref{eq:kprime} and \eqref{eq:gky} we have,
\begin{align}
    K'_x = \begin{cases} 
       B  &  \text{for $x < -p+r$} \\
       K_{x-r}  & \text{for $x \geq  -p+r$}.
       \end{cases}
       \label{eq:k'slices}
\end{align} 
In words, for $x \geq -p+r$ the slice $K'_x$ is a slice of $K$ shifted to the right by $r$. 
And for any $x \leq -p+r$,  the slice $K'_x$ is just $B$. See Figure \ref{fig:bana-proof-5}.

\paragraph{Another Ehrhard Symmetrization.}
We now apply another symmetrization, but this time along the subspace $W=\text{span}(e_2,\ldots,e_m)$. This will reduce $K$ and $K'$ to two-dimensional convex bodies.

 In this symmetrization step, we replace the slice $K_x \in \R^{m-1}$ of $K$ by the (one-dimensional) interval $I_x=(-\infty, \Phi^{-1}(\gamma^{m-1}(K_x))]$. Notice that $\gamma^1(I_x) = \gamma^{m-1}(K_x)$. 
 
 Formally, define the function $h_K:\R \rightarrow \R$ as
\[h_K(x) = \Phi^{-1} (\gamma^{m-1}(K_x)).\]
Then the symmetrization  of $K$ along $W$ is the 2-d body 
\[  K_W  :=  \{(x,t): x\in R, t \leq h_K(x)\}.\]
Similarly, we define $h_{K'}= \Phi^{-1} (\gamma^{m-1}(K'_x))$ for $K'$, and let $K'_W$ be the symmetrization of $K'$ along $W$.
See Figure \ref{fig:bana-proof-6}.

\begin{figure}[hbtp!]
    \centering
\includegraphics[scale=0.61, trim={6mm 5mm 10mm 10mm}]{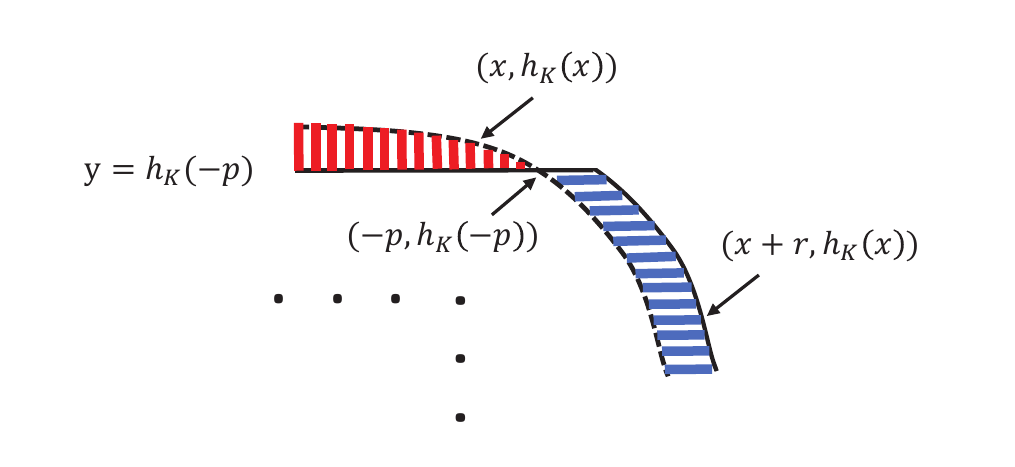}
    \caption{The two-dimensional bodies $K_W$ and $K'_W$ given by $h_K(x)$ and $h_{K'}(x)$. The bodies extend infinitely to the left and to the bottom}
    \label{fig:bana-proof-6}
\end{figure}

The bodies $K_W$ and $K'_W$ have the following properties. First, it is easily checked that $\gamma^2(K_W)  = \gamma^m(K)$ and $\gamma^2(K'_W) = \gamma^m(K')$ as symmetrization is measure-preserving.
The second remarkable property which will follow from Ehrhard's inequality in Section \ref{sec:bana-proof-ehrhard} is that this symmetrization preserves convexity.
\begin{proposition} As $K,K'$ are convex,
   $K_W, K'_W$ are also convex. 
\end{proposition}

\paragraph{The structure of $K_W$ and $K'_W$.} It is useful to understand the structure of $K_W$ and $K'_W$ more closely.

For $K$, we have $h_K(-p) = \Phi^{-1}(\gamma^{m-1}(B))$ by definition of $B$ in \ref{eq:banaB}.
As $x\mapsto K_x$ is decreasing in $x$ and $h_K(x) = \Phi^{-1}(\gamma_{m-1}(K_x))$, the function $h_K$ is decreasing. Convexity of $K_W$ also implies that $h_K$ is concave.
 See Figure \ref{fig:bana-proof-6}. 

For $K'$ by \eqref{eq:k'slices} we have,
  \begin{align*} 
   h_{K'}(x)= \begin{cases}
       h_K(-p) & \text{for $x< -p+r$}\\
       h_K(x-r)& \text{for $x\geq -p+r$}.
   \end{cases}
   \end{align*}
Notice that the boundaries of $K_W$ and $K'_W$ (given by the curves $(x,h_K(x))$ and $(x,h_{K'}(x))$) intersect at the point $P:=(-p,h_K(-p))$. See Figure \ref{fig:bana-proof-7}.

Finally, comparing $K_W$ to $K'_W$, we lose the (red) region $A$ and  gain the (blue) region $B$. Our goal is to show that  $\gamma^2(A) \leq \gamma^2(B)$.

\paragraph{The final simplification: Reducing $h_K$ to a line.}
At this stage, the only complicated quantity left is the function $h_K$. We now show how to replace $h_K$ by a line such that
 $\gamma^2(A)$ can only increase while $\gamma^2(B)$ can only decrease.
 
Let
$L(x)$ denote the line through the points $P=(-p,h_K(-p))$ and $Q = (-r/2,h_K(-r/2))$ on the curve $h_K(x)$, and let us replace the curve $h_K(x)$ by $L(x)$,  see Figure \ref{fig:bana-proof-7}.

Let us redefine the red region $A$ and blue region $B$ according to $L(x)$ and call them $C$ and $D$.  See Figure \ref{fig:bana-proof-7}.

\begin{figure}[hbtp!]
    \centering
\includegraphics[scale=0.6, trim={10mm 2mm 2mm 3mm}]{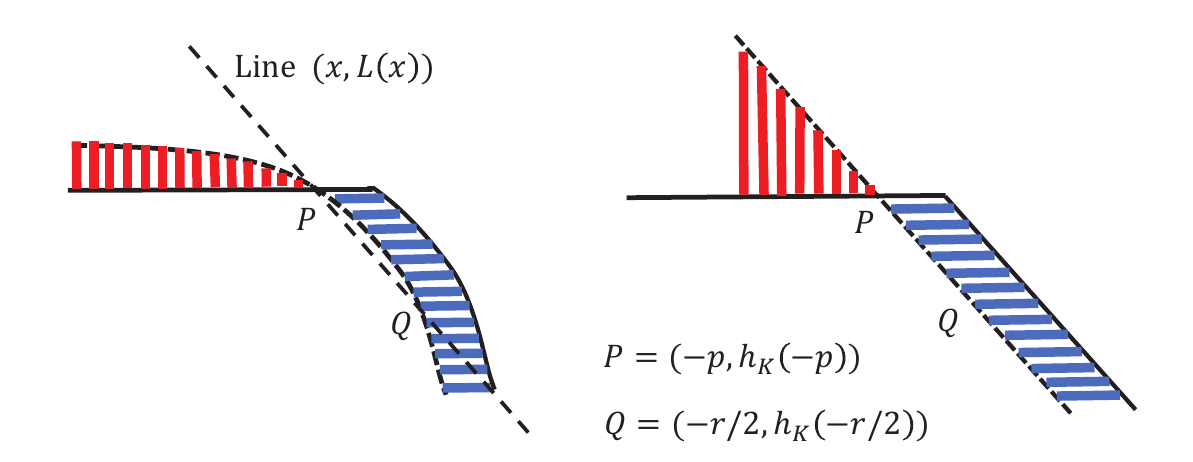}
    \caption{Replacing the curve $h_K(x)$ by the line $L(x)$. On the right the red and blue regions are defined according to $L(x)$.}
    \label{fig:bana-proof-7}
\end{figure}

\begin{lemma}
\label{lem:bana-hk-to-line}
Replacing $h_K(x)$ by $L(x)$ only increases $\gamma^2(A)$ and only decreases $\gamma^2(B)$. More precisely,
 $\gamma^2(C) \geq \gamma^2(A)$ and   $\gamma^2(D) \leq \gamma^2(B)$.
\end{lemma}
\begin{proof}
We show that the measure of each horizontal segment in $B$ only decreases and for those in $A$ it only increases.

Recall that $h_K$ is concave and decreasing.
The horizontal segments $(-\infty,h_K^{-1}(y))$ that comprise the red region $A$ can only get larger as $h_k^{-1}(y) \leq L^{-1}(y)$ for $y \geq h_K(-p)$. So this only increases $\gamma^2(A)$.

The blue region $B$ consists of segments $I_x=(x,x+r)$ for $x \geq -p$.
Notice that for any segment $I_a = [a,a+r]$ of length $r$, if $a < -r/2$ its Gaussian measure $\gamma^1(I_a)$ decreases 
 upon shifting $I_a$ to the left, and if $a\geq  -r/2$ then $\gamma^1(I_a)$ decreases upon shifting $I_a$ to the right.

Since $h_K$ is concave and $L$ is the chord between the points $(-p,h_K(-p))$ and $(-r/2,h_K(-r/2))$, we have $L(x)\ge h_K(x)$ for $x\ge -r/2$ and $L(x)\le h_K(x)$ for $x\le -p-d$.
So when we replace $h_K$ by $L(x)$ the interval $I_x$ moves leftward for each $x \in [-p,-r/2)$  and so $\gamma^1(I_x)$ only decreases. Similarly for $x > r/2$,  the interval $I_x$ moves rightward and $\gamma^1(I_x)$ also only decreases. Thus $\gamma^2(B)$ only decreases.  See Figure \ref{fig:bana-proof-7}.
\end{proof}

\section{Comparing $1$-dimensional intervals}
\label{sec:bana-proof-comparing-1-dim-intervals}
By Lemma \ref{lem:bana-hk-to-line}, to prove Theorem \ref{thm:bana-star-2} it suffices to show that $\gamma^2(C) \leq \gamma^2(D)$. We do this next, which will finish off the proof.

 Notice that $C$ consists of the left-infinite horizontal intervals $\ell_{x} := (-\infty,x)$ for $x\leq -p$.
And $D$ consists of horizontal intervals  $r_x := (x,x+r)$ for $x \geq -p$.
Also, each $\ell_x$ and $r_x$ is located at height $y=L(x)$.

\paragraph{Pairing intervals.} To show that $\gamma^2(C) \leq \gamma^2(D)$, we will pair the intervals of $C$ and $D$ as follows, and show the inequality for each pair. 
\begin{enumerate}
    \item Type 1: For each $d \geq p$, pair the interval $\ell_{-p-d}$ in $C$ with $r_{d}$.
    \item Type 2: For each $d \in [0,p)$, pair the interval $\ell_{-p-d}$ to extent half each with the intervals $r_{-d}$ and $r_{d}$. 
\end{enumerate}
Figure \ref{fig:bana-8} showing this pairing.

\begin{figure}[hbtp!]
    \centering
\includegraphics[scale=0.6, trim={3mm 2mm 2mm 3mm}]{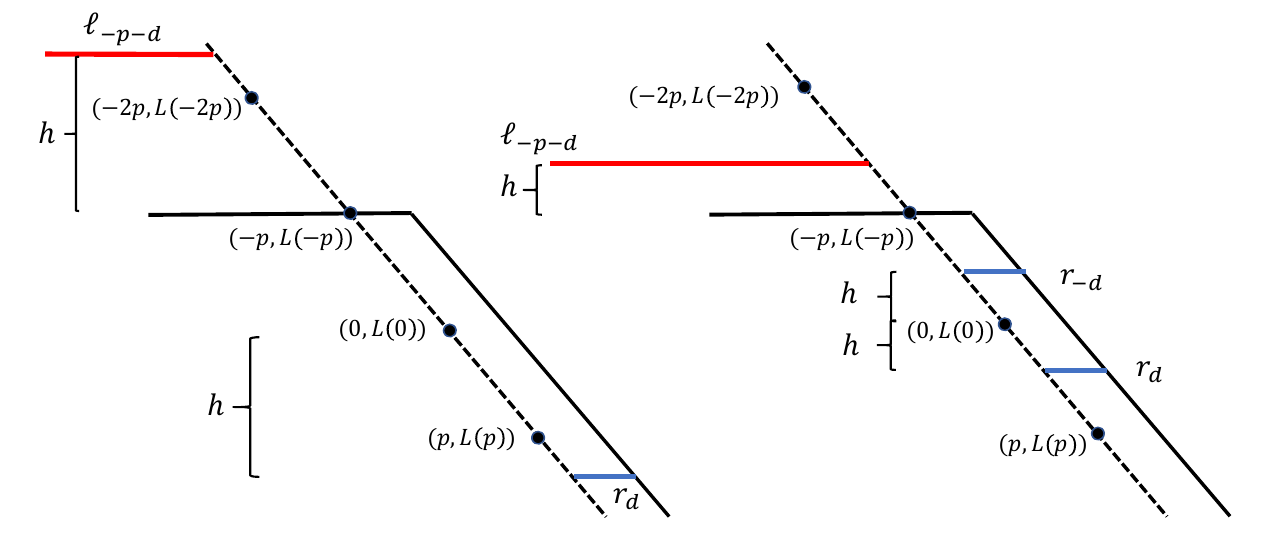}
    \caption{On the left is a type 1 pairing between the red interval $\ell_{-p-d}$ and $r_d$.
    On the right is a type 2 pairing. }
    \label{fig:bana-8}
\end{figure}

The rest of the section will be devoted to showing the following. 
\begin{theorem}
\label{thm:bana-final-1d-comp}
   For type 1 pairs, $\gamma^1(\ell_{-p-d}) \leq \gamma^1 (r_{d})$ for every $d\geq p$. 
    For type 2 pairs,  $\gamma^1(\ell_{-p-d}) \leq (\gamma^1 (r_{d})+ \gamma^1(r_{-d}))/2$ for each $d\in [0,p)$.
\end{theorem}
Even though Theorem \ref{thm:bana-final-1d-comp} only talks about the $1$-d measure of intervals (while their contribution to $\gamma^2(C), \gamma^2(D)$ also depends on their height $y$). However, the lemma below shows that this is not a problem and the heights $y$ only work in our favor.
\begin{lemma}
        Theorem \ref{thm:bana-final-1d-comp} implies that $\gamma^2(C) \leq \gamma^2(D)$.
\end{lemma}
\begin{proof}
The key observation is that $L(0)\geq 0$. 

Indeed as $\gamma^2(K_W) = \gamma^m(K) \geq 1/2$, the  origin $(0,0) \in K_W$ and thus $h_K(0)\geq 0$. By design, as $L(x) \geq h_K(x)$ for $x \geq -r/2$ (see Figure \ref{fig:bana-proof-7}), we have that $L(0)\geq 0$.

Let us first consider type $1$ interval pairs $\ell_{-p-d}$ and $r_{d}$ for $d\geq p$. Their $y$-coordinates are $L(-p-d)$ and $L(d)$ respectively. As $L(0) \geq 0$ and $L(x)$ is a line, we have that $|L(-p-d)| \geq |L(d)|$. So the factor $(2\pi)^{-1/2} \exp(-y^2/2)$ in the Gaussian density $\phi(x,y)$ due to the $y$-coordinate is only smaller for the interval $\ell_{-p-d}$.

Similarly, consider type $2$ interval pairs $\ell_{-p-d}$ and $r_{-d}, r_d$, for $0 \leq d \leq p$. Again as $L(0)\geq 0$, their $y$-coordinates satisfy $|L(-p-d)| \geq |L(-d)|\geq |L(d)|$, and the result follows by the argument above.
\end{proof}

\subsection{Comparing intervals and log-concavity}
We now focus on proving Theorem \ref{thm:bana-final-1d-comp}. This is the final piece which will complete the proof of Theorem \ref{thm:bana-star-2}.

Let us denote $\overline{\Phi}(x) = 1-\Phi(x) = \int_{t\geq x} \phi(t) dt$. We can write \[\gamma^1(r_d)= \overline{\Phi}(d)-\overline{\Phi}(d+r) 
\text{ and } \gamma^1(\ell_{-p-d}) =\Phi(-p-d) = \overline{\Phi}(p+d)\]
where we use that $\Phi(-x)= \overline{\Phi}(x)$ for all $x$.

As we are interested in the ratio of $\gamma^1(r_d)$ and $\gamma^1(\ell_{-p-d})$ as $d$ varies, let us define the function
\begin{equation}
\label{eq:bana-proof-fd}
     f(d) := \frac{\gamma^1(r_d)}{\gamma^1(\ell_{-p-d})} = \frac{\overline{\Phi}(d)-\overline{\Phi}(d+r)} {\overline{\Phi}(p+d)}.
\end{equation} 
The following property of $f$ will be very useful.

\begin{lemma}
    \label{bana:increasing-fn-lemma}
   For any fixed $r,p \geq 0$, the function $f(d)$ in \eqref{eq:bana-proof-fd} is non-decreasing for all $d\in \R$.
\end{lemma}
The proof of Lemma \ref{bana:increasing-fn-lemma} and later arguments use log-concavity, so let us recall this briefly, focusing only on the properties we need.
\paragraph{Log-concavity.} A non-negative function $f:\R^d \rightarrow \R$ is called log-concave if $\log f$ is concave.
Equivalently, for all $x,y$  and $0 \leq \lambda \leq  1$,
\begin{equation}
    \label{eq:log-concavity-equiv1}
    f(\lambda x + (1-\lambda) y) \geq f(x)^\lambda f(y)^{1-\lambda}.
\end{equation}  
Here we will only be interested in functions $f:\R \rightarrow \R$. We have the following equivalent characterization for such functions.
\begin{proposition}
\label{prop:log-conc-non-decreasing}
A smooth function $f:\R\rightarrow \R$ is log-concave iff $f(x)/f(x+a)$ is non-decreasing in $x$ for all $a\geq 0$.  
\end{proposition}
\begin{proof}
If $f$ is log-concave then $\log f$ is concave and hence its derivative $(\log f)'$ is non-increasing. So for any fixed $a\geq 0$,
\[ \left(\log \frac{f(x)}{f(x+a)}\right)' =  (\log f(x))' - (\log f(x+a))' \geq 0.\]
So $\log (f(x)/f(x+a))$ is non-decreasing, which implies that $f(x)/f(x+a)$ is also non-decreasing. The converse also follows similarly.
\end{proof}
Clearly the Gaussian density $\phi(x)$ is log-concave. We also have the following.
\begin{proposition}
    Both $\Phi$ and $\overline{\Phi}$ are log-concave.
\end{proposition}
\begin{proof}
It is an interesting exercise to show directly that $(\log \Phi)'$ 
and $(\log \overline \Phi)'$ are non-increasing
(this requires the useful fact $\phi'(x) = -x \phi(x)$ about the Gaussian density).

A more slick way is using the Prekopa-Leindler inequality. In particular, the log-concavity of $\overline{\Phi}$ follows by applying Lemma \ref{lm:ch3-bm-gauss} with $A = [x,\infty)$ and $B = [y,\infty)$ and using the definition \eqref{eq:log-concavity-equiv1}. 
\end{proof}

Lemma \ref{bana:increasing-fn-lemma} now follows by writing $f(d)$ in \eqref{eq:bana-proof-fd} as 
\[f(d) = \left( 1 - \frac{\overline{\Phi}(d+r)}{\overline{\Phi}(d)}\right)\left(\frac{\overline{\Phi}(d)}{\overline{\Phi}(d+p)} \right) \]
and observing that both the terms above are non-decreasing by Lemma \ref{prop:log-conc-non-decreasing} as $r,p\geq 0$.

We can now prove Theorem \ref{thm:bana-final-1d-comp}.

\paragraph{Type $1$ pairs.} 
Recall that these consist of intervals
$r_d$ and $\ell_{-p-d}$ for $d\geq p$.
By Lemma \ref{bana:increasing-fn-lemma}  we have 
\begin{equation}
    \frac{\gamma^1(r_{d}) }{\gamma^1(\ell_{-p-d})} = f(d) \geq f(p) =  \frac{\overline{\Phi}(p)-\overline{\Phi}(p+r)}{\overline{\Phi}(2p)}
    \label{bana:eq-final-comp-pair-1}
\end{equation}
We will show that $f(p) \geq 1$ in two steps.

\begin{lemma}
\label{lem:type1:comp1}
$\overline{\Phi}(p)-\overline{\Phi}(p+r) \geq \overline\Phi(p)^2$ and thus $f(p) \geq \overline\Phi(p)^2/\overline\Phi(2p)$.
\end{lemma}
\begin{proof}
Applying Lemma \ref{prop:log-conc-non-decreasing}  with $x=p$ and $a=r$, and as $r> 0$, we have that 
$\overline{\Phi}(p)/\overline{\Phi}(p+r)$ is non-decreasing in $p$. 
As $p\geq 0$, this gives 
\begin{align*}
    1-\frac{\overline{\Phi}(p+r)}{\overline{\Phi}(p)}  &  \geq  1-\frac{\overline{\Phi}(r)}{\overline{\Phi}(0)} = 1-2\overline{\Phi}(r) \qquad \qquad  \text{(as $\overline{\Phi}(0)=1/2$)}\\
& = \overline{\Phi}(-r) - \overline{\Phi}(r) = \gamma^1([-r,r])= \Phi(-p) =\overline{\Phi}(p) 
\end{align*} 
where the last equality uses the definition of $p$ in \eqref{eq:bana-pf-defn-p}. 
\end{proof}
\begin{lemma}
\label{lem:type2:comp2}
$\overline\Phi(p)^2/\overline\Phi(2p) \geq \overline\Phi(1)^2/\overline\Phi(2) \geq 1$. 
\end{lemma}
\begin{proof}
   Consider the function $g(x) =\overline{\Phi}(x)^2/\overline{\Phi}(2x)$. We claim that $\log g(x)$ and hence $g(x)$ is non-decreasing in $x$. Indeed
   \[(\log g(x))' = (2 \log \overline{\Phi}(x) - \log \overline{\Phi}(2x))'= 2\frac{\overline{\Phi}'(x)}{\overline{\Phi}(x)} -  2\frac{\overline{\Phi}'(2x)}{\overline{\Phi}(2x)} \geq 0,\]
where the inequality follows as $\overline{\Phi}'(x)/\overline{\Phi}(x) = (\log (\overline{\Phi}(x))'$ is non-increasing by the log-concavity of $\overline{\Phi}(x)$. 

As $p \geq 1$ by \eqref{eq:bana-pf-defn-p} and the assumption $\gamma^1([-r,r]) \leq \gamma^1((-\infty,-1])$, we get $g(p) \geq g(1) = \overline\Phi^2(1)/\overline\Phi(2)$, which can be checked numerically to be at least $1$.
\end{proof}
By \eqref{bana:eq-final-comp-pair-1} together with  Lemma \ref{lem:type1:comp1} and \ref{lem:type2:comp2} we gave that $\gamma^1(r_d) \geq \gamma^1(\ell_{-p-d})$ for all $d\geq p$.


\paragraph{Type 2 pairs.}
We now  consider the case when $d\in [0,p)$ and show that \[\frac{\gamma^1(r_{d})+\gamma^1(r_{-d})}{\gamma^1(\ell_{-p-d})}\geq 1.\] 
First, it suffices to show that $\gamma^1(r_d)/\gamma^1(\ell_{-p-d}) \geq 1/2$. This is because 
for any $d\geq 0$  
\[ \gamma^1(r_{d}) = \gamma^1([d,d+r]) \leq \gamma^1([-d,-d+r]) = \gamma^1(r_{-d}),\] as both intervals have length $r$ and the latter is closer to the origin.

Now Lemma \ref{bana:increasing-fn-lemma} and as $d\geq 0$,
\[  \frac{\gamma^1(r_d)}{\gamma^1(\ell_{-p-d})} = f(d) \geq f(0)=  \frac{\overline{\Phi}(0)-\overline{\Phi}(r)}{ \overline{\Phi}(p)} = \frac{1}{2}.\]
 Here, the last step uses that $\overline{\Phi}(p) = \overline{\Phi}(-r) - \overline{\Phi}(r) = 2 (\overline{\Phi}(0)-\overline{\Phi}(r))$,  
by the definition \eqref{eq:bana-pf-defn-p} of $p$.

\section{Ehrhard's inequality}
\label{sec:bana-proof-ehrhard}
We now describe perhaps the most important Gaussian isoperimetric-type inequality.

Let $\Phi$ denote the Gaussian distribution function
\[  \Phi(x) = \gamma^1((-\infty,x]) = (2\pi)^{-1/2}\int_{-\infty}^x \exp(-t^2/2) dt.\]
The function $\Phi^{-1}:[0,1] \rightarrow  \R \cup \{-\infty,\infty\}$ will play an important role.

\begin{theorem}
\label{thm:ehrhard} [Ehrhard-Borell inequality] \cite{Ehrhard1983, BORELL2003}
 For any measurable sets $A, B \subset \R^n$ and $0 \leq \lambda  \leq 1$, 
\begin{equation}
\label{ineq:ehrhard}
\Phi^{-1}(\gamma^n(\lambda A + (1-\lambda)B)) \geq \lambda \Phi^{-1} (\gamma^n(A)) + (1-\lambda) \Phi^{-1} (\gamma^n(B)).
\end{equation}
\end{theorem}

More pictorially, imagine associating the set $A$ with the interval $I_A=(-\infty,x_A)$ with the ``same'' Gaussian measure, i.e.,~$\gamma^n(A) = \gamma^1((-\infty,x_A])$. 
Notice that $\Phi^{-1} (\gamma^n(A)) =x_A$. 
Similarly, consider the interval $I_B=(-\infty,x_B)$ for set $B$. Theorem \ref{thm:ehrhard} says that point $x_\lambda$ corresponding to the set $\lambda A + (1-\lambda) x_B$ satisfies $x_\lambda \geq \lambda x_A + (1-\lambda) x_B$.

Let us also contrast this with the Gaussian Brunn-Minkowski inequality in Lemma \ref{lm:ch3-bm-gauss}, which states that 
\begin{equation}
    \label{eq:bm-bana-pf}\log \gamma^n(\lambda A+ (1-\lambda)B) \geq \lambda \log \gamma^n(A) + (1-\lambda) \log \gamma^n(B).
\end{equation}
Then \eqref{ineq:ehrhard} says that $\Phi^{-1} \circ \gamma^n$ is concave, while \eqref{eq:bm-bana-pf} shows that $\log \circ \,\gamma^n $ is concave.
It turns out that Ehrhard's inequality is stronger than Brunn-Minkowski for the Gaussian measure.

\paragraph{Gaussian symmetrization.}
We discuss Gaussian symmetrization more generally. Let $W$ be a subspace of $\R^n$ of dimension $n-k$, and let us decompose $\R^n = W^\perp \oplus W$ and view the point $x \in \R^n$ as $x = (y,w)$ where $ w= \Pi_W(x) $ and $y=\Pi_{W^\perp}(x)$.

Let $A \subset \R^n$ be a measurable set.
For $y \in {W}^\perp$,  consider the affine subspace $y + W$ and let $A_y = A \cap (y+W)$
be the $n-k$ dimensional slice obtained by intersecting $A$ with this affine subspace.
Consider the function $f:W^\perp \rightarrow \R$ defined as $f(y) = \Phi^{-1}(\gamma^{n-k}(A_y))$.
The Gaussian symmetrization of $A$ along $W$ is the $k+1$-dimensional body  
\[ A' =  \{(t,y): y \in W^{\perp}, t \leq f(y)\}.\] 
In Sections \ref{sec:bana-proof-symmetrize-u} and \ref{sec:bana-proof-reducing-two-dimensions}, we applied this for subspaces $W = \text{span}(e_1)$ and $W = \text{span}(e_2,\ldots,e_{m})$.

Observe that by construction $\gamma^{k+1}(A') = \gamma^n(A)$.
Moreover importantly, if $A$ is convex then Theorem \ref{thm:ehrhard} implies that the function $f(y)$ is concave and hence that $A'$ is also convex.

\section{Bibliographic Notes}
The proof here is due to Banaszczyk \cite{Bana98}, but our exposition is more detailed.

Theorem \ref{thm:ehrhard} was proved by
Ehrhard \cite{Ehrhard1983} for the case when $A$ and $B$ are convex, and the inequality for general case was proved by Borell \cite{BORELL2003}.
Several other proofs are also known \cite{handel2018, neeman20}.
Interestingly, it follows from \cite{Ehrhard1983},  that proving Theorem \ref{thm:ehrhard} for the case of $n=1$ implies the result for general $n$.

Ehrhard's inequality also implies various other inequalities such as  the Gaussian Isoperimetric inequality, Bobkov's inequality and Gross logarithmic Sobolev inequality. A nice survey is in \cite{Latala-icm02}.



\chapter{Hereditary Discrepancy and Factorization Norms}
\label{ch:gamma2}

We have now seen several tools for proving upper bounds on the
discrepancy of set systems and matrices, some of which also give
efficient algorithms for computing a coloring achieving the
discrepancy bound. In this chapter we address the question of the
complexity of approximating discrepancy and hereditary discrepancy. In
addition to showing some hardness results, we will show that
hereditary discrepancy is approximated well by the
$\gamma_2$ matrix factorization norm. Beyond giving an approximation
algorithm for hereditary discrepancy, this result has a number of
applications to discrepancy theory. It gives a powerful tool for
giving short proofs of upper and lower bounds on discrepancy, and for
proving various stability properties of discrepancy. 

\section{Complexity of Computing Discrepancy}

Perhaps the most basic computational question about discrepancy is
whether $\disc(A)$ can be computed efficiently given a matrix $A$ as
input. An easier question is whether the promise that $\disc(A)$ is
small, \eg, at most a universal constant, or even $0$, is enough to
efficiently find a coloring whose discrepancy is smaller than
Spencer's worst-case bound $O(\sqrt{n\log(2m/n)})$. Unfortunately, the
answer to both these questions is negative, as long as $\mathsf{P}
\neq \mathsf{NP}$, as shown in the following theorem.

\begin{theorem}\label{thm:ch7-disc-hard}
  For a set system $(\SS, [n])$ with $|\SS| = O(n)$ sets, it is $\mathsf{NP}$-hard to distinguish between the cases: (i) $\disc(\SS) = 0$,
or (ii) $\disc(\SS) = \Omega(\sqrt{n})$.
\end{theorem}

We will establish Theorem~\ref{thm:ch7-disc-hard} by reduction from
the $2$-$2$ Set Splitting problem, whose definition and hardness are given below. 

\begin{theorem}\label{thm:ch7-SS}(Guruswami~\cite{venkat-SS})
For a set system $(\TT, [n])$ with $m = O(n)$ sets, where $|T| = 4$
for all $T \in \TT$, and each $j \in [n]$ appears in at most $4$
sets, it is $\mathsf{NP}$-hard to distinguish between the cases:
\begin{enumerate}
\item $\disc{(\TT)} = 0$.
\item for all $\col \in \{\pm 1\}^n$, $|\{T \in \TT: \col(T) \not=
  0 \}| \geq \alpha m$, where $\alpha \approx 1/22$. 
\end{enumerate}
\end{theorem}

Note that Theorem~\ref{thm:ch7-SS} already shows that it is
$\mathsf{NP}$-hard to decide if a set system has discrepancy $0$, or
discrepancy at least $2$. By composing a $2$-$2$ Set Splitting
instance, in an appropriate way, with a set system satisfying a robust
discrepancy lower bound of $\Omega(\sqrt{n})$, we can amplify this gap and prove Theorem~\ref{thm:ch7-disc-hard}. The incidence matrix
of the set system used in this composition is given by the following lemma. In it, we
use the notion of a Hadamard matrix, i.e., an $m \times m$ matrix with
entries in $\{-1,+1\}$ such that $H^T H = mI$. Such matrices are known
to exist, and can be efficiently constructed, for all $m$ that are
powers of $2$ (among other values of $m$). 

\begin{lemma}\label{lm:ch7-hadamard}
  Let $H$ be an $m\times m$ Hadamard matrix such that
  one of its rows is the all-ones vector. Then the matrix
  $A$ obtained by replacing each $-1$ entry of
  $H$  with $0$ satisfies
  \(
  \|A\col\|_\infty \ge \frac{1}{3}\|\col\|_2
  \)
  for  any $\col \in \R^m$.
\end{lemma}
\begin{proof}
  Let $\col \in \R^m$ be arbitrary. We first observe that, by
  comparing the $\ell_\infty$ and $\ell_2$ norms, 
  \begin{align}
  \label{eq:hadamard}
    \| H\col \|_{\infty}^2
    \geq \frac{1}{m} \|H\col \|_2^2
    = \frac1m \col^T  H^T H \col
    = \col^T \col =\|\col \|_2^2.
  \end{align}
  Notice that $A = \frac12 (H + J)$ where $J$ is the $m\times m$
  all-ones matrix. Suppose, without loss of generality, that the first
  row of $H$ is the all-ones vector, and, therefore, the first row of
  $A$ is also the all-ones vector. Then
  \[
    \|A\col\|_\infty = \frac12 \|H\col + J\col\|_\infty
    = \frac12 \max_{i = 1}^m |(H\col)_i + (A\col)_1|.
  \]
  By \eqref{eq:hadamard}, we know that $|(H\col)_{i^*}|
  \ge \|\col\|_2$ for some $i^*$. If $|(A\col)_1| \le \frac13\|\col\|_2$, then, by the
  triangle inequality,
  \[
    \|A\col\|_\infty 
    \ge \frac12 (|(H\col)_{i^*}| - |(A\col)_1|) \ge  \frac13 \|\col\|_2.
  \]
  Otherwise, if $|(A\col)_1| > \frac13\|\col\|_2$, then 
  \(
  \|A\col\|_\infty  \ge |(A\col)_1| > \frac13\|\col\|_2.
  \)
\end{proof}

Notice that the requirement that the Hadamard matrix $H$ has an all-ones
row is not restrictive. We can replace every entry $H(i,j)$ in $H$ by
$H(i,j)H(1,j)$, and still get a Hadamard matrix, but now all entries
in the first row are equal to $1$. 

Theorem~\ref{thm:ch7-SS} and Lemma~\ref{lm:ch7-hadamard} allow for a
very simple hardness reduction that shows a statement analogous to
Theorem~\ref{thm:ch7-disc-hard}, but for matrices with entries in
$\{0,4\}$, as opposed to set systems. The reduction simply takes the
$m\times n$ incidence matrix $B$ of the set system $\TT$ in
Theorem~\ref{thm:ch7-SS}, pads it with $0$'s if necessary to make sure
that $m$ is a power of $2$, and outputs a matrix $\mat \eqdef A B$ for
the $m\times m$ matrix $A$ defined in Lemma~\ref{lm:ch7-hadamard}.

\begin{exercise}
  Analyze the reduction described in the paragraph above.
\end{exercise}

In order to get a reduction that, instead, outputs a set
system, we need a simple decomposition lemma, stated next. 

\begin{lemma}\label{lm:ch7-partition}
  Suppose that $(\TT, [n])$ is a set system such that each set in $\TT$
  has at most $s$ elements, and each element $j \in [n]$
  appears in at most $t$ sets in $\TT$. Then $\TT$ can be partitioned
  in polynomial time
  into $k \le st+1$ disjoint set systems $\TT_1, \ldots, \TT_k$, where
  each $\TT_i$ consists of disjoint subsets of $[n]$.
\end{lemma}
\begin{proof}
  We partition $\TT$ greedily, similarly to the greedy coloring
  algorithm for bounded degree graphs. We list the sets in $\TT$ in some
  order as $T_1, \ldots, T_m$, and, for each $i$, starting
  from $i=1$ up to $i=m$, we add $T_i$ into $\TT_\ell$, for the smallest $\ell$
  for which $\TT_\ell$ does not contain any set intersecting
  $T_i$. Since each $T_i$ can intersect at most $st$ other sets in
  $\TT$, there is always some $\ell \le st+1$ so that $T_i$ can
  be added to $\TT_\ell$.
\end{proof}

\begin{proof}[Proof of Theorem~\ref{thm:ch7-disc-hard}]
  Let $\TT$ be the set system from Theorem~\ref{thm:ch7-SS}, and let
  us partition it into $k\le 17$ set systems $\TT_1, \ldots, \TT_k$,
  each consisting of disjoint sets, using
  Lemma~\ref{lm:ch7-partition}. Let $B_\ell$ be the incidence matrix
  of $\TT_\ell$. We can assume that each $B_\ell$ has dimensions $m$
  by $n$ for some $m$ which is a power of $2$, by padding with
  zeros. Let us take $H$ to be an $m\times m$ Hadamard matrix, and
  define $A$ as in Lemma~\ref{lm:ch7-hadamard}. We now define $\mat_\ell \eqdef A B_\ell$, for
  $\ell \in [k]$. Notice that, since $\TT_\ell$ consists of disjoint
  sets, $B_\ell$ has at most a single $1$ in each column, so
  $A B_\ell$ is an $m\times n$ matrix with entries in $\{0,1\}$. We
  can then interpret $\mat_\ell$ as the incidence matrix of a set
  system $\SS_\ell$ on $[n]$. We define the set system $\SS$ as the
  union $\SS_1\cup \ldots \cup \SS_k$. $\SS$ has at most $km \le 17m =
  O(n)$ sets and can clearly be constructed in polynomial time.

  To finish the proof, we claim that the following two implications
  hold for an absolute constant $c > 0$:
  \begin{align}
    \disc(\TT) = 0 &\implies \disc(\SS) = 0;\label{eq:ch7-compl-sets}\\
    \min_{\col \in \{-1,+1\}^n} |\{T \in \TT: \col(T) \not=  0 \}| \geq \alpha m
                   &\implies \disc(\SS) \ge c\sqrt{\alpha m}.\label{eq:ch7-sound-sets}
  \end{align}
  The theorem then follows from Theorem~\ref{thm:ch7-SS}.

  We first prove \eqref{eq:ch7-compl-sets}. Since $\disc(\TT)$ is $0$,
  there exists a coloring $\col \in \{-1,+1\}^n$ for which $\disc(\TT,
  \col) = 0$. Then $B_\ell \col = 0$ for all $\ell \le k$, since each
  row of each $B_\ell$ is the indicator vector of a set in $\TT$. Therefore,
  \[
    \disc(\SS,\col) = \max_{\ell = 1}^k \disc(\SS_\ell,\col)
    =
    \max_{\ell=1}^k \|\mat_\ell \col\|_\infty
    =
    \max_{\ell=1}^k \|AB_\ell \col\|_\infty = 0.
  \]

  Finally, we prove \eqref{eq:ch7-sound-sets}. Take any $\col \in
  \{-1,+1\}^n$, and suppose that $|\{T \in \TT: \col(T) \not=  0 \}|
  \geq \alpha m$. This means that for some $\ell \le k$,
  \[
    |\{T \in \TT_\ell: \col(T) \not=  0 \}| \geq \alpha m/k
    \ge \alpha m/17.
  \]
  Therefore, $B_\ell\col$ has at least $\frac{\alpha m}{17}$ non-zero
  entries. Since all entries in $B_\ell\col$ are integers, this means
  that $\|B_\ell\col\|_2 \ge \sqrt{\alpha
      m/17}$. So, by Lemma~\ref{lm:ch7-hadamard},
  \[
    \disc(\SS,\col) \ge \|\mat_\ell\col\|_\infty
    =
    \|AB_\ell \col\|_\infty
    \ge \frac13 \sqrt{\alpha m/17}. 
  \]
  Since $\col \in \{-1,+1\}^n$ was arbitrary, this completes the proof
  of \eqref{eq:ch7-sound-sets}, 
and,
  therefore, the theorem. 
\end{proof}

\section{Approximating Hereditary Discrepancy}
  
In contrast with the strong hardness result for discrepancy given by
Theorem~\ref{thm:ch7-disc-hard}, we will see in the rest of this
chapter that hereditary discrepancy can be approximated efficiently up
to multiplicative poly-logarithmic factors. This is one concrete
example of hereditary discrepancy being a ``better behaved''
discrepancy notion. We will see other examples later on, also in this
chapter.

\subsection{Upper Bounds}

Recall that the hereditary discrepancy $\herdisc(\mat)$ of a matrix
$\mat$ is the maximum discrepancy $\disc(\mat(*,J))$ for any subset
$J$ of the columns of $\mat$. Since this is a maximum over an exponential
number of quantities, each of them intractable to approximate, it is
far from obvious that we can find useful tractable upper bounds on
$\herdisc(\mat)$. In this subsection we show that this is, in fact,
possible. 

Our starting point is the observation that if $f$ is a function on
$m\times n$ matrices which is monotonically non-increasing under
taking submatrices, \ie, $f(\mat(*,J)) \le f(\mat)$ for any $J
\subseteq [n]$, and that also satisfies $\disc(\mat) \le f(\mat)$, then we
automatically get the stronger conclusion $\herdisc(\mat) \le
f(\mat)$. Indeed, for any $J \subseteq [n]$, we have
\[
  \disc(\mat(*,J)) \le f(\mat(*,J)) \le f(\mat).
\]
We call such a function $f$ a {\em hereditary upper bound} on discrepancy.

We use the notation $\|\mat\|_{1\to 2}$ for the maximum $\ell_2$ norm
of a column of the matrix $\mat$, and $\|\mat\|_{2\to\infty}$ for the
maximum $\ell_2$ norm of a row of the matrix $\mat$.\footnote{More
  generally, for an $m\times n$ matrix $\mat$, and
  $p,q \in [1,\infty]$,
  $\|\mat\|_{p\to q} \eqdef \sup_{x\neq 0} \frac{\|\mat
    x\|_q}{\|x\|_p}$. }

Notice first that both $c\sqrt{\log(2m)}\|\mat\|_{1\to 2}$ and
$c\sqrt{\log(2m)}\|\mat\|_{2\to\infty}$ are hereditary upper bounds on
discrepancy (for some constant $c>0$). The first follows from the Banaszczyk  upper bound for
the Koml\'os problem (Theorem~\ref{thm:komlos-bana}), and the second
from Hoeffding's inequality and a union bound, applied to a uniformly
random coloring.
Both these upper bounds can be, unfortunately, very loose, {\em e.g.},~consider
the all-ones matrix. 

However, we already saw in Lemma~\ref{lm:fact-ub}
that these bounds can be interpolated in a useful way using
Banaszczyk's theorem. Lemma~\ref{lm:fact-ub}, in the formulation
in~\eqref{eq:fact-ub}, shows that, for any $m\times n$ matrix $\mat$,
\(
\disc(\mat) \lesssim \gamma_2(\mat)\sqrt{\log(2m)},
\)
where
\begin{equation}\label{eq:ch7-gamma2-defn}
 \gamma_2(\mat) \eqdef \inf\{\|U\|_{2\to\infty}\|V\|_{1\to2}: UV = \mat\}.
\end{equation}
The notation $\gamma_2$ comes from functional analysis,
where this function was first studied in connection with factoring
linear operators between Banach spaces through a Hilbert space. 

The following corollary follows from the above and the observation
that $\gamma_2(\mat)$ is a hereditary upper bound on discrepancy.

\begin{corollary}[of Lemma~\ref{lm:fact-ub}]\label{cor:ch7-fact-ub} 
  For any $m\times n$ matrix $\mat$, we have
  \[
    \herdisc(\mat) \lesssim \sqrt{\ln(2m)} \gamma_2(\mat).
  \]
\end{corollary}
\begin{proof}
  This follows immediately from Lemma~\ref{lm:fact-ub}
  and the observation that $\gamma_2(\mat)$ is monotonically
  non-increasing under taking submatrices. Indeed, for any $J
  \subseteq [n]$, and any  factorization $\mat = UV$, we have that
  $\mat(*,J) = UV(*,J)$, and $\|V(*,J)\|_{1\to 2} \le \|V\|_{1\to 2}$. Minimizing over all choices of $U$ and $V$
  gives us $\gamma_2(\mat(*,J)) \le \gamma_2(\mat)$. Thus, by
  Lemma~\ref{lm:fact-ub}, we have
  \[
    \disc(\mat(*,J)) \lesssim \sqrt{\ln(2m)} \gamma_2(\mat(*,J))
    \le
    \sqrt{\ln(2m)} \gamma_2(\mat).
  \]
  Taking the maximum on the left hand side over all choices of $J\subseteq[n]$ gives the
  corollary. 
\end{proof}

Note that the all-ones matrix is no longer a bad example for
Corollary~\ref{cor:ch7-fact-ub}, since both the hereditary discrepancy
and the $\gamma_2$ norm of this matrix are equal to $1$. 

\begin{exercise}\label{ex:ch7-fact-ub}
  Give an example in which Corollary~\ref{cor:ch7-fact-ub} is tight up
  to constant factors. That is, show that for infinitely many values of
  $m$, there exists an $m\times n$ matrix $\mat$ for which
  $\disc(\mat) \gtrsim \sqrt{\ln(2m)} \gamma_2(\mat)$.
\end{exercise}

We now show that $\gamma_2(\mat)$ also gives a nearly tight
lower bound on hereditary discrepancy (see Theorem
\ref{thm:ch7-fact-lb} for the formal statement). This will be our goal in the rest of this section.

\subsection{Semidefinite Program Formulation}
We start  by reformulating
$\gamma_2(\mat)$ as a semidefinite program (SDP). One important
consequence of this formulation is that $\gamma_2(\mat)$ is
efficiently computable, and, therefore, so is the upper bound we
proved. By semidefinite programming duality, this formulation also gives a dual formulation of
$\gamma_2(\mat)$ as a maximization problem. This dual formulation will
be essential in giving a lower bound on $\herdisc(\mat)$ in terms of
$\gamma_2(\mat)$.

Recall the definition of $\gamma_2(\mat)$ in
\eqref{eq:ch7-gamma2-defn}. Note first that we can assume that the
infimum is over factorizations $\mat = UV$ for which
$\|U\|_{2\to\infty} = \|V\|_{1\to 2} = \sqrt{\gamma_2(\mat)}$. This is because, for any $c >
0$, we can replace $U$ by $cU$ and $V$ by $\frac1c V$ to get another
factorization, and choosing
$c \eqdef \sqrt{\frac{\|V\|_{1\to   2}}{\|U\|_{2\to\infty}}}$ makes
the two matrix norms equal.

Note also that, since $\gamma_2(\mat) \le \|I\|_{2\to \infty} \|\mat\|_{1\to 2}=  \|\mat\|_{1\to 2}$, we can
further assume that $\|U\|_{2\to\infty}$ and $\|V\|_{1\to 2}$ are both
bounded by $\|\mat\|_{1\to 2}^{1/2}$. Then it follows by a standard
compactness argument that the infimum in the definition of
$\gamma_2(\mat)$ is achieved. 

Let us now consider some factorization $\mat = UV$ of the $m\times n$
matrix $\mat$ into a $m \times k$ matrix $U$ and a $k\times n$ matrix
$V$. The matrix
\[
  X= \begin{pmatrix}
    U\\
    V^T
  \end{pmatrix}
  \begin{pmatrix}
    U^T  & V 
  \end{pmatrix}
  =
  \begin{pmatrix}
    UU^T & UV\\
    V^TU^T & V^T V
  \end{pmatrix}
  =
  \begin{pmatrix}
    UU^T & \mat\\
    \mat^T & V^T V
  \end{pmatrix}
\]
is positive semidefinite, since it equals $W^T W$ for
$W = \begin{pmatrix}U^T&  V\end{pmatrix}$. Notice that the entry $X(i,i)$, for 
$1 \le i \le m$, equals $\|U(i,*)\|_2^2$, and the entry $X(m + j,m+j)$  equals $\|V(*,j)\|_2^2$ for
$1 \le j \le n$.

Conversely, suppose that $X$
is a semidefinite $(m+n) \times (m+n)$ matrix such that the submatrix
formed by the first $m$ rows and last $n$ columns equals $\mat$. Since
$X$ is positive semidefinite, we can write it as $W^T W$ for some
$(m+n) \times (m+n)$ matrix $W$. Defining $U$ to be the transpose of
the submatrix of $W$ formed by taking its first $m$ columns, and $V$
to be the submatrix of $W$ formed by its last $n$ columns, we get a
factorization $\mat = UV$. Moreover,
$\max_{i = 1}^m X(i,i) = \|U\|_{2\to\infty}^2$, and
$\max_{j=1}^n X(m+j,m+j) = \|V\|_{1\to 2}^2$. 
This equivalence directly gives the following lemma.
\begin{lemma}\label{lm:ch7-fact-sdp}
  For any $m\times n$ matrix $\mat$ we have
  \begin{align}
    \gamma_2(\mat) = &\min t \label{eq:ch7-sdp-obj} \\
                 \text{s.t. \,\,\, }    &X(i,m+j) = X(m+j,i) =\mat(i,j) \ \ \forall i\in [m], j \in [n]
                    \label{eq:ch7-sdp-eq}\\
                     &X(i,i) \le t\ \ \forall i \in [m+n]\\
                     &X \succeq 0\label{eq:ch7-sdp-psd}
  \end{align}
  where $X$ ranges over symmetric $(m+n) \times (m+n)$
  real matrices. 
\end{lemma}

Semidefinite programs share two important properties with linear
programs: they are efficiently solvable, in the sense that an
approximately optimal solution can be found in polynomial time to any
degree of accuracy; they satisfy a duality theory, i.e., we can 
derive a dual maximization semidefinite program to a given
minimization semidefinite program, so that the optimal values of the
two programs are equal.\footnote{These statements are imprecise: in
  order to be efficiently solvable and satisfy strong duality,
  semidefinite programs need to be sufficiently ``nice''. This turns
  out not to be an issue for our application.} Then, using
Lemma~\ref{lm:ch7-fact-sdp}, we can show that
$\gamma_2(\mat)$ is efficiently computable, and has a dual
maximization formulation. These facts, whose proofs we omit, are
captured in the following lemma.


\begin{lemma}\label{lm:ch7-fact-dual}
  For any $m\times n$ matrix $\mat$, and any $\varepsilon > 0$, we can
  compute a factorization $\mat = UV$ so that $\|U\|_{2\to\infty}
  \|V\|_{1\to 2} \le (1+\varepsilon)\gamma_2(\mat)$ in time polynomial
  in $m$, $n$, the maximum bitsize of any entry of $\mat$, and in
  $\log(1/\varepsilon)$.

  Moreover, $\gamma_2(\mat)$ satisfies the dual formula
  \begin{equation}
    \label{eq:ch7-fact-dual}
    \gamma_2(\mat) =
    \max\{\|P\mat Q\|_{\tr}: P, Q \text{ diagonal }, \tr(P^2) = \tr(Q^2) = 1\}.
  \end{equation}
  Above, $\|P\mat Q\|_{\tr}$ is the trace norm of the matrix $P\mat Q$,
  i.e., the sum of its singular values. 
\end{lemma}

\subsection{Lower Bound}
To lower bound the hereditary discrepancy by $\gamma_2(\mat)$ up to logarithmic factors, we need a general tool for proving lower bounds on hereditary
discrepancy. 
We next define an intermediate quantity called the determinant lower bound $\detlb(\mat)$. We use the connection between rounding and hereditary discrepancy, discussed Chapter~\ref{ch1-intro}, and a covering argument, to show that $\detlb(\mat)$ is a lower bound on hereditary discrepancy. Finally, we relate  $\gamma_2(\mat)$ and $\detlb(\mat)$.

\paragraph{The Determinant Lower bound and Hereditary Discrepancy.}
The determinant lower bound
  of an $m\times n$ matrix $\mat$ is defined as
  \begin{equation}
    \label{eq:ch7-detlb-defn}
    \detlb(\mat) \eqdef \max_k \max_{B \in S_k(\mat)} |\det(B)|^{1/k},
  \end{equation}
  where $S_k(\mat)$ is the set of $k\times k$ submatrices of
  $\mat$.
We have the following lower bound on $\herdisc(\mat)$.
\begin{lemma}\label{lm:ch7-detlb}
  For any $m\times n$ matrix $\mat$ we have the inequality
  \[
    \herdisc(\mat) \ge \frac12 \detlb(\mat).
  \]
\end{lemma}
\begin{proof}
We prove a slightly weaker bound
$\herdisc(\mat) \ge  \detlb(\mat)/4$,
  and leave the tighter bound as an exercise.
  We show that for any $k\times k$ square matrix $B$,
  \(
  \herdisc(B) \ge (1/4) |\det(B)|^{1/k}.
  \)
  The claimed bound then follows directly as $\herdisc(\mat) \ge
  \max_k \max_{B \in S_k(\mat)} \herdisc(B)$.   
  
\smallskip
  It follows from the transference theorem (Theorem~\ref{thm:ch1-transference}) that 
  \[
    \forall y \in [0,1]^k \ \exists z \in \{0,1\}^k\ \ 
  \|B (y-z)\|_\infty \le \herdisc(B).
  \]
  Let $K \eqdef \{x \in \R^k: \|Bx\|_{\infty} \le
  \herdisc(B)\}$. The equation above is equivalent to
  \(
    [0,1]^k \subseteq \bigcup_{z \in \{0,1\}^k} (K + z).
  \)
  Letting $\vol^k(\cdot)$ be the $k$-dimensional Lebesgue
  measure, by its translation invariance we have
  \begin{align}
    1 &= \vol^k([0,1]^k) \le
    \sum_{z \in \{0,1\}^k} \vol^k(K + z) = 2^k \vol^k(K).\label{eq:ch7-covering}
  \end{align}
  To turn this into a lower bound on $\herdisc(B)$, we need to
  relate it to $\vol^k(K)$. Let us assume that $B$ is invertible: 
  otherwise, the lower bound
  \(
  \herdisc(B) \ge \frac14 |\det(B)|^{1/k}=0
  \)
  is trivial. Then we can write
  \begin{align*}
    K & = \{x \in \R^k: \|B x\|_{\infty} \le   \herdisc(B)\}\\
    & = \{B^{-1} y: y\in \R^k, \|y\|_{\infty} \le \herdisc(B)\}.
  \end{align*}
  The right side is the image of the cube
  $[-\herdisc(B),\herdisc(B)]^k$ under multiplication by $B^{-1}$, and
  has volume $
  (2\herdisc(B))^{k} |\det(B^{-1})|$ $=
  2^k\herdisc(B)^k/|\det(B)|$.
  
Together with \eqref{eq:ch7-covering}, 
 we get that \[
    1 \le 2^k\vol^k(K) = 4^k\herdisc(B)^k/|\det(B)|. \]
  Re-arranging gives us the lower bound. 
\end{proof}

\begin{remark}\label{rem:vollb}
  The determinant lower bound can be far from tight when $m$ is
  large. For example, take $\mat$ to be a $2^n \times n$ matrix whose
  rows are all the vectors in $\{-1, +1\}^n$. Hadamard's inequality shows that
  $\detlb(\mat) \le \sqrt{n}$, but it is easy to see that $\disc(\mat)
  = n$. Intuitively, the reason for this gap is that the determinant
  lower bound only considers square submatrices of $\mat$. The proof
  above, however, can be modified slightly to establish a
  stronger volumetric lower bound on hereditary discrepancy, as we
  explain next. 

  Notice that the proof of the inequality \eqref{eq:ch7-covering} did
  not use that $B$ is a square matrix, and in fact works for any
  submatrix $B$ of $\mat$. Using the inequality for
  submatrices that consist of a subset of the columns of $\mat$, we
  get that, for any $S \subseteq [n]$, the convex body
  $\{x \in \R^S: \|\mat(*,S) x\|_{\infty} \le \herdisc(\mat)\}$ has
  $|S|$-dimensional volume at least $2^{-|S|}$, or, equivalently,
  $\{x \in \R^S: \|\mat(*,S) x\|_{\infty} \le 1\}$, has volume at
  least $(2\herdisc(\mat))^{-|S|}$. Letting 
  $\overline{K} \eqdef \{x \in \R^n: \|\mat x\|_{\infty} \le 1\}$,
  this means that
  \begin{equation}
    \label{eq:vollb}
    \herdisc(\mat) \ge \frac12 \vollb(\mat) \eqdef
    \min_{S \subseteq [n]}\vol^{S}(\overline{K} \cap \R^S)^{-1/|S|},
  \end{equation}
  where $\vol^S(\cdot)$ is the $|S|$-dimensional Lebesgue measure
  taken inside $\R^S$, and we used the observation that
  \(
  \overline{K} \cap \R^S = \{x \in \R^S: \|\mat(*,S) x\|_{\infty} \le 1\}.
  \)
   This volumetric lower bound takes advantage of
  ``tall'' matrices, and can be stronger than the determinant lower
  bound. Returning to the $2^n \times n$ matrix $\mat$ above, whose rows are
  all vectors in $\{-1,+1\}^n$, we have
  $\overline{K} = \{x \in \R^n: \|x\|_1 \le 1\}$, i.e., the
  $n$-dimensional $\ell_1$ ball, which has Lebesgue measure
  $2^n/n!$. So $\vollb(\mat) = (n!)^{1/n}/2 \approx n/2e$ by
  Stirling's approximation. Unlike the determinant lower bound, this
  matches the discrepancy of $\mat$  up to a constant factor.
\end{remark}

\begin{exercise}
  Prove the inequality \(\herdisc(\mat) \ge (1/2) \detlb(\mat)\) in
  Lemma~\ref{lm:ch7-detlb}. To do so, argue that
  \eqref{eq:ch7-covering} can be strengthened to $\vol^k(K) \ge 1$.
\end{exercise}

\paragraph{Relating $\detlb(\mat)$ and $\gamma_2(\mat)$.}
Our goal will be to show that $\gamma_2(\mat)\lesssim \log(2n) \detlb(\mat)$. 

Let us first see how to use
Lemma~\ref{lm:ch7-fact-dual} to get an inequality in the other
direction. Let us take any $k\times k$ submatrix $\mat(S,T)$ of $\mat$ that
achieves $\detlb(\mat)$. Define $P(i,i) = \frac{1}{\sqrt{k}}$ for $i \in S$,
$Q(j,j) = \frac{1}{\sqrt{k}}$ for $j \in T$, and define all other entries of $P$
and $Q$ to be $0$. Then, by \eqref{eq:ch7-fact-dual},
\[
  \gamma_2(\mat) \ge \|P\mat Q\|_{\tr} = \frac1k \|\mat(S,T)\|_{\tr},
\]
and notice that $\frac1k \|\mat(S,T)\|_{\tr}$ is the arithmetic mean of
the singular values of $\mat(S,T)$. Since $|\det(\mat(S,T))|^{1/k}$ is
the geometric mean of the singular values of the same matrix, the
AM-GM inequality gives us that
$\gamma_2(\mat) \ge \detlb(\mat)$.  

To reverse this inequality, we
then need an inequality that reverses the AM-GM inequality, in the
appropriate sense. The next lemma is sufficient for us for this
purpose. 

\begin{lemma}\label{lm:ch7-rev-AMGM}
  For any non-negative real numbers $\sigma_1 \ge \ldots \ge
  \sigma_r \ge 0$,
  \[
    \max_{k = 1}^r k(\sigma_1 \cdots \sigma_k)^{1/k}
    \ge
    \max_{k = 1}^r k\sigma_k
    \ge
    \frac{\sigma_1 + \ldots + \sigma_r}{1+\ln r}
  \]
\end{lemma}
\begin{proof}
  The first inequality is trivial. The second
  follows as,
  \begin{align*}
    \sum_{k = 1}^r \sigma_k =
    \sum_{k = 1}^r k\sigma_k\cdot\frac{1}{k}
    &\le \left(\max_{k = 1}^r k\sigma_k\right)\left(\sum_{k=1}^r \frac{1}{k}\right)
    \le \left(\max_{k = 1}^r k\sigma_k\right)(1+\ln r),
  \end{align*}
  as $\sum_{k=1}^r \frac{1}{k} \le 1 + \ln r$.
\end{proof}

Let us also recall the Cauchy-Binet formula, which states that for any
two $m\times n$ matrices $A,B$, where $n \le m$, we have
\[
  \det(A^T B) = \sum_{S \subseteq [m]: |S| = n} \det(A(S,*))\det(B(S,*)).
\]
For a short proof, see Section 3.2 in Tao's
book~\cite{Tao-randmat}. Note that applying the Cauchy-Schwarz
inequality to the right hand side of the Cauchy-Binet formula, we get
\begin{align}
  \det(A^T B)^2 &\le
     \det(A^T A) \det(B^T B).  \label{eq:ch7-CS-mat}
\end{align}

We need one final lemma before we can relate $\gamma_2(\mat)$ and
$\detlb(\mat)$. This lemma  allows us to extract a submatrix of
$\mat$ to use as a certificate that $\detlb(\mat)$ is large. 

\begin{lemma}\label{lm:ch7-submatrix}
  Let $n\le m$ be positive integers, and let $D$ be an $m\times m$ diagonal matrix with non-negative entries such that $\tr(D) =
  1$. Then for any $m\times n$ matrix $A$, there exists a set
  $S\subseteq [m]$ of cardinality $n$ for which 
  \[
    |\det(A(S,*))| \ge  \sqrt{n!\det(A^TDA)}.
  \]
\end{lemma}
\begin{proof}
  By the Cauchy-Binet formula and H\"older's inequality, we have
  \begin{align*}
    \det(A^T DA) 
                 &= \sum_{S\subseteq [m]: |S| = n}  \det(A(S,*))^2 \left(\prod_{i \in S} D(i,i)\right)\\
                 &\le \left(\max_{S \subseteq [m]: |S| = n}\det(A(S,*))^2\right)
                   \left(\sum_{S\subseteq [m]: |S| = n} \prod_{i \in S} D(i,i)\right).
  \end{align*}
  It remains to bound the sum on the right hand side. One can show that this sum is maximized when all $D(i,i)$ are equal,
  in which case it equals ${m \choose n} \frac{1}{m^n} <
  \frac{1}{n!}$. For an elementary proof of a similar bound,
  notice that, since all $D(i,i)$ are non-negative,
  \[
    1 = (D(1,1) + \ldots + D(m,m))^n
    \ge n! \sum_{S\subseteq [m]: |S| = n} \prod_{i \in S} D(i,i),
  \]
  because when we expand $(D(1,1) + \ldots + D(m,m))^n$ we get each
  multilinear term $\left(\prod_{i \in S} D(i,i)\right)$ $n!$ times,
  in addition to the other terms. 
\end{proof}

We are, finally, ready to prove that $\gamma_2(\mat)$ and
$\detlb(\mat)$ are equal up to a logarithmic factor.

\begin{lemma}\label{lm:ch7-fact-detlb}
  For any $m\times n$ matrix $\mat$ of rank $r$, we have
  \[
    \detlb(\mat) \le \gamma_2(\mat) \lesssim (1 + \log r) \detlb(\mat).
  \]
\end{lemma}
\begin{proof}
  As we argued above, the first inequality follows by considering the
  $k\times k$ submatrix $\mat(S,T)$ of $\mat$ achieving
  $|\det(\mat(S,T))|^{1/k} = \detlb(\mat)$, and setting $P$ and $Q$ in
  \eqref{eq:ch7-fact-dual} so that $P\mat Q = \frac1k \mat(S,T)$.
  Here we prove the second inequality, again using
  Lemma~\ref{lm:ch7-fact-dual}, as well as
  Lemmas~\ref{lm:ch7-rev-AMGM}~and~\ref{lm:ch7-submatrix}.

  Let $P$ and $Q$ achieve equality in \eqref{eq:ch7-fact-dual}, and
  let $\sigma_1 \ge \ldots \ge \sigma_r \ge 0$ be the $r$ largest singular
  values of $P \mat Q$. Since $P \mat Q$ has rank at most $r$, all
  other singular values of $P \mat Q$ are $0$, and $\gamma_2(\mat) =
  \sigma_1 + \ldots + \sigma_r$. Let $k \in [r]$ be the integer for
  which Lemma~\ref{lm:ch7-rev-AMGM} guarantees that
  \begin{equation}\label{eq:ch7-top-k}
    k(\sigma_1 \cdots \sigma_k)^{1/k}
    \ge \frac{\sigma_1 + \ldots + \sigma_r}{1+\ln r}
    = \frac{\gamma_2(\mat)}{1+\ln r}.
  \end{equation}
  The goal in the rest of the proof is to show that there exists a
  $k\times k$ submatrix $\mat(S,T)$ of $\mat$ for which $|\det(\mat(S,T))|^{1/k}$ is not
  much smaller than $k(\sigma_1 \cdots \sigma_k)^{1/k}$.

  Let $v_1, \ldots, v_k$ be the right singular vectors of $P \mat Q$
  associated with $\sigma_1, \ldots, \sigma_k$, and let $V$ be a
  matrix with columns $v_1, \ldots, v_k$. Define the $m\times k$
  matrix $A \eqdef \mat QV$. The matrix $PA$ then has singular values
  $\sigma_1, \ldots, \sigma_k$, or, equivalently, $A^TP^2A$ has
  eigenvalues $\sigma_1^2, \ldots,
  \sigma_k^2$. Lemma~\ref{lm:ch7-submatrix} applied to $A$ and $D
  \eqdef P^2$ now tells us that there exists a set $S \subseteq [m]$
  of cardinality $k$ so that
  \begin{align}
    |\det(\mat(S,*)QV)| &= |\det(A(S,*))|\notag\\
    &\ge \sqrt{k! \det(A^TP^2A)} = \sqrt{k!} \sigma_1 \cdots \sigma_k.    \label{eq:ch7-rows}
  \end{align}

  Next we would like to apply Lemma~\ref{lm:ch7-submatrix} again, in
  order to replace $\mat(S,*)QV$ with $\mat(S,T)$ for some cardinality
  $k$ set $T \subseteq [n]$. We first need to deal with the matrix
  $V$, however. To this end, apply \eqref{eq:ch7-CS-mat}  to the
  matrices $Q\mat(S,*)^T$ and $V$, to conclude that
  \begin{align}
    \det(\mat(S,*)Q^2 \mat(S,*)^T) &\ge
    \frac{\det(\mat(S,*)QV)^2}{\det(V^T V)} \notag\\
                                   &= \det(\mat(S,*)QV)^2,\label{eq:ch7-proj}
  \end{align}
  where we used that $V^T V = I$. 
  Then 
  \eqref{eq:ch7-rows} and \eqref{eq:ch7-proj} together give us
  \[
    \sqrt{\det(\mat(S,*)Q^2\mat(S,*)^T)} \ge
    \sqrt{k!} \sigma_1 \cdots \sigma_k.
  \]
  Finally, we apply Lemma~\ref{lm:ch7-submatrix} one more time, now
  with the $n\times k$ matrix $A \eqdef \mat(S,*)^T$ and $D \eqdef
  Q^2$, concluding that there exists a set $T \subseteq [n]$ of
  cardinality $k$ for which
  \[
    |\mat(S,T)| \ge \sqrt{k!\det(\mat(S,*)Q^2\mat(S,*)^T)} \ge
    k! \sigma_1 \cdots \sigma_k.
  \]
  Together with \eqref{eq:ch7-top-k}, and as $k! \ge (k/e)^k$, we can conclude that
  \begin{align*}
    \detlb(\mat)
    \ge
    |\det(\mat(S,T))|^{1/k}
    &\gtrsim
      k (\sigma_1 \cdots \sigma_k)^{1/k}
    \ge
      \frac{\gamma_2(\mat)}{1+\ln r}.  \qedhere
  \end{align*}
\end{proof}

Lemma~\ref{lm:ch7-fact-detlb} gives an approximation algorithm for the
determinant lower bound: computing $\gamma_2(\mat)$ gives a factor
$O(\log r)$ approximation for $\detlb(\mat)$ for any rank $r$ matrix
$\mat$. No better polynomial time computable approximation for
$\detlb(\mat)$ is known. 

Combining Lemmas~\ref{lm:ch7-detlb}~and~\ref{lm:ch7-fact-detlb}, we get
the following lower bound on hereditary discrepancy.

\begin{theorem}\label{thm:ch7-fact-lb}
  For any $m\times n$ matrix $\mat$ of rank $r$, we have
  \[
    \herdisc(\mat) \gtrsim  \frac{\gamma_2(\mat)}{1 + \log r}.
  \]
\end{theorem}

\subsection{The Approximation Theorem}

We can now state our main result.
\begin{theorem}\label{thm:ch7-fact-main}
  For any $m\times n$ matrix $\mat$ of rank $r$, we have
  \[
    \frac{\gamma_2(\mat)}{1 + \log r} \lesssim
    \herdisc(\mat)
    \lesssim
    \sqrt{\ln(2m)} \gamma_2(\mat).
  \]
\end{theorem}
Theorem~\ref{thm:ch7-fact-main} follows immediately from
Corollary~\ref{cor:ch7-fact-ub}, and Theorem~\ref{thm:ch7-fact-lb}.

In Exercise~\ref{ex:ch7-fact-ub} we showed that the upper bound in
Theorem~\ref{thm:ch7-fact-main} is tight up to constant factors. Soon
we will see a simple example of a matrix for which the lower bound is
tight.

\section{Some Properties}

There are many natural questions one can ask about the behavior of
hereditary discrepancy under simple transformations. For example, we
can ask how $\herdisc(\SS \cup \SS')$ compares to $\herdisc(\SS)$ and
$\herdisc(\SS')$ for two set systems $\SS$ and $\SS'$ on the same
universe. As another example, we can ask how $\herdisc(\mat^T)$
compares to $\herdisc(\mat)$ for a matrix $\mat$. None of these
questions has an obvious answer, but, nevertheless, we can give nearly tight
answers to them using Theorem~\ref{thm:ch7-fact-main}, and simple
properties of the $\gamma_2$ norm. Let us remark that it is essential
here to work with hereditary discrepancy - discrepancy itself does not
behave nicely with respect to most simple transformations. Two
examples of this are offered in the next exercises.

\begin{exercise}
  Give an example of two set systems $\SS$ and $\SS'$, both on the ground
  set $[n]$ for an arbitrarily large integer $n$, so that $\disc(\SS)
  = \disc(\SS') = 0$ but $\disc(\SS \cup \SS') \gtrsim n$.
\end{exercise}


\begin{exercise}
  Give an example of a $n\times n$ matrix with $\{0,1\}$ entries,
  for $n$ arbitrarily large, so that $\disc(\mat) = 0$ but
  $\disc(\mat^T) \gtrsim \sqrt{n}$.
\end{exercise}


The next lemma summarizes the basic properties of the $\gamma_2$ norm
that are useful in reasoning about hereditary
discrepancy. Below, we use $\otimes$ for the Kronecker product: given
two matrices $A$ and $B$, of dimensions, respectively, $m \times n$
and $p \times q$, $A\otimes B$ is a $(mp) \times (nq)$ matrix whose
rows are indexed by $[m]\times [p]$, whose columns are indexed by $[n]\times
[q]$, and whose entries are $(A\otimes B)((i_1, i_2), (j_1, j_2)) =
A(i_1,j_1)B(i_2,j_2)$. The key property of Kronecker products that we
repeatedly use is that $(A\otimes B)(A'\otimes B') =
(AA')\otimes(BB')$ for any matrices $A,A',B,B'$ for which these matrix
products are well-defined.

\begin{lemma}\label{lm:ch7-gamma2-props}
  The $\gamma_2$ function on matrices has the following properties:
  \begin{enumerate}
  \item $\gamma_2(\mat(S,T)) \le \gamma_2(\mat)$ for any $m\times n$
    matrix $\mat$ and any submatrix $\mat(S,T)$ of
    it;\label{pr:ch7-submatrix}
  \item duplicating rows and columns of $\mat$  keeps $\gamma_2(\mat)$ unchanged;\label{pr:ch7-duplicate}
  \item $\gamma_2(\mat) = \gamma_2(\mat^T)$ for any matrix $\mat$;\label{pr:ch7-transpose}
  \item $\gamma_2(\mat) = 0 \implies \mat = 0$;\label{pr:ch7-pos}
  \item $\gamma_2(c\mat) = |c|\gamma_2(\mat)$ for any matrix $\mat$
    and real value $c$;\label{pr:ch7-hom}
  \item $\gamma_2(\mat + \mat') \le \gamma_2(\mat) + \gamma_2(\mat')$;\label{pr:ch7-triangle}
  \item for any two matrices $\mat'$ and $\mat''$, each with $n$ columns, and the matrix 
    $\mat = \begin{pmatrix}\mat'\\ \mat''\end{pmatrix}$, 
    we have $\gamma_2(\mat) \le \sqrt{\gamma_2(\mat')^2 +
      \gamma_2(\mat'')^2}$; \label{pr:ch7-union}
  \item $\gamma_2(\mat \otimes \mat') = \gamma_2(\mat)
    \gamma_2(\mat')$ for any matrices $\mat$ and $\mat'$.\label{pr:ch7-tensor}
  \end{enumerate}
\end{lemma}
\begin{proof}
  Property~\ref{pr:ch7-submatrix} was noted previously. Next we
  give the proofs of some of the remaining properties, and leave
  the rest as exercises.

  Property~\ref{pr:ch7-triangle} can be proved in several ways,
  including by directly giving a factorization of $\mat + \mat'$ using
  optimal factorizations of $\mat$ and $\mat'$. Instead, let us prove
  it using the SDP characterization from
  Lemma~\ref{lm:ch7-fact-sdp}. Let $X, t$ be a solution to the
  semidefinite program \eqref{eq:ch7-sdp-obj}--\eqref{eq:ch7-sdp-psd}
  for $\mat$ that achieves the optimal objective value
  $\gamma_2(\mat)$, and let $X', t'$ be a solution for $\mat'$ that
  achieves objective value $\gamma_2(\mat')$. Then $X + X', t + t'$ is a
  feasible solution for $\mat + \mat'$, and achieves objective value
  $\gamma_2(\mat) + \gamma_2(\mat')$, which is an upper bound $\gamma_2(\mat + \mat')$.

\smallskip
  While property~\ref{pr:ch7-union} can also be shown using the SDP
  characterization, we give a direct
  factorization of $\mat = \begin{pmatrix}\mat'\\ \mat''\end{pmatrix}$
  instead. Let $\mat' = U' V'$ and $\mat'' = U'' V''$ be
  factorizations achieving, respectively,
  $\gamma_2(\mat') = \|U'\|_{2\to\infty} \|V'\|_{1\to 2}$, and
  $\gamma_2(\mat'') = \|U''\|_{2\to\infty} \|V''\|_{1\to
    2}$. 
    Rescale, so that
  $\|U'\|_{2\to\infty} = \|U''\|_{2\to\infty} = 1$,
  $\|V'\|_{1\to 2} = \gamma_2(\mat')$, and
  $\|V''\|_{1\to 2} = \gamma_2(\mat'')$.
  Consider then the factorization $\mat = UV$ where
  \[
    U =
    \begin{pmatrix}
      U' &0\\
      0  &U''
    \end{pmatrix}; \ \ \ \ 
    V =
    \begin{pmatrix}
      V'\\
      V''
    \end{pmatrix}.
  \]
  Then, clearly,
  \(
  \|U\|_{2 \to \infty}= 1.
  \)
  Moreover, as
  \(\|V(*,i)\|_2^2 = \|V'(*,i)\|_2^2 + \|V''(*,i)\|_2^2\)
  for any $i \in [n]$, we have
  \begin{align*}
    \|V\|_{1\to 2}^2 &= \max_{i = 1}^n \left( \|V'(*,i)\|_2^2 + \|V''(*,i)\|_2^2 \right)\\
    &\le
      \max_{i = 1}^n  \|V'(*,i)\|_2^2 + \max_{i=1}^n\|V''(*,i)\|_2^2\\
    &=
    \|V'\|_{1\to 2}^2 + \|V''\|_{1\to 2}^2.
  \end{align*}
  Thus,
  \(
  \gamma_2(\mat)^2 \le \|U\|^2_{2\to\infty} \|V\|^2_{1\to 2}
  \le
  \gamma_2(\mat')^2 + \gamma_2(\mat'')^2,
  \)
  as claimed.
  
\smallskip
  Finally, we prove property \ref{pr:ch7-tensor}.  
  To
  show
  $\gamma_2(\mat \otimes \mat') \le \gamma_2(\mat) \gamma_2(\mat')$ we
  can, once again, either use the SDP characterization, or give a
  direct factorization. Using the SDP characterization and Lemma~\ref{lm:ch7-fact-sdp} let $X, t$ be
  a solution to the semidefinite program
  \eqref{eq:ch7-sdp-obj}--\eqref{eq:ch7-sdp-psd} for $\mat$ achieving
  value $\gamma_2(\mat)$, and let $X', t'$ be
  a solution for $\mat'$ achieving value $\gamma_2(\mat')$. Then $X\otimes X', tt'$ is a
  feasible solution for $\mat\otimes \mat'$ achieving value
  $\gamma_2(\mat)\gamma_2(\mat')$. Here we use the fact that the
  Kronecker product of two positive semidefinite matrices is positive
  semidefinite. 
  
  To show the reverse inequality $\gamma_2(\mat\otimes
  \mat') \ge \gamma_2(\mat)\gamma_2(\mat')$, we use
  Lemma~\ref{lm:ch7-fact-dual}. Let $P,Q$ achieve equality in
  \eqref{eq:ch7-fact-dual} for $\mat$, let $P', Q'$ achieve
  equality in \eqref{eq:ch7-fact-dual} for $\mat'$. Then,
  \begin{align*}
    \gamma_2(\mat\otimes \mat') &\ge
    \|(P\otimes P') (\mat \otimes \mat') (Q\otimes Q')\|_{\tr}\\
    &=
    \|(P\mat Q)\otimes (P'\mat Q')\|_{\tr}
    =
   \|P\mat Q\|_{\tr}\|P'\mat Q'\|_{\tr}.
 \end{align*}
 Above we used the fact that, for any matrices $A$ and $B$,
 $\|A\otimes B\|_{\tr} = \|A\|_{\tr}\|B\|_{\tr}$, which can be verified
 easily using the singular value decomposition. 
\end{proof}

\begin{exercise}
  Prove the remaining properties
  \ref{pr:ch7-transpose}--\ref{pr:ch7-hom} in Lemma~\ref{lm:ch7-gamma2-props}.
\end{exercise}

Note that properties \ref{pr:ch7-pos}--\ref{pr:ch7-triangle} in
Lemma~\ref{lm:ch7-gamma2-props} show that $\gamma_2$ is, indeed, a
norm on matrices.

Lemma~\ref{lm:ch7-gamma2-props} and Theorem~\ref{thm:ch7-fact-main}
show that properties~\ref{pr:ch7-submatrix}--\ref{pr:ch7-tensor} carry
over from the $\gamma_2$ norm to hereditary discrepancy, up to
logarithmic factors. Properties \ref{pr:ch7-submatrix},
\ref{pr:ch7-pos}, and \ref{pr:ch7-hom} hold trivially for
hereditary discrepancy even without the extra logarithmic factors, but
we know no direct proof for the others. The next theorem rephrases
some of these properties for set systems.

\begin{theorem}\label{thm:ch7-props-ss}
  The following bounds hold:
  \begin{enumerate}
  \item \label{pr:ch7-ss-dual} For a set system $(\SS, \uni)$, define the dual set system
    $\SS^*$ on the universe $\SS$ by $\{S_\elem: \elem \in \uni\}$ where $S_\elem = \{S \in
    \SS: \elem \in S\}$; then 
    \(
    \herdisc(\SS^*) \lesssim  \log(2|\uni|)^{3/2}\herdisc(\SS),
    \)
  \item \label{pr:ch7-ss-union} For any set systems $\SS_1, \ldots, \SS_k$ defined on the same
    universe $\uni$, 
 \[
    \herdisc(\SS_1 \cup \cdots \cup \SS_k)
    \lesssim
    \log(2m)^{3/2} \left(\sum_{i=1}^k \herdisc(\SS_i)^2\right)^{1/2},\]
    where $m = \sum_{i=1}^k|\SS_i|$.
    
  \item \label{pr:ch7-ss-product}  for set systems $(\SS_1, \uni_1), \ldots, (\SS_k, \uni_k)$,
    let $\SS_1 \times \cdots \times \SS_k$ consist of all product sets $S_1
    \times \cdots \times S_k$ with $S_i \in \SS_i$; then
    \begin{multline*}
      \log(2m)^{ - \frac{k}{2}-1}
      \lesssim
      \frac{\herdisc(\SS_1 \times \cdots \times \SS_k)}{\herdisc(\SS_1) \cdots \herdisc(\SS_k)}
      \lesssim
      \log(2m)^{k + \frac12},
    \end{multline*}
    where $m = |\SS_1|\cdots |\SS_k|$.
  \end{enumerate}
\end{theorem}
\begin{proof}
  Each of the bounds follows from Theorem~\ref{thm:ch7-fact-main}, and
  respectively, properties~\ref{pr:ch7-transpose}, \ref{pr:ch7-union}, and
  \ref{pr:ch7-tensor} of Lemma~\ref{lm:ch7-gamma2-props}. 
\end{proof}

\section{Applications}

In this section we outline several applications of
Theorem~\ref{thm:ch7-fact-main} to proving upper and lower bounds on
the discrepancy of set systems.

\subsection{The $\gamma_2$ Norm of Intervals}

We start with a set system that, at first sight appears to be quite
uninteresting: the set system $\II_n$ of prefix intervals  of $[n]$, i.e., sets
of the type $\{1, \ldots i\}$. The hereditary discrepancy of $\II_n$
is $1$: we just need to alternate the signs the elements we color to
achieve this bound. It turns out however, that the $\gamma_2$ norm of
the incidence matrix of $\II_n$ is on the order of $\log
n$. This fact, which we prove below, shows that the lower bound in 
Theorem~\ref{thm:ch7-fact-lb} is tight up to constants. It is also
useful in the applications we develop below.

We use $T_n$ for the incidence matrix of $\II_n$, which we
write as the lower triangular matrix
\[
T_n(i,j) =
\begin{cases}
  1 &i \ge j\\
  0 &i < j
\end{cases}.
\]

\begin{lemma}\label{lm:ch7-gamma2-Tn}
  For the matrix $T_n$ above, $\gamma_2(T_n) = \Theta(\log n)$. 
\end{lemma}
\begin{proof}
  The upper bound follows from Theorem~\ref{thm:ch7-fact-lb}, since
  \[
    \gamma_2(T_n) \lesssim (1 + \log n) \herdisc(T_n) = 1 + \log n.
  \]
  For the lower bound $\gamma_2(T_n) \gtrsim 1+\log n$, we use
  Lemma~\ref{lm:ch7-fact-dual}. In particular, setting $P = Q =
  \frac{1}{\sqrt{n}}$ in \eqref{eq:ch7-fact-dual}, we see that
  \(
    \gamma_2(T_n) \ge \frac1n \|T_n\|_{\tr}.
  \)
  To prove a lower bound on $\|T_n\|_{\tr}$, we need to estimate the
  singular values of $T_n$. This can be done directly by computing the
  eigenvectors of $T_n^T T_n$, which have a simple closed
  form. Instead, we will consider a matrix that is easier to
  diagonalize. Consider the matrix $S_n = 2T_n - J_n$, where $J_n$ is
  the $n\times n$ all-ones matrix. In other words, $S_n$ is the matrix
  we get from $T_n$ by replacing all entries equal to $0$ with
  $-1$. By the triangle inequality,
  \begin{equation}\label{eq:ch7-Tn-lb}
    \gamma_2(T_n) \ge \frac1n \|T_n\|_{\tr}
    \ge \frac{1}{2n} (\|S_n\|_{\tr} - \|J_n\|_{\tr})
    = \frac{1}{2n} \|S_n\|_{\tr} - \frac12,
  \end{equation}
  since $J_n$ has only one nonzero singular value equal to $n$. Notice
  that $S_n$ is a normal matrix, i.e., $S_n^T S_n = S_nS_n^T$, and,
  therefore, the singular values of $S_n$ are equal to the absolute
  values of the eigenvalues of $S_n$. To compute $\|S_n\|_{\tr}$, it
  then suffices to compute its eigenvalues. We can diagonalize $S_n$
  over the complex numbers, using a transformation similar to the
  discrete Fourier transform. Let $\omega \in \mathbb{C}$ be a primitive $2n$-th
  root of unity, and, for $k \in \{0, \ldots, n-1\}$,  define the
  vector $u_k$ by $u_k(i) = \omega^{(2k + 1)(i-1)}$. What motivates
  this choice is that $\omega^{(2k + 1)n} = -1$. It is easy to
  verify that $u_k$ and $u_\ell$ are orthogonal (with the standard
  inner product over $\mathbb{C}^n$) if $k \neq \ell$. Moreover,
  notice that
  \begin{align*}
    (S_nu_k)(i)
    &=
      (1 + \ldots + \omega^{(2k + 1)(i-1)})
      - (\omega^{(2k + 1)i} +\ldots +  \omega^{(2k + 1)(n-1)})\\
    &= \frac{\omega^{(2k + 1)i} - 1}{\omega^{2k + 1} - 1}
      - \frac{\omega^{(2k + 1)n} - \omega^{(2k + 1)i}}{\omega^{2k + 1} - 1}\\
    &= \frac{2\omega^{(2k + 1)i}}{\omega^{2k + 1}-1}
      = \frac{2\omega^{2k + 1}}{\omega^{2k + 1}-1}u_k(i).
  \end{align*}
  Therefore, $u_k$ is an eigenvector of $S_n$ with eigenvalue
  $\frac{2\omega^{2k + 1}}{\omega^{2k + 1}-1}$. Thus, the
  singular values of $S_n$ are $2/|\omega - 1|,
  \ldots, 2/|\omega^{2n-1} - 1|$. Using elementary
  trigonometry,
  \(
  |\omega^{2k + 1} - 1| = 2\sin\left(\frac{\pi(2k+1)}{2n}\right),
  \)
  and, therefore,
  \[
    \|S_n\|_{\tr} = \sum_{k = 0}^{n-1} \frac{1}{\sin\left(\frac{\pi(2k+1)}{2n}\right)}
    \ge
    \frac{2n}{\pi}\sum_{k = 0}^{n-1} \frac{1}{2k+1},
  \]
  as $\sin x \le x$ for $x \ge 0$.
  Plugging back into \eqref{eq:ch7-Tn-lb}, we get that
  \[
    \gamma_2(T_n) \ge \frac{1}{\pi}\sum_{k = 0}^{n-1} \frac{1}{2k+1}
    \gtrsim 1 + \log n. \qedhere
  \]
\end{proof}

\begin{exercise}
  Using canonical intervals (see Section~\ref{sec:ch2-tusnady}), give a factorization
  showing that $\gamma_2(T_n) \lesssim 1 + \log n$. This gives a more
  direct proof of the upper bound in Lemma~\ref{lm:ch7-gamma2-Tn}.
\end{exercise}

\subsection{The Discrepancy of Boxes}

Let us now use what we have proved to show nearly tight bounds on the
discrepancy of boxes in $d$-dimensions, resolving
Problem~\ref{prob:tusnady} up to a fixed power of log.

Recall that
$\RR_d$ is the set of all axis aligned boxes in $\R^d$, i.e., sets of
the type $\rbox_{a,b} = \{u \in \R^d: a \le u \le b\}$ for $a,b \in
\R^d$. Recall also that $\disc(\RR_d, n)$ is the largest discrepancy
achieved by restricting $\RR_d$ to a set $P$ of $n$
points. Theorem~\ref{thm:ch7-fact-main}, Lemma~\ref{lm:ch7-gamma2-Tn},
and the properties of the $\gamma_2$ norm allow us to prove the
following bounds on $\disc(\RR_d,n)$.

\begin{theorem}\label{thm:ch7-tusnady}
  For any integer $d \ge 1$,
  \[
    \log(2n)^{d-1} \lesssim_d \disc(\RR_d,n) \lesssim_d \log(2n)^{d+\frac12},
  \]
  where the $\lesssim_d$ means that the constant
  depends on $d$.
\end{theorem}
\begin{proof}
  As we have done before with boxes, we will use the fact that we only
  need to consider anchored boxes, restricted to the grid
  $[n]^d$. Recall that $\AA_d$ is the set system of anchored boxes of the type $[0,b_1]\times\cdots \times [0,b_d]$ for $b \in \R^d$, and consider
  the set system $\SS \eqdef \AA_d|_{[n]^d}$ of subsets of the
  $n\times \cdots \times n$ grid induced by anchored boxes. As we observed in Section~\ref{sec:ch2-tusnady},
  \begin{equation}\label{eq:ch7-tusnady-grid}
   2^{-d}\disc(\RR_d,n) \le \herdisc(\SS) \le \disc(\RR_d,n^d).
 \end{equation}
  Thus, it suffices to bound $\herdisc(\SS)$ from
  both sides, and we do so by computing the $\gamma_2$ norm of the
  incidence matrix of $\SS$.

  Let then $\mat$ be the incidence matrix of $\SS$. The main
  observation is that $\SS = \II_n^{\times d}$, i.e., the product of
  $\II_n$ with itself $d$ times. Therefore,  $\mat =
  T_n^{\otimes d}$, i.e., the Kronecker
  product of $T_n$ with itself $d$
  times. Lemma~\ref{lm:ch7-gamma2-props} then shows that
  \(
  \gamma_2(\mat) = \gamma_2(T_n^{\otimes d}) = \gamma_2(T_n)^d,
  \)
  so, by Lemma~\ref{lm:ch7-gamma2-Tn}, 
  \begin{equation}\label{eq:ch7-gamma2-tusnady}
    \log(2n)^d \lesssim_d \gamma_2(\mat) \lesssim_d \log(2n)^d.
  \end{equation}
  Theorem~\ref{thm:ch7-fact-main} now implies
  \[
    \log(2n)^{d-1} \lesssim_d \herdisc(\SS) \lesssim_d \log(2n)^{d+\frac12}
  \]
  since $|\SS| \le n^d$. This implies the theorem via
  \eqref{eq:ch7-tusnady-grid}.
\end{proof}

Below we will see that the upper bound in
Theorem~\ref{thm:ch7-tusnady} can be improved further. The lower bound
is the best one known for Tusn\'ady's problem.

\subsection{Other Applications}

Let us note a couple of other discrepancy upper bounds that can be
proved using Lemma~\ref{lm:ch7-gamma2-props},
Theorem~\ref{thm:ch7-props-ss}, and
Theorem~\ref{thm:ch7-fact-main}. Neither of the two bounds below is
optimal, but they are both non-trivial and follow easily from the
results proved thus far.

\paragraph{Permutation families.} The set system induced by a single
permutation of $[n]$ clearly has hereditary discrepancy $1$, as we can list the
elements of $[n]$ (or a subset of it) in the order given by the permutation, and alternate
the signs of the coloring. Then, by property~\ref{pr:ch7-ss-union}.~of
Theorem~\ref{thm:ch7-props-ss}, we have that, for a set system $\SS$
induced by $k$ permutations of $[n]$, $\disc(\SS) \lesssim \sqrt{k} \log(kn)^{3/2}$.

\paragraph{Arithmetic Progressions.} Consider the system $\mathcal{AP}$
consisting of all arithmetic progressions restricted to $[n]$. A
classical result of Roth shows that $\disc(\mathcal{AP}) \gtrsim n^{1/4}$. We
can show an upper bound that matches Roth's lower bound up to
logarithmic factors. Let $\mathcal{AP}_d$ be the subset of $\mathcal{AP}$ consisting of
arithmetic progressions of difference $d$, i.e., sets of the form $a +
d, a+2d, \ldots, a + \ell d$ for some integers $a$ and $\ell$. Let
$\mathcal{AP}_{\ge s} \eqdef \bigcup_{d \ge s} \mathcal{AP}_{d}$ for a parameter $s$ we will
choose soon. Since
every arithmetic progression in $\mathcal{AP}_{\ge s}$ has difference at least
$s$, it can have size at most $\frac{n}{s}$, so, for the incidence
matrix $\mat_{\ge s}$ of $\mathcal{AP}_{\ge s}$, $\gamma_2(\mat_{\ge s}) \le
\sqrt{\frac{n}{s}}$. On the other hand, the hereditary discrepancy of
$\mathcal{AP}_d$ for any $d$ is at most $1$, since we can separately color each
congruence class $\mod d$, and alternate signs of the coloring within
a class. Therefore, by Theorem~\ref{thm:ch7-fact-main},
$\gamma_2(\mat_{d}) \lesssim \log(2n)$ for the incidence matrix
$\mat_d$ of $\mathcal{AP}_d$. Using property~\ref{pr:ch7-union} of
Lemma~\ref{lm:ch7-gamma2-props} we have, for the incidence matrix
$\mat$ of $\mathcal{AP}$,
\begin{align*}
  \gamma_2(\mat) &\le (\gamma_2(\mat_{\ge s})^2 +
  \gamma_2(\mat_{s-1})^2 + \cdots + \gamma_2(\mat_1)^2)^{1/2}\\
  &\lesssim \left(\frac{n}{s} + s\log(2n)^2\right)^{1/2}.
\end{align*}
Optimizing over $s$, we have $\gamma_2(\mat) \lesssim
n^{1/4}\sqrt{\log(2n)}$, and Theorem~\ref{thm:ch7-fact-main} gives us 
\(
\disc(\mathcal{AP}) \lesssim n^{1/4}\log(2n).
\)

\begin{exercise}
  Use canonical intervals to show that $\gamma_2(\mat)\lesssim
  n^{1/4}$ for the incidence matrix $\mat$ of the set system of
  arithmetic progressions restricted to $[n]$.
\end{exercise}
\section{Prefix Discrepancy}

Recall that in Lemma~\ref{lm:fact-ub-prefix} we showed that, for any
matrix $\mat$ with columns $u_1, \ldots, u_n \in \R^m$, the prefix discrepancy
\begin{equation}\label{eq:ch7-ub-prefix}
  \pdisc(\mat) \lesssim \gamma_2(\mat)\sqrt{\log(2mn)}.
\end{equation}
Theorem~\ref{thm:ch7-fact-main} and this inequality show that prefix
discrepancy is never much bigger than the hereditary discrepancy,
since, for any $m\times n$ matrix $\mat$,
\begin{equation}
  \label{eq:ch7-pdisc-herdisc}
  \pdisc(\mat) \lesssim \log(mn)^{1.5} \herdisc(\mat).
\end{equation}

In addition to this connection between prefix and hereditary
discrepancy, the upper bound~\eqref{eq:ch7-ub-prefix} sometimes allows us to
prove tighter discrepancy upper bounds than
Theorem~\ref{thm:ch7-fact-main}. We give two such examples below, the
first giving the best known  upper bound for Tusn\'ady's problem.

\medskip
\noindent {\bf Axis-aligned boxes.} As we will see in the next theorem, the
bound \eqref{eq:ch7-pdisc-herdisc} allows us to improve the upper
bound in Theorem~\ref{thm:ch7-tusnady} by effectively getting one
dimension ``for free''.

\begin{theorem}\label{thm:ch7-tusnady-ub}
  For any integer $d \ge 1$,
  \(
    \disc(\RR_d,n) \lesssim_d \log(2n)^{d-\frac12}.
  \)
\end{theorem}
\begin{proof}
  As in Theorem~\ref{thm:ch7-tusnady}, we only consider anchored
  boxes, restricted to sets in $[n]^d$. Equation
  \eqref{eq:ch7-pdisc-herdisc} allows us to restrict the sets we
  consider even more. Let $P \subseteq [n]^d$ be some $n$-point set,
  and let $\SS$ be the set system of all subsets of
  $P$ of the form $R \cap P$ where
  \(
  R = [0,b(1)] \times\cdots \times [0,b({d-1})]\times \R,
  \)
  for $b\in \R^{d-1}$. That is, $\SS$ consists of sets
  induced by anchored boxes that extend to infinity in the last
  coordinate. Equivalently, let $\pi_{d-1}$ be the orthogonal
  projection onto the first $d-1$ coordinates of $\R^d$; then $\SS$ is
  the set system consisting of all sets of the type $\{p: \pi_{d-1}(p)
  \in R\}$ where $R \in \AA_{d-1}$ is a $(d-1)$-dimensional anchored
  box. Since $\SS$ is isomorphic to $\AA_{d-1}|_{\pi_{d-1}(P)}$, equation \eqref{eq:ch7-gamma2-tusnady} in the proof
  of Theorem~\ref{thm:ch7-tusnady} gives us that the incidence matrix
  $\mat$ of $\SS$ satisfies $\gamma_2(\mat)\lesssim_d
  \log(2n)^{d-1}$.

  Let us order $P$ as $p_1, \ldots, p_n$ so that $p_1(d) \le \ldots
  \le p_n(d)$, and also order the columns of the incidence matrix $\mat$ of $\SS$
  so that the $i$-th column is the indicator vector of sets containing $p_i$.
  The key observation we make is that, for any anchored box $R \in
  \AA_d$, $R \cap P = S \cap \{p_1, \ldots, p_i\}$ for some $S \in
  \SS$ and $i \in [n]$. In particular, we can take $p_i$ to be the
  point with the largest $p_i(d)$ coordinate contained in $R$, and then
  take $S = \{p \in P: \pi_{d-1}(p) \in \pi_{d-1}(R)\}$. Therefore, for any
  coloring $\col:P \to \{-1,+1\}$, $|\col(R \cap P)|$ is bounded by
  $\pdisc(\mat, \col)$. The theorem now follows from
  \eqref{eq:ch7-ub-prefix} since
  \[
    \disc(\AA_d|_P) \le \pdisc(\mat)
    \lesssim_d \sqrt{\log(2n)} \gamma_2(\mat)
    \lesssim_d \log(2n)^{d-\frac12},
  \]
  where we used the fact that $\AA_d|_P$ consists of at most $n^d$
  distinct sets, as well as the estimate
  \(\gamma_2(\mat)\lesssim_d  \log(2n)^{d-1}\) above.
\end{proof}

\paragraph{Arithmetic Progressions.} Recall that $\mathcal{AP}$ is the
set system of arithmetic progressions restricted to $[n]$. We can use
equation~\eqref{eq:ch7-pdisc-herdisc} to improve our upper bound on
its discrepancy to $\disc(\mathcal{AP}) \lesssim n^{1/4}\sqrt{\log
  n}$. To see this, let us define $\SS$ to be the set of
\emph{maximal} arithmetic progressions, \ie, progressions inside $[n]$
that cannot be extended further to the left or to the right. We claim
that the incidence matrix $\mat$ of $\SS$ satisfies $\gamma_2(\mat)
\lesssim n^{1/4}$. To see this, notice that
\(
\mat =
\begin{pmatrix}
  \mat_{\ge \sqrt{n}}\\
  \mat_{< \sqrt{n}}
\end{pmatrix}
\), where $\mat_{\ge \sqrt{n}}$ and $\mat_{< \sqrt{n}}$ are the incidence matrix of maximal
arithmetic progressions of difference at least $\sqrt{n}$, and at most $\sqrt{n}$  respectively. 

Recall that each arithmetic progression in $[n]$ of difference $d$ has
cardinality at most $n/d$, and all maximal arithmetic
progressions of the same difference are disjoint. Thus, it follows that $\gamma_2(\mat_{\ge \sqrt{n}}) \le n^{1/4}$ as each of its rows has at
most $\sqrt{n}$ ones, and $\gamma_2(\mat_{<\sqrt{n}}) \le n^{1/4}$
as each of its columns has at most $\sqrt{n}$ ones, one for each
possible difference. So  property~\ref{pr:ch7-union} of
Lemma~\ref{lm:ch7-gamma2-props} gives us that $\gamma_2(\mat)\le
\sqrt{2} n^{1/4}$.  Now equation
\eqref{eq:ch7-ub-prefix} gives us that
\(
\pdisc(\mat) \lesssim n^{1/4} \sqrt{\log n},
\)
where the colums of $\mat$ are ordered so that the $i$-th column is
the indicator of the sets in $\SS$ containing $i$. This means that
there exists a coloring $\col:[n]\to \{-1,+1\}$ such that $|\col(S
\cap [j])| \lesssim n^{1/4} \sqrt{\log n}$ for each $j \in [n]$ and $S
\in \SS$. Since any arithmetic progression inside $[n]$ can be written
as $(S \cap [k]) \setminus (S \cap [j])$, we also have
$\disc(\mathcal{AP},\col) \lesssim n^{1/4}\sqrt{\log n}$.

\section{Bibliographic Notes}

The computational complexity of discrepancy was first investigated by
Charikar, Newman and Nikolov, who proved
Theorem~\ref{thm:ch7-disc-hard}, and analogous results for set systems
of bounded degree, and for set systems of bounded VC
dimension~\cite{dischard}. It was shown by Austrin, Guruswami, and
H{\aa}stad that hereditary discrepancy is $\mathsf{NP}$-hard to
approximate within factor $2-\varepsilon$ for any
$\varepsilon > 0$~\cite{AustrinGH13}. 

The $\gamma_2$ norm was first introduced in Banach space theory as a
way to measure how well a bounded linear operator $u:X \to Y$ between
Banach spaces $X$ and $Y$ factors through a Hilbert space. In
particular, $\gamma_2(u)$ is the infimum of the product of operator
norms $\|v\|\|w\|$ over all operators $v:X \to H$ and $w:H \to Y$ such
that $wv = u$ and $H$ is some Hilbert space. Analogously, if we
replace the Hilbert space with an $L_p$ space over some measure space
$(\Omega,\mu)$, we get the $\gamma_p$ norm. The $\gamma_2$ norm we
define here for matrices corresponds to the special case when $X =
\ell_1^n$ and $Y = \ell_\infty^m$.
Factorization of linear operators through
$L_p$ spaces was studied systematically by Kwapien~\cite{Kwapien72},
but has its origins already in the work of
Grothendieck~\cite{Grothendieck52}. In particular, the dual formulation
\eqref{eq:ch7-fact-dual} can be derived from the results of
Grothendieck (up to a factor of 2) and Kwapien; see also the appendix
of Pisier's survey on Grothendieck's inequality~\cite{Pisier-grothendieck}. The book by
Tomczak-Jaegermann surveys the basic facts and many
applications of the $\gamma_p$ norms in Banach space theory~\cite{Tomczak-Jaegermann}.

The connection between hereditary discrepancy and the $\gamma_2$ norm
is due to Matou\v{s}ek, Nikolov, and Talwar, who also proved
Theorems~\ref{thm:ch7-fact-main}~and~\ref{thm:ch7-tusnady}~\cite{MNT}. The
semidefinite formulation in Lemma~\ref{lm:ch7-fact-sdp}, and the
computability and duality results in Lemma~\ref{lm:ch7-fact-dual} are
due to Lee, Shraibman, and \v{S}palek~\cite{LeeSS08}. 

The $\gamma_2$ norm has also found other applications, \eg, to
communication complexity and learning
theory~\cite{LinialMSS07-signmatrices} (see also the
book~\cite{LeeS09}), and differential privacy~\cite{NTZ}. It is also
implicit in the work of Larsen on lower bounds for oblivious range
searching data structure~\cite{disc-larsen}, who also discovered the discrepancy upper
bounds stated in Lemma~\ref{lm:fact-ub} and
Corollary~\ref{cor:ch7-fact-ub}.

The determinant lower bound in Lemma~\ref{lm:ch7-detlb} was first
shown by Lov\'asz, Spencer, and Vesztergombi~\cite{LSV}. They also asked
if hereditary discrepancy can be bounded from above by a function of
the determinant lower bound. This was disproved by Hoffman, and his
counterexample can be found in Chapter~4 of Matou\v{s}ek's
book~\cite{matousek2010geometric}. Nevertheless, it is possible to
show that $\herdisc(\mat) \lesssim
\detlb(\mat)\sqrt{\log(2m)\log(2n)}$. The first result of this type
was proved by Matou\v{s}ek~\cite{Matousek-detlb}, and his (loose)
bound was strengthened by Jiang and
Reis~\cite{JiangReis22}.  Li and Nikolov showed this bound is
tight~\cite{LiNikolov24}. 
The volume
lower bound in Remark~\ref{rem:vollb} was observed by
Banaszczyk~\cite{Bana93}. It was further investigated by Dadush, Nikolov,
Talwar, and Tomczak-Jaegermann, who showed it matches the hereditary
discrepancy up to logarithmic factors for any finite dimensional norm,
and matches a hereditary notion of partial coloring discrepancy up to
constant factors~\cite{anynorm}.

The lower bound on $\gamma_2(T_n)$ in Lemma~\ref{lm:ch7-gamma2-Tn} is due to Mathias~\cite{Mathias}. He also showed that $\gamma_2(S_n) = \frac{1}{n} \|S_n\|_{\tr}$, and, therefore, $\gamma_2(T_n) \le \frac{1}{2n}\|S_n\| + \frac12$. Together with \eqref{eq:ch7-Tn-lb}, this gives us the value of $\gamma_2(T_n)$ up an additive approximation of $1$. Determining the exact value of $\gamma_2(T_n)$ as a function of $n$, and giving an optimal factorization remain intriguing open problems, with applications in differential privacy~\cite{DP-FTRL}.

The improved upper bound on the discrepancy of boxes in
Theorem~\ref{thm:ch7-tusnady-ub} is due to
Nikolov~\cite{tusnady-ub}. Essentially the same proof, together with a
decomposition theorem of Matou\v{s}ek~\cite{Matousek99-boxes}, gives similar upper bounds for
more general set systems induced by polytopes with prescribed facets. 
Moreover, the lower bound in Theorem~\ref{thm:ch7-tusnady}  also
extends to these more general set systems, under some technical conditions. These
results are also due to Nikolov~\cite{tusnady-ub}. 

The discrepancy of arithmetic progressions is bounded by $O(n^{1/4})$,
and this was shown by Matou\v{s}ek and Spencer using the partial
coloring method~\cite{MatousekSpencer-ap}: see the notes to
Chapter~\ref{ch:partial} for more on the discrepancy of arithmetic
progressions.




\chapter{Algorithmic Banaszczyk's bound for Koml\'os}
\label{ch:algo-komlos-bana}

\chaptermark{Algorithmic Koml\'os}

Let $\mat$ be an instance of the Koml\'{o}s problem (Problem~\ref{prob:komlos}), i.e., a matrix all of whose columns have $\ell_2$ norm at most $1$.
Recall that 
applying Banaszczyk's Theorem \ref{thm:bana} with $K = c [-1,1]^m$ with $c \approx \sqrt{\log 2m}$ so that $\gamma^m(K)\geq 1/2$, directly implies that $\disc(\mat) \lesssim \sqrt{\log 2m})$. However, the proof of Theorem~\ref{thm:bana} in Chapter~\ref{ch:bana-proof} does not give an efficient algorithm to find a coloring satisfying this bound.

We now describe an efficient algorithm to compute such a coloring.
This algorithm is based on using a discrete Brownian motion random walk process, similar to those in Chapter~\ref{ch:partial}, but using a stronger semidefinite program to guide the random walk. 
Later, in Chapter~\ref{ch:gram-schmidt}, we will see yet another algorithm called the Gram-Schmidt Walk, based on a different approach, that in fact works for general convex bodies $K$, and allows us to make Theorem \ref{thm:bana} fully algorithmic.

\section{Preliminaries}
Recall that in the Koml\'os problem, we are given an $m\times n$ matrix $\mat$, where each column $j\in [n]$ has squared $\ell_2$ norm $\sum_{i=1}^m \mat(i,j)^2\leq 1$. In other words, we have $\|\mat\|_{1\to 2} \le 1$.
By Remark \ref{rem:bana-komlos-trick}, we can assume that $n \leq m \leq n^2$, so that $\log n \approx \log m$.
As usual, we use $i$ and $j$ to index the rows and columns, and use $v_i$ to denote the $i$-th row of $\mat$.

Our goal is to show the following.
\begin{theorem}
There is a randomized polynomial time algorithm such that, for any $m\times n$ matrix $\mat$ all of whose columns have $\ell_2$ norm at most $1$, the algorithm finds a coloring $x \in \{-1,1\}^n$ with discrepancy $\|\mat x\|_\infty \lesssim  \sqrt{\log 2n}$.
\end{theorem}
Before describing the algorithm, we give some intuition.

\smallskip

Given the matrix $\mat$, we call a row $v_i$, $i\in [m]$, {\em large} if $\|v_i\|^2_2  \geq 4$, and {\em small} otherwise.
Intuitively, the difficulty of the problem is essentially due to large rows. In particular, if all the rows are small, a random coloring already works,  as for any row $i$ we have that $\E [\langle v_i,x \rangle^2] =  \|v_i\|^2_2 \lesssim 1$, and standard concentration bounds directly give that $|\langle v_i, x\rangle | \lesssim 
\sqrt{\log 2n}$  for all rows $v_i$ with high probability.  

 This suggests that a reasonable algorithm should try to carefully cancel out the discrepancy for large rows while ensuring that the coloring looks random for small rows.

 \paragraph {A Natural Algorithm that Fails.} We start by describing a very natural candidate algorithm, and showing a bad  example where it fails. This example will be instructive to understand, and it will also suggest a natural fix.
 
Consider the following algorithm. We start initially with the coloring $x_0=0^n$. At time $t=1,2,\ldots$, (as usual) we update the coloring $x_{t-1}$ to $x_{t} = x_{t-1} + \Delta x_t$, where  $\Delta x_t$  is some tiny random increment 
supported on coordinates  that are not yet fixed at time $t$, and chosen as follows. 

Let $n_t$ denote the number of active (not yet fixed) coordinates at $t$.
By averaging, there are at most $n_t/4$ large rows (restricted to the active coordinates), and
let $W_t$ be the subspace orthogonal to these rows. Suppose we pick $\Delta x_t$ randomly in $W_t$. 

 Clearly, in this algorithm every row incurs zero discrepancy as long as it is large. Moreover, as $\dim(W_t)\geq 3n_t/4$, intuitively one would expect that the walk still looks sufficiently random, and thus, with high probability, any row only incurs $O(\sqrt{\log n})$ discrepancy once it becomes small (as under a fully random coloring). 



\paragraph{An Instructive Bad Example.} Unfortunately, a serious problem with this idea is that constraining the walk to lie in the subspace $W_t$ can create arbitrary correlations between the active coordinates, leading to high discrepancy.
Consider the following example.

Suppose $v_1$ is a small row supported on $k\ll n$ coordinates (for concreteness, say, $\mat(1,j)=k^{-1/2}$ for $j \in[k]$ and $0$ otherwise). Let $W$ be the subspace orthogonal to the large rows. Now suppose $W$ is such that every vector $v \in W$ must satisfy $v(1) = v(2) = \cdots = v(k)$. But then, any update $\Delta x \in W$ must have $\Delta x(1)= \cdots = \Delta x(k)$, and hence updating $\Delta x(1)$ by $\pm \eps$ will incur discrepancy $\ip{v_1, \Delta x}  =\pm k \cdot k^{-1/2} \eps = \pm k^{1/2} \eps$. This should be contrasted with the case where the $\Delta x(i)$ are i.i.d. $\pm \eps$, where $\ip{v_1, \Delta x} \approx \pm k^{1/2} \cdot k^{-1/2} \eps = \pm \eps$.

Before reading on,
the reader may find it instructive to think about how one can  handle this example, even if the subspace $W$ is chosen adversarially as above.

\allowdisplaybreaks
\section{Algorithm}

In the example above, the correlations for the row $v_1$ can be handled by also setting $\Delta x(1)=0$ (and thus also $\Delta x(2)=\cdots =\Delta x(k) = 0$). 

However, it is not immediately clear how to address this problem in general. First, there could be multiple bad rows $v_i$ with coordinates correlated in arbitrary ways. Moreover, imposing additional constraints on $\Delta x$ to fix the currently bad rows can create further correlations and additional bad rows that need to be fixed again and so on. 

Interestingly, it turns out that there always exist 
good update directions $\Delta x$ that are

\smallskip 

(i) orthogonal to all large rows, and 

(ii) still look {\em random} (or better) in every direction.

\smallskip 

Moreover, these directions can be computed efficiently using a semidefinite program (SDP). 
 We now describe this SDP, stated below in a slightly more general setting than we need for our algorithm. We describe the overall algorithm, which will be similar to that in Section~\ref{sec:vec-herdisc}. We then analyze it in Section \ref{subsubsec:analysis_Banaszczyk}.
\subsection{SDP Formulation}
\label{subsec:sdp_formulation}


 Let $W \subset \R^h$ be any subspace with dimension $\dim(W) \le \delta h$. 
Consider the following SDP with the variable $U \in \R^{h\times h}$, and  with some scalar parameters $0\leq \kappa, \eta \leq 1$.
\begin{align}
 \text{(SDP)} \qquad   \langle ww^\top, U \rangle           &= 0                                          && \text{for all } w \in W  \label{sdp:orthog}\\
    \langle e_je_j^\top, U \rangle                               &\leq 1                                         && \text{for all } j \in [h] \label{sdp:jj} \\
    \tr(U)                                &\geq \kappa h                       \label{sdp:trace}         \\
    U                                     &\preceq \frac{1}{\eta} \diag(U)              \label{sdp:isotropic}\\
    U                                     &\succeq 0 \label{sdp:psd}
\end{align}
Here, $e_1,\ldots,e_h$ are the standard coordinate directions in $\R^h$, and $\diag(U)$ is the diagonal matrix we get from zeroing out all non-diagonal entries of $U$.
It is clear that this SDP can be solved in time polynomial in $h$.\footnote{In particular, for constraint \eqref{sdp:orthog}, one can pick an arbitrary basis $w_1,\ldots,w_{\delta h}$ for $W$ and set $\langle w_i w_i^T,U\rangle =0$ for each $i \in [\delta h]$. This ensures that $\|U^{1/2}w_i\|_2=0$ so that $U^{1/2}w_i ={\bf 0}$ for each $i$, and thus by linearity $U^{1/2}w =\bf{0}$ for all $w\in W$.}
To understand the constraints in this SDP, fix some feasible solution $U$ and consider the random vector $u =U^{1/2} g$ where $g \sim N(0,I_h)$. Clearly, $u$ has mean $\E[u]=0$ and covariance $\E[uu^\top]=U$. Let us think of $u$ as the coloring update  (ignoring the scaling). 
The constraints $\eqref{sdp:orthog}$ ensure that $\langle u,w\rangle=0 $ for all vectors $w\in W$, as  
\[\langle u,w\rangle = \langle U^{1/2} g,w \rangle = \langle g,U^{1/2} w \rangle  =0.\] 
constraints \eqref{sdp:jj} ensure that $\E[u(j)^2]\leq 1$ for all coordinates $j \in [h]$, as \[\E[u(j)^2] = \E[ \langle uu^T, e_je_j^T  \rangle] =  \langle U, e_je_j^T\rangle  \leq 1.\] 
For $\kappa >0$, the constraint \eqref{sdp:trace} rules out the trivial solution $U={\bf 0}$.

\medskip

\noindent {\bf Random-Like Coloring Update.} 
The constraint \eqref{sdp:isotropic} is more interesting. Notice that 
 it is equivalent to saying that \[\theta^T U \theta \leq (1/\eta)\, \theta^T \diag (U) \theta\] for all vectors $\theta = (\theta(1),\ldots,\theta(h)) \in \R^h$.
 
 As $\theta^T U \theta =\E[\langle \theta, u \rangle ^2]$, this says that \[\E\left[\langle \theta, u \rangle ^2 \right] = \E \Big[ \Big(\sum_{j=1}^h \theta(j) u(j)\Big)^2\Big] \leq \frac{1}{\eta} \sum_{j=1}^h \theta(j)^2\, \E[u(j)^2].\]
Notice that if we had $\eta=1$ and an equality above, then this is the same as saying that the random variables $u(1),\ldots,u(h)$ are pairwise independent. In our algorithm, we will set $\eta = 1/4$, so intuitively \eqref{sdp:isotropic} should be viewed as saying that the ``coloring update'' $u$ is almost pairwise independent.

We need the following result about the feasibility of this SDP.
 We will prove this later in
 Section \ref{subsec:sdp_feasibility}.

\begin{theorem}\label{thm:sdp_feasibility}(SDP Feasibility) 
For any subspace $W$ of $\R^h$ of dimension at most $\delta h$, the SDP given by \eqref{sdp:orthog}-\eqref{sdp:psd} is feasible
whenever $\delta + \kappa + \eta \leq 1$. 
\end{theorem}

\subsection{The Algorithm}
\label{subsubsec:alg_bana}
Let $\mat$ be the input matrix.
Fix a small step size $\gamma = 1/n^2$  and let $\delta = 1/n$. 
 Let $x_{t-1}$ be the coloring at the end of time step $t-1$ of the algorithm, and define $F_{t} = \{j: |x_{t-1}(j)| > 1-\delta\}$ to be the set of ``fixed'' coordinates at the end of time $t-1$ (or equivalently at the beginning of time $t$). We let $\act_t = [n]\setminus F_t$ be the set of active, i.e., unfixed coordinates.

Initialize $x_0(j) =0$ for $j\in [n]$
and $F_0= \emptyset$. 
Let $n_t := n-|F_t| = |\act_{t}|$.
At each time step $t$, we will only update coordinates in $\act_t$.

\medskip 

 \noindent {\bf Algorithm:} At each time step $t=1,2,\ldots$, repeat  until $n_t=0$.

 

\begin{enumerate}
\item Call a row $i$ \emph{large} (at time $t$)  if $\sum_{j \in \act_t} \mat(i,j)^2 \geq 4$. Let $W_t$ denote the subspace spanned by all the large rows. All rows that are not large are \emph{small}.
    \item Solve the SDP \eqref{sdp:orthog}-\eqref{sdp:psd} with variable $U_t \in \R^{\act_t \times \act_t}$, the subspace $W = W_t \subseteq \R^{\act_t}$, $h=n_t$ and $\kappa = \eta = 1/4$.  
 
     Let $\Delta x_t = \gamma U_t^{1/2} g_t$, where $g_t \sim N(0,I_{\act_t})$. 
     
     Update $x_t = x_{t-1} + \Delta x_t$.
\end{enumerate}

\noindent Let $T$ be the (random) time at which the algorithm terminates. The final coloring is given by $\col(i) = \sign(\col_T(i))$.

\section{Analysis} 
\label{subsubsec:analysis_Banaszczyk}
We first show that the SDP is always feasible at each time $t$, and that the algorithm runs in polynomial time. Then we bound the discrepancy.

\medskip 
\noindent 
{\bf Feasibility.}
Fix some time $t$. As $\sum_i \mat(i,j)^2\leq 1$ for each column $j$,  summing over the active columns gives $\sum_{j \in \act_t}\sum_i \mat(i,j)^2\leq |\act_t|= n_t$. So by averaging there can be at most $n_t/4$ large rows, and  $\dim(W_t) \leq n_t/4$ (and thus $\delta\leq 1/4$) in the SDP constraint \eqref{sdp:orthog}. (This argument is analogous to the proof of Lemma~\ref{lm:ch2-bf-counting}.)
As we set $\kappa=\eta=1/4$, we have that $\delta + \kappa + \eta \leq 1$, and thus by Theorem \ref{thm:sdp_feasibility} the SDP is feasible.

The algorithm also terminates with high probability in polynomial time, as at each time $t$, the squared norm
\[\E[\|x_t\|^2] = \E[\|\col_{t-1}\|^2] + \gamma^2 \Tr(U_t) \geq \E[\|x_{t-1}\|^2] + \gamma^2\]  rises  in expectation by at least $\gamma^2$, and is bounded by $n$ since, with high probability, $|x_t(j)| \le 1$ for all time steps $t$ and coordinates $j$.

 \medskip
%
We now bound the discrepancy.
\begin{theorem}
\label{thm:alg-komlos-bana}
  With high probability,  $\|\mat\col\|_\infty \lesssim \sqrt{\log 2n}$.   
\end{theorem}
Notice that we only need to bound the discrepancy from the time a row becomes small. 
Indeed, for any row $i$ that is large at time $t$, the vector $v_i$ always lies in the subspace $W_t$. Thus any row incurs $0$ discrepancy as long as it is large.


\medskip
\noindent {\bf Random processes with negative drift.} 
We first state a
useful concentration inequality for super-martingales with negative drift.
\begin{lemma}\label{fact:freedman_conc}
Let $\{Z_t: t = 0,1, \cdots\}$ be a real-valued random process with increments $\Delta Z_t := Z_t - Z_{t-1}$, where 
there exists an $\alpha> 0$ such that, for all $t$, with probability 1, $|\alpha \Delta Z_t| \leq 1$, and 
\[
\E_{t-1}[\Delta Z_t] \leq - \alpha\, \E_{t-1}[(\Delta Z_t)^2].
\] Here $\E_{t-1}[\cdot]$ denotes the conditional expectation $\E[\cdot | Z_1, \cdots, Z_{t-1}]$. Then for all $\xi \geq 0$, we have that   
$
\Pr[ Z_t - Z_0 > \xi ] \leq \exp(- \alpha \xi).$
\end{lemma}
\begin{proof}
By Markov's inequality,
\[ \Pr[Z_t - Z_0 \geq \xi ]  = \Pr[\exp(\alpha (Z_t-Z_0)) \geq \exp(\alpha \xi)] \leq \frac{\E[\exp(\alpha (Z_t-Z_0))] }{\exp(\alpha \xi)}.\]
So, it suffices to show that $\E[ \exp(\alpha (Z_t-Z_0)) ]  \leq 1$. 
Now,
\begin{align*}
\E_{t-1}\left[\exp(\alpha (Z_t-Z_0))\right] = &  
  \exp(\alpha (Z_{t-1}-Z_0))\E_{t-1}\left[\exp(\alpha \Delta Z_t)\right] \\
 \le &  \exp(\alpha (Z_{t-1}-Z_0)) \E_{t-1}[1+\alpha \Delta Z_t + (\alpha \Delta Z_t)^2] \\
  \leq & \exp(\alpha (Z_{t-1}-Z_0)). \qquad 
\end{align*}
Here the first inequality uses that $\alpha |\Delta Z_t|\leq 1$ and $e^y \leq 1 + y + y^2$ for $ |y| \leq 1$, and the second that 
 $\E_{t-1}[\Delta Z_t] \leq -\alpha \E_{t-1}[(\Delta Z_t)^2]$.

The result now follows by the law of iterated expectations.
\end{proof}

\paragraph{Proof of Theorem \ref{thm:alg-komlos-bana}.}
%
%
Fix some row $i$. Let $t_i$ be the time when it first becomes small, and 
let $\theta$ denote the row $v_i$ restricted to the coordinates that are alive at $t_i$ (and set to $0$ otherwise). Note that $\|\theta\|^2_2 \leq 4$. 
Our goal is to show that
$|\langle \theta,x_T - x_{t_i} \rangle | \lesssim \sqrt{\log 2n}$, with high probability.

\smallskip
 
Let us write $\theta = b + s$, where $b$ contains the big entries of $\theta$ with magnitude at least $1/\sqrt{\log 2n}$, with the other entries set to $0$,
and $s$ consists of the rest of the entries, with the big entries set to $0$. 
\smallskip 

The discrepancy due to $b$ is trivial to bound. As $b$ has big entries and norm $\|b\|_2 \leq \|\theta\|_2 \le 2$, it can have at most $4 \log 2n$ non-zero entries. By the Cauchy-Schwarz inequality,
\[
\big|\langle b, x_T - x_{t_i} \rangle \big| \leq \|b\|_2 \cdot \Big(\sum_{j: b(j) \neq 0} (x_T(j) - x_{t_i}(j))^2\Big)^{1/2} \lesssim \sqrt{\log 2n}.
\]

\medskip
We now focus on bounding the discrepancy of $s$. 

\noindent {\bf Regularized Discrepancy.}
To do this, we will apply Lemma \ref{fact:freedman_conc} to the {\em regularized discrepancy} of the row $s$, defined as
\[
Z_t := \langle s, x_t \rangle + \beta \sum_{j=1}^n s(j)^2 (1 - x_t(j)^2), 
\]
where we set $\beta :=\sqrt{\log 2n}$. 

$Z_t$ consists of two components: the discrepancy of $s$, and the ``energy'' term $\beta \sum_{j=1}^n s(j)^2 (1 - x_t(j)^2)$, which quantifies how far the colors of coordinates that contribute to the discrepancy are from $\{-1,+1\}$.
Intuitively, $Z_t$ is a good proxy for the discrepancy of $s$ at time $t$. Moreover, it has the nice property that one can ``charge'' the increase in discrepancy to the decrease in energy. We show these properties next.

As $\sum_j s(j)^2 \leq \|\theta\|^2 \leq 4$ and $1-x_t(j)^2 \geq 0$ for all $j\in [n]$,  
\[\langle s, x_t \rangle  \leq Z_t \leq \langle s, x_t \rangle + 4 \beta,\]
$Z_t$ then tracks the actual discrepancy of $s$ to within $O(\beta) = O(\sqrt{\log n})$, by our choice of $\beta$. 
\medskip

\noindent {\bf Showing the Negative Drift.} 
Next, we claim that $\Delta Z_t$ satisfies the negative drift condition in Lemma \ref{fact:freedman_conc} with $\alpha = \Omega(\beta)$. 
Indeed, as $x_t = x_{t-1} + \Delta x_t$, and writing $x_t(j)^2 =  x_{t-1}(j)^2 + 2 \Delta x_t (j) x_{t-1}(j) + \Delta x_t(j)^2$, a direct computation gives
\begin{align} \label{eq:dZ_t}
\Delta Z_t = \langle s - 2 \beta x_{t-1} \circ s^2, \Delta x_t \rangle  - \beta \langle s^2, (\Delta x_t)^2 \rangle,
\end{align} 
where we use $s^2$ and $\Delta x_t^2$ to denote the vector with entries $s(j)^2$ and  $(\Delta x_t(j))^2$ respectively, and $x_{t-1} \circ s^2$ the vector with entries $x_{t-1}(j) s(j)^2$. 

Let us now consider $\E[\Delta Z_t]$ and $\E[(\Delta Z_t)^2]$.
As $\E[\Delta x_t]=0$, we have 
\begin{equation}
    \label{eq:sdp-komlos-disc-0}\E[\Delta Z_t] = - \beta \,\, \E \langle s^2, \Delta x_t^2 \rangle.
\end{equation}
Next as $\Delta x_t = \gamma U_t^{1/2} g$, and ignoring $O(\gamma^3)$ terms as $\gamma$ is tiny,\footnote{As the algorithm runs for $\tilde{O}(1/\gamma^2)$ time steps with high probability, the total contribution of these terms over all time steps can be trivially bounded by $o(1)$.} we have
\begin{align}
\E[(\Delta Z_t)^2] 
& = \E  \langle s - 2 \beta x_{t-1} \circ s^2, \Delta x_t \rangle^2  \nonumber \\
& \lesssim \E  \langle (s - 2 \beta x_{t-1} \circ s^2)^2, \Delta x_t^2 \rangle \label{eq:sdp-komlos-disc-1}\\
& \lesssim \E \langle s^2, \Delta x_t^2 \rangle.\label{eq:sdp-komlos-disc-2}
\end{align}
The first inequality \eqref{eq:sdp-komlos-disc-1} uses that for any vector $v$, we have
\begin{equation} \E [\langle v,\Delta x_t\rangle^2] \lesssim \E[\langle v^2, \Delta x_t^2\rangle].
\label{eq:sdp-komlos-bana-nice-eq}
\end{equation}
This follows as $\E[\Delta x_t \Delta x_t^\top] = \gamma^2 U$, and so the SDP constraint \eqref{sdp:isotropic} implies that
$\E[\Delta x_t \Delta x_t^\top]  \preceq O(1) \cdot \diag(\E[\Delta x_t \Delta x_t^\top])$. 
This immediately implies \eqref{eq:sdp-komlos-bana-nice-eq}.
The inequality \eqref{eq:sdp-komlos-disc-2} uses that   $|2 \beta x_{t-1}(j) s(j)^2| \lesssim|s(j)|$, which follows as $|s(j)| \leq 1/\sqrt{\log2n}$, $\beta = \sqrt{\log2n}$  and $ |x_{t-1}(j)| \leq 1$ for each coordinate $j$.

Together, \eqref{eq:sdp-komlos-disc-0} and \eqref{eq:sdp-komlos-disc-2} give that, for some $\alpha = \Omega(\beta)$, \[\E[\Delta Z_t] \leq  -\alpha\E[(\Delta  Z_t)^2].\] 

\medskip
\noindent {\bf Bounding the discrepancy.}
Applying Lemma \ref{fact:freedman_conc} with $\alpha = \Omega(\beta)$ and $\xi = \Theta(\beta^{-1} \log n) $ gives 
\[
\Pr[Z_T - Z_{t_i} \geq \xi] = \exp(-\alpha \xi ) = 1/\poly(n). 
\]
As $\xi = O(\sqrt{\log2n})$ by our choice of $\beta$, this implies that $Z_T \leq Z_{t_i} + O(\sqrt{\log n})$  with high probability.
As
$\langle s, x_T \rangle \leq Z_T$ 
and $Z_{t_i} = O(\beta)$ (as $\langle v_i,x_{t_i}\rangle=0$, since $v_i$ was large until time $t_i$, and hence the updates $\Delta x_t$ were orthogonal to it), this implies that
$
\langle s, x_T - x_{t_i} \rangle = O(\sqrt{\log n})$. This proves the result. \qedhere

\section{SDP Feasibility}
\label{subsec:sdp_feasibility}
We now prove Theorem \ref{thm:sdp_feasibility}. We first state two useful facts.

\begin{lemma}
\label{col1matr}
Given any $m\times n$ matrix $M$ with columns $m_j$ satisfying $\|m_j\|_2\le 1$ for all $j\in[n]$ and $\beta\in(0,1]$, there exists a subspace $S \subseteq \mathbb{R}^n$ such that (i) $\dim(S)\ge (1-\beta)n $ and (ii) $\forall y\in S$, $\|My\|_2^2 \le (1/\beta)\|y\|_2^2$.
\end{lemma}
\begin{proof}
Let $M=\sum_{i=1}^n \sigma_i p_i q_i^T$ by the singular value decomposition of $M$, with singular values $0\le \sigma_1\le\dots\le\sigma_n$ and $\{p_i : i\in [n]\}, \{q_i : i\in [n]\}$ are the two sets of orthonormal  vectors.
Then,
\[
\sum_{i=1}^n\sigma_i^2 =\textrm{Tr}[\sum_{i=1}^n \sigma_i^2 q_iq_i^T]=\textrm{Tr}[M^TM]=\sum_{i=1}^n \|m_i\|_2^2 \le n
\]
So at least $(1-\beta)n$ of the $\sigma_i^2$s have value at most $1/\beta$, and thus $\sigma_1^2\le\dots\le\sigma^2_{(1-\beta)n} \le 1/\beta$.
Let $S=\textrm{span}\{q_1,\dots,q_{(1-\beta)n}\}$. For $y\in S$,
\begin{align*}
\|My\|_2^2 & =  \|\sum_{i\leq n} \sigma_ip_iq_i^T y\|_2^2 
  =   \|\sum_{i \leq (1-\beta)n} \sigma_ip_iq_i^T y\|_2^2 \nonumber\\
 & =  \sum_{i\leq (1-\beta)n} \sigma_i^2 (q_i^Ty)^2  \leq    \frac{1}{\beta}\sum_{i\leq (1-\beta)n}(q_i^Ty)^2 \le \frac{1}{\beta}\|y\|_2^2
\end{align*}
where the third equality uses that the $p_i$ are orthonormal, and the last step uses that the $q_i$ are orthonormal.
\end{proof}

\begin{lemma}
\label{diagdom}
Given an $n\times n$ PSD matrix $G$ and $\beta\in(0,1]$, there exists a subspace $S\subseteq \mathbb{R}^n$ satisfying (i) $\dim(S)\ge (1-\beta) n$
and (ii) $\forall y\in S$, $y^TGy\le (1/\beta) y^T\diag(G)y$.
\end{lemma}
\begin{proof}
Let $U=G^{1/2}$. Then $G=U^TU$, and $w^TGw\le (1/\beta)w^T\diag(G)w$ is equivalent to
\[ \|Uw\|_2^2\le (1/\beta)\|\diag(G)^{1/2}w\|_2^2.\]
Let $N \subseteq [n]$ be the set of coordinates $i$ with $G_{ii} > 0$. For $i \notin N$, as $G_{ii}=0$, the $i^{th}$ column of $U$ must be identically zero, and we trivially have that  $\|Ue_i\|_2^2= \frac{1}{\beta}\|\diag(G)^{1/2}e_i\|_2^2=0.$

As $\textrm{span}\{e_i:i\in N\}$ is orthogonal to the directions in $[n]\setminus N$, it 
suffices to show that there is a $(1-\beta)|N|$ dimensional subspace $Z$ in $\textrm{span}\{e_i:i\in N\}$ such that $ \|Uy\|_2^2 \leq  (1/\beta)  \|\diag(G)^{1/2}y\|_2^2$ for each $y \in Z$. Then we can simply return the overall subspace $Z \oplus \textrm{span}\{e_i: i \in [n]\setminus N\}$ which has dimension $(1-\beta)|N| + (n-|N|) \geq (1-\beta)n$.

So let us assume that $N=[n]$, so that  $G_{ii} > 0$ for all $i\in N$ and $\diag(G)$ is invertible.
Let $\widetilde{U}=U \diag(G)^{-1/2} $. The $\ell_2$-norm of each column in $\widetilde{U}$ is $1$, and by Lemma \ref{col1matr}, there is a subspace $\widetilde{S}$ of dimension at least $(1-\beta)|N|$ such that $\|\widetilde{U}\widetilde{y}\|_2^2 \leq  \frac{1}{\beta} \|\widetilde{y}\|_2^2$ for each $\widetilde{y}\in \widetilde{S}$.
Setting $y=\diag(G)^{-1/2}\widetilde{y}$ gives
 \[ \|U y \|_2^2 =  \|\widetilde{U}\tilde{y}\|_2^2 \leq (1/\beta) \|\widetilde{y}\|_2^2 = (1/\beta) \|\diag(G)^{1/2} y\|_2^2, \]  
and $S = \{\diag(G)^{-1/2}\widetilde{y}: \widetilde{y}\in \widetilde{W}\}$ gives the desired subspace. 
\end{proof}
We can now prove Theorem \ref{thm:sdp_feasibility}.
\begin{proof}(Theorem \ref{thm:sdp_feasibility})
Consider the following SDP
\begin{equation}
\label{eq:sdp_primal} \tag{Primal SDP}
\begin{aligned}
    \max \quad & \Tr(X) \\
    s.t. \quad & X \cdot ww^\top = 0 \ , && \text{for all } w \in W ,\\
    & X \cdot e_i e_i^\top \leq 1 \ , &&\text{for all } i \in [h] , \\
    & X \preceq \frac{1}{\eta} \diag(X) , \\
    & X \succeq 0 .
\end{aligned}
\end{equation}
We will show the optimal value of \eqref{eq:sdp_primal} is at least $\kappa h$.

\smallskip
\noindent 
We do so by considering the dual SDP. 
Let $\gamma_w \in \R$, $\alpha_i \geq 0$, $H \succeq 0$ be the Lagrangian multipliers for the constraints of \eqref{eq:sdp_primal}. 
Note that $H \in \R^{h\times h}$.
Then the dual SDP is given by 
\begin{align}
\label{eq:sdp_dual} 
& \min \quad \sum_{i \in [h]} \alpha_i   \qquad  \tag{Dual SDP} \\
& s.t. \sum_{w \in W} \gamma_w ww^\top + \sum_{i \in [h]} \alpha_i e_i e_i^\top + H - \frac{1}{\eta} \diag(H)  \succeq I, \label{eq:dual-sdp-komlos} \\
& \quad H \succeq 0, \gamma_w \in \R, \ \text{and } \alpha_i \geq 0 \text{ for all } i \in [h] .\nonumber 
\end{align}
We will show that {\em every} feasible solution to \eqref{eq:sdp_dual} must have objective value at least $\kappa h$.
By strong duality,\footnote{Strictly speaking, to apply strong duality for SDPs one needs that the dual have some solution strictly in its interior. In our case this is easily verified by choosing, say, $\gamma_w=0$ for all $w$, $\alpha_i=2/\eta$ for all $i$ and $H=I$.  } this would imply that \eqref{eq:sdp_primal} must have a solution with value at least $\kappa h$.

\medskip
\noindent {\bf Every dual solution has large value.}
Fix some solution $(\gamma,H,\alpha)$ to \eqref{eq:sdp_dual}.
By Lemma \ref{diagdom} applied to the matrix $H$, there exists some subspace $S \subset \R^h$ with $\dim(S) \geq (1 - \eta) h$, such that \[y^\top H y \leq (1/\eta) y^\top \diag(H) y \qquad \text{ for all vectors $y \in S$}.\]

Consider the subspace $S_{\mathsf{neg}} := W^\perp \cap S$. Then
as  $\dim(W^\perp)\geq (1-\delta) h$ and $\dim(S) \geq (1 - \eta) h$, we have that $\dim(S_{\mathsf{neg}}) \geq \big (1 - \delta - \eta  \big) h \geq \kappa h$, where we use that $\delta + \kappa + \eta \leq 1$. Moreover,  
for any vector  $y \in S_{\mathsf{neg}}$, 
\begin{align} \label{eq:neg_subspace_SDP}
& y^\top \Big( \sum_{w \in W} \gamma_w ww^\top + H - \frac{1}{\eta} \diag(H) \Big) y \nonumber \\  
& = y^\top (  H - \frac{1}{\eta} \diag(H) ) y  \leq 0 .
\end{align}
Consider some orthonormal set of vectors $v_1, \cdots, v_{\kappa h} \subset S_{\mathsf{neg}}$.
As $(\gamma,H,\alpha)$ is a feasible dual solution, it satisfies \eqref{eq:dual-sdp-komlos}. Now, taking the trace inner product with  $\sum_{ j \in [\kappa h]} v_jv_j^T$  on both sides of \eqref{eq:dual-sdp-komlos}, and using \eqref{eq:neg_subspace_SDP} gives that 
\begin{align}
\label{eq:key-sdp-komlos-bana-alg}
 \left\langle\sum_{j \in [\kappa h]}  v_j v_j^\top, \sum_{i \in [h]} \alpha_i e_i e_i^\top \right\rangle\geq \left\langle \sum_{j \in [\kappa h]} v_j v_j^\top, I\right\rangle = \kappa h.
\end{align}
As $\sum_{j \in [\kappa h]} v_j v_j^\top \preceq I$, we have $\sum_{j \in [\kappa h]} \langle v_j v_j^\top ,  e_ie_i^\top \rangle \leq 1$ for each $i \in [h]$, and thus \eqref{eq:key-sdp-komlos-bana-alg} implies have that  the dual objective
$\sum_{i \in [h]} \alpha_i  \geq \kappa h$.
\end{proof}

\section{Bibliographic Notes}
Theorem \ref{thm:alg-komlos-bana} was first proved by Bansal, Dadush and Garg in \cite{BDG16}. However, the original proof was more complicated, and our exposition here is substantially simpler and is based on the ideas in \cite{BansalG17}. 

Several alternate ways to prove Theorem \ref{thm:alg-komlos-bana} have also been discovered since then, using a variety of different algorithmic ideas.
In particular, Levy, Ramadas and Rothvoss \cite{LevyRR17} gave a deterministic algorithm based on a nice application of the multiplicative updatie method. More recently, Bansal, Laddha and Vempala  \cite{BansalLV22} gave an algorithm using the barrier function methods, and Pesenti and Vladu \cite{PesentiV23} gave another algorithm based on regularization. Both these algorithms are also deterministic.

\chapter{The Gram Schmidt Walk Algorithm}
\label{ch:gram-schmidt}
We now give an efficient algorithm for Banaszczyk's result as stated
in Theorem \ref{thm:bana} (assuming access to a membership oracle for
the convex set $K$).
As discussed in Section \ref{sec:bana-core}, it is not clear how to
work directly with a general $K$ or use the arguments in
Chapter~\ref{ch:bana-proof} or Chapter~\ref{ch:algo-komlos-bana} to give an algorithm. 

Instead we will use the following equivalent formulation described in Section \ref{sec:bana-subgaussian}.


\begin{theorem}
\label{thm:bana-subgaussian2}
Given arbitrary vectors $v_1,\ldots,v_n \in \R^m$ with
$\|v_i\|_2 \leq 1$, there is a distribution $D$ over colorings
$x \in \{-1,1\}^n$ such that the resulting discrepancy vector
$\sum_{j=1}^n x(j) v_j$, where $x$ is sampled from $D$, has a
mean-zero $O(1)$-sub-Gaussian distribution.

Moreover, given $v_1, \ldots, v_n$, a coloring $x$ can be sampled from
$D$ in polynomial time.
\end{theorem}

We will describe a randomized algorithm called the Gram-Schmidt Walk (GS Walk), that given $v_1,\ldots,v_n$ produces a random  coloring $x$ (sampled from some implicit distribution $D$) such that the resulting discrepancy vector $Y$ is $\sqrt{20}$-sub-Gaussian.

\smallskip 
Before describing the algorithm, it is instructive to consider the following two examples to get some intuition for the sub-Gaussianity condition in Theorem \ref{thm:bana-subgaussian2}. 

\begin{example}
\label{ex:gs-1}
    Suppose that $v_1,\ldots,v_n$ are orthogonal unit vectors. 
    
    We claim that in this case, a uniformly random $\pm 1$ coloring suffices. Indeed fix some vector $\theta \in \R^m$. Then  
    \[\ip {\theta,Y}=
\ip {\theta, \sum_{j=1}^n x(j) v_j} = \sum_{j=1}^n  x(j) \ip{\theta,v_j}.\]
For a uniformly random coloring $x \in \{-1,1\}^n$, each $x(j)$ is an
independent and uniform sample in $\{-1,+1\}$, and hence $x(j) \ip{\theta,v_j}$
is $O(|\ip{\theta,v_j}|)$-sub-Gaussian by Hoeffding's lemma. By
independence $\sum_{j=1}^n  x(j) \ip{\theta,v_j}$ is $\sigma$-sub-Gaussian
with $\sigma\lesssim \sqrt{\sum_{j=1}^n \ip{ \theta, v_j}^2 }$, and the
right hand side is at most $ \|\theta\|_2$  as the $v_j$ are orthonormal.

\begin{remark}
    The randomness is crucial; if we just output a fixed coloring $x$,
    then $Y=\sum_{j=1}^n x(j)v_j$ is $\sqrt{n}$-sub-Gaussian as
    $\ip{Y,\theta} = \|Y\|_2 = \sqrt{n}$ for the unit vector $\theta = Y/\|Y\|_2$.
\end{remark} 
\end{example} 

 \begin{example}
 \label{ex:gs-2}
On the other extreme, suppose that $v_1 = \cdots = v_n =v $ are all
identical to some $v$ with $\|v\|_2=1$. Then a uniformly random
coloring is  very bad as $Y = \left(\sum_{j=1}^n x(j)\right) v$, and
\(
  \E [\ip{Y,\theta}^2] = \E\left(\sum_{j=1}^n x(j)\right)^2 = n
\)
for
$\theta =v$, while we require that $\E [\ip{Y,\theta}^2] \lesssim
1$ for an $O(1)$-sub-Gaussian $Y$ and a unit vector $\theta$.

The right thing to do in this example, of course, is to pair up the
signs $x(j)$ and cancel them out as much as possible. We can
then make sure that $\|Y\|_2 \lesssim 1$ with probability $1$.
\end{example}

The GS Walk algorithm will handle these two extreme scenarios in a unified way. It will exploit the linear dependencies to cancel things exactly as much as possible, and at the same time have enough randomness so that  the resulting distribution is $O(1)$-sub-Gaussian.

We describe this algorithm next in Section~\ref{sec:gs:algorithm}.
In Section~\ref{sec:gs-covariance} we show that the covariance of the output discrepancy vector $Y$ satisfies $\text{Cov}(Y) \preceq I$ and finally in Section~\ref{sec:gs-subgaussian} we prove the sub-Gaussianity.

\section{The Gram-Schmidt Walk}
\label{sec:gs:algorithm}
We first describe the algorithm informally.

Let $v_1,\ldots,v_n \in \R^m$ be the input vectors. The algorithm begins with the fractional coloring $x_0 =(0,\ldots,0)$ and updates it iteratively over time until the final coloring $x_T \in \{-1,1\}^n$. Once  a coordinate reaches $\pm 1$ it is frozen and not updated anymore, otherwise it is alive.

To update the coloring at time $t$, we compute an update direction
$u_t$ based on which coordinates are still alive, and set $x_t = x_{t-1} + \delta_t u_t$ where the step size $\delta_t$ is chosen randomly such that $\E[\delta_t]=0$ and at least one new variable gets frozen.

\paragraph{Computing the update.} The idea behind the update step (which is the crux of the algorithm) is the following. 

Assume for simplicity that no coordinate is frozen.
Let $B$ denote the matrix with columns $v_1,\ldots,v_n$. 
If the coloring $x$ is updated by $u$, the discrepancy vector changes by $Bu = u(1) v_1 + u(2) v_2 +  \ldots + u(n) v_n$.

Suppose that we are required to set $u(n)=1$ (to ensure progress on coordinate $n$),   
but can freely set $u(1),\ldots,u(n-1)$ in $\R$ (not just in  $[-1,1]$). Then intuitively, we should try to use this flexibility to cancel things to to make $Bu$ as small as possible. 

Concretely, let us try to minimize $\|Bu\|_2$ subject to $u(n) = 1$.
Then we explicitly have
\begin{align}
\label{eq:gs-bu}
    \min_{u: u(n)=1} \| Bu\|_2 & =  \min_{u(1),\ldots,u(n-1) } \left\|v_n  + \sum_{j=1}^{n-1} u(j) v_j  \right\|_2  \nonumber \\
    & = \|\Pi_{V^\perp} (v_n)\|_2,
\end{align} 
where $V=\text{span}(v_1,\ldots,v_{n-1})$ is the subspace spanned by $v_1,\ldots,v_{n-1}$, $V^\perp$ is the subspace orthogonal to $V$, and for any subspace $W$, $\Pi_W$ denotes the orthogonal projection operator onto  $W$. 

To see \eqref{eq:gs-bu}, set $u(1),\ldots,u(n-1)$ so that
$\sum_{i=1}^{n-1} u(i) v_i = - \Pi_V(v_n)$ and hence $Bu = v_n -
\Pi_V(v_n) =\Pi_{V^\perp} (v_n)$. This is also
the least possible value of $\|Bu\|_2$, because $\sum_{j=1}^{n-1} u(j)
v_j \in V$, so
\[
  \|Bu\|_2 \ge \|\Pi_{V^\perp}(Bu)\|_2 = \|\Pi_{V^\perp} (v_n)\|_2.
\]

\paragraph{Notation.}
 The algorithm is described formally below. See also Figure \ref{fig:enter-label}.
Let $x_{t-1}$ and $ A_{t-1}$ denote the coloring and the set of alive coordinates at the beginning of time $t$. 
Let
 $p_t$ be the largest-indexed (alive) element in $A_{t-1}$ at time $t$. 
 
 We compute the vector $u_t$ to achieve the minimum in \eqref{eq:gs-bu} (restricted to alive coordinates and $p_t$ as the index $n$), and update $x_t$ randomly along $u_t$
 so that $x_t$ stays in $[-1,1]^n$ and least one more variable reaches $\pm 1$.
 We call $p_t$ the \textit{pivot}  at time $t$, and it plays a special role at time $t$.

\begin{figure}[hbtp!]
    \centering
    \includegraphics[scale=0.7, trim=5mm 10mm 20mm 5mm]{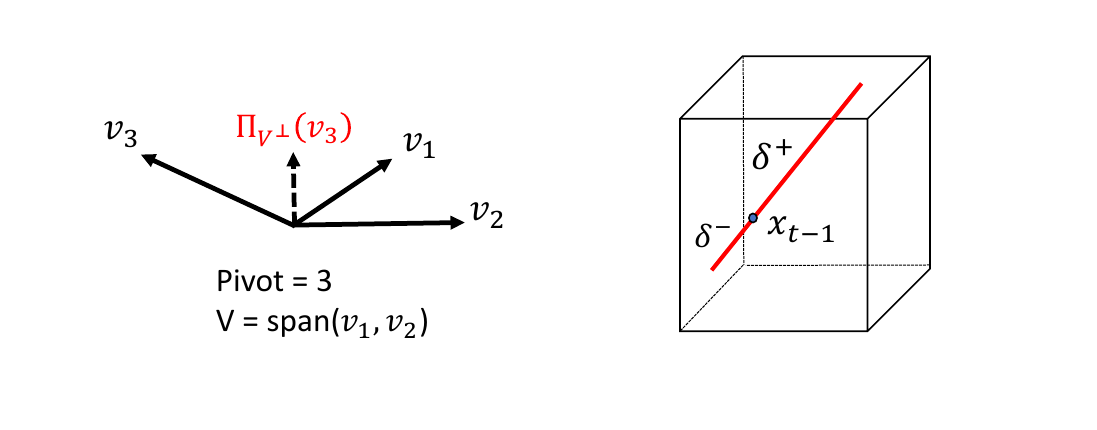}
    \caption{The left shows the pivot and discrepancy update direction. The right shows how the coloring updates.}
    \label{fig:enter-label}
\end{figure}

\noindent {\bf Algorithm Description.} 
Let $B$ denote the input matrix with columns $v_1,\ldots,v_n$. 
Initialize $x_0=(0,\ldots,0)$ and $A_0=[n]$.\\

For $t=1,\ldots,n$ do the following:
\begin{enumerate}

\item 
Compute  $u_t = \text{argmin}_{u} \|Bu\|_2$, subject to $u(p_t)=1$ for the pivot $p_t$, and $u(j)=0$ for each frozen variable $j \notin A_{t-1}$.

\item Let $\delta^{-}_t <0< \delta^{+}_t$ be the maximal negative and positive values for $\delta$, such that 
$|x_{t-1} + \delta u_t| \in [-1, 1]^n$ for all $\delta \in
[\delta_t^-, \delta_t^+]$.
Let
\[
\delta_t \eqdef
\begin{cases}
\delta^{-}_t & \text{with probability }  \delta^{+}_t/(\delta^{+}_t -  \delta^{-}_t)\\
\delta^{+}_t & \text{with probability }  -\delta^{-}_t/( \delta^{+}_t -  \delta^{-}_t)
\end{cases}\;.
\]
\item Set $x_t = x_{t-1} + \delta_t u_t$, and update $A_t \eqdef \{i \in [n]: |x_t(i)| < 1\}$.
\end{enumerate}

\subsection{Basic Properties}
Let us note some basic properties of the algorithm and define some notation that will be useful later.
\paragraph{Coloring update.}
Let $\Delta x_t := x_t - x_{t-1} = \delta_t u_t$ denote the coloring update  at time $t$, and
let $\Omega_{t-1}$ denote all the random choices $\delta_1,\ldots,\delta_{t-1}$ made by the algorithm  by time $t-1$.
Note that $\E[\delta_t| \Omega_{t-1}]=0$, and thus each color $x_t(i)$ evolves as a martingale. 

The step sizes $\delta_t^-,\delta_t^+$  ensure that at least one alive variable always reaches $\pm 1$.
We can assume that exactly one variable gets frozen per step (else we can divide the step into several steps). Thus the algorithm terminates in exactly $n$ steps. 

\paragraph{Discrepancy update.} Let $d_t := Bx_t$ denote the discrepancy vector at the end of time $t$, and let
$\Delta d_t := d_t - d_{t-1}$ denote its update at $t$.

At time $t$, let $V_t$ denote the subspace spanned by the alive vectors excluding the pivot, that is by $\{ v_i:i \in A_{t-1}, i \neq  p_t\}$. 
Notice that $V_0 \supseteq V_1 \supseteq \ldots \supseteq V_n = \emptyset$, is a (random) nested sequence produced by the execution of the algorithm, where we define $V_0= \text{span}\{v_1,\ldots,v_n\}$.
Equivalently, $V_0^\perp \subseteq V_1^\perp \subseteq \ldots \subseteq V_n^\perp$.

As $Bu_t = \Pi_{V_t^\perp}(v_{p_t})$ by \eqref{eq:gs-bu}, we can explicitly write
\begin{equation}
\label{eq:gs-disc-update-explicit1}
    \Delta d_t = B \Delta x_t =  \delta_t  B u_t = \delta_t \Pi_{V_t^\perp}(v_{p_t}). 
\end{equation} 

\paragraph{Phases.}
When the coloring is updated at time $t$, two things may happen:
\begin{enumerate}
    \item 
 Some non-pivot $i \neq p_t$ gets frozen. Then we update $A_t = A_{t-1}\setminus \{i\}$ and (crucially) keep the pivot unchanged, i.e.,~$p_{t+1}=p_t$.
    \item The current pivot gets frozen, i.e., $x_{t}(p_t)$ reaches $\pm 1$. Then we update $A_t =A_{t-1}\setminus \{p_t\}$ and choose a new pivot $p_{t+1} \in A_{t}$ (as the largest indexed element in $A_t$).
\end{enumerate}

For a pivot $p_t$ at time $t$, let $f_t\leq t$ and $\ell_t \geq t$
denote the first (resp.~last) time when $p_t$ was the pivot. We call
the time interval $[f_t,\ell_t]$ a {\em phase}. Thus the phases
partition the time steps $[n]$ into maximal intervals, where inside
each interval the pivot remains unchanged. 

\subsection{Examples} It is instructive to see the algorithm on Examples \ref{ex:gs-1} and \ref{ex:gs-2}.

If the $v_i$ are orthogonal, we claim that it produces a uniformly random coloring.
Indeed, as $\Pi_{V_t} (v_{p_t})=0$, we have $\Pi_{V_t^\perp} (v_{p_t})=v_{p_t}$ and thus  $u_t(p_t)=1$ and $u_t(i)=0$ for $i\neq p_t$. In other words, the algorithm only updates $x(p_t)$ at time $t$ and sets it independently to $\pm 1$. At each step it will choose a new pivot and thus output a uniformly random $\pm 1$ coloring.

If the $v_i$ are identical, we claim that the algorithm will cancel out the colors as much as possible as desired. 
Indeed, as long as at least two elements are alive, we have $\Pi_{V_t}(v_{p_t}) = v_{p_t}$ and thus $\Pi_{V_t^\perp}(v_{p_t}) = 0$ and so $\Delta d_t=0$ by \eqref{eq:gs-disc-update-explicit1}. So the final coloring will have an equal number of $\pm 1$ (up to the parity of $n$).

\section{Bounding the Covariance}
\label{sec:gs-covariance}

Before proving Theorem \ref{thm:gs-subg} in Section \ref{sec:gs-subgaussian},
we first bound the covariance of $d_n$.
This already contains most of the conceptual ideas and avoids various technical details.
In fact, we prove the following optimum covariance bound. 
\begin{theorem}
\label{thm:gs-covariance}
 The final discrepancy vector has covariance \[\E[d_n d_n^T ] \preceq I.\]
Equivalently, $\E[\ip{d_n,\theta}^2] \leq \|\theta\|_2^2$ for all
$\theta \in \R^m$.
\end{theorem}

Let us fix some $\theta \in \R^m$.
We will track how $\E[\ip{d_t,\theta}^2]$ evolves as the algorithm proceeds.

As $d_n = \sum_{t=1}^n \Delta d_t$ and $\Delta d_t = \delta_t Bu_t$, we have
\[\ip{d_n,\theta} = \sum_{t=1}^n\ip{ \Delta d_t,\theta} = \sum_{t=1}^n \delta_t \ip{Bu_t, \theta}.\]
As the increments $\delta_t$ satisfy $\E[\delta_t|\Omega_{t-1}]=0$
(and the update direction $u_t$  and the pivot $p_t$ are completely
determined by $\Omega_{t-1}$), the sequence $ \ip{d_t,\theta}$ has martingale increments. Thus we have that,
\begin{equation}
\label{eq:gs-disc1}
\E[\ip{d_n,\theta}^2] = \E \left( \sum_{t=1}^n \delta_t   \ip{Bu_t, \theta} \right)^2 = \E \left[ \sum_{t=1}^n \delta_t^2   \ip{Bu_t, \theta}^2 \right].\end{equation}
\paragraph{A naive attempt.} 
As $\|Bu_t\|_2 = \|\Pi_{V_t^\perp}(v_{p_t})\|_2 \leq \|v_{p_t}\|_2 \leq 1$, one can bound the right side of \eqref{eq:gs-disc1}
 by $\E[\sum_t \delta_t^2 \|\theta\|_2^2]$. However, this can be as large as $\Omega(n)\|
\theta\|_2^2$ in general (instead of the desired bound of $\|\theta\|_2^2$).

The problem is that the bound $\|\Pi_{V_t^\perp}(v_{p_t})\|_2 \leq 1 $ is too crude, and does not exploit the linear algebraic properties of the algorithm; in particular how the subspaces $V_t$ evolve and how the pivots $p_t$ are chosen. 


To this end, we  will decompose the discrepancy increment $\Delta d_t = \delta_t Bu_t = \delta_t \Pi_{V_t^\perp}(v_{p_t})$ further based on evolution of $V_t$ during the phase where $p_t$ was the pivot. 

\paragraph{Nested subspaces and decomposing $\Pi_{V_t^\perp}(v_{p_t})$ further.}
 Recall that for any two subspaces nested $A\subseteq B$, we can
 write $B$ as the direct sum of the orthogonal subspaces $A$ and $B \cap A^\perp$. Equivalently,  
 the projection $\Pi_{B}$ onto $B$ can be written as  $\Pi_{B} =
 \Pi_{B \cap A^\perp} + \Pi_{A}$, the sum of projections onto $B \cap
 A^\perp$ and $A$. In other words, $\Pi_{B} - \Pi_A = \Pi_{B \cap A^\perp}$.

Applying this to the sequence of subspaces $V_0^\perp \subseteq V_1^\perp \subseteq \ldots \subseteq V_n^\perp$ produced by the algorithm, let us define 
\begin{equation}
\label{eq:gs-pi_t} \Pi_t :=   \Pi_{V_t^\perp}  - \Pi_{V_{t-1}^\perp} = \Pi_{V_t^\perp \cap V_{t-1}},  \end{equation}
 for $t=1,\ldots,n$. Notice that the $\Pi_t$ are mutually orthogonal,
 \ie, $\Pi_{s}\Pi_{t} = \Pi_t \Pi_s = 0$ whenever $s < t$. This is the case since
 the subspace $V_{s}^\perp \cap V_{s-1} \subseteq V_s^\perp \subseteq
 V_{t-1}^\perp$ is orthogonal to $V_{t-1}$.
 
We have the following useful fact.
\begin{lemma}
\label{lem:gs-bu_t-sum}
    At any time $t \in [n]$, we have that
    $Bu_t =  \Pi_{V_{t}^\perp}(v_{p_t})= \sum_{t'=f_t}^t \Pi_{t'} (v_{p_t})$,
    where $p_t$ is the pivot and $f_t$ is the start time of the phase containing $t$. 
\end{lemma}
Henceforth, we denote $P_t:= \sum_{t'=f_t}^t \Pi_{t'}$, so that 
$Bu_t = P_t v_{p_t}$.
\begin{proof}
As $v_{p_t}$ was alive  at time $f_t-1$ and it was not the pivot at that time it  lies in the subspace $V_{f_t-1}$. So  $\Pi_{V_{f_t-1}^\perp}(v_{p_t})=0$ and thus,
    \[  \Pi_{V_{t}^\perp}(v_{p_t}) = \Pi_{V_{t}^\perp}(v_{p_t}) - \Pi_{V_{f_t-1}^\perp}(v_{p_t}) 
     = \sum_{t'\in [f_t,t]} \Pi_{t'} (v_{p_t}). \qedhere\]
\end{proof}
We can now show the following refined bound on $\E[\ip{d_n,\theta}^2]$.
\begin{lemma} \label{lem: gs-ref-bd} 
\[\E[\ip{d_n,\theta}^2] \leq  \E\left[  \sum_{t=1}^n \|\Pi_{t} (\theta) \|_2^2\sum_{t'=t}^{\ell_t}  \delta_{t'}^2\right]  ,\] where $\ell_t$ is the last time of the phase containing $t$.
\end{lemma}
\begin{proof}
By  \eqref{eq:gs-disc1} and Lemma \ref {lem:gs-bu_t-sum}, we have 
\begin{equation}
    \E[\ip{d_n,\theta}^2] = \E \left[ \sum_{t=1}^n \delta_t^2 \ip{Bu_t, \theta}^2\right] = \E \left[ \sum_{t=1}^n \delta_t^2 \ip{P_t v_{p_t}, \theta}^2\right], \label{gs:covariance-bd}
\end{equation} 
As $P_t$ is self-adjoint and  $\|v_{p_t}\|_2\leq 1$, by the Cauchy-Schwarz inequality
   \[\ \ip { P_t v_{p_t}, \theta }^2  =  \ip { v_{p_t},P_t \theta }^2 
     \leq \| v_{p_t}\|^2_2 \,\|P_t\theta\|^2_2 \leq \|P_t\theta\|^2_2.\]
     Further, as the $\Pi_{t'}$ are mutually orthogonal,
     \[\|P_t\theta\|^2_2 = \left\| \sum_{t' \in [f_t,t] }\Pi_{t'} (\theta) \right\|_2^2 = \sum_{t' \in [f_t,t]} \|\Pi_{t'}(\theta)\|_2^2.\] 
Plugging this into \eqref{gs:covariance-bd} and interchanging summations gives
   \[ \E[\ip{d_n,\theta}^2] \leq \E\left[ \sum_{t=1}^n \delta_t^2 \sum_{t'\in [f_t,t]} \|\Pi_{t'}(\theta)\|_2^2\right]  = \E\left[ \sum_{t=1}^n \   \|\Pi_{t}(\theta)\|_2^2 \sum_{t'\in[t,\ell_t]} \delta_{t'}^2 \right]. \qedhere\] 
\end{proof}

\paragraph{A martingale argument.} We can now finish the proof of Theorem \ref{gs:covariance-bd}.
We only need the following standard fact about martingales. 

Consider a real-valued martingale  $X_t$ starting at some fixed $X_0$
with increments $\delta_t = X_t - X_{t-1}$, where  $|X_t |\leq 1$ for
all $t$ with probability $1$.
\begin{lemma} \label{lem:gs-martingale} Let
  $\tau \eqdef \min\{t: |X_t| = 1\}$ be the first (random) time when
  the martingale $X_t$ above reaches $+1$ or $-1$, and assume that,
  for some fixed $N$, $\tau \le N$ with probability $1$. For any
  $t\geq 1$, the expected quadratic variation of future increments
  satisfies
  $\E\left[ \sum_{t'= t}^\tau \delta_{t'}^2 | \Omega_{t-1}\right] =
  1-X_{t-1}^2$, where $\Omega_{t-1}$ denotes all the random choices up
  to time $t-1$.
\end{lemma}
\begin{proof}
  We assume that $t \le \tau$, since the lemma is trivial otherwise.
  Let us define the auxiliary random variables $Y_s \eqdef X_{t+s-1}^2 - \sum_{t' =
    t}^{t+s-1} \delta_{t'}^2$, for any integer $s \ge 1$, and $Y_{0}
  \eqdef X_{t-1}^2$. We claim that, conditional on $\Omega_{t-1}$,
  $Y_0, Y_1, \ldots$ forms a
  martingale sequence. Indeed, for any $s \ge 1$ we have
  \begin{align*}
    \E[Y_s - Y_{s-1}|\Omega_{t+s-2}]
    &= \E[X_{t+s-1}^2  - X_{t+s-2}^2 - \delta_{t+s-1}^2|\Omega_{t+s-2}]\\
    &= \E[2\delta_{t+s-1}X_{t+s-2} | \Omega_{t+s-2}]  = 0,
  \end{align*}
  by the martingale property of $X_t$. Thus, by the optional stopping
  theorem, and using $X_{\tau}^2 = 1$,
  \[
    \E\left[1 - \sum_{t'=t}^\tau \delta_{t'}^2|\Omega_{t-1}\right]
    =\E[Y_{\tau - t+1}|\Omega_{t-1}]
    = Y_0 = X_{t-1}^2.
  \]
  The lemma follows after re-arranging.
\end{proof}

\paragraph{Proof of Theorem \ref{gs:covariance-bd}.}
Fix some time $t$, and condition on $\Omega_{t-1}$. This fixes $V_t$, $\Pi_t$ and the pivot $p_t$. As $p_t$ stays unchanged until its color reaches $\pm 1$ (which happens at time $\ell_t$), by Lemma \ref{lem:gs-martingale} 
\[ \E \bigg[ \sum_{t'\in [t,\ell_t]} \delta_{t'}^2\, |\Omega_{t-1} \bigg] = 1- x_{t-1}(p_t)^2 \leq 1.\]
Averaging over $\Omega_{t-1}$ gives that  $\E \left[ \sum_{t' \in [t, \ell_t]} \delta_{t'}^2 \|\Pi_t (\theta)\|_2^2\right] \leq \|\Pi_t (\theta)\|_2^2,$  and thus by Lemma \ref{lem: gs-ref-bd} 
\[ \E[\ip{d_n,\theta}^2] \leq \sum_{t=1}^n \|\Pi_{t} (\theta) \|_2^2
  = \left\|\sum_{t=1}^n\Pi_t(\theta)\right\|_2^2
 \leq \|\theta\|_2^2, \] 
where we use that the  $\Pi_t$ are mutually orthogonal. 

\section{Proving Sub-gaussianity}

\label{sec:gs-subgaussian}
We now prove the sub-Gaussianity property. In particular, we will show the following.
\begin{theorem}
\label{thm:gs-subg}
For any vector
$\theta \in \R^m$, we have 
\begin{equation}
\label{eq:subg-bound-thm}
    \E[\exp(\ip{d_n,\theta})] \leq \exp(10 \|\theta\|_2^2).
\end{equation}
\end{theorem}

The proof will follow a similar structure as the bound on the covariance. We fix some $\theta \in \R^m$ and track the evolution of $\E[\exp(\ip{d_t,\theta})]$ over time.   

\paragraph{Proof Overview.}
Clearly $\exp(\ip{d_0,\theta})=1$ as $x_0 = \bf{0}$ initially. To show that \eqref{eq:subg-bound-thm} holds eventually at the end of the algorithm,
a natural idea to define a suitable potential $\Phi_t$ so that (roughly) the increase in $d_t$ can be charged to the decrease in $\Phi_t$.
More precisely, 
suppose we have some $\Phi_t$ such that 
$Z_t = \ip{d_t,\theta} + \Phi_t$
satisfies the following properties:
 \[ \E[\exp(Z_t)|\Omega_{t-1}] \leq \exp(Z_{t-1}) \qquad \text{(Super-martingale)} \] 
 \[ \Phi_0 \leq  {10\|\theta\|_2^2}  \text{ and  } \Phi_n=0.  \qquad   \text {(Boundary conditions)}\]
Then Theorem \ref{thm:gs-subg} follows directly as $\E[\exp(Z_n)] \leq \exp(Z_0)$ by super-martingale condition, and $Z_n = \ip{d_n,\theta}$ and $Z_0 = \Phi_0 \leq 10\|\theta\|_2^2$ by the boundary conditions.

Denoting $Z_t = Z_{t-1} + \Delta Z_t = \Delta d_t + \Delta\Phi_t$,  the first property is same as showing that
 $\E[\exp(\Delta Z_t)|\Omega_{t-1}] \leq 1$,
 or equivalently
 \begin{equation}
     \label{eq:gs-supmartingale}\E[ \exp(\ip{\Delta d_t,\theta} + \Delta \Phi_t) | \Omega_{t-1}]  \leq 1. 
 \end{equation}
At a high level this is what we do, but one complication is that \eqref{eq:gs-supmartingale} does not hold when $|\ip{\Delta d_t,\theta}| \gg 1$. 
This introduces some technical complications. In particular, we work with a modified $Z_t$ 
that ignores time steps where $|\ip{\Delta d_t,\theta}|$ can be large, and bound the contribution of such steps to $\ip{d_n,\theta}$ separately.

The motivates the notion of {\em bad} times that we define formally below. We then define the potential $\Phi_t$ and $Z_t$ and give the  detailed proof.
\subsection{Notation and the Potential}
\paragraph{Bad times.}
Recall that by Lemma \ref{lem:gs-bu_t-sum}, $\Delta d_t= \delta_t Bu_t = \delta_t P_t v_{p_t}$.
We call a time $t$  {\em bad} if  $\|P_t\theta \|_2 > 1/4$.
Note that whether $t$ is bad or not is completely determined by
$\Omega_{t-1}$ (the random choices before $t$), since $P_t$ is.
\paragraph{The potential $\Phi_t$.} 
Recall that $V_t = \text{span}\{ v_i : i\in A_{t-1}, i\neq p_t\}$.
Consider the potential
\begin{align*} \Phi_t := 
\begin{cases}
        \|\Pi_{V_t} (\theta) \|_2^2 + (1-x_t(p_t)^2) \|P_t \theta \|_2^2
           \quad &  \text{if $t$ is not bad }\\
                \| \Pi_{V_t} \theta \|_2^2  \quad  &  \text{if $t$ is bad }
\end{cases}
\end{align*}
The term $(1-x_t(p_t)^2) \|P_t \theta \|_2^2$ tracks progress toward freezing the pivot.

Let's compute the boundary conditions. 
\smallskip

(i) Initially $\Phi_0 \leq \|\theta\|_2^2$ as $P_0=0$ and,

(ii) Eventually $\Phi_n =0$  as $V_n=\{0\}$ and $x_t\in \{-1,1\}^n$ at $t=n$.
\allowdisplaybreaks
\paragraph{The process $Z_t$.}
We define $Z_t \eqdef Y_t +  2\Phi_t$, where $Y_t \eqdef Y_0 + \sum_{t' =
  1}^t\Delta Y_{t'}$ is the process defined by $Y_0=0$ and the following increments for $t=1,\ldots,n$
\begin{align*}
\Delta Y_t := 
	\begin{cases}
			\ip{\Delta d_t,\theta} & \mbox{if }t \text{ is not bad} \\
		  0  & \text{ if $t$ is bad}
	\end{cases}
\end{align*}
Intuitively, one should think of $Y_t$ is $\ip{d_t,\theta}$, except that we do not update $Y_t$ at time $t$ if $t$ is a bad time (so while $d_t $ updates to $d_{t-1} + \Delta d_{t-1}$, we keep $Y_t$ unchanged to $Y_{t-1}$).

Notice that initially $Z_0 = Y_0 + 2 \Phi_0 = 2 \Phi_0 \leq  2 \|\theta\|_2^2$. 
Eventually at time $t=n$, using $\Phi_n=0$,
\begin{equation}
        Z_n  = Y_n + 2 \Phi_n = Y_n  
       = \sum_{t=1}^n \Delta Y_t =   \sum_{t \notin B}   \ip{\Delta d_t,\theta}  \label{eq:gs-zn}
\end{equation}
where $B$ denotes the set of all the bad times $t$.

\allowdisplaybreaks
\subsection{The Key Lemmas} We will show the following key lemmas, which will directly imply Theorem \ref{thm:gs-subg}.
\begin{lemma} (Super-martingale property.) 
\label{lem:gs-boundz}
$ \E[\exp(\Delta Z_t)| \Omega_{t-1}] \leq 1$ for  each $t \geq 1$. Thus,  $\E[\exp(Z_n)] \leq  \exp(Z_{0}) = \exp(2 \|\theta\|_2^2).$
 \end{lemma}

\begin{lemma} (Contribution of bad times.)
\label{lem:gs:bound_badtimes} Let $B$ be the set of bad times $t$ with $\|P_t \theta\|_2> 1/4$. Then $|\sum_{t \in B} \ip{\Delta d_t,\theta}| \leq 8\|\theta\|_2^2$. 
\end{lemma}
Let us see how these directly give Theorem \ref{thm:gs-subg}.
\begin{proof}(\emph{Theorem \ref{thm:gs-subg}}.)
Writing $d_n = \sum_{t \notin B} \Delta d_t + \sum_{t\in B} \Delta d_t$, as the sum of increments over bad and non-bad times,
\begin{align*} 
\ip{d_n,\theta} &= \underbrace{\sum_{t \notin B} \ip{ \Delta d_t,\theta}}_{=Z_n\,\,\, \eqref{eq:gs-zn}} + \underbrace{\sum_{t \in B} \ip{ \Delta d_t,\theta}}_{\leq 8\|\theta\|_2^2 \text{ (Lem. \ref{lem:gs:bound_badtimes})}} \leq Z_n + 8\|\theta\|_2^2.
\end{align*}
As $\E[\exp(Z_n)] \leq  \exp(2 \|\theta\|_2^2) $ by Lemma \ref{lem:gs-boundz}, we get the desired bound 
\[\E[\exp( \ip{d_n,\theta})] \leq \E[ \exp(Z_n)]  \exp(8 \|\theta\|_2^2) \leq   \exp(10 \|\theta\|_2^2). \qedhere\]
\end{proof} 

\subsection {Proof of Lemma \ref{lem:gs:bound_badtimes}}
Our goal here is to bound the contribution of bad time steps. We first bound the discrepancy incurred during any sub-interval of a phase.
\begin{lemma}
\label{lem:gs-discphase}
    Let $k$ be any phase with start and end times $b_k$ and $e_k$. Then $\left|\sum_{t=f}^g \ip{ \Delta d_t,\theta}\right| \leq  2 \|P_{g}\theta\|_2$ for any times $f,g \in [b_k, e_k]$. 
\end{lemma}
    \begin{proof}
    Let $p$ denote the pivot during phase $k$.
 Then $\Delta d_t = \delta_t P_t v_p$ at any time $t$ in the phase, and $P_t = \sum_{t'=b_k}^t \Pi_{t'}$. So we have
  \begin{align}
     \sum_{t=f}^g \Delta d_t 
& =\sum_{t=f}^{g} \delta_t  P_t(v_p)  = \sum_{t=f}^{g} \delta_t \left(\sum_{t'=b_k}^t \Pi_{t'}(v_p) \right)
\nonumber \\
 & =  \sum_{t'=b_k}^{g}  \Pi_{t'}(v_p) \left(\sum_{t= \max(t',f)}^{g} \delta_t \right). 
 \qquad \text{(interchanging sums)} \label{eq:gs-derivation-subg1}
\end{align} 
As the pivot $p$ stays unchanged during the phase and $\delta_t= \Delta x_t (p)$ for each $t$ in the phase, a key observation is that  for any $u,v \in [b_k,e_k]$,
\[\left|\sum_{t=u}^{v} \delta_{t}\right| = |x_{v}(p)-x_{u-1}(p) | \leq 2.\]
Using this bound in \eqref{eq:gs-derivation-subg1}, by the triangle inequality,
\[  \left|\sum_{t=f}^g \ip{ \Delta d_t,\theta}\right| \leq  2  \sum_{t'=b_k}^{g}  |\ip{\Pi_{t'}(v_p),\theta} |. \]
As $\Pi_{t'}$ is a projection, it is self-adjoint, and $\Pi_{t'}^2 =
\Pi_{t'}$. Using this, and applying the Cauchy-Schwarz inequality twice,
gives us that  
\begin{align*} \sum_{t'=b_k}^{g} |  \ip{\Pi_{t'}(v_{p}),\theta}| 
      & 
 \leq   \sum_{t'=b_k}^{g} \| \Pi_{t'}(v_{p}) \|_2 \|\Pi_{t'} \theta\|_2 \\
   &  \leq \underbrace{\Big( \sum_{t'=b_k}^{g}  \|\Pi_{t'}(v_{p})\|_2^2 \Big)^{1/2}}_{=\, \|P_g(v_{p})\|_2 \leq \|v_p\|_2\leq 1 } \underbrace{\Big( \sum_{t'=b_k}^{g}  \|\Pi_{t'} \theta\|_2^2 \Big)^{1/2}}_{  =\, \|P_g\theta \|_2  }  
   \leq  \|P_g \theta \|_2.  \end{align*}
   The last step uses that $P_g = \sum_{t'=b_k}^g \Pi_{t'}$, and as the projections $\Pi_{t'}$  are mutually orthogonal  $\|P_g w \|_2^2 = \sum_{t'=b_k}^g \|\Pi_{t'}(w)\|_2^2 $
for any vector $w$.  
\end{proof}
We now finish the proof of Lemma \ref{lem:gs:bound_badtimes}.
\begin{proof}[Proof of \em Lemma \ref{lem:gs:bound_badtimes}.]
Let $B_k$ denote the set of bad time steps in phase $k$. Notice that if time $t$ is bad, then all subsequent times $t'>t$ in the same phase as $t$ remain bad as $P_t \preceq P_{t'}$. 
So the bad times in $B_k$ form an interval and we can apply Lemma~\ref{lem:gs-discphase}, which gives that  
\[\sum_{t \in B_k} \ip{\Delta d_t,\theta}   \leq 2 \|P_{e_k} \theta\|_2  \leq 8 \|P_{e_k} \theta\|_2^2,\]
where the last step uses that $\|P_{e_k}\theta\|_2> 1/4$ whenever $B_k$ is non-empty.

Thus we can bound the contribution of bad times as,
\begin{align*}
  \left|\sum_{t \in B}  \ip{\Delta d_t,\theta} \right|
&  \leq \sum_{k: B_k \neq \emptyset} \left|\sum_{t \in B_k} \ip{\Delta d_t,\theta} \right| 
  \leq  \sum_{k:B_k\neq \emptyset} 8 \|P_{e_k} \theta\|_2^2
  \leq  8 \|\theta\|_2^2.
\end{align*}
Here, the last step follows as the projection operators $P_{e_k}$ are mutually orthogonal for different phases $k$, since $P_{e_k} =\sum_{t = b_k}^{e_k} \Pi_t$, and the $\Pi_t$ are mutually orthogonal.
\end{proof}

\subsection{Proof of Lemma \ref{lem:gs-boundz}.}
We now show that  $\E[\exp(\Delta Z_t)|\Omega_{t-1}] \leq 1$.\\

First, suppose that $t$ is bad. Then this holds trivially as $\Delta Y_t =0$ (by design) and so $\Delta Z_t = 2\Delta \Phi_t$. But then, $\Delta \Phi_t \leq 0$ as 
\[\Phi_t = \|\Pi_{V_t} \theta\|_2^2 \leq \|\Pi_{V_{t-1}} \theta\|_2^2  \leq \Phi_{t-1} \quad \text{ (as $V_t \subset V_{t-1}$)}.\] Thus $\Delta Z_t \leq 0$ with probability $1$.\\

Henceforth we assume that $t$ is not bad.
It will be easier to analyze the change $\Delta Z_t = Z_t-Z_{t-1}$ in two steps:
\begin{enumerate}
\item \emph{Structural changes:} Due to the subspace $V_{t-1}$ changing to $V_t$ and (possibly) the  pivot changing from $p_{t-1}$ to $p_t$. This does not affect the discrepancy $d_{t-1}$ (and hence $Y_{t-1}$), but causes $\Phi_{t-1}$ to change to change to the intermediate potential
  \[
  \Phi_{t}' \eqdef \|\Pi_{V_{t}}(\theta)\|_2^2 + (1-x_{t-1}(p_t)^2) \|P_{t}(\theta)\|_2^2.
  \]
\item \emph{Coloring update:} Due to update of $x_{t-1}$ to $x_t:= x_{t-1} +\Delta x_t$. This affects the discrepancy (and causes $Y_{t-1}$ to change to $Y_t$) and also the potential  $\Phi_t'$ to change to $\Phi_t$.
\end{enumerate}
We can then decompose $\Delta Z_t$ as
\(
\Delta Z_t = (\Phi_t'- \Phi_{t-1})  + (\Delta Y_t + \Phi_t - \Phi_t').
\)
We bound each of the two parenthesized expressions on the right hand side separately.

Let us first consider the effect of structural changes.
\begin{lemma}\label{lm:gs-subg-struc}
$\Phi_t'\le  \Phi_{t-1}$ holds with probability $1$.
\end{lemma}
\begin{proof}
  If
  $p_t \neq p_{t-1}$, then a new phase begins at $t$ and $P_t= \Pi_t$. 
In that case, 
\begin{align*}
    \Phi_t' &= \|\Pi_{V_{t}}(\theta)\|_2^2 + (1-x_{t-1}(p_t)^2) \|\Pi_{t} \theta\|_2^2  \\
    &\leq
     \|\Pi_{V_{t}}(\theta)\|_2^2 + \|\Pi_{t} \theta\|_2^2 =  \|\Pi_{V_{t-1}}(\theta)\|_2^2 \le \Phi_{t-1},
\end{align*}
where we used that $\Pi_{V_t}$ and $\Pi_t$ are orthogonal, and $\Pi_{V_{t-1}} = \Pi_{V_t} + \Pi_t$.

Now suppose that $p_t=p_{t-1}$. Note that
\[
  \|\Pi_{V_{t}}(\theta)\|_2^2 - \|\Pi_{V_{t-1}}(\theta)\|_2^2 = -\|\Pi_t(\theta)\|_2^2,
\]
since  $\Pi_{V_{t-1}} = \Pi_{V_{t}} +\Pi_t$ and  $\Pi_t$ is orthogonal to $\Pi_{V_{t}}$. Moreover,
\(
\|P_{t} (\theta)\|_2^2 - \|P_{t-1}(\theta)\|_2^2 = \|\Pi_t(\theta)\|_2^2,
\)
because $P_t = P_{t-1}  + \Pi_t$  and $\Pi_t$ and $P_{t-1}$ are orthogonal. 
Then,
\[
  \Phi_t'- \Phi_{t-1} \le -\|\Pi_t(\theta)\|_2^2 +
                        (1-x_{t-1}(p_t)^2)\|\Pi_t(\theta)\|_2^2  \le 0. \qedhere
\]
\end{proof}

We now consider the effect of changing $x_{t-1}$ to $x_t$.
\begin{lemma}\label{lm:gs-subg-col}
    $\E[\exp(\Delta Y_t + \Phi_t - \Phi_t')|\Omega_{t-1}] \leq 1$.
\end{lemma}
\begin{proof}
Recall that $\delta_t= x_{t}(p_t)-x_{t-1}(p_t) $ and $\Delta d_t = \delta_t P_t v_{p_t}$, 
where $p_t$ is the pivot at time $t$. 
So, 
\[\Delta Y_t + \Phi_t - \Phi_t' =  \underbrace{\delta_t \ip{P_t v_{p_t},\theta}}_{\Delta Y_t} \underbrace{- 2(2 \delta_t x_{t-1}(p_t) + \delta_t^2 ) \|P_t \theta\|_2^2}_{\Phi_t - \Phi_t'}.\]
Let us denote $a = \ip{P_t v_{p_t},\theta} , b = \|P_t \theta\|_2^2$ and $x = x_{t-1}(p_t)$. As $a,b$ and $x$
are fixed by $\Omega_{t-1}$, our goal is to show that
\begin{equation} \E_{\delta_t} [\exp(\delta_t a -2(2 \delta_t x + \delta_t^2 ) b)] \leq 1 \label{eq:gs-final} \end{equation}
Recall that 
$\delta_t  \in \{-\delta^-_t, \delta^+_t\}$ has a two-point distribution with mean $\E[\delta_t|\Omega_{t-1}] = 0$. To simplify computations in proving \eqref{eq:gs-final}, let us imagine instead a (finely discretized) continuous random walk starting from $0$ and terminating when it reaches $\delta^-_t$ or $\delta^+_t$. 
Since the random walk is a martingale, the final distribution of $\delta_t$ will be supported on \(\{-\delta^-_t, \delta^+_t\}\) and have expected value $0$, so it will have the same distribution as the $\delta_t$ described in the GS algorithm.

In more detail, let us assume that $\delta^-_t$, and $\delta^+_t$ are rational numbers; the lemma would then follow for general $\delta^-_t$, and $\delta^+_t$ by a standard approximation argument. Let us choose some arbitrarily small $\eps> 0$ so that $\delta^-_t$, and $\delta^+_t$ are both integer multiples of $\eps$. Define $\delta_{t,0} \eqdef 0$, and $\delta_{t,s} = \delta_{t,s-1}$ when $\delta_{t,s-1} \in \{-\delta^-_t, \delta^+_t\}$, and $\delta_{t,s} + \eps\sigma_s$ otherwise,  where $\sigma$ is uniform in $\{-1,+1\}$. Let $\tau$ be the first time $s$ when $\delta_{t,s} \in \{-\delta^-_t, \delta^+_t\}$ and set $\delta_t \eqdef \delta_{t,\tau}$. As mentioned above, $\delta_{t,0}, \delta_{t,1}, \ldots$ is a martingale sequence, and, by the optional stopping theorem, $\delta_t = \delta_{t,\tau} \in \{-\delta^-_t, \delta^+_t\}$ has expectation $0$, so it has the correct distribution.

It suffices to show that the sequence defined from $\delta_{t,s}$ by $E_s \eqdef \exp(\delta_{t,s} a -2(2 \delta_{t,s} x + \delta_{t,s}^2 ) b)$ is a supermartingale. Then \eqref{eq:gs-final} follows from the optional stopping theorem. For $s-1 \ge \tau$, $\E[E_s|\delta_{t,s-1}] = E_{s-1}$ trivially, since $\delta_{t,s} = \delta_{t,s-1}$. Therefore, we only need to show that $\E[E_s|\delta_{t,s-1}] \le E_{s-1}$ for $s \le \tau$. To do so, we need to prove the inequality
\begin{equation*}
  \E[\exp(\sigma_s\eps a - 4\sigma_s \eps b(x+\delta_{t,s-1})|\delta_{t,s-1}] \le  \exp({2b\eps^2}).
\end{equation*}
Let us replace $x + \delta_{t,s-1}$ above with $x'$, and notice that $|x'| \le 1$ by the choice of $\delta_{t}^-$ and $\delta_t^+$. Then, it is enough to show that, for a uniformly random $\sigma \in \{-1,+1\}$, and any $x' \in [-1,+1]$,
\begin{equation}
\label{eq:final-gs-computation}
\E[\exp( \sigma \eps (a - 4  bx'))]  \leq \exp(2 b \eps^2).
\end{equation}
By Hoeffding's lemma and the Cauchy-Schwarz inequality, we have
\[
  \E[\exp( \sigma \eps (a - 4  bx'))]
  \le \exp(\eps^2(a - 4bx')^2/2) \le \exp(\eps^2 (a^2 + 16b^2)),
\]
where we also used that $(x')^2 \le 1$.
Now, as $a = \ip{P_t v_{p_t},\theta} =\ip{v_{p_t}, P_t \theta} \leq \|P_t \theta \|_2 = \sqrt{b}$, we have $a^2 \leq b$. Also, as $t$ is a good time, $\|P_t \theta\|_2\leq 1/4$, so $b \leq 1/16$ and hence $16b^2\leq b$. This gives that $a^2 + 16b^2 \leq  2b$, which proves \eqref{eq:final-gs-computation}. 
\end{proof}

Now Lemma~\ref{lem:gs-boundz} follows from Lemma~\ref{lm:gs-subg-col}~and~\ref{lm:gs-subg-struc}. This also completes the proof of Theorem~\ref{thm:bana-subgaussian2}.

\section{Bibliographic Notes}
The Gram-Schmidt Walk (GS Walk) algorithm in Section \ref{sec:gs:algorithm} is due to Bansal, Dadush, Garg and Lovett~\cite{BansalDGL19}. They showed that the distribution $D$ in Theorem \ref{thm:bana-subgaussian2} produced by the GS Walk is  $\sqrt{40}$-sub-Gaussian.
Our presentation follows their approach while simplifying some of the arguments.

Remarkably, motivated by applications to the design of randomized controlled trials,  Harshaw, S{\"{a}}vje, Peng and Spielman \cite{HarshawSSW24} showed that the distribution $D$ output by GS walk algorithm is in fact 1-sub-Gaussian. This is clearly optimal --- as already for $n=1$, the random vector $\pm e_1$ is exactly 1-sub-Gaussian. Proving this refined bound requires more care and we do not do this here.
The tight covariance bound in Theorem \ref{thm:gs-covariance} is also due to \cite{HarshawSSW24}.

\chapter{Online Discrepancy}
\label{ch:online}

So far in all the discrepancy problems we considered, the vectors $v_1,\ldots,v_n \in \R^m$ were given to us upfront. Another very natural setting is where the vectors are revealed in an online manner and when $v_t$ arrives at time $t$, its sign $x(t)$  must be chosen immediately and irrevocably, without the knowledge of future vectors.
The goal is to keep the discrepancy $\|d_t\|_\infty$ small at all times $t$, where  $d_t = x(1) v_1 + \ldots + x(t) v_t$. In a sense, this is very similar
to prefix discrepancy, where we also need to minimize $\max_{t = 1}^n
\|d_t\|_\infty$. However, the difference from the online setting is that in (offline) prefix discrepancy problems, we can choose the coloring
$x$ with full advance knowledge of $v_1, \ldots, v_n$. We will see in this
chapter that a simple online discrepancy minimization algorithm also
gives the best algorithmic results for a number of prefix discrepancy
problems we have already considered. 

\paragraph{Online Koml\'{o}s Problem.}

As the primary example of an online discrepancy problem, we focus here on the online Koml\'{o}s setting where the vectors satisfy $\|v_t\|_2 \leq 1$. It is useful to think of $n$ as being much larger than $m$ (the problem is already non-trivial for $m=2$). 

Notice that sampling each $x(t)$ independently and uniformly can be
thought of as an online algorithm, and incurs discrepancy $O(\sqrt{n})$
in the Koml\'os setting.
Unfortunately, no online algorithm can do better in general. In
particular, suppose that $m=2$, and each vector $v_t$ is picked by an
adversary that knows the discrepancy vector $d_{t-1}$. Then the adversary can pick $v_t$
orthogonal to $d_{t-1}$ so that $\|d_t\|_2^2 = \|d_{t-1}\|_2^2 +
\|v_t\|_2^2$ irrespective of the sign $x(t)$ chosen by the online
algorithm. Repeating this for $n$ steps gives $\|d_n\|_\infty \ge
\sqrt{\frac{n}{2}}$.

Interestingly, substantially better results can be obtained if the vectors $v_t$ are chosen in a less adversarial manner.

\paragraph{Oblivious Adversary Model.} 
A standard model in online computation is that of an oblivious adversary. Here the adversary still knows the online algorithm and can pick the worst case sequence of vectors accordingly. However it must do so in advance before the online algorithm begins its execution. 

Clearly, oblivious adversaries do not help for deterministic
algorithms --- as the adversary knows exactly what the algorithm will do on each input, it can still pick the worst possible sequence in advance.
However, randomization can be very helpful here as the adversary cannot see the specific random choices made by the algorithm (even though it knows the algorithm itself).

\section{Self-Balancing Walk}
\label{sec:sw-alg}
We now describe an elegant algorithm called the  Self-Balancing Walk (SW) with the following guarantee.

\begin{theorem}
\label{thm:als-21} 
For any
$\delta \in (0,1)$, there is a polynomial time randomized online algorithm against oblivious
adversaries, that, for any input vectors $v_1, \ldots, v_n$ with
$\|v_t\|_2 \le 1$ for each $t$,  with probability $1-\delta$ satisfies that $\|d_t
\|_\infty =O(\log(mn/\delta))$ for all $t \in [n]$.
\end{theorem}

Theorem \ref{thm:als-21} is remarkable in many ways. 
It gives $O(\log mn)$ discrepancy w.h.p.~for the online Koml\'{o}s
problem. This matches the best known algorithmic bound for the offline
\emph{prefix} Koml\'os problem that we discussed in
Chapter~\ref{ch:bana-intro}, and comes close to the non-constructive
$O(\sqrt{\log n})$ bound from  Theorem \ref{thm:prefix-komlos}.
Finally, the algorithm is very fast to implement and only requires computing a single inner product at each step. 

\paragraph{The Gaussian Walk (GW) Algorithm.} To gain some intuition for the SW algorithm, we first describe another algorithm with weaker requirements.  We allow our algorithm to choose, at each time step $t$, a random color $x(t)$ which can be any real number (possibly outside the interval $[-1,1]$), but must satisfy $\E[x(t)^2] = 1$. 

Remarkably, one can ensure that $\E[d_t(i)^2] \le 1$ holds for each coordinate $i \in [m]$ and each time $t$.\footnote{Finding such a distribution over $x_t$ is equivalent to constructing a feasible SDP solution with vector discrepancy at most $1$, as discussed later.} Moreover, the (random) color $x(t)$ can be computed online, using the following algorithm.\\

\noindent For $t=1,\ldots,n$ do the following:
\begin{enumerate}
    \item Given the vector $v_t$ at time $t$, compute
    \(
        b_t =\langle v_t, d_{t-1} \rangle.
    \)
    \item Sample $x(t)$ from the Gaussian distribution $N(-b_t, 1-\E[b_t^2])$.\footnote{We will show later that this variance is always non-negative.}
    \item Set $d_t = d_{t-1} + x(t) v_t$.
\end{enumerate}

Note here $b_t$ is a random variable as $d_{t-1}$ is random, and the expectation $\E[b_t^2]$ in Step 2
 is over the randomness up to time $t-1$. 
 
 The choice to make $x(t)$ Gaussian is not essential, and we could have chosen any random variable with mean $-b_t$ and variance $1-\E[b_t^2]$. The bias $-b_t$ of $x(t)$ plays an important role --- it makes sure that, if $\ip{v_t, d_{t-1}}$ is large, then the discrepancy vector $d_t$ is pushed towards the origin in the direction of $v_t$. To see this more precisely, we can write $x(t)$ as $x(t) = -b_t + y(t)$ where $y(t)\sim N(0, 1-\E[b_t^2])$, independent of $d_{t-1}$. Then $d_t = d_{t-1} -b_tv_t + y(t) v_t$.

We have the following guarantee.

\begin{theorem}\label{thm:gauss-prefix-komlos}
     For any $v_1, \ldots, v_n$ provided by an oblivious adversary such that $\|v_t\|_2 \le 1$ for all $t$, the GW algorithm above is well-defined, and, at any time step $t$, satisfies $\E[x(t)^2]=1$ and $\E[d_td_t^T]\preceq I$.\footnote{Recall that the notation $A\preceq B$, for two symmetric $m\times m$ matrices $A$ and $B$, means that $B-A$ is positive semidefinite.}
\end{theorem}
\begin{proof}
We will show using induction on $t$ that $\E[d_{t}d_{t}^T]\preceq I$. 
This holds trivially for the base case $t=0$ since $d_0 = 0$. Consider some $t\geq 1$ and assume
inductively that $\E[d_{t-1}d_{t-1}^T]\preceq I$. 

First note that this implies that the algorithm is well-defined at time $t$, as  $\E[b_t^2] 
=\E[\langle v_t, d_{t-1} \rangle^2]\le \|v_t\|_2^2
\le 1$, so the variance of $x(t)$ is non-negative.
As $\E[b_t^2] \le 1$, it is also easily seen that $\E[x(t)^2] = 1$.
    

  We now show the inductive step that $\E[d_{t}d_{t}^T]\preceq I$. Let us define $d'_t := d_{t-1} -b_t v_t$, so that $d_t = d'_t + y(t)v_t$.
  A key observation is that
  \[d'_t = d_{t-1} - b_t v_t = d_{t-1} - \langle d_{t-1},v_t \rangle v_t =  (I - v_tv_t^T)d_{t-1}.\]
 Let us further define $\sigma_t^2:= \E[y(t)^2] = 1-\E[b_t^2]$. We thus have,
    \begin{align*}
    \E[d_td_t^T]
    &= \E[(d'_t + y(t) v_t)(d'_t + y(t) v_t)^T]\\
    &= \E[(d'_t) (d'_t)^T] + \sigma_t^2 v_t v_t^T\\
    &= (I-v_t v_t^T)\E[d_{t-1}d_{t-1}^T](I-v_t v_t^T) + \sigma_t^2v_t v_t^T\\
    &\preceq (I-v_t v_t^T)^2 + \sigma_t^2v_t v_t^T \\
    & = I - (2- \|v_t\|_2^2 - \sigma_t^2) v_tv_t^T  \preceq I.
    \end{align*}
    Here, the second equality uses that $y(t)$ is independent of $d'_t$ and has mean $0$, the fourth line uses the inductive hypothesis $\E[d_{t-1}d_{t-1}^T]\preceq I$, and the last inequality follows as $\|v_t\|_2^2 \le 1$ and $\sigma_t^2 \le 1$.
%
\end{proof}

\begin{remark}
    Theorem~\ref{thm:gauss-prefix-komlos} is similar to Theorem~\ref{thm:bana-prefix-subgaussian}. Indeed, since $d_t$ is jointly Gaussian, has mean $0$, and has variance at most $1$ in every direction, it is also $1$-subgaussian. However, unlike Theorem~\ref{thm:bana-prefix-subgaussian}, we allow the ``colors'' to be any real number (instead of just $\pm 1$). On the other hand, we achieve better constants, and have an online algorithm. 
\end{remark}

\smallskip

\noindent {\bf An explicit Vector Coloring.} As a corollary of Theorem~\ref{thm:gauss-prefix-komlos}, we get a vector coloring for the prefix Koml\'os problem achieving vector discrepancy $1$. This is a strengthening of Theorem~\ref{thm:komlos-vec-disc}. 

\begin{theorem}\label{thm:vec-prefix-komlos}
    For any vectors $v_1, \ldots, v_n \in \R^m$ with $\|v_j\|_2 \le 1$ for all $j \in [n]$, there exist vectors $w_1, \ldots w_n \in \R^d$ for some positive integer $d$, such that $\|w_j\|_2 = 1$ for all $j \in [n]$, and 
    \[
    \max_{j=1}^n \max_{i = 1}^m \left\|\sum_{k = 1}^j w_k v_k(i)\right\|_2 \le 1.
    \]
\end{theorem}
\begin{proof}
    Let $x := (x(1), \ldots, x(n))$ be the random colors determined by the GW algorithm, and define the positive semidefinite matrix $X$ with entries $X(i,j) = \E[x(i)x(j)]$. Theorem~\ref{thm:gauss-prefix-komlos} implies that $X(j,j) = 1$ for all $j \in [n]$, and also that $V_j X V_j^T\preceq I$ for all $j \in [n]$, where $V_j$ is the $m\times n$ matrix whose first $j$ columns are $v_1, \ldots, v_j$, and the remaining $n-j$ columns are $0$. Now the theorem is implied by taking $w_1, \ldots, w_n$ to be any vectors such that $\ip{w_i,w_j} = X(i,j)$. 
\end{proof}

\paragraph{The Self-Balancing Walk (SW) Algorithm.}
We now formulate the Self-Balancing Walk (SW) algorithm that proves Theorem~\ref{thm:als-21}. The SW algorithm is very similar to the GW algorithm, with some key differences. We need to make sure that each color $x(t)$ is in $\{-1,+1\}$. Since this forces the bias to be in $[-1,1]$, we will make the bias smaller in proportion to $\ip{v_t, d_{t-1}}$, and, moreover, the algorithm will fail when the bias is outside the interval $[-1,1]$. Below, we let $c = 2\log (4mn/\delta)$, and, as before, let $d_t$ denote the discrepancy vector at time $t$. Initially, $d_0=0$.\\

\noindent For $t=1,\ldots,n$ do the following:
\begin{enumerate}
    \item Given the vector $v_t$ at time $t$, compute \begin{equation}
    \label{eq:sw-bt}
        b_t =\langle v_t, d_{t-1} \rangle/c.
    \end{equation}
    \item If $|b_t| > 1$, abort the algorithm.
Otherwise, set
\begin{equation} 
\label{eq:sw-rule}
x(t) =  \begin{cases} -1  \text { \,\,\, w.p. }  (1+b_t)/2 \\ 
  +1 \text{ \,\,\,\, otherwise}.
\end{cases}
\end{equation}
\item Set $d_t = d_{t-1} + x(t) v_t$.
\end{enumerate}
\vspace{2mm}

To see the parallel between the GW and SW algorithms, notice that,
conditioned on $d_{t-1}$, the sign $x(t)$ satisfies $\E[x(t) | d_{t-1}]=-b_t$, and, as $b_t = \ip{v_t,d_{t-1}}/c$, the discrepancy $d_t$ satisfies
\[\E[d_t |d_{t-1}] = d_{t-1} - b_t v_t = \left(I - \frac{1}{c}v_t v_t^T\right) d_{t-1}.\]

\section{Analysis}
We now prove Theorem \ref{thm:als-21}. The proof follows the proof of Theorem~\ref{thm:gauss-prefix-komlos}, but is more involved. Instead of just bounding the second moments of the discrepancy vector $d_t$, we show that $d_t$ is subgaussian. Moreover, we need to account for the possibility that the algorithm aborts. To deal with this second issue, it will be convenient to consider a modified algorithm that never aborts but may sometimes output a non-sense ``sign'' $x(t)$ with $|x(t)|>1$. Consider the following algorithm. 
\paragraph{Modified Walk (MW).} At each time $t$, compute $b_t = \ip{d_{t-1},v_t}/c$ exactly as before. 

\vspace{2mm}

(i) If $|b_t|> 1$, set $x(t) =-b_t$ deterministically.
\vspace{2mm}

(ii) Else if $|b_t|\leq 1$, then set $x(t)$ randomly exactly as in \eqref{eq:sw-rule}.\\

Notice that MW never aborts, but may set $|x(t)|>1$ in step (i).

\paragraph{Coupling SW and MW.} Notice that if we use the same random choices in SW in \eqref{eq:sw-rule} and MW in step (ii), then they both behave identically until the first (random) time $\tau$ when $|b_{\tau}|>1$. In particular, at this time SW will abort, but MW will output $x_\tau$ with $|x_\tau|>1$. 

For this reason it will suffice to analyze the MW algorithm. 

\paragraph{Sub-gaussianity of $d_t$.}
Recall 
from Section \ref{sec:bana-subgaussian} that 
a mean-zero vector-valued random variable  $X \in \R^m$  is $\sigma$-sub-Gaussian if
\begin{equation}
    \label{eq:als-subg-general}
\E[\exp( \langle X,\theta \rangle )] \leq \exp (\sigma^2 \|\theta\|^2/2) \quad \text{for all $\theta \in \R^m$},
\end{equation}
and that such an $X$ satisfies
\begin{equation}
\label{eq:als-tail}
    \Pr[ |\ip{X,\theta} | \geq x] \leq 2 \exp(-x^2/(2\sigma^2\|\theta\|^2)) \quad \text{for all $\theta \in \R^m$}.
\end{equation}
MW satisfies the  following remarkable guarantee.
\begin{theorem}
\label{thm:als-subgaussian}
    At each time $t$, the discrepancy vector $d_t$ produced by MW is
    $\sqrt{c}$-sub-Gaussian. 
\end{theorem}
Interestingly, Theorem \ref{thm:als-subgaussian} directly implies Theorem \ref{thm:als-21}---by simply setting $c = 2 \ln (4mn/\delta)$ to ensure that $\tau > n$ with high probability.

\begin{proof}[Proof of Theorem \ref{thm:als-21}] 
As the adversary is oblivious, $v_1,\ldots,v_n$ are fixed in advance (and do not depend upon the randomness of the algorithm). 

Fix a time $t$.
By Theorem \ref{thm:als-subgaussian}
the vector $d_t$ produced by MW is $\sqrt{c}$-sub-Gaussian. So, applying \eqref{eq:als-tail} with $X=d_{t-1}$ and $\theta = v_t$, and using that $\|v_t\|_2\leq 1$ and $\sigma^2 =c$, 
\[\Pr[|b_t|>1] = \Pr[|\ip{d_{t-1},v_t}| \geq c] \leq 2 e^{-c/2} = \frac{\delta}{2mn}. \] 
By a union bound over the $n$ time steps, the probability that SW aborts at some $t\leq n$ is at most $\frac{\delta}{2m} \leq \frac{\delta}{2}$.

Similarly, for a fixed time $t$ and standard coordinate vector $e_i$, applying  \eqref{eq:als-tail} with $X=d_t, \theta = e_i$,  $x=c$, and $\sigma^2=c$ gives  \[\Pr[|\ip{ d_t,e_i}|> c] \leq 2 e^{-c/2} \leq \frac{\delta}{4mn}.\] Thus, by a union bound over all $t \in [n],i \in [m]$, the probability that $\|d_t\|_\infty >c$ for any time $t$ is at most $\delta/2$.
\end{proof}

\remark
We can also derive the following more general theorem, using an
argument analogous to the proof of
Theorem~\ref{thm:bana-subg-suff}. While an even more general statement
is possible, we state the theorem for a symmetric convex body $K$ in $\R^m$ with non-empty interior. We use
$\|y\|_K \eqdef \inf\{t: y \in tK\}$ for the norm with unit ball
$K$. We use $B_{\ell_2^m}$ for the unit ball of the $\ell_2^m$
norm.

\begin{theorem}\label{thm:als-genK}
  Let $v_1, v_2, \ldots, v_n \in \R^m$ satisfy $\|v_t\|_2 \leq 1$ for
  all $t \in [n]$, and suppose that $K$ is a symmetric convex  body in
  $\R^m$ such that $rB_{\ell_2^m} \subseteq K$. Then, for any $\delta
  \in (0,1)$, and for an appropriate choice of $c$, with probability at
  least $1-\delta$ the SW algorithm satisfies
  \[
    \|d_t\|_K \lesssim \sqrt{\log(n/\delta)}(\E\|\grv\|_K + r\sqrt{\log(n/\delta)}),
  \]
  where  $\grv$ is a standard Gaussian random vector in $\R^m$.
\end{theorem}

The proof of Theorem~\ref{thm:als-genK} is similar to the proof of
Theorem~\ref{thm:als-21} above. We pick $c$ on the order of
$\log(n/\delta)$, which is enough to argue that the probability that
SW aborts is at most $\frac{\delta}{2}$. Then we estimate $\max_{t =
  1}^n \|d_t\|_K$ using the fact that $d_t$ is $\sqrt{c}$-sub-Gaussian,
and the high probability version of Talagrand's comparison inequality
(see, \eg, Exercise~8.6.5 in~\cite{vershynin}). We leave the details
as an exercise.

Theorem~\ref{thm:als-genK} can be used to give an efficient (online!)
algorithm that nearly matches other prefix discrepancy bounds we have
already seen, in particular those derived from Theorem~\ref{thm:bana2}. In particular, the bounds on prefix discrepancy in the
$\ell_p$ norm in Theorem~\ref{thm:sser-ellp} can be matched up to an
additional $O(\sqrt{\log n})$ multiplicative factor.

\subsection{Proof of Theorem \ref{thm:als-subgaussian}}
We now prove Theorem \ref{thm:als-subgaussian}. 

Fix some time $t$.
Conditioned on $d_{t-1}$, as the ``sign'' $x(t)$ satisfies
$\E[x(t)| d_{t-1}]=-b_t$, let us write $y(t) \eqdef x(t)  + b_t$, 
so that  $\E[y(t)|d_{t-1}]=0$. 
In particular
 if $|b_t|>1$ then $y(t)=0$, and otherwise
\begin{equation}
\label{eq:als-rt}
    y(t)  = \begin{cases} -(1 - b _t) \text{ w.p. }  (1+b_t)/2\\ 
             +(1 +b _t) \text{ w.p. }  (1-b_t)/2
             \end{cases}
\end{equation}
As $b_t = \langle v_t,d_{t-1}\rangle/c$, we can write  
\[ d_t = d_{t-1} + x(t) v_t =   d_{t-1} - b_t v_t + y(t) v_t  =  \left(I - \frac{1}{c}v_tv_t^T\right) d_{t-1} + y(t)v_t.\]
Let us denote $A_t : = (I - v_t v_t^T/c)$, so that $d_t = A_t d_{t-1} + y(t) v_t$.
Note that $A_t$ only depends on $v_t$, and hence it does not depend on the randomness of the algorithm as the adversary is oblivious.

\paragraph{The inductive argument.}
We now prove the sub-Gaussianity of $d_t$ by induction on $t$.
Fix some test vector $\theta \in \R^m$. We will show that \[\E [\exp(\ip{d_t,\theta})] \leq \exp(c \|\theta\|_2^2/2).\]
This clearly holds for $t=0$ as $\ip{d_t,\theta}=0$.
Assume inductively that 
\begin{equation}
    \label{eq:als-induction}
    \E[\exp(\ip{d_{t-1},u})]\leq \exp(c\|u\|_2^2/2) \quad \text{ for all } u \in \R^m.
\end{equation}
As $d_t = A_t d_{t-1} +y(t) v_t $, we have 
\begin{align}
\E [\exp(\ip{d_t,\theta }) |d_{t-1}] & =   \E [\exp(\ip{A_t d_{t-1}+y(t) v_t,\theta}) |d_{t-1}] \notag  \\
 &=   \exp( \ip{A_t d_{t-1}, \theta }) \, {\E [\exp(y(t) \ip{ v_t,\theta }) | d_{t-1}]} \notag \\
& \leq \exp( \ip{A_t d_{t-1}, \theta })   \exp(\ip{v_t,\theta}^2/2).\label{eq:als} 
 \end{align}
 The last step follows from Hoeffding's lemma,
applied to $y(t)$ conditioned on $d_{t-1}$, and noting that
$\E[y(t)|d_{t-1}]=0$, and that $y(t)$ takes values in an interval of
length $2$. 

Taking expectation over $d_{t-1}$ in \eqref{eq:als}  gives
\begin{align}
    \E [\exp(\ip{d_t,\theta })] 
                      & \leq  \exp(\ip{v_t,\theta }^2/2)\, \E [\exp( \ip{A_t d_{t-1}, \theta })] \notag \\
                      & =  \exp(\ip{v_t,\theta }^2/2) \,\E [\exp( \ip{ d_{t-1}, A_t \theta })]  \notag \\
                      &  \leq \exp(\ip{v_t,\theta }^2/2) \,\exp( c \|A_t \theta \|_2^2 /2\label{eq:als-final} ) 
\end{align}
where the second line uses that $A_t$ is self-adjoint and the last step uses the inductive hypothesis \eqref{eq:als-induction} for $d_{t-1}$ with $u=A_t \theta $.

Now, as $A_t \theta = \theta - \frac{\ip{v_t,\theta}}{c} v_t$, 
\begin{align*}
    \|A_t \theta \|_2^2  & =  \left\| \theta -  \frac{ \ip{v_t,\theta }}{c} v_t \right\|_2^2 \\
    & =  \|\theta\|_2^2 - \frac{2   \ip{v_t,\theta }^2}{c} +\frac{\ip{v_t,\theta }^2 \|v_t\|_2^2}{c^2} \\
    & \leq  \|\theta\|_2^2 -   \frac{\ip{v_t,\theta }^2}{c}.  \tag{as $c \geq 1$, $\|v_t\|^2 \leq 1$}
\end{align*}
Plugging this in \eqref{eq:als-final} gives the desired bound 
\[
 \E [\exp(\ip{d_t,\theta })]  
 \leq \exp ( c \|\theta \|_2^2/2).\]




\section{Bibliographic Notes} 
 Online discrepancy was first studied by Spencer \cite{Spencer77}.
Besides the online Koml\'{o}s problem we consider here, several other discrepancy problems have also been studied in the online setting. 
B\'ar\'any \cite{Barany79} considered the general setting where the vectors $v_t$ are chosen from some fixed set $S$, and showed almost matching upper and lower  bounds on the online discrepancy achievable by any deterministic algorithm.

Spencer \cite{Spencer86Prefix} considered the setting where the vectors $v_t$ satisfy $\|v_t\|_\infty \leq 1$
and showed that for $n=m$ any online algorithm must incur discrepancy $\Omega(\sqrt{n \log n})$. In contrast, the offline  prefix discrepancy is $O(\sqrt{n})$.
Bansal and Spencer \cite{BansalS20} considered the setting where the vectors $v_t$ are chosen uniformly and independently from $\{-1,1\}^m$ and gave an online strategy that at any time $t$ ensures that $\|d_t\|_\infty = O(\sqrt{m})$ with high probability.


Another very interesting question is the online carpooling problem \cite{AjtaiANRSW98}. Here there are $n$ nodes
and edges $e_t=(u_t,v_t)$ arrive over time that must be oriented from
$u_t\rightarrow v_t$ or $v_t\rightarrow u_t$. The goal is to minimize
the discrepancy at each vertex, equal to the absolute value of the
difference between in-degree and out-degree.
If the edges arrive randomly from the complete graph the discrepancy is $O(\log \log n)$ at any fixed time \cite{AjtaiANRSW98}.
For arbitrary graphs and oblivious adversaries Theorem \ref{thm:als-21} implies an $O(\log (nT))$ discrepancy, where $T$ is the time horizon. 
A prominent open question is whether a $\text{polylog}(n)$ bound (independent of $T$) exists. To this end,
\cite{AjtaiANRSW98} gave a
randomized algorithm with $O(\sqrt{n \log n})$ discrepancy at any fixed time, and also showed a lower bound of $\Omega(\log^{1/3} n)$ for any randomized algorithm.

For the online Koml\'{o}s problem, the Self-Balancing Walk algorithm in Section \ref{sec:sw-alg} is due to Alweiss, Liu and Sawhney  \cite{alweiss2020discrepancy}. Their original proof
follows a different approach and is based on showing that the distribution of the discrepancy vector is stochastically dominated by a suitably scaled Gaussian distribution. The exposition here gives better constants and is simpler, and is based on a personal communication by Arun Jambulapati. The weaker Gaussian Walk algorithm and the corresponding vector coloring in Theorem~\ref{thm:vec-prefix-komlos} are unpublished results of Nikolov.

    Very recently, Kulkarni, Reis, and  Rothvoss \cite{rothvoss-online}
    showed the {\em existence} of an online algorithm with
    $O(\sqrt{\log mn})$ discrepancy. Moreover, this algorithm matches
    the guarantees of Theorem~\ref{thm:bana-prefix-subgaussian} while being online.
       The bound     $O(\sqrt{\log mn})$ is also the best achievable
       by any online algorithm for the Koml\'os problem, in contrast to offline prefix Koml\'{o}s where $O(1)$ discrepancy may be possible.
    The algorithm of Kulkarni, Reis, and Rothvoss is highly non-constructive\footnote{Roughly speaking, given a time horizon $n$, it computes a distribution over exponentially many decision trees, where each decision tree specifies a deterministic online coloring strategy for every possible sequence of incoming vectors.}, 
    but it strongly suggests that a simpler and more efficient algorithm might exist. 
    Finding an efficient algorithm with $O(\sqrt{\log mn})$
    discrepancy for the online Koml\'{o}s problem is an outstanding open problem. 
    
    As further evidence, Liu, Sah and Sawhney \cite{LiuSS22} gave an efficient  Gaussian Fixed Point Walk algorithm that gives a partial $\{-1,0,1\}$ coloring with discrepancy $O(\sqrt{\log mn})$ where the color $0$ is used at most $4\%$ of the time.
    In fact, their algorithm ensures that the discrepancy vector is distributed exactly 
   as the standard Gaussian $N(0,I_m)$.  
    In another remarkable result, Chewi, Gerber, Rigollet and Turner \cite{ChewiGRT22} showed  that if one relaxes the colors to be two-dimensional unit vectors, then there is an efficient algorithm that ensures that the discrepancy is exactly distributed as the standard Gaussian.

\bibliographystyle{plain}
\bibliography{discrepancy}

\end{document}